\documentclass[12pt]{amsart}
\usepackage{amsmath,amssymb,amsthm}
\usepackage{latexsym}
\usepackage{graphicx}
\usepackage{enumerate}

\title[]{Long-Time Asymptotics for Solutions of the NLS Equation with a Delta Potential and Even Initial Data}
\author{Percy Deift and Jungwoon Park}
\address{Courant Institute of Mathematical Sciences, New York University, New York, NY, USA.}

\newtheorem{theorem}{Theorem}
\newtheorem{proposition}{Proposition}[section]
\newtheorem{definition}[proposition]{Definition}
\newtheorem{lemma}[proposition]{Lemma}
\newtheorem{corollary}[proposition]{Corollary}

\theoremstyle{remark}
\newtheorem{remark}[proposition]{Remark}
\newtheorem{notation}[proposition]{Notation}

\numberwithin{equation}{section}

\def\rb{\mathbb{R}}
\def\cb{\mathbb{C}}

\def\bal{\begin{aligned}}
\def\eal{\end{aligned}}
\def\beq{\begin{equation}}
\def\eeq{\end{equation}}
\def\bpm{\begin{pmatrix}}
\def\epm{\end{pmatrix}}

\def\bsm{\bigl(\begin{smallmatrix}}
\def\esm{\end{smallmatrix}\bigr)}

\DeclareMathOperator{\Res}{Res}
\DeclareMathOperator{\sech}{sech}

\def\dsize{\text{}}
\def\ovl{\overline}
\def\inv#1{{#1}^{-1}}
\def\rd{\mathrm{d}}

\def\mc#1{\mathcal{#1}}
\def\mf#1{\mathfrak{#1}}

\def\first#1{\widetilde{{#1}}_\delta}
\def\firste#1{\widetilde{{#1}}_\delta}
\def\second#1{\widetilde{{#1}}_\delta^{[\cdot]}}
\def\seconde#1{\widetilde{{#1}}_\delta^{[\cdot]}}

\begin{document}
\maketitle

\begin{abstract}
We consider the one-dimensional focusing nonlinear Schr\"odinger equation (NLS) with a delta potential and even initial data.
The problem is equivalent to the solution of the initial/boundary problem for NLS on a half-line with Robin boundary conditions at the origin.
We follow the method of Bikbaev and Tarasov which utilizes a B\"acklund transformation
to extend the solution on the half-line to a solution of the NLS equation on the whole line.
We study the asymptotic stability of the stationary 1-soliton solution of the equation under perturbation
by applying the nonlinear steepest-descent method for Riemann-Hilbert problems introduced by Deift and Zhou.
Our work strengthens, and extends, earlier work on the problem by Holmer and Zworski.
\end{abstract}

\section{Introduction}

	The nonlinear Schr\"{o}dinger (NLS) equation with an external potential $V(x)$
(Gross-Pitaevskii (GP) equation)
\beq
\label{eq:gpe}
i u_t + \frac 12 \Delta u + V(x) u + \kappa |u|^2 u = 0 \ ,
\ \ \ \ \ \kappa = \pm 1
\eeq
arises as a model for a wide variety of phenomena in physics. In particular (see e.g. \cite{PS})
equation \eqref{eq:gpe} provides a model for the evolution of Bose-Einstein condensates in dilute
boson gases at very low temperatures: in this case $V(x)$ is called the
trapping potential and $\kappa = +1$ or $-1$ depending on whether the
interaction of the bosons is attractive or repulsive, respectively.
Furthermore, in this Bose-Einstein model, $|u(x)|^2$ is the boson density and $\mc M = \int |u(x)|^2 dx$ is the
total number of bosons present in the system.  In the spatially inhomogeneous case
($V(x)\not \equiv c$), exact solutions of \eqref{eq:gpe} are hard to come by, even
in one dimension. Further reduction in the one-dimensional case
to the delta-potential at $x = 0$, $V(x) = q\delta_0 (x)$, still leaves a
formidable and currently much studied problem (see e.g. \cite{HZ} and the references
therein). However, as noted by Fokas, in one-dimension with $V(x) = q\delta_0 (x)$,  $q\in\rb$,
and $u_0 (x) = u(x,t=0)$ {\it even}, equation \eqref{eq:gpe} becomes integrable.
This is because the delta function introduces a jump in the derivative at
$x = 0$, $\frac 12 (u_x (0+) - u_x (0-)) + q u(0) = 0$, and if $u(x)$ is
even this relation reduces to
\beq
\label{eq:mixed_bdry}
u_x (0+ ) +q  u(0) = 0
\eeq
In other words, \eqref{eq:gpe} with even initial data reduces to the initial-boundary value (IBV)
problem for NLS on a half-line
\begin{equation}
\label{eq:nls_plus}
i u_t + \frac{1}{2} u_{xx} +|u|^2 u= 0 , \ \ x>0,
\end{equation}
with homogeneous boundary conditions \eqref{eq:mixed_bdry} at $x = 0$: such problems are known to be
integrable by an extension of the inverse scattering method
(see \cite{Sk}, \cite{Ta}, \cite{Fo}, \cite{Kh}, \cite{BT} and \cite{Ta2}).

	In \cite{HZ} the authors consider the GP equation in one-dimension with $V = q \delta_0$
and $\kappa =1$,
\beq
\label{eq:nls_delta}
iu_t + \frac 12 u_{xx} + q \delta_0 (x) u + |u|^2 u = 0
\eeq
with initial data of the form
\beq
\label{eq:nls_delta_initial}
u(x,0) = v_{\lambda} (x) + w (x)
\eeq
where $q$ is small, $w(x)$ is even and of order $O(q)$, and $v_{\lambda}$ has the special form
\beq
\label{E:gs}
v_{\lambda} (x) = \lambda {\rm sech} (\lambda |x| + {\rm tanh}^{-1}
(q/\lambda )), \ \ \lambda > |q|.
\eeq
The data $v_{\lambda}$ corresponds to the nonlinear ground state of the condensate obtained by
minimizing the energy $\mc H = \int_\rb [\frac 12 |u_x |^2 - q \delta_0 |u|^2
- \frac 12 |u|^4 ] dx$ subject to $\mc M = \int_\rb |u(x)|^2 dx = 2(\lambda -q)$.
Associated with $v_{\lambda}$, one has the stationary solution
$u_{\lambda} (x,t) = e^{i\lambda^2 t /2} v_{\lambda} (x)$ for \eqref{eq:nls_delta}
corresponding to $w \equiv 0$. The main result in \cite{HZ} concerns the asymptotic
stability of this ground state condensate under perturbations,
$w = O(q)$, $q \ll 1$. The authors in \cite{HZ} prove in particular that for
$1 \leq t \leq c|q|^{-2/7}$,
\beq
\label{eq:hz_result}
u(0,t) = e^{i\hat{\lambda}^2 t/2}
  \left( \hat{\lambda} - \sqrt{\frac 2{\pi t}}
    e^{i(\hat{\lambda}^2 t/2 + \pi /4)}
     \int_{0}^\infty w(x)dx \right) + O(q /t^{3/2} )
\eeq
for some explicit $\hat{\lambda}$, $\hat{\lambda} \sim v_{\lambda} (0)$
for $q$ small. The authors use PDE-Hamiltonian systems methods and do not
utilize the integrability of the system explicitly.
In this paper we analyze the above problem \eqref{eq:nls_delta}\eqref{eq:nls_delta_initial},
but now utilizing the full force of the integrability of the IBV
problem together with the steepest-descent method for Riemann-Hilbert Problems (RHPs) in \cite{DZ3}\cite{DZ}.
As indicated above, there are a number of different ways to show that the IBV problem
\eqref{eq:nls_plus}\eqref{eq:mixed_bdry} is completely integrable. We will follow the method of
Bikbaev and Tarasov \cite{Ta}\cite{BT}\cite{Ta2}, which is based in turn on Khabibullin \cite{Kh}.
Part of this paper is devoted to reformulating the method of
Bikbaev and Tarasov in the language of RHP's, so that the method in \cite{DZ3}\cite{DZ} can be applied.

Our main results  are as follows. We use the standard notations of the scattering and inverse scattering method
for the focusing NLS equation. We refer the reader, who may not be familiar with these notations,
for example, $a(z)$, B\"acklund extension, etc., to Sections 3 and 4 below.
We note that for $u(x,0)=v_{\lambda}(x)+\epsilon w(x)$, $w(x)\in H^{1,1}=\{w\in L^2(\rb): w$ is absolutely continuous, $xw(x), w'\in L^2(\rb)\}$,
the equation \eqref{eq:nls_delta} has a unique global (weak) solution in $H^{1,1}$,
i.e. $t\mapsto u(t)=u(x,t)$ is a continuous map from $\rb^+$ into $H^{1,1}(\rb)$ satisfying
\beq
\label{eq:weak_sol_H_q}
\bal
   & u(t) = e^{-iH_q  t/2} u_0 + i \int_0^t e^{-iH_q  (t-s)/2} |u(s)|^2 u(s) \rd s, \\
   & u(t=0) = u_0 =u_0(x) \in H^{1,1}(\rb). \\
\eal
\eeq
Here $H_q $ is the self-adjoint operator $-\frac {d^2}{dx^2} - 2q  \delta_0$ on $\rb$ with domain
$$ \bal D(H_q )=\{f\in L^2: \
& f \text{ is absolutely continuous,}  \\
& f' \text{ is absolutely continuous in } \rb\setminus \{0\}, \\
& f', f'' \in L^2, f'(0+) - f'(0-) + 2q  f(0)=0  \}.
\eal $$
We will also consider global weak solutions $u(t)$ of \eqref{eq:nls_plus} with boundary conditions \eqref{eq:mixed_bdry}
in the sense that $u(t)$ is a continuous map from $\rb^+$ into $H^{1,1}(\rb^+) = \{f\in L^2(\rb^+): f$ is absolutely continuous
$, f', xf \in L^2(\rb^+)\}$ which solves
\beq
\label{eq:weak_sol_H_q_plus}
\bal
   & u(t) = e^{-iH_q ^+ t/2} u_0 + i \int_0^t e^{-iH_q ^+ (t-s)/2} |u(s)|^2 u(s) \rd s, \\
   & u(t=0) = u_0 =u_0(x) \in H^{1,1}(\rb^+). \\
\eal
\eeq
Here $H_q ^+$ is the self-adjoint operator $-\frac {d^2}{dx^2}$ on $\rb^+$ with domain
$$\bal  D(H_q ^+)=\{f\in L^2(\rb^+):
& f \text{ and }f'\text{ are absolutely continuous}, \\
& f'' \in L^2(\rb^+), f'(0+) + q f(0) =0  \}.\\
\eal$$

\begin{definition}
We say that $u(t)$ solves a \emph{HNLS$_q^+$} if $u(t)$ is a (global, weak) solution to  \eqref{eq:nls_plus} with \eqref{eq:mixed_bdry}.
If $u(t)$ solves \eqref{eq:nls_plus} on $\rb^-$ with \eqref{eq:mixed_bdry}, we say that $u(t)$ solves a \emph{HNLS$_q^-$}.
\end{definition}

Also we will consider global weak solutions $u(t)$ of NLS \eqref{eq:focusing_nls} on the line,
by which we mean a continuous map from $\rb^+$ into $H^{1,1}(\rb)$ such that
\beq
\label{eq:weak_sol_H_0}
\bal
   & u(t) = e^{-iH_0 t/2} u_0 + i \int_0^t e^{-iH_0 (t-s)/2} |u(s)|^2 u(s) \rd s, \\
   & u(t=0) = u_0 \in H^{1,1}(\rb). \\
\eal
\eeq
Here $H_0$ is the self-adjoint operator $-\frac {d^2}{dx^2}$ on $\rb$ with domain
$$
D(H_0)=\{f\in L^2: f \text{ and }f'\text{ are absolutely continuous}, f', f'' \in L^2\}.
$$
Unless stated otherwise, whenever we discuss a solution of \eqref{eq:nls_delta} in $H^{1,1}(\rb)$,
\eqref{eq:nls_plus} with boundary conditions \eqref{eq:mixed_bdry} in $H^{1,1}(\rb^+)$, or \eqref{eq:focusing_nls} in $H^{1,1}(\rb)$,
we always mean the global weak solutions described above.
On a number of occasions, however, particularly in our motivations for the Bikbaev and Tarasov's method,
we will also consider classical solutions to HNLS$_q^+$, etc.,
i.e., the solutions which are $C^2$ with respect to $x$ and $C^1$ with respect to $t$.
In addition, on a few occasions, we will also need to consider solutions of NLS in $H^{k,k}(\rb)$, $k\geq1$,
where $H^{k,j} = \{u \in L^2(\rb): u, u', \cdots, u^{(k-1)}$ are absolutely continuous, $u^{(k)}(x), x^ju(x) \in L^2(\rb)\}$, $k\geq0$, $j\geq0$.
We will show that \eqref{eq:weak_sol_H_q}, \eqref{eq:weak_sol_H_q_plus} and \eqref{eq:weak_sol_H_0}
indeed have unique solutions in $H^{1,1}$, in Section \ref{sec:sol_eqs} below.
Moreover, the solutions $u(t)=u(t;u_0)$ of these equations depend on the initial data $u_0$, uniformly for $t$ in compact subsets of $\{t\geq0\}$.

\begin{theorem}[Asymptotics of $u(x,t)$ as $t\to \infty$]
\label{T:main}
Let $0< |q|<\mu_0 $. Suppose that $u(x,t)$ is the (unique, weak, global) solution of the equation \eqref{eq:nls_delta} with initial data
$$ u( x,0) = v_{\mu_0} (x) + \epsilon w(x),$$
where $w$ is even and $ \|w\|_{H^{1,1}( \mathbb{R} )}\leq c$.
Let $u_0^e(x)$ be the B\"acklund extension  of $ u(x,0) |_ {\rb^+}$ to $\rb$ with respect to $q$.
Let $a(z)$ and $r(z)$ be the scattering function and the reflection coefficient of $u_0^e(x)$, respectively.
Denote $z_0 = |x|/t$. Fix $ 0 < \kappa < \frac{1}{4} $ and $M>1$. Then, there exists an $\chi_0 =\chi_0(\mu_0) > 0$ such that the following holds:
For any $0\leq \epsilon \leq \chi_0 |q|^{\frac 12}$, $a(z)$ has one simple zero if $q>0$ and at most two simple zeros for $q<0$.
In both cases the zeros lie in $i\rb^+$. We denote the zeros by $z_1 = i\mu_1$, $z_2 = i\mu_2$.
Set
\beq
\label{eq:l_j}
l_j = \frac{1}{\pi} \int_{0}^{\infty} \frac{s }{s^2 + \mu_j^2} \log ( 1+|r(s)|^2 ) \rd s,
\eeq
and
\beq
\label{eq:delta_j}
\hat{\delta}_j = \exp \bigg[ \frac 1{2\pi i} \int_0^{z_0}  \frac{\log ( 1+|r(s)|^2 )}{s-i\mu_j}  \rd s\bigg].
\eeq
\beq
\label{eq:upsilon_j}
\bal
& \upsilon_j = -\sqrt{\frac{\mu_j-q}{\mu_j+q}} e^{i(\mu_j^2 t/2 + \rho_j + l_j)} \hat{\delta}_j^{2},
\ \ \hat \upsilon_j = \frac{\mu_j+q}{\mu_j-q} \upsilon_j, \\
& \upsilon_j^x = \upsilon_j e^{-\mu_j |x|}, \ \  \hat \upsilon_j^x = \hat \upsilon_j e^{-\mu_j |x|},
\eal
\eeq
for $j=1,2$, where $\rho_1$ is described in \eqref{eq:gamma_one_zero} if $a(z)$ has one zero,
and $\rho_1, \rho_2$ are given in \eqref{eq:gamma_two_zero} if $a(z)$ has two zeros.
Let $k_1$, $k_2$, $p_1$, $p_2$, $p_3$ and $p_4$ be given in \eqref{eq:k_1}, \eqref{eq:k_2}, \eqref{eq:p_1_p_2} and \eqref{eq:p_3_p_4}.

\vspace{2 mm}

\noindent\textnormal{(i)} Suppose that $a(z)$ has one simple zero at $z_1=i\mu_1\in i\rb^+$ as given in \eqref{eq:perturbed_mu},

\noindent$\bullet$ For $|x|\leq M$ and $t\geq 1$\textnormal{:}
\beq
\label{eq:sol_small_x_large_t_one_zero}
 u(x,t) = e^{i( \mu_1^2 t/2 + \rho_1 + l_1)} v_{\mu_1}(x) + O(\epsilon |q|^{-\frac 12} t^{-({\frac 12}+\kappa)}).
\eeq

\noindent$\bullet$ For $|x|\geq 1/M$ and $t\geq 1$\textnormal{:}
\beq
\label{eq:sol_large_x_large_t_one_zero}
\bal  u(x,t)
& = -\frac{k_1}{\sqrt t} - \frac {2\mu_1 (\upsilon_1^x + p_1 )(\ovl{p_2} \ovl{\upsilon_1^x} +1)}
{|\upsilon_1^x + p_1|^2 + |p_2\upsilon_1^x +1|^2}+O( \epsilon |q|^{-\frac 12} t^{-({\frac 12}+\kappa)}),\\
\eal
\eeq

\noindent\textnormal{(ii)} If $a(z)$ has two simple zeros at $z_1=i\mu_1$, $z_2=i\mu_2$
\textnormal{(}$\mu_1 > \mu_2$\textnormal{)}, set $\tau \equiv |q|\sqrt t$.

\noindent$\bullet$ For $|x|\leq M$, as $\tau \to \infty$\textnormal{:}
\beq
\label{eq:sol_small_x_large_t_two_zeros}
 \bal u(x,t)
& = e^{i(\mu_1^2 t/2 + \rho_1 + l_1)} \mu_1 \sech (\mu_1 x - \tanh^{-1}(q/\mu_1)) \\
& \quad - \frac {2\mu_2 (\hat\upsilon_2^x - s_0)(1+ \hat \upsilon_1^x \ovl{s_0})}
    {|\hat\upsilon_2^x - s_0|^2  + |1+ \hat \upsilon_1^x \ovl{s_0}|^2}
     + O(\epsilon |q|^{-\frac 12} t^{-(\frac 12 + \kappa)}),\\
\eal
\eeq
where
\beq
\label{eq:s_0}
s_0 = -\frac{2\mu_1}{\mu_1 - \mu_2} \frac{\hat \upsilon_1^x- \hat\upsilon_2^x}
{|\hat \upsilon_1^x|^2+1}.
\eeq

\noindent$\bullet$ For $|x| \geq 1/M$, as $\tau \to \infty$\textnormal{:}
\beq
\label{eq:sol_medium_x_large_t_two_zeros}
 \bal
& u(x,t) = -\frac{k_1}{\sqrt t} - \frac {2\mu_1 (\hat\upsilon_1^x + p_1 )(\ovl{p_2} \ovl{\hat\upsilon_1^x} +1)}
{|\hat\upsilon_1^x + p_1|^2 + |p_2\hat\upsilon_1^x +1|^2} -  \frac { 2\mu_2 s_1 \ovl{s_2}} {|s_1|^2 + |s_2|^2}\\
& \qquad\qquad  + O(\epsilon qe^{-\tau^2} + \epsilon^2 q(z_0^2 + q^2)^{-1} t^{-1} + \epsilon |q|^{-\frac 12} t^{-(\frac 12 + \kappa)}). \\
 \eal \eeq
where
\beq
\label{eq:s}
s = \frac{2\mu_1}{\mu_1 - \mu_2} \frac{(p_2 \hat \upsilon_1^x +1)(\hat \upsilon_2^x+p_3)-(\hat \upsilon_1^x + p_1)(p_4 \hat \upsilon_2^x + 1)}
{|\hat \upsilon_1^x + p_1|^2+|p_2 \hat \upsilon_1^x +1|^2},
\eeq
and
\beq
\label{eq:s_j}
\bal
& s_1 = \hat \upsilon_2^x+p_3 - (\ovl{p_2} \ovl{\hat \upsilon_1^x} +1) s, \\
&s_2 = p_4 \hat \upsilon_2^x + 1 + (\ovl{\hat \upsilon_1^x} + \ovl{p_1})s. \\
\eal
\eeq

\end{theorem}

\begin{remark}
This result should be compared with \cite{FIS}
where the authors obtain the long-time asymptotics for the solution of NLS on half-line
using the method introduced by Fokas \cite{Fo2}.
\end{remark}

\begin{remark}
For $q<0$, case (ii) is generic in the sense that
there exists an open dense subset $\mathcal{U}\subset H^{1,1}(\rb^+) $ such that
any $w \in \mathcal{U}$ falls into case (ii) for all sufficiently small $\epsilon > 0$.
\end{remark}
\begin{remark}
If $\epsilon=0$, we have $\mu_1 = \mu_0$, $\rho_1 = 0$, $r(z) \equiv 0$ and hence $l_1=0, \hat{\delta}_1=1$. This matches the exact solution $u(x,t) = e^{i \mu_0^2 t/2} v_{\mu_0}(x)$ for $x\in\rb$.
\end{remark}
\begin{remark}
If the initial data has sufficient smoothness and decay,
the error term can be sharpened to order $O\big(\epsilon |q|^{-\frac 12} \frac{\log t}{t}\big)$
in \eqref{eq:sol_small_x_large_t_one_zero}, \eqref{eq:sol_large_x_large_t_one_zero}
and \eqref{eq:sol_small_x_large_t_two_zeros}.
\end{remark}
\begin{remark}
As we will see in the proof of the Theorem, \eqref{eq:sol_large_x_large_t_one_zero} and \eqref{eq:sol_medium_x_large_t_two_zeros} hold for all $x\in\rb$.
In the overlapping region $1/M \leq |x| \leq M$, as $z_0 = O(t^{-1})$,
$$
\hat{\delta_1} = 1+ O(\epsilon^2 t^{-1}) \text{ and } k_1 = O(r(z_0)) = O(\epsilon |q|^{-\frac 12} t^{-\frac 12}),
$$
and hence $p_1, p_2 = O(\epsilon |q|^{-\frac 12} t^{-1})$.
This implies that \eqref{eq:sol_small_x_large_t_one_zero} and \eqref{eq:sol_large_x_large_t_one_zero} are consistent.
In the case that $a(z)$ has two zeros, it is easily verified that
\eqref{eq:sol_small_x_large_t_two_zeros} and \eqref{eq:sol_medium_x_large_t_two_zeros} match as well.
\end{remark}

\begin{remark}
It can be verified that the leading terms in \eqref{eq:sol_small_x_large_t_two_zeros} represent the 2-soliton solution for NLS (cf. \cite{FT}).
\end{remark}

When $\epsilon=q$, the following result recovers, in particular, \eqref{eq:hz_result} and provides a more detailed description of the asymptotics of $u(x,t)$ for all $t\ll q^{-2}$.

\begin{theorem}[Asymptotics of $u(x,t)$ for $t\ll q^{-2}$, $q$ small]
\label{T:main2}
Suppose $u(x, t)$ solves \eqref{eq:nls_delta} with initial data
$$ u( x,0) = v_{\mu_0} (x) + q w(x),$$
where $ w(x)$  is real, even and $ \|w\|_{H^{1,1}( \mathbb{R} )}\leq c$.
Let $u_0^e(x)$ be the B\"acklund extension  of $ u(x,0) |_ {\rb^+}$ to $\rb$ with respect to $q$
and let $a(z)$ be the scattering function of $u_0^e(x)$. Define
\beq\bal
\label{eq:K_z}
& w_0=\int_0^\infty w(s) \rd s,\ \ w_1 = \int_{-\infty}^\infty w(s)v_{\mu_0} (s) \rd s,\\
& K(z) = -\mathrm{Re} \int_0^\infty e^{-isz} w(s)  \frac{z^2 - 2iz\mu_0 \tanh(\mu_0 s) - \mu_0^2}{z^2+\mu_0^2} \rd s. \\
\eal\eeq
Denote $\tau = |q|\sqrt{t}$ and $z_0 = |x|/t$. Fix $C, M>1$. Then, there exists a $q_0=q_0(\mu_0)$ such that the following holds: For any $0<|q|<q_0$, $a(z)$ has one simple zero if $q>0$ and at most two simple zeros for $q<0$. In both cases the zeros lie in $i\rb^+$.
We denote the zeros by $z_1 = i\mu_1$, $z_2 = i\mu_2$.
Set
\beq
\label{eq:nu_j_sm_t}
 \upsilon_1 = -\sqrt{\frac{\mu_1-q}{\mu_1+q}} e^{i\mu_1^2 t/2 }, \ \ \upsilon_1^x = \upsilon_1 e^{-\mu_1 |x|},
\eeq
Let $p_1$ and $p_2$ be given in \eqref{eq:p_12_small_q}.
\vspace{2 mm}

\noindent\textnormal{(i)} If $a(z)$ has one simple zero at $z_1=i\mu_1\in i\rb^+$, then $\mu_1 = \mu_0 + q w_1 +O(q^2)$. Moreover,

\noindent$\bullet$ for the region $|x|\leq M$ \textnormal{:}
\beq
\label{eq:sol_small_x_one_zero}
  u(x,t) =
\left\{
\bal
& e^{i \mu_1^2 t/2} \Bigg[ v_{\mu_1}(x) - q w_0 \sqrt{\frac{2}{\pi t}} \Big(e^{i\Omega}\sech^2\mu_1 x - e^{-i\Omega} \tanh^2 \mu_1 x \Big)  \Bigg] \\
&  \qquad\quad \ \ + O\Big( \frac {q}{\sqrt{t}} \big(t^{-\frac 14} + |q|^{\frac 12} + \tau |\log \tau| \big)\Big),  \ \ t\leq \frac 12 |q|^{-2}, \\
& e^{i \mu_1^2 t/2} v_{\mu_1}(x) + O(q^2), \qquad \qquad \qquad\qquad\quad \ t \geq \frac 1C |q|^{-2}, \\
\eal
\right.
\eeq
where $\Omega(x,t) = -x^2/(2t) + \mu_1^2 t/2 + \pi/4$.

\noindent$\bullet$ For $|x|\geq 1/M$ \textnormal{:}
\beq
\label{eq:sol_large_x_one_zero}
  u(x,t) =
\left\{
\bal
& i(z_0 - i\mu_1)p_1 - \frac {2\mu_1 (\upsilon_1^x + p_1 )(\ovl{p_2} \ovl{\upsilon_1^x} +1)}
{|\upsilon_1^x + p_1|^2 + |p_2\upsilon_1^x +1|^2} \\
&  \qquad \qquad+ O\Big( \frac {q}{\sqrt{t}} \big(t^{-\frac 14} + |q|^{\frac 12} + \tau |\log \tau| \big)\Big), \   t\leq \frac 12 |q|^{-2}, \\
& e^{i \mu_1^2 t/2} v_{\mu_1}(x) + O(q^2), \qquad \qquad \qquad\qquad\quad \ t \geq \frac 1C |q|^{-2}. \\
\eal
\right.
\eeq

\noindent\textnormal{(ii)} If $a(z)$ has two simple zeros $z_1=i\mu_1\in i\rb^+$ and $z_2=i\mu_2\in i\rb^+$\textnormal{:} necessarily $q<0$.
In this case, $\mu_1 = \mu_0 + q w_1  +O(q^2)$ and $\mu_2 = -q + O(q^3)$.
Moreover, the asymptotic behavior of $u(x,t)$ is the same as in one-zero case up to error estimates
in \eqref{eq:sol_small_x_one_zero} and \eqref{eq:sol_large_x_one_zero}.

\end{theorem}

\begin{remark}
As we will see in the proof of the Theorem, \eqref{eq:sol_large_x_one_zero} holds for all $x\in\rb$.
\end{remark}

\begin{remark}
The results in Theorem \ref{T:main} and Theorem \ref{T:main2} are complementary to each other.
The asymptotic form \eqref{eq:sol_small_x_large_t_one_zero}, for example, provides a decay rate in the error term as $t\to\infty$ whereas \eqref{eq:sol_small_x_one_zero} does not. If $t \ll |q|^{-2}$, however,
$$
\epsilon |q|^{-\frac 12} t^{-({\frac 12}+\kappa)} \geq |q|^{\frac 12} t^{-\frac 34} \gg \frac {|q|}{\sqrt t}.
$$
Hence the solution is described in greater detail by \eqref{eq:sol_small_x_one_zero} in this region.
\end{remark}

\begin{remark}
The above result shows that, in both cases, whether $a(z)$ has one or two zeros, the leading terms are the same for $t\ll |q|^{-2}$.
In particular, in the case that $a(z)$ has two zeros, the contribution from the second zero is not noticeable for $t \ll |q|^{-2}$.
The one-zero case and the two-zero case indeed behave differently, however, for larger times, $t\gg |q|^{-2}$ (see Theorem \ref{T:main}).
Hence the critical region where the solution changes its behavior is $t\sim q^{-2}$.
\end{remark}

\begin{remark}
The determining role played by the parameter $\tau=|q|\sqrt t$ can be traced,
at least at the formal level, to the following observation:
the B\"acklund transformation taking NLS from the half line to the full line introduces terms of the form $\frac q{z\pm iq}$,
which are small away from $z=\mp q$.
On the other hand the solution of the Cauchy problem involves integrals of the form $\int f(z) \frac q{z\pm iq} e^{-itz^2/2}$,
from which we see that the leading order contributions take the form $\sim f(\mp iq) e^{itq^2/2} = f(\mp iq) e^{i\tau^2/2}$,
in which the role of $\tau$ is prominently displayed.
\end{remark}

\begin{remark}
Note that in the cross-over region, $\frac 1C \leq \tau \leq C$
the formulae \eqref{eq:sol_small_x_large_t_two_zeros}, \eqref{eq:sol_medium_x_large_t_two_zeros},
\eqref{eq:sol_small_x_one_zero} and \eqref{eq:sol_large_x_one_zero} describe the solution only up to an error $O(q^2) = O(t^{-1})$.
\end{remark}

\begin{remark}
Holmer and Zworski \cite{HZ} used ${\hat{\lambda}} = \mu_0 + qw_1$ in \eqref{eq:hz_result}
to describe the evolution of solutions to \eqref{eq:nls_delta} for time interval $1 \ll t\leq c|q|^{-2/7}(\ll c|q|^{-2})$.
From \eqref{eq:sol_small_x_one_zero} we see that as $\mu_1={\hat{\lambda}} + O(q^2)$,
$$ \bal u(0,t)
& = e^{i \mu_1^2 t/2} \bigg( v_{\mu_1}(0) - q w_0 \sqrt{\frac{2}{\pi t}} e^{i(\mu_1^2 t/2 + \pi/4)} \bigg) + o\bigg(\frac{q}{\sqrt t}\bigg)\\
& = e^{i {\hat{\lambda}}^2 t/2} \bigg( v_{{\hat{\lambda}}}(0) - q w_0 \sqrt{\frac{2}{\pi t}} e^{i({\hat{\lambda}}^2 t/2 + \pi/4)} \bigg) + o\bigg(\frac{q}{\sqrt t} + q^2 t\bigg),\\
\eal
$$
which agrees to leading order with the result \eqref{eq:hz_result} of \cite{HZ} for times $t \ll |q|^{-2/3}$.
Of course $|q|^{-2/7} \ll |q|^{-2/3}$, and we see in particular that the result in \cite{HZ} is in fact true in the larger interval $1 \ll t \ll |q|^{-2/3}$. In order to describe the asymptotics up to time $t \sim |q|^{-2}$, however, we need to use the more accurate frequency $\mu_1^2$ in place of ${\hat{\lambda}}^2$ in \eqref{eq:hz_result}.
\end{remark}

In order to gain some perspective on the nonlinear problem \eqref{eq:nls_plus}\eqref{eq:mixed_bdry},
it is useful to consider first the linear case:
\beq
\label{eq:lin_SE}
\bal
& iu_t + \frac 12 u_{xx} = 0, \ \ \ x\geq 0, \ \ t\geq 0 , \\
&u(x,0) = u_0(x), \ \ \ x>0, \\
&u_x(0+,t) + q u(0,t) = 0, \ \ \ t\geq 0. \\
\eal
\eeq
for some $q \in \rb$. Here we assume that $u_0(x)$ is smooth and decays at some appropriate rate as $x \to \infty$.
The standard, functional analytic way to solve this problem utilizes the self-adjoint operator $H_q^+ = - d^2/dx^2$ acting in $L^2(0,\infty)$
with the domain $D(H_q^+) = \{ u \in L^2(0,\infty) : u, u' \text{ absolutely continuous}, u''\in L^2(0,\infty), u_x(0+)+q u(0+)=0\}$.
For $q<0$, $H_q^+$ has purely absolutely continuous spectrum on $(0,\infty)$, but for $q>0$, $H_q^+$  has, in addition,
an $L^2(\rb^+)$ eigenvalue at $-q^2$. The solution $u(x,t)$ of \eqref{eq:lin_SE} is then expressed
in terms of the spectral theory of $H_q^+$ as follows: for $q<0$,
\beq
\label{eq:cont_spec_rep}
u(x,t) = \int_0^\infty e^{-itz^2/2} r(z) f_q(x,z) \rd z,
\eeq
and for $q>0$,
\beq
\label{eq:lin_u_e_pos_alp}
\bal u(x,t)
& =\int_0^\infty e^{-itz^2/2} r(z) f_q (x,z) \rd z \\
&+ 2q e^{itq^2/2}e^{-q x} \int_0^\infty e^{-q y} u_0(y) \rd y. \\
\eal
\eeq
Here $r(z)$ is determined by the initial data $u_0(x)$ and the functions $f_q (x,z) =  \frac {e^{ixz}}{iz+q} -  \frac {e^{-ixz}}{-iz+q}$ are the continuum eigenfunctions of $H_q^+$.

In the nonlinear case on the whole line, $-\infty<x<\infty$, the method of Zakharov-Shabat \cite{ZS} shows how to solve NLS
in terms of the spectral theory of the associated Lax operator $\partial_x - L$ where
\beq
\label{eq:ZS_Lax_operator}
L = iz\sigma + \bpm 0 & u(x)\\ -\overline{u(x)} &0 \epm , \ \ \ \sigma = \frac 12 \sigma_3, \ \ \ \sigma_3 = \bpm 1 & 0\\ 0 &-1 \epm.
\eeq
On the half line, $x>0$, the method of Fokas \cite{Fo} provides, in an appropriate sense, an ``$H_q^+$-analog" of $L$,
and hence a solution procedure for HNLS$_q^+$.
Alternatively, however, we can try to solve the IBV problem by an analog of the ``method of images" from electrostatics, i.e.
we extend the solution $u(x,t)$ to a solution $u_e(x,t)$ of the differential equation on the whole real line,
in such a way that the boundary condition at $x=0$ is automatically satisfied.

In particular for the linear problem \eqref{eq:lin_SE}
we seek a $C^1$ function $u_e(x,t)$ that decays appropriately as $|x|\to\infty$ on $\rb$ such that
\begin{enumerate}[(i)]
\item $u_e(x,t)$ solves the Schr\"odinger equation $i\partial_t u_e + \frac 12 \partial_x^2 u_e = 0$ on $\rb \times \rb^+$,
\item $u_e(x,t=0) = u_0(x)$, $x>0$,
\item $(u_e)_x(x,t) + q u_e (x,t)$ is an odd function , $t\geq 0$.
\end{enumerate}
Condition (iii) immediately implies that $(u_e)_x(0,t) + q u_e (0,t)=0$
and so for $x>0$, $u_e(x,t)$ is a solution to \eqref{eq:lin_SE} with $u_e(x,0) = u_0(x)$, $x>0$.

To ensure (iii) for all $t\geq 0$, it is sufficient, by linearity and homogeneity, to verify condition (iii) at $t=0$, i.e.
$$
(u_e)_x(-x,0) + q u_e (-x,0) = -\big((u_e)_x(x,0) + q u_e (x,0)\big), \ \ \ x>0.
$$
Integrating this relation we obtain
\beq
\label{eq:lin_sol}
u_e(-x,0) = u_0(x) + 2q \int_0^x e^{q(x-s)} u_0(s) \rd s , \ \ \ x>0.
\eeq
Set
\beq
\label{eq:lin_ext}
\bal u_e(x,0) =
\left\{
\bal
& u_0(x), \qquad\qquad\qquad\qquad\qquad\qquad \ \ x>0, \\
& u_0(-x) + 2q \int_0^{-x} e^{-q(x+s)} u_0(s) \rd s, \ \ x<0. \\
\eal
\right.
\eal
\eeq
A direct calculation shows that $u_e(x,0)$ is indeed $C^1$ and gives rise, in principle, to a solution of
the Schr\"odinger equation satisfying (i)(ii)(iii) and hence a solution of \eqref{eq:lin_SE}.

There are, however, some technical considerations. If $q<0$, then one sees from \eqref{eq:lin_sol} that $u_e(x,0)$ decays
(at a rate dictated by $q$ and $u_0(x)$) as $x\to -\infty$. This means, in particular, that $u_e(x,t)$ solves a PDE on $\rb$
with initial data in some standard class for which existence and uniqueness are easily established.
Indeed, one can solve (i)(ii)(iii) using Fourier theory (equivalently, using the spectral theory of $H=-d^2/dx^2$ in $L^2(\rb)$)
\beq
\label{eq:fou_rep}
u_e(x,t) = \frac 1{2\pi} \int_\rb e^{i(xz-tz^2/2)} \hat u (z) \rd z
\eeq
where $\hat u(z)$ decays appropriately as $|z|\to \infty$. To ensure the oddness of $(u_e)_x(x,t) + q u_e(x,t)$,
we see that $r(z) \equiv (iz+q) \hat u(z)$ must be odd, from which it follows that
\beq
\label{eq:lin_r}
r(z) =  \int_\rb e^{-ixz} F(x) \rd x
\eeq
where $F(x)=(u_0)_x(x) + q u_0(x)$, $x>0$ and $F(x) = -F(-x)$ for $x<0$. To check (ii), we note that for $x>0$,
$(u_e)_x(x,0) + q u_e(x,0) = \int e^{ixz} r(z) \rd z = (u_0)_x(x) + q u_0(x)$ and so $e^{q x} u_e(x,0) = e^{q x} u_0(x) + c$.
Letting $x\to +\infty$, we see that $c=0$ and so $u_e(x,0) = u(x)$ for $x>0$. Thus,
\beq
\label{eq:lin_u_e}
\bal u_e(x,t)
& = \int_\rb e^{i(xz-tz^2/2)} \frac {r(z)}{iz+q} \rd z \\
& = \int_0^\infty e^{-itz^2/2} r(z) \Big( \frac {e^{ixz}}{iz+q} -  \frac {e^{-ixz}}{-iz+q} \Big)\rd z \\
\eal
\eeq
where $r(z)$ is determined from \eqref{eq:lin_r}. In other words the method of images leads us back
to the functional analytical approach using the operator $H_q^+$.
If $q>0$, however, $u_e(x,0) \sim e^{q |x|} \int_0^\infty e^{-q s} u_0(s)\rd s$ as $x\to -\infty$
so that generically $u_e(x,0)$ does not decay as $x\to -\infty$. Thus the method of images breaks down.
Said differently, the functional analytic solution procedure \eqref{eq:lin_u_e_pos_alp} for $q>0$ cannot be ``unfolded"
to a ``method of images" on the whole line with $u_e(x)$ sufficiently smooth and decaying.

The method of \cite{Ta}\cite{BT}\cite{Ta2} to solve HNLS$_q^+$ is a nonlinear analog
of the ``method of images": we seek an extension $u_e(x,t)$ of the solution $u(x,t)$ of \eqref{eq:nls_plus} such that
$u_e(x,t)$ solves NLS, $i\partial_t u_e + \frac{1}{2}\partial_x^2 u_e +|u_e|^2 u_e = 0$, on the whole line and
the boundary condition $(u_e)_x(0,t) + q u_e (0,t)=0$ is automatically satisfied.
As NLS is not linear (and also not homogeneous), choosing $(u_e)_x(x,0) + q u_e (x,0)$ to be odd no longer works.
The situation is more subtle: The key ingredient in \cite{Ta}\cite{BT}\cite{Ta2} is the notion of a B\"acklund transformation. As we will show (see Remark \ref{rmk:nonlinearization} below), the method of \cite{Ta}\cite{BT}\cite{Ta2} is not just a nonlinear analog of the method of images for \eqref{eq:lin_SE}, but in fact a nonlinearization of the method in an appropriate sense.

We recall (see e.g. \cite{RS}) that a B\"acklund transformation refers very generally to a mapping $\mathcal{B}$
whereby a solution, say $u$, of one differential equation is transformed into a solution, say $v=\mathcal{B}(u)$, of another equation.
The two equations may be the same in which case $\mathcal{B}$ is called an auto-B\"acklund transformation.
B\"acklund transformations have their origin in the transformation theory of pseudospherical surfaces
by Bianchi and B\"acklund in the 1880's. (For such surfaces the equation that underlies the theory is the sine-Gordon equation (see \cite{RS}).)
In the modern theory of integrable system, in particular, the B\"acklund transformation has emerged as a major tool.
The map $u \mapsto \mathcal{B}(u) \equiv -u$ is a trivial auto-B\"acklund transformation for NLS:
More interesting, for example, is the transformation (see e.g. \cite{De}) $w\mapsto W$ where
\beq
\label{eq:Kdv_backlund}
\bal
& W(x,t) = -2 \frac {d^2}{dx^2} \log w(x,t), \\
& w(x,t) = e^\theta + \alpha e^{-\theta}, \ \ \alpha>0, \ \ \ \theta = \beta x - 4 \beta^3 t, \ \ \ \beta\in\rb,
\eal
\eeq
which takes the solution $w(x,t)$ of the (integrable!) linear equation
$$
   w_t + 4 w_{xxx} = 0
$$
into a (soliton) solution $W(x,t)$ of the KdV equation
$$
W_t - 6WW_x + W_{xxx} = 0.
$$
The idea in \cite{Ta}\cite{BT}\cite{Ta2} is to use a B\"acklund transformation to extend the solution $u(x,t)$ in $x>0$ to NLS \eqref{eq:nls_plus} to a solution $u_e(x,t)$ on the whole line, in such a way that the boundary condition $u_e(0,t) + q u_e(0,t)=0$ is automatically satisfied, as desired. In addition the transformation should be such, as in the linear case, that the extension $u_e(x,t)$ is smooth and decays as $x\to \pm\infty$ so that the standard methods of scattering/inverse scattering theory, and hence Riemann-Hilbert methods can be applied.
Riemann-Hilbert theory provides the precise analog in the nonlinear case of the Fourier transform method given in \eqref{eq:fou_rep}.
As in the linear case (cf. $q>0$), however, there are limitations on the applicability of this ``method of images."
(see Remark \ref{rmk:limit_images} below)

The paper is organized as follows.
In Section 2, we prove the existence of solutions of equations \eqref{eq:weak_sol_H_q}, \eqref{eq:weak_sol_H_q_plus} and \eqref{eq:weak_sol_H_0}.
In Section 3, we describe the scattering and inverse scattering method
for the focusing NLS equation on the line using Riemann-Hilbert techniques.
In Section 4, we discuss B\"acklund transformation and the method of Bikbaev and Tarasov
to solve NLS on the half-line with boundary conditions \eqref{eq:mixed_bdry}.
A significant part of this section involves the rephrasing of the method of Bikbaev and Tarasov as a RHP.
In Sections 5, 6 and 7, we apply the steepest-descent method for RHP in \cite{DZ3}\cite{DZ} to analyze $u(x,t)$ as $t\to\infty$. (cf. Theorem \ref{T:main})
Finally in Section 8 we combine the results from Section 4 together with Riemann-Hilbert method
to obtain a more detailed description of the solution $u(x,t)$ for times $t=O(q^{-2})$ (cf. Theorem \ref{T:main2}).

We need to clarify a point of terminology.
We have defined a B\"acklund transformation very generally as a mapping
which takes solutions of one differential equation into the solution of another differential equations.
In particular, by \eqref{eq:backlund_commute_x}, the map $\psi \mapsto \tilde \psi = (z+P) \psi$ is a B\"acklund transformation
taking solutions of $(\partial_x - L)\psi = 0$ to solutions of $(\partial_x - \tilde L)\tilde \psi = 0$.
But \eqref{eq:backlund_commute_x} also shows that $\partial_x - \tilde L = (z+P)(\partial_x - L)(z+P)^{-1}$ is invertible
if and only if $\partial_x - L$ is invertible (apart, possibly, from points $z$ in the spec(-P)).
Said differently, the operators $(i\sigma)^{-1}(\partial_x - Q)$ and $(i\sigma)^{-1}(\partial_x - \tilde Q)$
are iso-spectral (modulo points in spec(-P)).
Such iso-spectral transformations are usually referred to in the inverse scattering community as \emph{Darboux transformations}.
There is considerable freedom in the choice of Darboux transformations,
and we see that the B\"acklund transformation that we utilize in this paper
taking solutions of NLS to solutions of NLS, are constructed from Darboux transformations
$\partial_x - L \mapsto \partial_x - \tilde L$ with a specific choice of parameters.
The choice of parameters in the Darboux transformations, depends in general, on the tack at hand.
For example, a different choice of parameters gives rise to an operator $ \partial_x - \tilde L$
with poles added in or removed from the spectrum of $\partial_x - L$.
Such Darboux transformations are discussed in the appendix.

\section{Solutions of equations \eqref{eq:weak_sol_H_q}, \eqref{eq:weak_sol_H_q_plus} and \eqref{eq:weak_sol_H_0}}
\label{sec:sol_eqs}

\begin{theorem}
Solutions to equations \eqref{eq:weak_sol_H_q}, \eqref{eq:weak_sol_H_q_plus} and \eqref{eq:weak_sol_H_0}
in $H^{1,1}(\rb)$, $H^{1,1}(\rb^+)$ and $H^{1,1}(\rb)$,respectively, exist and are unique for all $t\geq0$.
\end{theorem}

The proof of this theorem is (essentially) standard.
For the convenience of the reader we provide some of the details for HNLS$_q^+$ \eqref{eq:weak_sol_H_q_plus}.
We proceed in steps.
Let $U_t \equiv e^{-iH_q^+ t/2}$ be the one parameter unitary operator generated by $H_q^+$ in $L^2(\rb^+)$.
If $u(t)=u(x,t)$ solves \eqref{eq:weak_sol_H_q_plus}, set
$$
\breve u(t) \equiv U_{-t} u(t) = u_0 + i \int_0^t U_{-s} |u(s)|^2 u(s) \rd s.
$$
For any $\phi \in C_0^\infty(0,\infty)$, as $U_t$ is unitary in $L^2(\rb^+)$,
$$ \bal  \frac d{dt} (\phi, u(t))
& = \frac d{dt} (U_{-t} \phi, \breve u(t)) \\
& = \big(-\frac i2 U_{-t}  H_q^+ \phi, \breve u(t) \big) + \big(U_{-t} \phi, i U_{-t} |u(t)|^2 u(t) \big), \\
\eal$$
where $(f, g) = \int_0^\infty \ovl{f} g$ is the standard inner product in $L^2(\rb^+)$.
Thus,
\beq
\label{eq:sol_weak_test_ftn}
\frac d{dt} (\phi, u(t))  = \frac i2 (\phi'', u) + i(\phi, |u(t)|^2 u(t)).
\eeq
Let $X = C([0,T], H^{1,1}(\rb^+))$, $0<T<\infty$, denote the set of continuous maps from $[0,T]$ to $H^{1,1}(\rb^+)$.
Equip X with a norm, $\|v\|_X \equiv \sup_{0\leq t \leq T} \|v(t)\|_{H^{1,1}(\rb^+)}$.
For $v=v(t)\in X$, define an operator
$$
(T_q v)(t) = U_t v(0) + i \int_0^t U_{t-s} |v(s)|^2 v(s) \rd s, \ \ 0\leq t \leq T.
$$

\vspace{0.1in}
\noindent \emph{Step 1 (Local existence and uniqueness)}: $T_q$ is a locally Lipschitz continuous mapping from $X$ to itself.
\vspace{0.1in}

From standard spectral theory(see \cite{CL}), we have
\beq
\label{eq:spec_rep_H_plus}
\bal
& U_t w = \frac 1{2\pi} \int_0^\infty e^{-iz^2 t/2} (f_q, w)(z) f_q(x,z) \rd z \\
& \qquad\qquad   + 2q \Big(\int_0^\infty e^{-q y} w(y) \rd y \Big) e^{-q x + iq^2 t/2}, \ \ w\in L^2(\rb^+), \\
\eal
\eeq
for $q>0$ where
$$
f_q(x,z) = e^{ixz} - \frac {q + iz}{q - iz} e^{-ixz}.
$$
If $q<0$, the second term is omitted in the spectral representation \eqref{eq:spec_rep_H_plus}.
As $U_t$ is unitary in $L^2(\rb^+)$, $\|U_t w\|_{L^2(\rb^+)} = \|w\|_{L^2(\rb^+)}$.
Now for $w\in H^{1,1}(\rb^+)$ define
$$
w_q(x) \equiv
\left\{
\bal
& w'(x) + q w(x), \ \ x\geq 0, \\
& -(w'(-x) + q w(-x)), \ \ x<0, \\
\eal
\right.
$$
Then, a direct calculation shows that
\beq
\label{eq:H_q_plus_H_0}
(U_t w)_q = e^{-iH_0 t/2} w_q.
\eeq
As $e^{-iH_0 t/2}$ is unitary in $L^2(\rb)$, it immediately follows that
$\|(U_t w)'\|_{L^2(\rb^+)} \leq c_1 \|w\|_{H^{1,1}(\rb^+)}$.
Moreover, by using integration by parts, we obtain $\|x(U_t w)\|_{L^2(\rb^+)}\leq (c_2 t + c_3) \|w\|_{H^{1,1}(\rb^+)}$.
Assembling the above results, we see that $\|U_t w\|_{H^{1,1}(\rb^+)} \leq (c_4 t + c_5) \|w\|_{H^{1,1}(\rb^+)}$
and hence for $v\in X$,
$$ \bal
& \|(T_q v)(t)\|_{H^{1,1}(\rb^+)}  \\
& \qquad \leq \|U_t v(0)\|_{H^{1,1}(\rb^+)} + \int_0^t \|U_{t-s} |v(s)|^2 v(s)\|_{H^{1,1}(\rb^+)} \rd s \\
& \qquad \leq (c_4 t + c_5) \bigg( \|v(0)\|_{H^{1,1}(\rb^+)} + \int_0^t \||v(s)|^2 v(s)\|_{H^{1,1}(\rb^+)} \rd s \bigg).\\
\eal
$$
Thus, $\|T_q v\|_X \leq (c_4 T + c_5)(1+ c_6 T \|v\|_X^2) \|v\|_X < \infty$.
For $v_1, v_2\in X$,
\beq
\label{eq:lipschitz}
\bal
& \|T_q v_1 - T_q v_2 \|_X  \\
& \qquad\leq \sup_{0\leq t\leq T} \int_0^t\|U_{t-s}(|v_1(s)|^2 v_1(s)-|v_2(s)|^2 v_2(s))\|_{H^{1,1}(\rb^+)} \rd s \\
& \qquad\leq c_7 T(c_4 T + c_5) \sup_{0\leq t\leq T} (\|v_1\|_{H^{1,0}}^2 + \|v_2\|_{H^{1,0}}^2) \|v_1 - v_2 \|_X, \\
\eal
\eeq
where ${H^{1,0}}=\{w\in L^2: w$ is absolutely continuous $w'\in L^2\}$ is the first Sobolev space.
This proves Step 1.
But then, it follows by a standard fixed point argument that there exists a unique (weak) solution to HNLS$_q^+$ locally in $H^{1,1}(\rb^+)$.
Let $u(t)=u(x,t)$, $0\leq t \leq T$, be the solution to HNLS$_q^+$ with $u(t=0)=u_0 \in H^{1,1}(\rb^+)$.

\vspace{0.1in}
\noindent \emph{Step 2}: If $u_0 \in H^{1,1}(\rb^+)$ and $u_0'\in L^1(\rb^+)$, then $u_x(x,t)$ is continuous in $x\in \rb^+$
and the boundary condition $u_x(0+,t) + q u(0,t) = 0$ holds for all $t>0$.
Moreover,
\beq
\label{eq:bounded_u_x}
|u_x(x,t)| \leq \frac c{\sqrt t}, \ \ 0 < t\leq T, \ 0<T<\infty.
\eeq
\vspace{0.1in}

From \eqref{eq:weak_sol_H_q_plus} and \eqref{eq:H_q_plus_H_0}, we have
\beq
\label{eq:u_t_alpha}
(u(t))_q = e^{-iH_0 t/2} (u_0)_q + i \int_0^t e^{-iH_0 (t-s)/2} (|u(s)|^2 u(s))_q \rd s.
\eeq
For $w\in H^{1,1}(\rb^+)$ such that $w'\in L^1(\rb^+)$, observe that $w_q \in L^1(\rb)$.
Using the classical formula for the kernel of $e^{-iH_0 t/2}$, we obtain from \eqref{eq:H_q_plus_H_0},
\beq
\label{eq:H_0_fourier_alpha}
(U_t w)_q = \sqrt{\frac 2 {\pi t}} \int_\rb w_q (s+x) e^{i(\frac {s^2}{2t} - \frac{\pi}4)} \rd s, \ \ t>0.
\eeq
Hence, $(U_t w)_q$ is continuous in $x$
and $|(U_t w)_q| \leq \frac c{\sqrt t} \|w_q\|_{L^1(\rb^+)}$.
Setting $w=u_0$ in \eqref{eq:H_0_fourier_alpha}, we see that the first term in \eqref{eq:u_t_alpha} is continuous in $x$.
Clearly $|u(t)|^2 u(t) \in H^{1,1}(\rb^+) \subset L^1(\rb^+)$ and $(|u(t)|^2 u(t))' \in L^1(\rb^+)$.
As $\frac 1{\sqrt {t-s}}$ is integrable on $[0,t]$, $0\leq t\leq T$, the second term in \eqref{eq:u_t_alpha} is also continuous in $x$.
Thus, $(u(x,t))_q$ is continuous and, as $u(x,t)$ is clearly continuous, we conclude that $u_x(x,t)$ continuous.
As $w_q$ is odd, it follows from \eqref{eq:H_0_fourier_alpha} that $(U_t w)_q(x=0)=0$
and hence the boundary condition $u_x(0+,t) + q u(0+,t) = 0$, $0<t\leq T$, immediately follows from \eqref{eq:u_t_alpha}.

\vspace{0.1in}
\noindent \emph{Step 3 (Conserved quantities)}: $\mc M \equiv \int_0^\infty |u|^2 \rd x$ and $\mc H \equiv \int_0^\infty |u_x|^2 - |u|^4 \rd x - q |u(0)|^2$ are constant in $t$.
\vspace{0.1in}

The usual approach here is to prove by direct calculation
that $\mc M$ and $\mc H$ are conserved for classical solutions (with smooth initial data),
and then to use an approximation argument.
However, the existence of classical solutions  for HNLS$_q^+$ does not appear to be an elementary matter.
The technical issue is that if a smooth function $u$ satisfies the boundary condition $u'(0) + q u(0) = 0$,
then $|u|^2 u$ in general does not.
In other words, although $|u|^2 u$ is smooth, $|u|^2 u$ may not be in the domain of the operator $H_q^+$
and so an iteration in a higher Sobolev space fails.

On the other hand, classical solutions to \eqref{eq:weak_sol_H_q_plus} satisfying the boundary conditions,
with arbitrary orders of smoothness and decay, can indeed be shown to exist (see Remark \ref{prop:classical_sol_NLS} below).
The proof of the existence of such smooth solutions, relies ultimately on Riemann-Hilbert methods.
(cf. proof of Proposition \ref{prop:classical_sol_NLS})
Such solutions could be used, in particular, to verify Step 3 for general $u\in H^{1,1}(\rb^+)$ by approximation as just described.
However, we will present a direct proof of Step 3
which does not use approximation by smooth solutions.
Relying on Riemann-Hilbert techniques to prove the existence of global solutions of \eqref{eq:weak_sol_H_q_plus}
does not seem appropriate.
Even more to the point, similar issues arise in the analysis of \eqref{eq:weak_sol_H_q}.
In this case, solutions to the equation which are smooth enough to yield a direct proof of the constancy of $\mc M$ and $\mc H$,
are not even known to exist for general (non-even) data (cf. Step 5 below).
So one needs the analog for \eqref{eq:weak_sol_H_q} of the proof given below, to prove global existence.

Suppose first that $u_0\in H^{1,1}(\rb^+)$ and $u_0' \in L^1(\rb^+)$.
As $U_t$ is unitary, we have for $\breve u(t) = U_{-t} u(t)$,
$$ \bal
\frac {d}{dt} \int_0^\infty |u|^2 \rd x
& =  \frac {d}{dt} \int_0^\infty |\breve u|^2 \rd x
= \int_0^\infty \breve u_t \ovl{\breve u} +  \breve u \ovl{\breve u_t} \rd x \\
& = \int_0^\infty (iU_{-t}|u|^2 u) \ovl{\breve u} +  \breve u (\ovl{iU_{-t}|u|^2 u}) \rd x \\
& = \int_0^\infty (i|u|^2 u) \ovl{u}  + u (\ovl{i|u|^2 u}) \rd x =0,\\
\eal
$$
and hence $\int_0^\infty |u|^2 \rd x$ is constant. For $\mc H$, we first have
\beq
\label{eq:deriv_u_sqr}
\bal \int_0^\infty |u_q|^2 \rd x
& = \int_0^\infty |u_x + q x|^2 \rd x \\
& = \int_0^\infty |u_x|^2 \rd x - q |u(0,t)|^2 + q^2 \int_0^\infty |u|^2 \rd x. \\
\eal
\eeq
On the other hand, again as $e^{iH_0 t/2}$ is unitary in $L^2(\rb)$,
it follows from \eqref{eq:u_t_alpha} and the oddness of $u_q (x)$ that
$$\bal
& \frac {d}{dt} \int_0^\infty |u_q|^2 \rd x \\
& \ \ =\frac 12 \frac {d}{dt} \int_\rb |u_q|^2 \rd x = \frac 12 \frac {d}{dt} \int_\rb |e^{iH_0 t/2} u_q|^2 \rd x \\
& \ \ = \frac 12 \int_\rb \ovl{e^{iH_0 t/2} u_q} (e^{iH_0 t/2} u_q)_t + e^{iH_0 t/2} u_q (\ovl{e^{iH_0 t/2} u_q})_t\rd x. \\
& \ \ = \frac 12 \int_\rb \ovl{e^{iH_0 t/2} u_q} (ie^{iH_0 t/2} (|u|^2 u)_q) + e^{iH_0 t/2} u_q (\ovl{e^{iH_0 t/2} (|u|^2 u)_q})\rd x. \\
& \ \ = \frac i2 \int_\rb \ovl{u_q} (|u|^2 u)_q - u_q (|u|^2 \ovl u)_q \rd x. \\
& \ \ = i \int_0^\infty (\ovl{u_x + q u}) ((|u|^2 u)' + q |u|^2 u) \\
& \qquad\qquad\qquad - (u_x + q u) ((|u|^2 \ovl u)' + q |u|^2 \ovl u)\rd x. \\
& \ \ = i \int_0^\infty u^2 \ovl{u}_x^2  - \ovl{u}^2 u_x^2  \rd x \equiv D(t)\\
\eal
$$
Hence, by \eqref{eq:deriv_u_sqr} together with the fact that $\int_0^\infty |u|^2$ is constant,
\beq
\label{eq:D_t}
\frac {d}{dt} \Big(\int_0^\infty |u_x|^2 \rd x - q |u(0,t)|^2 \Big) = D(t).
\eeq
Let $j(x)$ be any $C^\infty$-function supported on $(-1,1)$ such that $j\geq0$ and $\int_\rb j(x) \rd x = 1$,
and set $j_\epsilon (x) \equiv \frac 1 \epsilon j(\frac x\epsilon)$.
For any function $w\in H^{1,1}(\rb^+)$, denote $w_\epsilon(x) \equiv \int_0^\infty j_\epsilon(x-y) w(y) \rd y$.
Note that $j_\epsilon(x-y) \in C_0^\infty(0<y<\infty)$ for $x>\epsilon$.
Hence from \eqref{eq:sol_weak_test_ftn}, we have for $x > \epsilon$,
$$
(u_\epsilon)_t = \frac i2 (u_\epsilon)_{xx} + i(|u|^2 u)_\epsilon.
$$
Thus, integrating by parts and noting that $(u_\epsilon)_x=(u_x)_\epsilon$ for $x > \epsilon$, we obtain
\beq
\label{eq:u_epsilon_fourth}
 \bal \frac d{dt} \int_\epsilon^\infty |u_\epsilon|^4
& = 2\int_\epsilon^\infty \ovl{u}_\epsilon (\ovl{u}_\epsilon)_t u_\epsilon^2
+ \ovl{u}_\epsilon^2 u_\epsilon (u_\epsilon)_t \\
& =  D_\epsilon (t) + \delta_1^\epsilon(t)+ \delta_2^\epsilon(t),\\
\eal
\eeq
where
$$ \bal
& D_\epsilon (t)= i\int_\epsilon^\infty u_\epsilon^2 (\ovl{u}_\epsilon)_x^2 - \ovl{u}_\epsilon^2 (u_\epsilon)_x^2
, \\
& \delta_1^\epsilon(t) = i\big(u_\epsilon^2 \ovl{u}_\epsilon (\ovl{u}_\epsilon)_x - \ovl{u}_\epsilon^2 u_\epsilon (u_\epsilon)_x \big)(\epsilon), \text{ and } \\
& \delta_2^\epsilon(t) =  2i \int_\epsilon^\infty \ovl{u}_\epsilon^2 u_\epsilon (|u|^2 u)_\epsilon - \ovl{u}_\epsilon u_\epsilon^2 (|u|^2 \ovl {u})_\epsilon. \\
\eal
$$
As $j_\epsilon$ is an approximate identity, we see that for each $0<t\leq T$
$\delta_1^\epsilon(t)$ and $\delta_2^\epsilon(t) \to 0$ as $\epsilon \downarrow 0$:
here we need the fact that $u_x(x,t)$ is continuous in $x$ and satisfies the boundary condition $u_x(0+,t) + q u(0,t) = 0$.
Also $D_\epsilon (t) \to D(t)$.
Moreover, it follows from \eqref{eq:bounded_u_x}, and the fact that $t\mapsto u(t)$ is continuous, and hence bounded for $0<t\leq T$,
that $D_\epsilon (t) + \delta_1^\epsilon(t)+ \delta_2^\epsilon(t)$ is bounded by $\frac c{\sqrt t}$, uniformly for $\epsilon>0$.
Thus from \eqref{eq:D_t} and \eqref{eq:u_epsilon_fourth}, we obtain
$$ \bal
& \int_0^\infty |u_x(x,t)|^2 \rd x - q |u(0,t)|^2 - \int_\epsilon^\infty |u(x,t)|^4 \rd x \\
& \qquad = \int_0^\infty |u_x(x,0)|^2 \rd x - q |u(0,0)|^2 - \int_\epsilon^\infty |u(x,0)|^4 \rd x \\
& \qquad\qquad\qquad + \int_0^t D(s) - D_\epsilon (s) - \delta_1^\epsilon(s)- \delta_2^\epsilon(s) \rd s.\\
\eal $$
Now it follows by dominated convergence theorem that $\mc H$ is constant.

Step 3 for general data $u_0 \in H^{1,1}(\rb^+)$, then follows by approximation, $u_0 \to \check u_0 = u_0 \chi_R$, $R\to\infty$,
where $\chi_R \in C_0^\infty[0,\infty)$, $\chi_R(x) = 1$ for $0\leq x\leq R$, and $\chi_R(x) = 0$ for $x>R+1$.

\vspace{0.1in}
\noindent \emph{Step 4 (Global existence}: The solution $u(x,t)$ of HNLS$_q^+$ in $H^{1,1}(\rb^+)$ exists globally for all $t\geq 0$.
\vspace{0.1in}

From the basic estimates, we have for any constant $M>0$,
$$
\|u\|_{L^\infty(\rb^+)}^2 \leq 2\|u\|_{L^2(\rb^+)} \|u_x\|_{L^2(\rb^+)} \leq M\mc M + \frac 1M \|u_x\|_{L^2(\rb^+)}^2,
$$
and hence
$$\bal \|u_x\|_{L^2(\rb^+)}^2
& = \Big| \mc H + \int_0^\infty |u|^4 + q |u(0)|^2 \Big|  \\
& \leq |\mc H| + \|u\|_{L^\infty(\rb^+)}^2 (\mc M + |q|) \\
& \leq |\mc H| + \Big(M\mc M + \frac 1M \|u_x\|_{L^2(\rb^+)}^2 \Big) (\mc M + |q|). \\
\eal$$
Thus, if we set $M = 2(\mc M +|q|)$, then, as $\mc M$ and $\mc H$ are constant, $\|u_x\|_{L^2(\rb^+)}$ is uniformly bounded.

As Lipschitz constant in \eqref{eq:lipschitz} depends only on the $H^{1,0}$-norm of $u(t)$,
it follows by standard arguments using the a priori bound on $\|u_x(x,t)\|_{L^2}$, $\|u(x,t)\|_{L^2}$ that
the solution of HNLS$_q^+$ exists globally.

\vspace{0.1in}
\noindent \emph{Step 5}: Set $\hat u_0(x) \equiv u_0(|x|)$ and $\hat u(x,t) \equiv u(|x|,t)$, $x\in \rb$.
Then, $\hat u(x,t)$ solves \eqref{eq:weak_sol_H_q} with $\hat u(t=0) = \hat u_0$.
\vspace{0.1in}

The proof follows immediately by applying the following result,
\beq
\label{eq:ext_even}
(e^{-iH_q t/2} \hat w)(x) = (e^{-iH_q^+ t/2} w)(x), \ \ x>0
\eeq
for any $w\in L^2(\rb^+)$, and $\hat w(x) = w(|x|)$.
This identity follows directly by using the spectral representation \eqref{eq:spec_rep_H_plus} for $e^{-iH_q^+ t/2}$
and the spectral representation (cf. \cite{CL})
$$ \bal
& (e^{-iH_q t/2} \hat w)(x) \\
& \quad = \frac 1{4\pi} \int_\rb e^{-iz^2t/2} \frac {z^2}{z^2+q^2}
\big[ (f_1, \hat w)(z) f_1(x,z) + (f_2, \hat w)(z) f_2(x,z) \big] \rd z \\
& \qquad\qquad\qquad + q \Big( \int_\rb e^{-q |y|} \hat w(y) \rd y \Big) e^{-q |x|}
\eal $$
in the case $q>0$, where $f_1(x,z) = e^{ixz}$, $x>0$, and $f_1(x,z) = \big(1+\frac{q}{iz} \big)e^{ixz}-\frac{q}{iz}e^{-ixz}$, $x<0$,
and $f_2(x,z) = f_1(-x,z)$.
In the case $q<0$, the second term is omitted.

\vspace{0.1in}

For NLS on the line, solutions to IVP exist globally in $H^{k,k}(\rb)$ for all $k\geq1$.

\begin{proposition}
\label{prop:NLS_sol_H_k_k}
Let $u_0 \in H^{k,k}(\rb)$, $k\geq1$.
Then NLS on $\rb$ has a (unique) solution with $u(t=0)=u_0$ in $H^{k,k}(\rb)$, i.e.,
there exists a continuous map $t\mapsto u(t)\in H^{k,k}(\rb) \subset H^{1,1}(\rb)$, $u(t=0)=u_0$
such that $u(t)$ solves \eqref{eq:weak_sol_H_0}.
\end{proposition}

\begin{remark}
A standard argument shows that if $u(t)$ solves NLS in $H^{k,k}(\rb)$ for $k\geq3$,
then $u(t)$ is a classical solution.
\end{remark}

\section{Scattering and Inverse Scattering}
\label{sec:scat_inv}
We now recall the Riemann-Hilbert version (see \cite{BC}\cite{Zh1}\cite{Zh2}; see also \cite{DZ4}\cite{BDT}) of scattering and inverse scattering for the focusing NLS equation on the line,
\beq
\label{eq:focusing_nls}
\bal
& iu_t + \frac 12 u_{xx} + |u|^2 u = 0, \ \ t\geq0, \ \ -\infty<x<\infty, \\
& u(x,t=0) = u_0(x). \\
\eal
\eeq
Let $L$ denote the Zakharov-Shabat Lax operator in \eqref{eq:ZS_Lax_operator} above and set
\beq
\label{eq:operator_time}
E = -\frac i2 z^2 \sigma - \frac 12 z \bpm 0 & u\\ -\overline{u} &0 \epm + \frac 12 \bpm i|u|^2 & iu_x\\ i\overline{u}_x & -i|u|^2 \epm.
\eeq
The operators $\partial_x - L$ and $\partial_t - E$ form a Lax pair for the focusing NLS equation in the sense that the compatibility of $\partial_x \psi = L\psi$ and $\partial_t \psi = E\psi$, i.e. $\partial_t (\partial_x \psi) = \partial_x (\partial_t \psi)$ is equivalent to NLS \eqref{eq:focusing_nls}, i.e., $u=u(x,t)$ is a (classical) solution to \eqref{eq:focusing_nls}
if and only if $\partial_x - L$ and $\partial_t - E$ commute as operators on smooth functions $\psi=\psi(x,t)$, i.e.,
\beq
\label{eq:L_E_commute}
(\partial_x - L)(\partial_t - E) = (\partial_t - E)(\partial_x - L).
\eeq
Alternatively, NLS is equivalent to an iso-spectral deformation of the operator
\beq
\label{eq_iso_operator}
T = T(u) = (i\sigma)^{-1} \bigg( \partial_x - \bpm 0 & u(x)\\ -\overline{u(x)} &0 \epm \bigg).
\eeq
The following facts about solutions $\psi$ of $T\psi = z \psi$, or equivalently,
\beq
\label{eq:scattering_problem}
 \psi_x  = L\psi = (iz\sigma + Q)\psi, \ \ Q(x)=\bpm 0 & u(x)\\ -\overline{u(x)} &0 \epm,
\eeq
play a basic role in the theory that follows. The proofs are direct calculations.
Let $\sigma_2=\bsm 0& -i \\ i&0 \esm$ be the second Pauli matrix.

\begin{proposition}
\label{prop:scattering_symmetry}
Let $\psi(x,z)$ solve \eqref{eq:scattering_problem}. Then,

\noindent \textnormal{ (i)} $\ovl{\sigma_2 \psi(x,\ovl{z})\sigma_2}$ also solves \eqref{eq:scattering_problem} and

\noindent \textnormal{ (ii)} $\sigma_3 \psi(-x,-z)\sigma_3$ solves \eqref{eq:scattering_problem} with $Q(x)$ replaced by $Q(-x)$.
\end{proposition}
Of course the factors $\sigma_2$ and $\sigma_3$ on the right in (i),(ii) can be omitted: we include them here in view of the applications that follow.

The scattering map $\mathcal{S}$ (see Definition \ref{def:scattering} below) maps $T$ to its scattering data.
We now describe the properties of $\mathcal{S}$ when $u$ lies in the Sobolev space
$H^{1,1} = \{u \in L^2(\rb): u'(x), xu(x) \in L^2(\rb)\}$.
This is sufficient for our purpose as we will only consider NLS and \eqref{eq:nls_delta} in $H^{1,1}$. (cf. \cite{DZ})
(For the action of $\mathcal{S}$ on the other spaces e.g.
$H^{k,j} = \{u \in L^2(\rb): u, u', \cdots, u^{(k-1)}$ are absolutely continuous, $u^{(k)}(x), x^ju(x) \in L^2(\rb)\}$, $k\geq0$, $j\geq 1$, or Schwartz space, see \cite{Zh2}).

We use the standard notation $\cb^+ = \{ \text{Im} \ z >0\}$, $\cb^- = \{ \text{Im} \ z <0\}$, $e_1 = (1,0)^T$ and $e_2 = (0,1)^T$.
For $u\in H^{1,1}$ and $z \in \cb^+$, the equation
\beq
\label{eq:m_2_minus_lsp}
 \phi_x  = (iz\sigma + Q)\phi
\eeq
has a unique column solution $\phi= (\phi_1, \phi_2)^T$ such that
\beq
\label{eq:m_2_minus}
\phi(x,z)e^{ixz/2} \to e_2 \text { as } x \to -\infty.
\eeq
As $x\to\infty$,
$$
\phi(x,z)e^{ixz/2} \to a(z)e_2.
$$
where the \emph{scattering function} $a(z)$ is
\begin{itemize}
\item analytic in $\cb^+$,
\item continuous in $\ovl{\cb^+}$ and
\item $a(z) \to 1$ as $z\to \infty$ in $\ovl{\cb^+}$.
\end{itemize}
We say that $u\in H^{1,1}(\rb)$ is \emph{generic} if $a(z)$ is non-zero in $\ovl{\cb^+}$ except at a finite number of points
$z_1, \cdots, z_n \in \cb^+$ where it has simple zeros $a(z_k)=0$, $a'(z_k)\neq 0$, $k=1, \cdots, n$.
The set of generic functions $u(x)$ is an open dense subset of $H^{1,1}(\rb)$ (see \cite{Zh2}\cite{BDT}), which we denote by $\mathcal{G}$.
Set $Z_+ = \{ z_1 , \cdots, z_n \}$, $Z_- = \{ \overline{z_1} , \cdots, \overline{z_n} \}$, and $Z = Z_+ \cup Z_-$.

\begin{theorem}[Direct Scattering: see \cite{BC}\cite{Zh2}; also cf. \cite{BDT}]
\label{thm:direct_scattering}
Let $u \in \mathcal{G} \subset H^{1,1}$.
\begin{enumerate}
\item[\textnormal{(1)}] For $z\in \cb\setminus (\rb \cup Z)$, the equation $T\psi = z \psi$ has a unique solution $\psi = \psi(x,z)$ such that
\begin{enumerate}
\item[\textnormal{(i)}] $\psi(x,z) e^{-ixz\sigma}$ is bounded for $x\in \rb$ and
\item[\textnormal{(ii)}] $\psi(x,z) e^{-ixz\sigma}\to I$ as $x\to -\infty$.
\end{enumerate}

\item[\textnormal{(2)}] Points $z\in Z$ correspond to $L^2(\rb)$ eigenvalues of $T$, $Th = zh$, $h=(h_1, h_2)^T \in L^2(\rb)$.
Moreover, $z=z_k$ is an $L^2$ eigenvalue if and only if $z=\ovl{z_k}$ is an $L^2$ eigenvalue.
The associated eigenfunctions $h$ have asymptotics
$$
\bal h=h(x,z_k)
& = (e_1 + o(1)) e^{ixz_k/2} \ \ \text{ as } x \to +\infty, \\
& = \gamma(z_k)(e_2 + o(1)) e^{-ixz_k/2} \ \ \text{ as } x \to -\infty, \\
\eal
$$
and
$$
\bal h=h(x,\ovl{z_k})
& = (e_2 + o(1)) e^{-ix\ovl{z_k}/2} \ \ \text{ as } x \to +\infty, \\
& = \gamma(\ovl{z_k})(e_1 + o(1)) e^{ix\ovl{z_k}/2} \ \ \text{ as } x \to -\infty, \\
\eal
$$
where $\gamma(z_k)= -\ovl{\gamma(\ovl{z_k})} \neq 0$.

\item[\textnormal{(3)}] For all $x\in \rb$, the solution $\psi(x,z)$ in \textnormal{(1)}
\begin{itemize}
\item is analytic in $\cb\setminus (\rb \cup Z)$,
\item has simple poles at $Z$ such that
$$ \bal
& \Res\displaylimits_{z = z_k} \psi(x,z)= \lim_{z \to z_k} \psi(x,z) \bpm 0&0\\c(z_k)&0 \epm, \\
& \Res\displaylimits_{z = \ovl{z_k}} \psi(x,z)= \lim_{z \to \ovl{z_k}} \psi(x,z) \bpm 0&c(\ovl{z_k})\\0&0 \epm. \\
\eal $$
where $c(z_k) = \gamma(z_k)/ a'(z_k)$, $c(\ovl{z_k}) = -\ovl{c(z_k)}$, $k=1, \cdots, n$, and
\item has continuous boundary values $\displaystyle{\psi_\pm (x,z) = \lim_{\epsilon \downarrow 0} \psi(x,z\pm i\epsilon)}$, $z\in\rb$,
satisfying the \underline{jump relation}
\beq
\psi_+(x,z) = \psi_-(x,z)v(z), \ \ z\in\rb,
\eeq
where the \underline{jump matrix} $v(z)$ has the form
\beq
v(z) = \bpm 1+|r(z)|^2 & r(z) \\ \ovl{r(z)} & 1 \epm,
\eeq
and the \underline{reflection coefficient} $r(z) \in H^{1,1}$.
\end{itemize}

\item[\textnormal{(4)}] For all $x\in\rb$,
\begin{itemize}
\item $\psi(x,z) e^{-ixz\sigma} \to I$ as $z\to\infty$ in $\cb^+$ or in $\cb^-$.
\item $\psi(x,z) e^{-ixz\sigma} = I + \frac {m_1(x)}{z} + o \big(\frac 1z \big)$ as $z\to\infty$ in any cone $|\textnormal{Im} z|> c|\textnormal{Re} z|$, $c>0$.
\end{itemize}
\end{enumerate}
\end{theorem}

\begin{remark}
A solution $\psi$ of $T\psi=z\psi$, $z\in \cb\setminus\rb$, satisfying (1)(i)(ii) is known as a \emph{Beals-Coifman solution}.
For emphasis we denote $\psi=\psi_{BC}$.
\end{remark}

\begin{remark}
\label{rmk:det_constant}
Note that as the vector field in \eqref{eq:m_2_minus_lsp} has trace zero, $\det \psi(x,z)$ is independent of $x$.
Hence by (1)(i),
\beq
\det \psi(x,z) =1, \text{ for all } x\in \rb, \ z\in \cb\setminus (\rb\cup Z).
\eeq
\end{remark}

\begin{remark}
\label{rmk:not_generic}
If $u\in H^{1,1}(\rb)$ is not generic, the zero set of $a(z;u)$ can be extremely complicated.
For example see \cite{Zh1}. Let $z_0>0$ be an arbitrary positive number and let $\epsilon>0$ be small.
Let $D^+ = \{z\in\cb^+:|z-z_0|<\epsilon\}$. Let $a(z)$ be any function which is analytic in $D^+$,
continuous up to the boundary $\partial D^+$, and non-zero in $\partial D^+ \cap \cb^+$.
Then there exists $u$ in Schwartz space, $\mc S(\rb)$, such that $\Delta = \{z\in D^+: a(z)=0\} \subset$ spec$\{T(u)\}$.
Thus we see, in particular, that there exist a Schwartz function $u$ with $L^2(\rb)$-eigenvalues $\{z_k\}\subset \cb^+$
accumulating at essentially arbitrary rates onto the real axis.
The long time asymptotics of solutions of NLS with such initial data $u$,
is correspondingly extremely complicated and not yet fully understood.
This is in sharp contrast to the case where $u\in\mc G$(see below).
\end{remark}

We note that the second column of $\psi(x,z)$, $z\in\cb^+$, is in fact the solution $\phi=(\phi_1, \phi_2)^T$ in \eqref{eq:m_2_minus_lsp}\eqref{eq:m_2_minus} above. Define
$$
a(z) \equiv \ovl{a(\ovl{z})}, \ \ z\in\cb^-.
$$
The asymptotic behavior of $\psi(x,z)e^{-ixz\sigma}$ in (1)(i)(ii) above is in $H^{1,1}$ in the following sense.

\begin{proposition}
\label{prop:BC_tail}
Let $z\in \cb\setminus(\rb\cup Z)$. Let $\chi^+(x)$ be a $C^\infty$ function such that $0\leq \chi^+ \leq 1$, $\chi^+(x)=1$ for $x\geq 1$ and $\chi^+(x)=0$ for $x\leq 0$. Let $\chi^- = 1 - \chi^+$. Then,
$$\bal
& \|(\psi(x,z)e^{-ixz\sigma}-I)\chi^-\|_{H^{1,1}} \leq c(u), \\
& \|(\psi(x,z)e^{-ixz\sigma}-a(z)^{-\sigma_3})\chi^+)\|_{H^{1,1}} \leq c(u), \\
\eal $$
\end{proposition}

\begin{proof}
Let $z\in\cb^+\setminus Z$. The other case $z\in\cb^-\setminus Z$ is similar. Let $m(x,z)= (m_1, m_2) = \psi(x,z)e^{-ixz\sigma}$. Then (cf. \cite{BC},\cite{Zh2}), as $a(z)\neq 0$, $m_1$, $m_2$ are obtained as solutions of the following Volterra integral equations
\beq
\label{eq:m_1_BC}
m_1(x,z) = a(z)^{-1} e_1 -\int_x^{\infty} \bpm 1&0\\0&e^{iz(y-x)} \epm \bpm 0&u\\-\overline{u}&0 \epm m_1(y,z) \rd y,
\eeq
and
\beq
\label{eq:m_2_BC}
m_2(x,z) = e_2 +\int_{-\infty}^x \bpm e^{-iz(y-x)}&0\\0&1 \epm \bpm 0&u\\-\overline{u}&0 \epm m_2(y,z) \rd y,
\eeq
We consider \eqref{eq:m_2_BC}: the analysis of \eqref{eq:m_1_BC} is similar.
The solution $m_2(x,z)$ of \eqref{eq:m_2_BC} is obtained by standard iterations, which yields the bound
\beq
\label{eq:bound_sup}
\sup_{x\in\rb} |m_2(x,z)| \leq e^{ \int_{-\infty}^{\infty} |u(y)| \rd y}.
\eeq
Writing $m_2=((m_2)_1, (m_2)_2)^T$, \eqref{eq:m_2_BC} becomes
\beq
(m_2)_1(x,z) = \int_{-\infty}^x e^{-iz(y-x)} u(y) (m_2)_2(y,z) \rd y,
\eeq
and
\beq
\label{eq:m_2_2}
(m_2)_2(x,z) = 1- \int_{-\infty}^x \ovl{u}(y) (m_2)_1(y,z) \rd y,
\eeq
From Lemma \ref{lem:H11_integral} below,
$$
\|(m_2)_1\|_{H^{1,1}(\rb)} \leq c\|u (m_2)_2\|_{H^{0,1}(\rb)} \leq c\|(m_2)_2\|_{L^\infty(\rb)} \|u\|_{H^{0,1}(\rb)},
$$
and hence from Lemma \ref{lem:H11_integral} again,
$$
\|((m_2)_2-1)\chi^-\|_{H^{1,1}(\rb)}  \leq c\|u\|_{H^{0,1}} \|(m_2)_1\|_{H^{1,1}} \leq c \|u\|_{H^{0,1}}^2 \|(m_2)_2\|_{L^\infty}.
$$
As $x\to \infty$, we re-write \eqref{eq:m_2_2} as
\beq
(m_2)_2(x,z) = 1- \int_{-\infty}^\infty \ovl{u}(y) (m_2)_1(y,z) \rd y + \int_x^{\infty} \ovl{u}(y) (m_2)_1(y,z) \rd y.
\eeq
The quantity $1- \int_{-\infty}^\infty \ovl{u}(y) (m_2)_1(y,z) \rd y$ is precisely $a(z)$.
Arguing as above we see that
$$
\|((m_2)_2-a(z))\chi^+\|_{H^{1,1}(\rb)}  \leq c \|u\|_{H^{0,1}}^2 \|(m_2)_2\|_{L^\infty}.
$$
This completes the proof of the Proposition.
\end{proof}

\begin{lemma}
\label{lem:H11_integral}
Let $ f(x)\in H^{0,1}(\rb)$, $g(x) \in H^{1,1}(\rb)$ and $q > 0$. Then,
$$
\begin{aligned}
& \bigg\|\int_{\langle x\rangle}^\infty f(y) g(y)  \rd y \bigg\|_{H^{1,1}(\rb^+)} \leq c \|f\|_{H^{0,1}(\rb^+)}\|g\|_{H^{1,1}(\rb^+)} ,\\
& \bigg\|\int_{\langle x\rangle}^\infty f(y) e^{-q(y-\langle x\rangle)} \rd y\bigg\|_{H^{1,1}(\rb)} \leq c \|f\|_{H^{0,1}(\rb)}, \\
& \bigg\|\int_{-\infty}^{\langle x\rangle} f(y) e^{-q(\langle x\rangle-y)}  \rd y \bigg\|_{H^{1,1}(\rb)} \leq c \|f\|_{H^{0,1}(\rb)}, \\
\end{aligned}
$$
where c depends on $q$.
\end{lemma}
\begin{proof}
For $x>0$,
$$
\begin{aligned}
\bigg| \int_x^\infty f(y) g(y) \rd y \bigg|
& \leq  \int_x^\infty \frac {y^2+1}{x^2+1} |f(y)| |g(y)| \rd y \\
& \leq  \frac 1{x^2+1}  \|f\|_{H^{0,1}(\rb^+)}\|g\|_{H^{0,1}(\rb^+)}.\\
\end{aligned}
$$
and hence we obtain the inequality for $\big\|\int_{\langle x\rangle}^\infty f g \big\|_{H^{0,1}(\rb^+)}$.
As $\frac{d}{dx} \int_x^\infty f g = -f g $, the first inequality follows.
For the second inequality, observe that
$$
\int_x^\infty f(y) e^{-q(y-x)} \rd y = \int_0^\infty f(x+y) e^{-q y} \rd y.
$$
and hence by Minkowski's inequality,
$$
\begin{aligned}
& \bigg\|(1+|{\langle x\rangle}|) \int_{\langle x\rangle}^\infty f(y) e^{-q(y-\langle x\rangle)} \bigg\|_{L^2(\rb)} \\
& \leq \int_0^\infty \big\|(1+|{\langle x\rangle}|) f(\langle x\rangle+y) \big\|_{L^2(\rb)} e^{-q y} \rd y \\
& \leq c \|f\|_{H^{0,1}(\rb)} \int_0^\infty (1+y) e^{-q y} \rd y \leq c \|f\|_{H^{0,1}(\rb)} .\\
\end{aligned}
$$
Now, $\frac{d}{dx} \int_x^\infty f e^{q(x-y)}  = -f + \int_x^\infty q f e^{q(x-y)}$,
which lies in $L^2$, again by the Minkowski's inequality. This proves the the second inequality.
A simple change of variables shows that the third inequality follows from the second inequality.
\end{proof}

In \cite{ZS} the authors consider solutions $\phi$ of the Lax equation $T\phi = z\phi$ which are different from, but related to, the boundary values of the Beals-Coifman solutions $\psi_{\pm}(x,z) = ({\psi_1}_{\pm}(x,z), {\Psi_2}_{\pm}(x,z))$.
For $z\in \rb$, set
\beq
\label{eq:zs_akns_def}
\bal
& \psi^+(x,z) \equiv (a(z) {\psi_1}_+ (x,z), \ovl{a(z)} {\psi_2}_- (x,z)), \\
& \psi^-(x,z) \equiv ({\psi_1}_- (x,z), {\psi_2}_+ (x,z)), \\
\eal \eeq
Then $T\psi^{\pm} = z \psi{\pm}$, $z\in\rb$ and $\psi^{\pm}(x,z)$ have asymptotics,
\beq
\label{eq:zs_asymp}
\bal
& m^+(x,z) \equiv \psi^+(x,z)e^{-ixz\sigma} \to \text{ I as } x\to +\infty, \\
& m^-(x,z) \equiv \psi^-(x,z)e^{-ixz\sigma} \to \text{ I as } x\to -\infty, \\
\eal
\eeq
Moreover, the solutions $\psi^{\pm}$, $z\in\rb$, are uniquely determined by the conditions \eqref{eq:zs_asymp}.
From Proposition \ref{prop:scattering_symmetry},
\beq
\label{eq:symm_ZS_sol}
\psi_1^\pm(x,z) = i\sigma_2 \ovl{\psi_2^\pm(x,\ovl z)} = \bpm 0 & 1 \\ -1 & 0 \epm \ovl{\psi_2^\pm(x,\ovl z)}.
\eeq
As $\det \psi^\pm (x,z) \equiv 1$, $z\in\rb$, we conclude that for $x,z\in\rb$,
\beq
\label{eq:psi_pm_det}
\|\psi_j^\pm (x,z)\|^2 = |(\psi_j^\pm)_1 (x,z)|^2 + |(\psi_j^\pm)_2 (x,z)|^2 = 1, \ \ j=1,2.
\eeq

A direct calculation using the jump relation for $\psi_{\pm}$ shows that
\beq
\label{eq:scattering_matrix}
\psi^+(x,z) = \psi^-(x,z) S(z), \ \ z\in\rb
\eeq
where $S$ is the \emph{scattering matrix} and
$$
S(z) = \bpm a(z)&-\ovl{b(z)} \\ b(z) & \ovl{a(z)} \epm, \ \ b(z) \equiv a(z) \ovl{r(z)}, \ \ z \in \rb.
$$
We have
\beq
\label{eq:scattering_det}
|a(z)|^2 + |b(z)|^2 = \det S(z) = 1, \ \ z\in \rb.
\eeq
Note that if $u$ is generic, $a(z)\neq 0$ for $z\in\rb$ and as $a(z)\to1$ as $|z|\to\infty$, $\|b\|_{L^\infty(\rb)}\leq c<1$.
Let
\beq
\label{eq:gamma_set}
\Gamma_+ = \{\gamma(z_k), z_k \in Z_+\}.
\eeq
Knowledge of $(a,b,\Gamma_+)$ is equivalent to the knowledge of $(r,Z_+, K_+)$.
Indeed if $a$, $b$ and $\Gamma_+$ are known, then clearly $r(z)=\ovl{b(z)}/\ovl{a(z)}$ and $c(z_k) = \frac{\gamma(z_k)}{a'(z_k)}$, $k=1, \cdots, n$ are known.
Moreover, $Z_+$ is just the set of the zeros of $a(z)$ in $\cb^+$.
Conversely, as $|a(z)|^{-2} = 1+|r(z)|^2$, it is easy to verify that
\beq
\label{eq:a_exp_formula}
a(z) = \prod_{z_k \in Z_+} \frac{z-z_k}{z-\ovl{z_k}} e^{-l(z)}.
\eeq
where
\beq
\label{eq:l_z}
l(z) = \frac 1{2\pi i} \int_\rb \frac {\log \big(1+|r(s)|^2\big)}{s-z}   \rd s, \ \ z\in\cb^+.
\eeq
But then $b(z)=a(z)\ovl{r(z)}$, $z\in\rb$, and using the above formula \eqref{eq:a_exp_formula} it is easy to deduce that
\beq
\label{eq:b_H11}
b\in H^{1,1}(\rb).
\eeq
Also,
\beq
\label{eq:gamma_c_a}
\gamma(z_k)=c(z_k)a'(z_k), \ \ k=1, \cdots, n.
\eeq
We will use this fact repeatedly below without further comment.
The solutions $\psi^\pm$ are called the \emph{ZS-AKNS solutions} of $T\psi = z\psi$ (see \cite{AKNS}).
\eqref{eq:zs_akns_def} shows that $\psi_1^+$ and $\psi_2^-$, the first and second columns of $\psi^+$ and $\psi^-$, respectively, have analytic continuations to $\cb^+$. Also, $\psi_1^-$ and $\psi_2^+=\ovl{a(z)} {\psi_2}_- (x,z)=\ovl{a(\ovl{z})} {\psi_2}_- (x,z)$ and $\psi_2^+$ have analytic continuations to $\cb^-$.
Whenever we mention ZS-AKNS solutions in the sequel for $z\in\cb\setminus\rb$, we always mean the analytic continuations of $\psi^+$, $\psi_2^-$ and $\psi_1^-$, $\psi_2^+$ to $\cb^+$ (respectively $\cb^-$). Note that
\begin{equation}
\label{eq:zs_akns_anal_cont}
\psi_{ZS}(x,z) =
\left\{
\begin{aligned}
&(\psi_1^+(x,z) , \psi_2^-(x,z)) =\psi_{BC}(x,z) \bpm a(z) & 0 \\ 0 & 1 \epm, \ \ z \in \cb^+\\
&(\psi_1^-(x,z) , \psi_2^+(x,z)) =\psi_{BC}(x,z) \bpm 1 & 0 \\ 0 & \ovl{a(\ovl{z})} \epm, \ \ z \in \cb^-.
\end{aligned}
\right.
\end{equation}
In particular, as $\det \psi_{BC}(x,z) = 1$, we see that
\begin{equation}
\label{eq:zs_akns_det}
\det \psi_{ZS}(x,z) =
\left\{
\begin{aligned}
&a(z), \ \ z \in \cb^+\\
&\ovl{a(\ovl{z})}, \ \ z \in \cb^-.
\end{aligned}
\right.
\end{equation}
Also, from \eqref{eq:scattering_matrix}, we see that
\beq
\label{b_as_det}
b(z) =\det(\psi_1^-, \psi_1^+)(x,z), \ \ z\in\rb.
\eeq
Clearly by \eqref{eq:zs_akns_det}, $a(z)=0$ for $z\in\cb^+$ if and only if $\psi_1^+$ is proportional to $\psi_2^-$.
But as $\psi_1^+$ and $\psi_2^-$ decays exponentially as $x\to\pm\infty$, respectively,
we see that $a(z)=0$, $z\in\cb^+$ if and only if $z$ is an $L^2(\rb)$-eigenvalue of the operator $T$.
Similarly for $z\in\cb^-$, $a(z)=0$ if and only if $\psi_1^-$ is proportional to $\psi_2^+$.
Also, $a(z)=0$ if and only if $z$ is an $L^2(\rb)$-eigenvalue of the operator $T$.
Let $K=K_+ \cup K_- = \{ c(z_1), \cdots , c(z_n) \} \cup \{ c(\ovl{z_1}), \cdots , c(\ovl{z_n}) \}$.
\begin{definition}
\label{def:scattering}
The \underline{scattering map} $\mathcal{S}$ takes $u\in \mathcal{G}$ to its \emph{scattering data} $\mathcal{S}=\mathcal{S}(u)=(r, Z_+, K_+)$
$$
H^{1,1} \supset \mathcal{G} \ni u \mapsto \mathcal{S}(u) = (r(\cdot;u), Z_+(u), K_+(u)) \in H^{1,1} \times \cb_+^n \times (\cb\setminus 0)^{n}.
$$
where $n=n(u)$ depends on $u$.
\end{definition}

The inverse scattering map $\mathcal{S}^{-1}$ is constructed by solving a normalized RHP (see \cite{BC}, \cite{Zh2}).
The use of RHP's to solve inverse scattering problems, goes back to the work of Shabat \cite{Sh}.
For any $n\geq 0$, let
\begin{itemize}
\item $r\in H^{1,1}$,
\item $Z=Z_+ \cup Z_-$ where $Z_+ \in \cb_+^n$ and $Z_- = \ovl{Z_+} \in \cb_-^n$,
\item $K = K_+ \cup K_-$ where $K_+ = \{ c(z_1), \cdots , c(z_n) \} \in (\cb\setminus 0)^n$
and $K_- = \{ c(\ovl{z_1})=-\ovl{c(z_1)}, \cdots , c(\ovl{z_n})=-\ovl{c(z_n)} \}\in (\cb\setminus 0)^n$.
\end{itemize}
Set $v(z) = \bsm 1+|r(z)|^2 & r(z) \\ \ovl{r(z)} & 1 \esm$ and for $x\in\rb$,
$$ \bal v_x(z)
& = e^{ixz \textnormal{ad}\sigma} v(z) \equiv e^{ixz\sigma} v(z) e^{-ixz\sigma}
 = \bpm 1+|r_x(z)|^2 & r_x(z) \\ \ovl{r_x(z)} & 1 \epm, \\
\eal $$
where $r_x(z) = r(z)e^{ixz}$.
For fixed $x\in\rb$, let $m=m(x,z)$ solve the following normalized RHP:
\beq
\label{eq:jump_inv_scat}
\eeq

\begin{enumerate}[(i)]
\item $m(x,\cdot)$ is analytic in $\cb\setminus(\rb\cup Z)$.
\item $m(x,\cdot)$ has continuous boundary values on the axis satisfying
$$
m_+(x,z) = m_-(x,z) v_x(z), \ \ z \in \rb,
$$
\item $m(x,\cdot)$ has simple poles at $Z$ satisfying
$$ \bal
& \Res\displaylimits_{z = z_k} m(x,z)= \lim_{z \to z_k} m(x,z) v_x(z_k), \ \ v_x(z_k)=\bpm 0&0\\c_x(z_k)&0 \epm, \\
& \Res\displaylimits_{z = \ovl{z_k}} m(x,z)= \lim_{z \to \ovl{z_k}} m(x,z)  v_x(\ovl{z_k}), \ \ v_x(\ovl{z_k})=\bpm 0&c_x(\ovl{z_k})\\0&0 \epm. \\
\eal $$
where $c_x(z_k)=e^{-ixz_k} c(z_k)$, $c_x(\ovl{z_k})=e^{ix\ovl{z_k}} c(\ovl{z_k}) = -\ovl{c_x(z_k)}$.
\item $m(x,z) \to I$ as $z\to \infty$.
\end{enumerate}
Note that, by Theorem \ref{thm:direct_scattering}, if $u(x,z)$ is the Beals-Coifman solution for $u(x)$,
then $m(x,z)$ $\equiv \psi(x,z)e^{-ixz\sigma}$ is a solution of the above normalized RHP with $\mathcal{S}(u)=(r, Z_+, K_+)$.

\begin{remark}
\label{rmk:det_1}
If $m=m(x,z)$ solves the above RHP, $\det m(x,z) \equiv 1$.
Indeed, as $\det v_x(z) \equiv 1$, $\det m$ is analytic across $\rb$.
Also, $\det m$ is analytic at $z_k$:
this is because the residue condition is easily shown to be equivalent to the requirement that
$m(x,z) = M(x,z)\big(I+\frac {v_x(z_k)}{z-z_k} \big)$ where $M(x,z)$ is analytic and invertible near $z=z_k$.
On the other hand, $\det m(x,z) \to 1$ as $z\to\infty$ and so $\det m(x,z)\equiv 1$ by Liouville's theorem.
\end{remark}

\begin{remark}
\label{rmk:RHP_unique}
If the above RHP has a solution $m=m(x,z)$, then it is unique.
Indeed, if $\hat m=\hat m(x,z)$ is another solution of the RHP, a simple standard calculation shows that
the ratio $\hat m m^{-1}$ is an entire function which goes to I as $z\to\infty$.
By Liouville's theorem, it follows that $\hat m = m$.
\end{remark}

We now introduce a Cauchy operator on $\rb$,
$$ \bal
& Ch(z) = \int_\rb \frac {h(s)}{s-z} \frac{\rd s}{2\pi i}, \ \ z\in \cb\setminus\rb, \\
& C^{\pm}h(z) = \lim_{\epsilon \downarrow 0} Ch(z\pm i\epsilon), \ \ z \in \rb. \\
\eal $$
The operators $C^\pm$ are bounded, self-adjoint projections in $L^2(\rb)$, $\|C^\pm \|_{L^2(\rb) \to L^2(\rb)} = 1$, satisfying the important identity
\beq
\label{eq:cauchy_bd_id}
C^+ - C^- = \text{id}.
\eeq
Define the operator on $2\times 2$ matrix valued functions
\beq
\label{eq:cauchy_v_op}
C_v h = C^-(h(v-I)), \ \ h\in L^2.
\eeq
Clearly, $C_v$ is also bounded in $L^2$. For $x\in\rb$, let $\mu=\mu_x(z)$ be a $2\times 2$ matrix valued function on $\rb\cup Z$ with the following properties. On $\rb$, all the entries of $\mu-I=\mu(z)-I$ lie in $L^2(\rb)$ and for each $\zeta \in Z$, $\mu(\zeta) \in M_{2 \times 2}(\cb)$.
Suppose that $\mu=\mu_x$ solves the following system of coupled singular integral equations in $\rb\times Z$,
\beq
\label{eq:sing_mu}
\bal
& \big((1-C_{v_x})\mu_x\big)(z) = I + \sum_{\zeta \in Z} \frac{\mu_x(\zeta)v_x(\zeta)}{\zeta-z}, \quad z\in\rb, \\
& \bal \mu_x(\zeta) = I
    & + \frac 1{2\pi i} \int_\rb \frac {\mu_x(s)(v_x(s)-I)}{s-\zeta} \rd s \\
    & +  \sum_{\zeta' \in Z\setminus\{\zeta\}} \frac{\mu_x(\zeta')v_x(\zeta')}{\zeta'-\zeta}, \qquad\qquad \zeta\in Z. \\ \eal \\
\eal
\eeq
Set
\beq
\label{eq:eq_m_polar}
m(z) = m(x,z) \equiv I + \frac 1{2\pi i} \int_\rb \frac {\mu_x(s)(v_x(s)-I)}{s-z} \rd s + \sum_{\zeta \in Z} \frac{\mu_x(\zeta)v_x(\zeta)}{\zeta-z}
\eeq
for $z\in\cb\setminus(\rb\cup Z)$. Then, a direct calculation using \eqref{eq:cauchy_bd_id} shows that $m(z)$ solves the RHP \eqref{eq:jump_inv_scat}. Equations \eqref{eq:sing_mu} have a unique solution $\mu=\mu_x$ in $(I+L^2(\rb))\times (M_{2\times 2}(\cb))^{2n}$ by the following argument (see \cite{Zh2} for details): First one shows that the equations are Fredholm of index zero, and then one shows that the associated kernel of the equations is $\{0\}$
(More properly, one shows that equations \eqref{eq:sing_mu} are of type $\mathcal{F}\mathbf{x}=\mathbf{y}$ where $\mathcal{F}$ is a Fredholm operator with $\text{ind}(\mathcal{F})=0$ and $\text{Null}(\mathcal{F})=\text{dim Ker}(\mathcal{F})=0$. Thus $\text{codim}(\mathcal{F})= -\text{ind}(\mathcal{F}) + \text{Null}(\mathcal{F})=0$ and so $\mathcal{F}$ is a bijection).
Hence the solution $\mu_x$ exists and is unique in $(I+L^2(\rb))\times (M_{2\times 2}(\cb))^{2n}$ for each $x\in\rb$.
Moreover, as $z\to\infty$ in a cone $|\textnormal{Im} z|> c|\textnormal{Re} z|$, $c>0$,
$$
m(z) = m(x,z) = I + \frac {m_1(x)}{z} + o \Big(\frac 1z \Big).
$$
Set
$$
    u(x) = -i(m_1(x))_{12}.
$$
Then one shows that $r\in H^{1,1} \Rightarrow u(x) \in H^{1,1}$.
\begin{definition}
The \underline{inverse scattering map} $\mathcal{I}$ takes $(r,Z_+,K_+)\in H^{1,1} \times \cb_+^n \times (\cb\setminus 0)^{n}$ to the function $u(x) \in H^{1,1}$,
$$
(r,Z_+,K_+) \mapsto (r,Z_+ \cup \ovl{Z_+}, K_+ \cup (-\ovl{K_+})) \mapsto u(x).
$$
\end{definition}

\begin{theorem}
\label{thm:sca_inverse}
The maps $\mathcal{S}: \mc G \to H^{1,1} \times \cb_+^n \times (\cb\setminus 0)^{n}$
and $\mathcal{I}: H^{1,1} \times \cb_+^n \times (\cb\setminus 0)^{n} \to \mc G$ are continuous and inverse to each other.
\end{theorem}

\begin{remark} (see \cite{BC} \cite{Zh2}; also cf. \cite{DZ})
The direct scattering problem yields the function $m(x,z)=m(x,z;u)=\psi(x,z;u)e^{-ixz\sigma}$
where $\psi$ is a Beals-Coifman solution of $\frac{d\psi}{dx} = (iz\sigma + \bsm 0 & u(x)\\ -\overline{u(x)} &0 \esm)\psi$.
On the other hand, the RHP yields the function $m(x,z) = m(x,z;r,Z_+, K_+)$.
We will simply use the notation $m(x,z)$; the precise meaning will be clear from the context.
Of course, if $(r,Z_+,K_+)=\mathcal{S}(u)$, then by the above theorem $m(x,z;u) =  m(x,z;r,Z_+, K_+)$.
\end{remark}

\begin{theorem}[Time evolution of the scattering data: see \cite{ZS}\cite{Zh1}]
\label{thm:evol_scatt_data}
Suppose $u_0 \in \mathcal{G} \subset H^{1,1}$. Then the (weak, global) solution $u(t)$ of \eqref{eq:focusing_nls} with $u(t=0)=u_0$ remains in $\mathcal{G}$, $u(t) \in \mathcal{G}$ for $t\geq0$,
and
$$\bal
& r(t) = r(u(t)) = r(z;u_0)e^{-iz^2t/2}, \ \ z\in\rb, \\
& Z_+(t) = \{ z_j(t) = z_j(u(t)) = z_j(u_0), j=1,\cdots,n \}, \\
& K_+(t) = \{ c_j(t) = c_j(u(t)) = c_j(u_0)e^{iz_j^2 t/2}, j=1,\cdots,n \}, \\
\eal$$
or equivalently,
$$\bal
& a(t) = a(u(t)) = a(z;u_0), \ \ z\in \ovl{\cb^+}, \\
& b(t) = b(u(t)) = b(z;u_0)e^{iz^2t/2}, \ \ z\in\rb, \\
& Z_+(t) = \{ z_j(t) = z_j(u(t)) = z_j(u_0), j=1,\cdots,n \}, \\
& \Gamma_+(t) = \{ \gamma_j(t) = \gamma_j(u(t)) = \gamma_j(u_0)e^{iz_j^2 t/2}, j=1,\cdots,n \}. \\
\eal$$
Conversely, if $r_0\in H^{1,1}(\rb)$ and
$$\bal
& Z_+^0 = \{ z_1^0 ,\cdots, z_n^0\}\subset \cb^+, \\
& K_+^0 = \{ c_1^0 ,\cdots, c_n^0\}\subset \cb\setminus \{0\}, \\
\eal$$
then $u(t)= \mathcal{I}\big(r_0(\diamondsuit)e^{-i\diamondsuit^2t/2},Z_+^0, \{c_j^0e^{i(z_j^0)^2t/2} \}\big)$
solves NLS with initial data $u_0 = \mc I (r_0, Z_+^0, K_+^0)$.
In particular, if $u_0 \in \mc G$ and $r_0 =r(z; u_0)$, $\{z_j^0=z_j(u_0)\}$, $\{c_j^0=c_j(u_0)\}$,
then the (weak, global) solution $u(t)$ of NLS with $u(0)=u_0$ is given by
$$\bal u(t)
&= \mathcal{I}(\mathcal{S}(u(t))) \\
&= \mathcal{I}(r(u(t)),Z_+(t)=Z_+, K_+(t)) \\
&= \mathcal{I}\big(r(\diamondsuit; u_0)e^{-i\diamondsuit^2t/2},\{z_j(u_0)\}, \{c_j(u_0)e^{iz_j^2t/2} \}\big).
\eal$$
\end{theorem}

The long-time behavior of $u(t)$ is then inferred by evaluating $\mathcal{I}$ in terms of the RHP \eqref{eq:jump_inv_scat} and using the steepest descent method of \cite{DZ3}.

\begin{remark}
Let $\psi_{ZS}(x,t,z) \equiv (\psi_1^+, \psi_2^-)(x,t,z)$ for $z\in\cb^+$.
As the operator $\partial_t - E$ in \eqref{eq:operator_time} commutes with $\partial_x - L$ in \eqref{eq:ZS_Lax_operator},
$(\partial_x - L)((\partial_t - E)\psi_{ZS}) = (\partial_t - E)((\partial_x - L)\psi_{ZS}) = 0$, and so
$(\partial_t - E)\psi_{ZS}$ also solves \eqref{eq:scattering_problem}.
For $z\in\cb\setminus\rb$ such that $a(z)=\det \psi_{ZS} \neq 0$, we have $(\partial_t - E)\psi_{ZS} = \psi_{ZS} M$ for some matrix $M=M(z,t)$.
As $\psi_{ZS} e^{-ixz\sigma}$ is bounded as $x\to\pm\infty$, we see that $M(z,t) = \frac{iz^2 \sigma}2$, or equivalently,
\beq
\label{eq:evolution_psi_1}
\frac {\partial}{\partial t} \bpm (\psi_1^+)_1 \\ (\psi_1^+)_2 \epm =  \frac {i}2\bpm |u|^2 & u_x +i z u \\  \ovl{u}_x -i z \ovl{u} & z^2+|u|^2 \epm \bpm (\psi_1^+)_1 \\ (\psi_1^+)_2 \epm
\eeq
and
\beq
\label{eq:evolution_psi_2}
\frac {\partial}{\partial t} \bpm (\psi_2^-)_1 \\ (\psi_2^-)_2 \epm =  \frac {i}2\bpm -z^2+|u|^2 & u_x +i z u \\  \ovl{u}_x -i z \ovl{u} & |u|^2 \epm \bpm (\psi_2^-)_1 \\ (\psi_2^-)_2 \epm.
\eeq
\end{remark}

\section{The B\"acklund Transformation, the B\"acklund Extension and the Solution Procedure}
\label{sec:backlund}
Most of the results in this section are due to Bikbaev and Tarasov \cite{BT}\cite{Ta2}.
For completeness, what follows is a presentation of the results in \cite{BT}\cite{Ta2} with some interpretations, additions and alternative proofs.
Let $q \in \mathbb{R}$ and let $u(x)$ be an $L^1_{\text{loc}}$ complex-valued function on the line or the half-line $\rb^+$ or $\rb^-$.
Let $Q=\bsm 0&u\\ -\bar{u}&0 \esm$ as above. Consider the following $2\times 2$ matrix ODE,
\beq
\label{eq:backlund}
\left\{ \bal
& P_x = (Q - i[\sigma,P])P-PQ, \\
& P_0 \equiv P(0) = -iq \sigma_3. \\
\eal\right.
\eeq

\begin{proposition}
\label{p:backlund}
If $u\in L^1_{\text{loc}}(\rb)$, then there exists a unique global solution $P(x)$ of \eqref{eq:backlund} on $\rb$.
If $u\in L^1_{\text{loc}}(\rb^\pm)$, then there exists a unique global solution $P(x)$ of \eqref{eq:backlund} on $\rb^\pm$, respectively.
\end{proposition}

\begin{proof}
Let $\Psi^0(x,z) = (\Psi_1^0(x,z),\Psi_2^0(x,z))$ solve \eqref{eq:scattering_problem} with $\Psi^0(0,z) = I$.
Equation \eqref{eq:backlund} can be linearized in the following way.
Note that \eqref{eq:backlund} can be rewritten in an iso-spectral form
$$
P_x = [Q-i\sigma P, P].
$$
Hence
\begin{equation}
\label{eq:sol_P}
P(x) = \varphi(x) P_0 (\varphi(x))^{-1}.
\end{equation}
for some invertible matrix $\varphi(x)$, $\varphi(0)=I$.
But then we must have $P_x = [\varphi'(x)(\varphi(x))^{-1}, P]$,
so we can choose $\varphi'\varphi^{-1}=Q-i\sigma P = Q-i\sigma \varphi P_0 \varphi^{-1}$ or $\varphi' = Q\varphi - i\sigma \varphi P_0 $.
Hence if $\varphi(x) = (\varphi_1(x),\varphi_2(x))$ we see that
$\varphi_1' = Q\varphi_1 + i(iq)\sigma \varphi_1$ and $\varphi_2' = Q\varphi_2 + i(-iq)\sigma \varphi_2$.
As $\varphi(0)=I$, we must have $\varphi(x) = (\Psi_1^0(x,iq), \Psi_2^0(x,-iq))$.
But by Proposition \ref{prop:scattering_symmetry}, $\Psi_2^0(x,-iq) = \overline{\sigma_2 \Psi^0(x,iq) \sigma_2}e_2 = \bsm 0&-1 \\ 1&0 \esm \ovl{\Psi_1^0(x,iq)}$.
Thus $\varphi(x) = \bigl(\begin{smallmatrix} \xi_1&-\overline{\xi_2}\\\xi_2&\overline{\xi_1}\end{smallmatrix}\bigr)$,
 where $\xi = \bigl(\begin{smallmatrix} \xi_1 \\ \xi_2 \end{smallmatrix}\bigr) = \Psi_1^0(x,iq)$.
In particular, $\det \varphi(x) = |\xi_1(x)|^2 + |\xi_2(x)|^2 \neq 0$ so that $\varphi(x)$ is invertible.
We conclude that $P = \varphi P_0 \varphi^{-1}$, $\varphi = \bsm \xi_1&-\overline{\xi_2}\\\xi_2&\overline{\xi_1} \esm$ is the desired unique global solution of \eqref{eq:backlund}.
\end{proof}

Henceforth $P(x)$ denotes the unique solution to \eqref{eq:sol_P} for $u(x) \in L^1_{\text{loc}}$.
For a column vector ${\mathfrak{b}}(x)=\bsm {\mathfrak{b}_1}(x) \\ {\mathfrak{b}_2}(x) \esm$ such that ${\mathfrak{b}}(x) \neq 0$ for all $x$, define
\begin{equation}
\label{eq:backlund_ratio_ftn}
\mathcal{F}({\mathfrak{b}}) \equiv \frac{\mathfrak{b}_1\overline{\mathfrak{b}_2}}{|{\mathfrak{b}_1}|^2 + |{\mathfrak{b}_2}|^2}. \end{equation}
Note that $\mathcal{F}({\mathfrak{b}})$ in \eqref{eq:backlund_ratio_ftn} is determined by the ratio of the entries in ${\mathfrak{b}}$.
From \eqref{eq:sol_P} we see that
\begin{equation}
\label{eq:P_matrix}
P(x) = \varphi P_0 \varphi^{-1} = \frac{-i q}{|\xi_1|^2 + |\xi_2|^2}
\bpm |\xi_1|^2 - |\xi_2|^2 & 2\xi_1\overline{\xi_2}\\ 2\overline{\xi_1}\xi_2 & -|\xi_1|^2 + |\xi_2|^2 \epm,
\end{equation}
where $\varphi_1 = \bigl(\begin{smallmatrix} \xi_1 \\ \xi_2 \end{smallmatrix}\bigr)$.
Denote
\beq
\label{eq:u_tilde}
\tilde{u} \equiv u - i P_{12} = u - 2q \mathcal{F}(\varphi_1).
\eeq
Then, from \eqref{eq:P_matrix},
\beq
\label{eq:tilde_Q}
\tilde{Q} \equiv Q - i [\sigma, P] = \bpm 0&\tilde{u}\\ -\overline{\tilde{u}}&0 \epm.
\eeq
In particular, by \eqref{eq:backlund},
\beq
\label{eq:cont_at_x_0}
\tilde u(0) = u(0).
\eeq

The choice of $P$ in \eqref{eq:backlund} is motivated by the following fact which is easily verified by direct computation:
Let $\psi(x,z)$ be a solution of \eqref{eq:scattering_problem},
\beq
\label{eq:scattering_problem_1}
\psi_x = L\psi = (iz\sigma+Q)\psi
\eeq
and set
$$
\tilde\psi(x,z) = (z+P(x))\psi(x,z).
$$
Then $P(x)$ solves \eqref{eq:scattering_problem} if and only if $\tilde\psi(x,z)$ is an eigensolution of
\beq
\label{eq:scattering_problem_tilde}
\tilde\psi_x = \tilde L \tilde \psi = (iz\sigma+\tilde Q)\tilde \psi, \ \ \tilde L=iz\sigma+\tilde Q,
\eeq
where $\tilde Q$ is given by \eqref{eq:tilde_Q}.
Equivalently, $P(x)$ solves \eqref{eq:scattering_problem} if and only if
\beq
\label{eq:backlund_commute_x}
(z+P)(\partial_x - L) = (\partial_x - \tilde L) (z+P).
\eeq
In particular, we see that $P(x)$ induces a B\"acklund transformation $\psi \mapsto \tilde \psi$ for scattering equations of type \eqref{eq:scattering_problem}.

\begin{remark}
Note that $const=\det \tilde \psi(x,z) = \det(z+P(x))\det\psi(x,z) = \det(z+P(x)) const$, so $\det(z+P(x))$ is independent of $x$.
This is consistent with the fact that \eqref{eq:backlund} is an iso-spectral deformation $P_0$.
\end{remark}

We call the transformation $\mc{B}_q : u \mapsto \tilde u =\mc{B}_q(u)$, the \emph{B\"acklund transformation} of $u(x)$ with respect to $q$.
If $u(x)$ is only defined on $\rb^+$ or $\rb^-$, then we call the map $\mc{B}_q^\pm : u \mapsto \tilde u =\mc{B}_q^\pm(u)$, the B\"acklund transformation of $u(x)$ with respect to $q$ on $\rb^+$ or $\rb^-$, respectively.
Clearly, $\mc{B}_q$ maps $L^1_{\text{loc}}(\rb) \to L^1_{\text{loc}}(\rb)$ and
$\mc{B}_q^\pm$ maps $L^1_{\text{loc}}(\rb^\pm) \to L^1_{\text{loc}}(\rb^\pm)$.

In general, B\"acklund transformations $u\mapsto\tilde u$ involve a choice of parameters,
symbolically $\tilde u = F(u; c_1, c_2, \cdots)$, for some functional $F$,
which takes, for example, one scattering problem \eqref{eq:scattering_problem_1} into another \eqref{eq:scattering_problem_tilde},
If $u(t)$ solves a dynamical equation in $t$, it is often possible to choose $c_1=c_1(t)$, $c_2=c_2(t)$ appropriately
to ensure that $\tilde u(t) = F(u(t); c_1(t), c_2(t), \cdots)$ also solves an (or maybe, the same) dynamical equation.
In this way the B\"acklund transformation for scattering problems, for example,
gives rise to (auto-)B\"acklund transformations for the dynamical systems.
For example, the B\"acklund transformation in \eqref{eq:Kdv_backlund} can also be viewed (see e.g. \cite{De})
as taking the trivial solution $W_{triv}(x,t)\equiv0$ of KdV to the solution \eqref{eq:Kdv_backlund} of KdV
via a B\"acklund transformation of Schr\"odinger operators
$$
H_{triv} = -\frac{d^2}{dx^2} + W_{triv} \mapsto H(t) = -\frac{d^2}{dx^2} + W(t)
$$
where $W(t)= W(x,t) = -2 \frac {d^2}{dx^2} \log \big( e^{\beta x} +  c e^{-\beta x} \big)$.
Only if $c$ is chosen so that $c=c(t)= q e^{8\beta^3 t}$, does $W_{triv} \mapsto W(t)$ result in a solution of KdV.
As we will see, the method in \cite{BT}\cite{Ta2} can be viewed as a B\"acklund transformation
taking functions on $\rb^+$ to functions on $\rb$, where the parameters are chosen automatically
in such a way that the solution $u(t)=u(x,t)$ of HNLS$_q^+$ is taken to a solution $u^e(t) = u^e(x,t)$ of \eqref{eq:focusing_nls} on $\rb$.
Said differently, as $\partial_x - \tilde L$ is conjugate to $\partial_x - L$, $\partial_x - \tilde L=(z+P)(\partial_x - L)(z+P)^{-1}$,
any iso-spectral deformation of $\partial_x - L$ will give rise to an iso-spectral deformation of $\partial_x - \tilde L$.
The challenge here is to choose $P$ appropriately so that $\partial_x - \tilde L$ solves the same iso-spectral deformation (viz. NLS) as $\partial_x - L$.

\begin{remark}
\label{rmk:trivial_P}
If $q = 0$, then $P(x) \equiv 0$ and $\tilde{u}(x) = u(x)$.
\end{remark}


Let $\mc R u(x) \equiv u(-x)$. As we now show, up to the reversal $\mc R$, the B\"acklund transformation is reciprocal.
\begin{lemma}
\label{lem:reciprocal_backlund}
Let $q \in \rb$.
\begin{enumerate}
\item[\textnormal{(i)}] If $u \in L^1_{\text{loc}}(\rb)$, then $\mc{R}\mc{B}_q \mc{R} \mc{B}_q u(x)= u(x)$, $x\in\rb$.
\item[\textnormal{(ii)}] If $u \in L^1_{\text{loc}}(\rb^\pm)$, then $\mc{R} \mc{B}_q^\pm u \in L^1_{\text{loc}}(\rb^\mp)$ and   $\mc{R}\mc{B}_q^\mp \mc{R} \mc{B}_q^\pm u(x)= u(x)$, $x\in\rb^\pm$, respectively.
\end{enumerate}
\end{lemma}

\begin{proof}
We prove (i): the proof of (ii) is similar. Suppose that $u\in L^1_{\text{loc}}(\rb)$.
Let $u_1 = \mc R \mc B_q (u)$ and let $P_1(x) = \sigma_3 P(-x) \sigma_3$.
Observe from \eqref{eq:P_matrix} that $P(x) = -P^*(x)$ where $P^*$ is an adjoint of $P$.
From \eqref{eq:tilde_Q},
$$\begin{aligned} Q(-x)
& = \tilde{Q}(-x)+i[\sigma, P(-x)] = \tilde{Q}(-x)-i[\sigma,P_1(x)]\\
& = Q_1(x)-i[\sigma,P_1(x)],\\
\end{aligned} $$
where $Q_1 = \bigl(\begin{smallmatrix} 0&u_1 \\ -\overline{u_1}&0\end{smallmatrix}\bigr)$. In addition,
$$ \begin{aligned} (P_1)_x(x)
& = (-P_1^*)_x (x) = \sigma_3 P_x^*(-x) \sigma_3 = \sigma_3 (\tilde{Q}P - PQ)^*(-x)\sigma_3 \\
& = \sigma_3 (-Q^*P^* + P^*\tilde{Q}^*)(-x)\sigma_3 = \sigma_3 (-QP + P\tilde{Q})(-x)\sigma_3 \\
& = Q(-x)P_1(x) - P_1(x)\tilde{Q}(-x) = Q(-x)P_1(x) - P_1(x)Q_1(x) \\
& = (Q_1(x)-i[\sigma,P_1(x)]) P_1(x) - P_1(x) Q_1(x). \\
\end{aligned} $$
Combining with $P_1(0) = P(0)$, we conclude by uniqueness that $\tilde{Q}_1(x) = Q_1(x)-i[\sigma,P_1(x)] = Q(-x)$.
Hence $\mc R \mc B_q (u_1) = u$, as desired.
\end{proof}

\begin{corollary}
$\mc B_q$ is a bijection from $L^1_{\text{loc}}(\rb)$ onto $L^1_{\text{loc}}(\rb)$. $\mc B_q^\pm$ is a bijection from $L^1_{\text{loc}}(\rb^\pm)$ onto $L^1_{\text{loc}}(\rb^\pm)$, respectively.
\end{corollary}
\begin{proof}
As $\mc R^2 = 1$, $\mc{R}\mc{B}_q \mc{R} \mc{B}_q = 1$ implies $\mc{B}_q \mc{R} \mc{B}_q \mc{R}= 1$
and the bijectivity in $L^1_{\text{loc}}(\rb)$ follows. Similarly, $\mc{R}\mc{B}_q^\mp \mc{R} \mc{B}_q^\pm = 1_{L^1_{\text{loc}}(\rb^\pm)}$ and $\mc{B}_q^\pm \mc{R} \mc{B}_q^\mp \mc{R}= 1_{L^1_{\text{loc}}(\rb^\pm)}$,
which implies the bijectivity for $L^1_{\text{loc}}(\rb^\pm)$.
\end{proof}

\begin{remark}
 Clearly if $u\in L^1_{\text{loc}}(\rb)$, then
\beq
\label{eq:B_halfline}
(\mc B_q u)|_{\rb^\pm} = \mc B_q^\pm (u|_{\rb^\pm}).
\eeq
\end{remark}

\begin{proposition}
\label{prop:bijectivity_backlund}
$\mc B_q$ is a bijection from $H^{1,1}(\rb)$ onto $H^{1,1}(\rb)$. $\mc B_q^\pm$ is a bijection from $H^{1,1}(\rb^\pm)$ onto $H^{1,1}(\rb^\pm)$.
\end{proposition}
\begin{proof}
We only consider the case $\mc B_q^\pm$ and $q>0$. The other cases $\mc B_q^\pm$, $q<0$, and $\mc B_q$, $q \neq 0$ are similar.
As $\mc R$ is a bijection from $H^{1,1}(\rb^\pm)$ onto $H^{1,1}(\rb^\mp)$,
it is enough to prove that $\mc B_q^+$ and $\mc B_q^-$ map $H^{1,1}(\rb^\pm)$ into $H^{1,1}(\rb^\pm)$, respectively;
the bijectivity of $\mc B_q^+$ then follows from Lemma \ref{lem:reciprocal_backlund}(ii).
We consider $B_q^+$: the case $B_q^-$ is similar.
The proof hinges on the existence of the solutions $g(x,z)$ and $h(x,z)$ of the differential equation
\beq
\label{eq:linear_spectral_column_vector}
\psi_x = (iz\sigma + Q)\psi.
\eeq
on $\rb^+$ with the properties that
\beq
\label{eq:tail_behavior_g}
g(x,z) = (e_1 + r_1(x,z)) e^{ixz/2}, \ \ r_1 \in H^{1,1}(\rb^+).
\eeq
and
\beq
\label{eq:tail_behavior_h}
h(x,z) = (e_2 + r_2(x,z)) e^{-ixz/2}, \ \ r_2 \in H^{1,1}(\rb^+).
\eeq
Note that $h$ is not unique; for any constant $c$, $h+cg$ solves \eqref{eq:linear_spectral_column_vector} with \eqref{eq:tail_behavior_h} if $h$ solves \eqref{eq:linear_spectral_column_vector} with \eqref{eq:tail_behavior_h}.
The solution $g=\hat g e^{ixz/2}$ is uniquely specified by the asympotics $\hat g(x,z) \to e_1$ as $x\to\infty$, and
can be constructed by solving the integral equation
$$
\hat g(x,z) = e_1 - \int_x^\infty \bpm e^{iz(y-x)/2} &0\\0&1 \epm Q(y) \hat g(y,z) \rd y
$$
(cf. \eqref{eq:m_1_BC}) as in the proof of Proposition \eqref{prop:BC_tail}.
The solution $h$, however, cannot be obtained in a similar way via a Volterra equation.
For the reader's convenience, we prove the existence of $h$ following \cite{CL} pp. 104--105.
Note that the non-uniqueness of $h$ is reflected in the arbitrary choice of $x_0$.

For any given $u\in H^{1,1}(\rb^+)$, fix $x_0 >0$ so that $\int_{x_0}^\infty |u| < \frac 12$.
For $z\in\cb^+$, we consider the following integral equation for $g(x,z)=(g_1, g_2)^T$,
\beq
\label{eq:psi_2_plus_general}
\bal h(x,z)
& = e^{-ixz/2} e_2 + \int_{x_0}^x \bpm e^{-iz(y-x)/2} &0\\0&0 \epm Q(y) h(y,z) \rd y \\
& \qquad - \int_x^\infty \bpm 0 &0\\0&e^{iz(y-x)/2} \epm Q(y) h(y,z) \rd y, \ \ x\geq x_0. \\
\eal
\eeq
A direct calculation shows that $h(x,z)$ solves \eqref{eq:linear_spectral_column_vector}.
Setting $f(x,z) = h(x,z) e^{ixz/2}$, the integral equation \eqref{eq:psi_2_plus_general} becomes
\beq
\label{eq:m_2_plus_general}
f(x,z) = e_2 + (\mc T f)(x,z), \ \ x\geq x_0,
\eeq
where $\mc T$ is an integral operator defined by
$$
\bal (\mc T f)(x,z)
&  = \int_{x_0}^x \bpm e^{-iz(y-x)} &0\\0&0 \epm Q(y) f(y,z) \rd y \\
& \ \  - \int_x^\infty \bpm 0 &0\\0&1 \epm Q(y) f(y,z) \rd y, \ \ f\in L^\infty[x_0, \infty). \\
\eal
$$
Note that as $z\in \cb^+$ and $u\in L^1[x_0, \infty)$, $\mc T$ is well-defined.
Let $f_0 = e_2$ and define $f_{k+1} = e_2 + \mc T f_k$, $k\geq 0$, inductively.
Then,
$$
\|f_{k+1} - f_k\|_{L^\infty[x_0, \infty)} \leq \frac c{2^k}, \ \ k\geq 0.
$$
Indeed, $\|f_1 - f_0\|_{L^\infty[x_0, \infty)}\leq c$ and for $k\geq 1$,
$$\bal
& \|f_{k+1} - f_k\|_{L^\infty[x_0, \infty)} = \|\mc T(f_{k} - f_{k-1})\|_{L^\infty[x_0, \infty)} \\
& \quad \leq  \|f_{k} - f_{k-1}\|_{L^\infty[x_0, \infty)} \int_{x_0}^\infty |u| \leq \frac 12 \|f_{k} - f_{k-1}\|_{L^\infty[x_0, \infty)}. \\
\eal
$$
Hence $f\equiv f_0 + \sum_{k=1}^\infty f_{k} - f_{k-1}$ converges in $L^\infty[x_0, \infty)$ and solves the integral equation \eqref{eq:m_2_plus_general}.
Writing $f(x,z)=(f_1, f_2)^T$, \eqref{eq:m_2_plus_general} becomes
$$
\bal
& f_1(x,z) = \int_{x_0}^x e^{-iz(y-x)} u(y) f_2(y,z) \rd y, \\
& f_2(x,z) = 1+ \int_x^\infty \ovl{u}(y) f_1(y,z) \rd y, \\
\eal
$$
As $\|f\|_{L^\infty[x_0, \infty)} \leq c$, by Lemma \ref{lem:H11_integral}, $f_1 \in H^{1,1}[x_0,\infty)$
and therefore again by Lemma \ref{lem:H11_integral}, $f_2-1 \in H^{1,1}[x_0,\infty)$.
Now by the standard ODE theory, $h(x,z)$ extends uniquely to a solution of \eqref{eq:linear_spectral_column_vector} on $\rb^+$,
also denoted by $h(x,z)$. From the above calculations, we see \eqref{eq:tail_behavior_h}.
Now we fix $z=iq$. As $g(x,iq)$ and $h(x,iq)$ are clearly linearly independent for $x\geq 0$,
\beq
\label{eq:psi_0_1_lin_combi}
 \Psi_1^0(x,iq) = c_1 g(x,iq) + c_2 h(x,iq).
\eeq
for some constants $c_1$, $c_2$.
By \eqref{eq:P_matrix} and \eqref{eq:u_tilde}, we must show that $\mc F(\Psi_1^0(x,iq)) \in H^{1,1}(\rb^+)$.
Suppose $c_2 \neq 0$. Then as $x\to \infty$,
$$
\Psi_1^0(x,iq) = c_1 e^{-q x/2} \bpm 1 +r_3(x)\\ r_4(x) \epm, \ \ r_j\in H^{1,1}(\rb^+), \ \ j=3,4.
$$
Hence
$$
\mc F(\Psi_1^0(x,iq)) =\frac{(1+r_3) \ovl{r_4}}{|1+r_3|^2 + |r_4|^2}  \in H^{1,1}(\rb^+).
$$
On the other hand, if $c_2\neq 0$, then as $x\to \infty$,
\beq
\label{eq:psi_1_0_c_2_nonzero}
\Psi_1^0(x,iq) =c_2 e^{q x/2} \bpm r_5(x)\\ 1+r_6(x) \epm, \ \ r_j\in H^{1,1}(\rb^+), \ \ j=5,6.
\eeq
Thus,
$$
\mc F(\Psi_1^0(x,iq)) =\frac{r_5 \ovl{(1+r_6)}}{|r_5|^2 + |1+r_6|^2},
$$
which again lies in $H^{1,1}(\rb^+)$. This completes the proof.
\end{proof}

\begin{lemma}
\label{lem:P_tail}
Suppose $u\in H^{1,1}(\rb)$, and let $P(x)$ solve \eqref{eq:backlund}. Then, $P(x) \to -i\beta_{\pm} \sigma_3$ as $x \to \pm\infty$ where $\beta_{\pm}^2 = q^2$.
\end{lemma}

\begin{proof}
From \eqref{eq:P_matrix}, $P(x) = \frac{-i q}{|\xi_1|^2 + |\xi_2|^2}
\bsm |\xi_1|^2 - |\xi_2|^2 & 2\xi_1\overline{\xi_2}\\ 2\overline{\xi_1}\xi_2 & -|\xi_1|^2 + |\xi_2|^2 \esm$,
where $\xi = \bsm \xi_1 \\ \xi_2 \esm = \Psi_1^0(x,iq)$, $\Psi_1^0(0,iq) = e_1$.
Suppose first that $q>0$. Then by \eqref{eq:psi_0_1_lin_combi}, $ \Psi_1^0(x,iq) = c_1 g(x,iq) + c_2 h(x,iq)$
for some constants $c_1$, $c_2$. From \eqref{eq:tail_behavior_g}\eqref{eq:tail_behavior_h},
$g(x,iq) = (e_1 + r_1(x,iq)) e^{-q x/2}$, $h(x,iq) = (e_2 + r_2(x,iq)) e^{q x/2}$
where $r_1(x,iq)$, $r_2(x,iq) \in H^{1,1}(\rb^+)$.
If the second component of $g(x,iq)=(g_1(x,iq), g_2(x,iq))^T$ vanishes at $x=0$, then necessarily $c_2=0$.
But then $\frac {\xi_2}{\xi_1} = \frac{g_2(x,iq)}{g_1(x,iq)} \to 0$ as $x\to+\infty$, and so $P(x) \to -iq \sigma_3$ as $x\to+\infty$.
If the second component of $g(x,iq)$ does not vanish at $x=0$, then $c_2 \neq 0$.
Then, as in \eqref{eq:psi_1_0_c_2_nonzero}, $\xi = \Psi_1^0(x,iq) = c_2 e^{q x/2} \bsm r_5(x)\\ 1+r_6(x) \esm$, where $r_5$, $r_6\in H^{1,1}(\rb^+)$.
It now follows that $\frac {\xi_1}{\xi_2} \to 0$ as $x\to+\infty$ and so $P(x)\to iq\sigma_3$ as $x\to+\infty$.
Thus $P(x)\to -i\beta_+\sigma_3$ as $x\to+\infty$, where $\beta_+=q$ if $g_2(0,iq) =0$ and $\beta_+=-q$ if $g_2(0,iq)\neq0$.

As in the case $x\to+\infty$, there are solutions $g(x,iq)$, $h(x,iq)$ of \eqref{eq:linear_spectral_column_vector} with the property that $g(x,iq) = (e_1 + r_1(x,iq)) e^{-q x/2}$ and $h(x,iq) = (e_2 + r_2(x,iq)) e^{q x/2}$
where $r_1(x,iq)$ and $r_2(x,iq)$ now belong to $H^{1,1}(\rb^-)$.
(As opposed to the case $x>0$, $h(x,iq)$ is now uniquely determined and $g(x,iq)$ is not unique).
Again $\xi=\Psi_1^0(x,iq) = c_1 g(x,iq) + c_2 h(x,iq)$.
If the second component of $h(x,iq)$ vanishes at $x=0$, then $c_1= 0$. Hence $\frac {\xi_1}{\xi_2} = \frac{h_1(x,iq)}{h_2(x,iq)}\to 0$ as $x\to-\infty$, and so $P(x)\to -iq\sigma_3$ as $x\to-\infty$.
On the other hand, as before, if the second component of $h(x,iq)$ does not vanish at $x=0$, then $c_1 \neq 0$
and we find that $P(x)\to iq\sigma_3$ as $x\to-\infty$.
Thus $P(x)\to -i\beta_-\sigma_3$ as $x\to-\infty$, where $\beta_-=q$ if $h_2(0,iq) =0$ and $\beta_-=-q$ if $h_2(0,iq)\neq0$.
If $q<0$, the situation is similar, and again $\beta_{\pm}^2 = q^2$.
\end{proof}

\begin{remark}
\label{rmk:beta_depend_psi_1_0}
From the above proof, we see that as $x\to\infty$,
if $\Psi_1^0(x,iq)$ grows exponentially, then $\beta_+=-|q|$ and
if $\Psi_1^0(x,iq)$ decays exponentially, then $\beta_+=|q|$.
On the other hand, as $x\to -\infty$,
if $\Psi_1^0(x,iq)$ grows exponentially, then $\beta_-=|q|$ and
if $\Psi_1^0(x,iq)$ decays exponentially, then $\beta_-=-|q|$.
For example, if $q>0$ and $a(i|q|) \neq 0$,
then $\Psi_1^0(x,iq)=c_1 \psi_1^+(x,iq)+ c_2 \psi_2^-(x,iq)$ for some constants $c_1$, $c_2$.
Hence,
\beq
\label{eq:equiv_ZS_component_beta}
\bal
& (\psi_1^+)_2(0,iq) = 0 \iff c_2 = 0 \iff \beta_+=\beta_-=q,  \\
& (\psi_2^-)_2(0,iq) = 0 \iff c_1 = 0 \iff \beta_+=\beta_-=-q, \text{ and } \\
& \bal (\psi_1^+)_2(0,iq) \neq 0, (\psi_2^-)_2(0,iq) \neq 0
    & \iff c_1 \neq 0, c_2 \neq 0 \\ &  \iff \beta_+=-q, \ \ \beta_-=q. \\ \eal
\eal
\eeq

\end{remark}

The following calculation is standard in the inverse scattering literature.
\begin{proposition}
\label{prop:tilde_scattering_data}
Let $q \in \rb\setminus \{0\}$ be given. Suppose $u\in H^{1,1}(\rb)$ and let $\tilde u = \mc B_q u\in H^{1,1}(\rb)$.
Let $S(z) = \bsm a(z)&-\ovl{b(z)} \\ b(z)&\ovl{a(z)} \esm$, $\tilde S(z) = \bsm \tilde a(z)&-\ovl{\tilde b(z)} \\ \tilde b(z)&\ovl{\tilde a(z)} \esm$ be the scattering matrices for $u$ and $\tilde u$, respectively.
Then,
\beq
\label{eq:S_tilde}
\tilde S(z) = (z-i\beta_- \sigma_3)S(z) (z-i\beta_+ \sigma_3)^{-1}, \ \ z\in\rb,
\eeq
where $P(x)\to -i\beta_\pm \sigma_3$ as $x\to \pm \infty$, $\beta_\pm^2 = q^2$.
Equivalently,
\beq
\label{eq:tilde_a_b}
\tilde a(z) =  \frac{z-i\beta_-}{z-i\beta_+} a(z), \ \ \tilde b(z) = \frac{z+i\beta_-}{z-i\beta_+} b(z) , \ \ z\in\rb.
\eeq
In particular, if $a(i|q|)\neq 0$, then
\beq
\label{eq:tilde_generic}
\tilde u \in \mc G,  \text{ if } \ u \in \mc G.
\eeq
Moreover,
\beq
\label{eq:tilde_zeros}
\tilde Z_+ = Z_+ \text{ or } \tilde Z_+ = Z_+ \cup \{i|q|\}
\eeq
In both cases,
\beq
\label{eq:tilde_norming}
\tilde \gamma(z_k) = \frac{z_k+i\beta_-}{z_k-i\beta_+} \gamma(z_k), \ \ z_k \in Z_+,
\eeq
and in the second case,
\beq
\label{eq:tilde_norming_alpha}
\tilde \gamma(i|q|) =
\left\{\bal
& \frac{(\psi_1^+)_2(0,iq)}{(\psi_2^-)_2(0,iq)}, \ \ \text{ if } \ q>0, \\
& \frac{(\psi_1^+)_1(0,-iq)}{(\psi_2^-)_1(0,-iq)}, \text{ if } \ q<0. \\
\eal\right.
\eeq
\end{proposition}

\begin{proof}
From \eqref{eq:scattering_problem_tilde} and Lemma \ref{lem:P_tail}, we see that
\beq
\label{eq:tilde_ZS}
\tilde \psi^\pm(x,z) = (z+P(x))\psi^\pm (x,z)(z-i\beta_\pm \sigma_3)^{-1}, \ \ z\in\rb,
\eeq
where $\psi^\pm$, $\tilde \psi^\pm$ are the ZS-AKNS solutions for $u$ and $\tilde u$, respectively.
Substituting these relations into $\tilde \psi^+ =\tilde \psi^- \tilde S$ and utilizing $\psi^+=\psi^- S$,
we immediately obtain \eqref{eq:S_tilde} and hence \eqref{eq:tilde_a_b}.

Now suppose that $u\in\mc G$ and $a(i|q|)\neq0$.
Then $\tilde a(z) \neq 0$, $z\in\rb$, and $\tilde a(z) \to 1$ as $z\to\infty$.
Then relationship $\tilde a(z) =  \frac{z-i\beta_-}{z-i\beta_+} a(z)$ continues analytically to $\cb^+\setminus\{i\beta_+\}$.
Now clearly $i\beta_\pm \not\in Z_+$. Hence $\tilde a(z_k)=0$ if and only if $a(z_k)=0$.
and as $a'(z_k)\neq 0$, it follows that $z_k$ is a simple zero for $\tilde a(z)$.
From \eqref{eq:tilde_ZS}, we have by analytic continuations
$$ \bal
& \tilde\psi_1^+(x,z_k) = \frac{z_k+P(x)}{z_k-i\beta_+} \psi_1^+(x,z_k), \ \
\tilde\psi_2^-(x,z_k) = \frac{z_k+P(x)}{z_k+i\beta_-} \psi_2^-(x,z_k). \\
\eal$$
Substituting these relations into $\tilde\psi_1^+(x,z_k) = \tilde\gamma(z_k) \tilde\psi_2^-(x,z_k)$
and utilizing $\psi_1^+(x,z_k) = \gamma(z_k) \psi_2^-(x,z_k)$, we immediately obtain \eqref{eq:tilde_norming}.
From \eqref{eq:tilde_a_b} we see that $a(i|q|)=0$ if and only if $\beta_+ = -|q|$, $\beta_- = |q|$.
In this case, the zero at $i|q|$ is also simple.
In particular, $\tilde u \in \mc G$ in both cases in \eqref{eq:tilde_zeros}.
From \eqref{eq:tilde_ZS}, for $z=i|q|$,
$\tilde f^+(x) \equiv \frac{i|q|+P(x)}{2i|q|} \psi_1^+(x,i|q|)$ and
$\tilde f^-(x) \equiv \frac{i|q|+P(x)}{2i|q|} \psi_2^-(x,i|q|)$
are two solutions of $\frac{d\tilde\psi}{dx}=(-|q|\sigma + \tilde Q)\tilde\psi$ with asymptotics,
$\frac 1{2i|q|} \bsm i(|q|-\beta_+)&0\\0&i(|q|+\beta_+)\esm (e_1+o(1))e^{-|q|x/2}$ and
$\frac 1{2i|q|} \bsm i(|q|-\beta_-)&0\\0&i(|q|+\beta_-)\esm (e_2+o(1))e^{|q|x/2}$
as $x\to\pm\infty$, respectively.
Hence $\tilde f^+(x) = \tilde \psi_1^+(x,i|q|)$, $\tilde f^-(x) = \tilde \psi_2^-(x,i|q|)$.
By definition $\tilde f^+(x) = \tilde\gamma(i|q|) \tilde f^-(x)$.
Evaluating this relation at $x=0$, we find
$$
\bsm i(|q|-q)&0\\0&i(|q|+q)\esm \psi_1^+(0,i|q|) =
\tilde\gamma(i|q|) \bsm i(|q|-q)&0\\0&i(|q|+q)\esm \psi_2^-(0,i|q|)
$$
and so \eqref{eq:tilde_norming_alpha} follows.
It is easy to see that $\tilde\gamma(i|q|)$ is non-zero and finite.
\end{proof}

In general, if $u(t)$ solves NLS, $\tilde u(t) = \mc B_q u(t)$ will not solve NLS,
because $\beta_\pm(t)$ are not continuous functions of t.
For example, we see from the proof of Lemma \ref{lem:P_tail} that for $q>0$, say,
$\beta_+=q$ if $g_2(0,iq) =0$ and $\beta_+=-q$ if $g_2(0,iq)\neq0$.
But $g=(g_1, g_2)^T$ is just the ZS-AKNS solution $\psi_1^+(x,iq)=\psi_1^+(x,iq;u)$
and there is no reason why $(\psi_1^+)_2(0,iq)$ cannot pass through $0$ as $u(t)$ evolves under NLS.
In particular, we may have $(\psi_1^+)_2(0,iq; u(t=0))=0$ but $(\psi_1^+)_2(0,iq; u(t))\neq0$ for $t>0$
and hence $\beta_+(t)$ is not continuous at $t=0$.

Now from \eqref{eq:equiv_ZS_component_beta} $\tilde u$ has an eigenvalue at $z=iq$, $q>0$,
if and only if $(\psi_1^+)_2(0,iq; u)$ and $(\psi_2^-)_2(0,iq; u)$ are both non-zero.
Thus if $(\psi_1^+)_2(0,iq; u(t))=0$ for $t=0$
but $(\psi_1^+)_2(0,iq; u(t))$ and $(\psi_2^-)_2(0,iq; u(t))$ are non-zero for $t>0$,
we see by \eqref{eq:tilde_a_b} that $\partial_x \tilde \psi = (iz\sigma+\tilde Q)\tilde \psi$ has no $L^2$-eigenvalue at $z=iq$ for $t=0$,
but $\partial_x \tilde \psi = (iz\sigma+\tilde Q)\tilde \psi$ has an $L^2$-eigenvalue at $z=iq$ for $t>0$.
Furthermore, in the case when $(\psi_1^+)_2(0,iq; u(t))$ and $(\psi_2^-)_2(0,iq; u(t))$ are both non-zero,
the norming constant $\tilde\gamma(i|q|)$ in \eqref{eq:tilde_norming_alpha} will not, in general, evolve appropriately,
i.e. $\tilde \gamma(i|q|;u(t)) \neq \tilde \gamma(i|q|;u(0))e^{-iq^2 t/2}$.
In terms of the discussion preceding  Remark \ref{rmk:trivial_P},
the ``parameters" in the B\"acklund transformation $u\mapsto \tilde u = \mc B_q u$ are not chosen correctly.

As we now show, however, if $u(t)$ solves NLS on $\rb^+$ together with the boundary condition \eqref{eq:mixed_bdry},
then $\tilde u(t) = \mc B_q^+ u(t)$ is also a solution of NLS on $\rb^+$ (but with $q$ replaced by $-q$).
Thus in this case, the ``parameters" in the B\"acklund transformation are correctly chosen automatically.
The following lemma gives the necessary and sufficient conditions on $P=P(x,t)$
in order that $\mc B_q^+ u(t)$ solves NLS.

\begin{lemma}
\label{lem:equiv_NLS_evol}
Let $u(t)$ be a classical solution in $H^{1,1}$.
Suppose that $P$ solves \eqref{eq:backlund} and let $\tilde u(t) = \mc B_q^+ u(t)$ be the B\"acklund transformation of $u(t)$.
Let $L$, $\tilde L$ denote the Lax operators in \eqref{eq:ZS_Lax_operator}
and define $E$, $\tilde E$  as in \eqref{eq:operator_time} with $u$, $\tilde u$, respectively.
Then,
\beq
\label{E_P_commute}
(\partial_t - \tilde E)(z+P) = (z+P)(\partial_t - E)
\eeq
if and only if
\beq
\label{E_L_commute}
(\partial_t - \tilde E)(\partial_x - \tilde L) = (\partial_x - \tilde L)(\partial_t - \tilde E).
\eeq
In other words, condition \eqref{E_P_commute} is necessary and sufficient for $\tilde u(t)$ to solve NLS.
\end{lemma}

\begin{proof}
If \eqref{E_P_commute} holds, then by \eqref{eq:L_E_commute} and \eqref{eq:backlund_commute_x},
$$ \bal
& (\partial_t - \tilde E)(\partial_x - \tilde L)(z+P) \\
& \quad = (\partial_t - \tilde E)(z+P)(\partial_x - L)  = (z+P)(\partial_t - E)(\partial_x - L)\\
& \quad  = (z+P)(\partial_x - L)(\partial_t - E)= (\partial_x - \tilde L)(z+P)(\partial_t - E) \\
& \quad  = (\partial_x - \tilde L)(\partial_t - \tilde E)(z+P), \\
\eal$$
which clearly proves \eqref{E_L_commute}.
Conversely, suppose that \eqref{E_L_commute} holds.
Define $Q_1\equiv\bsm i|u|^2 & iu_x\\ i\overline{u}_x & -i|u|^2 \esm$
and similarly $\tilde Q_1$ with $u$ replaced by $\tilde u$.
Then,
\beq
\label{eq:backlund_Q_1}
\tilde Q_1 - Q_1 = P_x = (\tilde Q P - PQ).
\eeq
Indeed, $(\tilde Q P - PQ)_{12} = (P_{12})_x = i(\tilde u - u)_x = (\tilde Q_1 - Q_1)_{12}$
and $(\tilde Q P - PQ)_{11} = \tilde u P_{21} - (-\ovl u) P_{12} = i(|\tilde u|^2 - |u|^2) = (\tilde Q_1 - Q_1)_{11}$.
Equation \eqref{eq:backlund_Q_1} now follows by symmetry.
Set
$$
\Delta \equiv (\partial_t - \tilde E)(z+P) - (z+P)(\partial_t - E).
$$
Then, by \eqref{eq:tilde_Q} and \eqref{eq:backlund_Q_1},
\beq
\label{eq:Delta_with_Q1}
 \bal \Delta
& = P_t - z(\tilde E - E) - (\tilde E P - PE) \\
& = \frac 12 z^2 (\tilde Q - Q + i[\sigma,P]) -\frac z2 \big(\tilde Q_1 - Q_1 - (\tilde Q P - PQ)\big) \\
& \qquad + \big(P_t - \frac 12 (\tilde Q_1 P - PQ_1) \big)\\
& = P_t - \frac 12 (\tilde Q_1 P - PQ_1). \\
\eal
\eeq
In particular, $\Delta$ is independent of $z$.
Again by \eqref{eq:L_E_commute} and \eqref{eq:backlund_commute_x}, we see that
$$ \bal
& (\partial_x - \tilde L)(\partial_t - \tilde E) (z+P) \\
& \quad = (\partial_t - \tilde E)(\partial_x - \tilde L) (z+P)=(\partial_t - \tilde E)(z+P)(\partial_x - L)  \\
& \quad= \Delta (\partial_x - L) + (z+P)(\partial_t - E)(\partial_x - L) \\
& \quad=\Delta (\partial_x - L) + (z+P)(\partial_x - L)(\partial_t - E) \\
& \quad= \Delta (\partial_x - L) + (\partial_x - \tilde L)(z+P)(\partial_t - E)  \\
& \quad= \Delta (\partial_x - L) - (\partial_x - \tilde L)\Delta + (\partial_x - \tilde L)(\partial_t - \tilde E) (z+P)\\
\eal$$
and hence
$$\bal
& \Delta (\partial_x - L) - (\partial_x - \tilde L)\Delta = 0, \ \ z\in\cb \\
& \iff iz[\sigma, \Delta] - (\Delta_x - \tilde Q \Delta + \Delta  Q ) = 0, \ \ z\in\cb \\
& \iff [\sigma, \Delta]=0, \ \ \Delta_x = \tilde Q \Delta - \Delta Q. \\
\eal$$
The first equation above implies that $\Delta$ is diagonal.
But then, $\tilde Q \Delta$ and $\Delta Q$ are off-diagonal and it follows that $\Delta_x=0$,
which implies that $\Delta$ is constant in $x$.
As $\tilde u(0)=u(0)$ and $P(x=0,t)=-iq \sigma_3$, $(\tilde Q_1 P - PQ_1)$ is off-diagonal at $x=0$ and $P_t(x=0)=0$,
and it follows from \eqref{eq:Delta_with_Q1} that $\Delta$ is also off-diagonal at $x=0$,
and hence we must have $\Delta = 0$ at $x=0$.
Thus, $\Delta \equiv 0$ for all $x$, which completes the proof.
\end{proof}

In motivating the approach in \cite{BT} and \cite{Ta2},
the following result plays a key role.

\begin{proposition}
\label{prop:IBV_rel_backlund}
Let $q \in \rb\setminus\{0\}$  be given.
Let $t \mapsto u(t) = u(x,t), x\geq0$ be $C^2$ with respect to $x$ and $C^1$ with respect to $t$
such that $u_{xx}$ and $u_t$ decay rapidly as $x\to+\infty$. Let $u_0(x) = u(x,0)$.
Then, the following statements are equivalent.
\begin{enumerate}
 \item[\textnormal{(i)}] $u(t)$ solves HNLS$_q^+$ with initial data $u_0$.

 \item[\textnormal{(ii)}] $\mc B_q^+ u(t)$ solves HNLS$_{-q}^+$ with initial data  $\mc B_q^+ u_0$.
\item[\textnormal{(iii)}] $\mc R \mc B_q^+ u(t)$ solves HNLS$_q^-$ with initial data  $\mc R \mc B_q^+ u_0$.
\end{enumerate}
\end{proposition}

\begin{proof}
Clearly (ii) and (iii) are equivalent. We will first show that (i) implies (ii).
Suppose that $q>0$. The case $q<0$ is similar and will be discussed below.
Define $E$ as in \eqref{eq:operator_time} and $\tilde E$ with $u$ replaced by $\tilde u = \mc B_q^+ u$.

\vspace{0.1in}
\noindent \emph{Claim}: (i) implies that $(\partial_t - \tilde E)(z+P) = (z+P)(\partial_t - E)$, $x,t\geq0$.
\vspace{0.1in}

\noindent Let $g(x,t,z)=(g_1, g_2)^T$, $x\geq0$, $z\in\ovl{\cb^+}$
be as in the proof of Proposition \ref{prop:bijectivity_backlund} with $u(x)=u(x,t)$.
As the operator $\partial_t - E$ in \eqref{eq:operator_time} commutes with $\partial_x - L$ in \eqref{eq:ZS_Lax_operator},
$(\partial_x - L)((\partial_t - E)g) = (\partial_t - E)((\partial_x - L)g) = 0$, and so
$(\partial_t - E)g$ also solves \eqref{eq:scattering_problem}.
As $(\partial_t - E)g e^{-ixz/2} \to \frac{iz^2}4 e_1$ as $x\to\infty$, we see that $(\partial_t - E)g = \frac{iz^2}4 g$, or equivalently,
\beq
\label{eq:evolution_g}
\frac {\partial g}{\partial t}  =  \frac {i}2\bpm |u|^2 & u_x +i z u \\  \ovl{u}_x -i z \ovl{u} & z^2+|u|^2 \epm g.
\eeq

If $g_2(0,iq, t=0)= 0$, then as $u_x(0,t)+q u(0,t)=0$,
it follows from \eqref{eq:evolution_g} that $g_2(0,iq,t) = 0$ for all $t\geq0$.
Let $P(x,t)$ solve \eqref{eq:backlund} with $u(x,t)$ and $P(x,t) \to -i\beta_+(t)\sigma_3$ as $x\to \infty$.
From \eqref{eq:psi_0_1_lin_combi}, we see that $\Psi^0_1(x,iq,t)=c_1(t) g(x,iq,t)$, $c_1(t) \neq 0$
and hence $\beta_+(t) = q$ by Remark \ref{rmk:beta_depend_psi_1_0}.
From \eqref{eq:sol_P}, we have $P(x,t) = \varphi(x,t) P_0 (\varphi(x,t))^{-1}$ where $\varphi=(\varphi_1, \varphi_2) = (\Psi_1^0(x,iq,t), \Psi_2^0(x,-iq,t))$.
Now, $(\varphi_1)_t = (c_1 g(iq))_t = \big(c_1'/c_1 - iq^2/4 + E(iq)\big) \varphi_1$.
As $\varphi_2 = \bsm 0&-1\\1&0\esm \ovl{\varphi_1}$, we have $(\varphi_2)_t = \big(\ovl{(c_1'/c_1)} + iq^2/4 + E(-iq)\big) \varphi_2$
and hence $\varphi_t = \varphi D  + \frac 12 (iq^2 \sigma \varphi - iq Q \varphi \sigma_3 + Q_1 \varphi)$,
where $D = \bsm c_1'/c_1 - iq^2/4 & 0 \\ 0 & \ovl{(c_1'/c_1)} + iq^2/4 \esm$,
which implies $\varphi_t \varphi^{-1}= \varphi D \varphi^{-1} + \frac 12 (iq^2 \sigma +QP + Q_1)$,
Thus, from the above calculations together with $P^2 = -q^2 I$,
$$\bal P_t
& = (\varphi_t \varphi^{-1})P - P(\varphi_t \varphi^{-1}) \\
& = \frac 12 \big(iq^2 [\sigma,P] + (QP-PQ)P + Q_1P-PQ_1\big) \\
& = \frac 12 \big((\tilde Q-Q)P^2 + (QP-PQ)P + Q_1P-PQ_1\big) \\
& = \frac 12 (\tilde Q_1 P - PQ_1), \\
\eal$$
and hence the Claim follows from \eqref{eq:Delta_with_Q1}.
But then, from Lemma \ref{lem:equiv_NLS_evol} we see that
the operators $\partial_t - \tilde E$, $\partial_x - \tilde L$ commute with each other.
Hence, $\tilde u(t)$ solves NLS on $\rb^+$.
As $\tilde Q(0) = Q(0)$, $\tilde Q_x(0) = Q_x(0) - i[\sigma, P_x(0)] = -q Q(0) - i[\sigma, \tilde Q(0) P_0 - P_0 Q(0)] = q \tilde Q(0)$
and hence  $\tilde u(t)$ solves HNLS$_{-q}^+$.

If $g_2(0,iq,t=0) \neq 0$, then from Remark \ref{rmk:beta_depend_psi_1_0}, we have $\beta_+(t=0) = -q$.
Set $v_0 = \mc R \mc B_{q}^+ u_0 \in H^{1,1}(\rb^-)$ and let $v(t)$ solve HNLS$_{q}^-$ with initial data $v_0$.
Let $G(x,t,z)=(G_1, G_2)^T$ be the (unique) solution of the spectral problem
$\Big(\partial_x - \Big(iz\sigma + \bsm 0 & v(t) \\ -\ovl{v(t)} & 0 \esm \Big)G = 0$
such that $G(x,t,z)e^{ixz/2} \to e_2$ as $x\to-\infty$, $z\in\ovl{\cb^+}$.
Let $P(x,t)$ be as above. Then by the proof of \eqref{eq:symmetry_ZS_AKNS22_1},
$$
G(x,t,z=0) = \frac{1}{z+i\beta_+(t=0)} \sigma_1 (z-\ovl{P(-x,t=0)}) \ovl{g(-x,-\ovl{z},t=0)}.
$$
In particular, we see that $G_2(0,iq,t=0)=0$
and it follows that $\Psi_1^0(x,iq;v_0) = c_1 G(x,iq,t=0)$, $x\leq 0$.
We are now in a similar situation to the case $g_2(0,iq,t=0)=0$, but for $x\leq 0$,
and we can conclude that $\mc R \mc B_q^- v(t)$ solves HNLS$_q^+$
with initial data $\mc R \mc B_q^- \mc R \mc B_{q}^+ u_0 = u_0$.
Hence $u(t) = \mc R \mc B_q^- v(t)$, which implies that $ \mc R \mc B_{q}^+ u(t)(=v(t))$ solves HNLS$_q^-$.
This shows that (i)$\Longrightarrow$(iii) in the case  $g_2(0,iq,t=0) \neq 0$,
and so completes the proof of the implication (i)$\Longrightarrow$(ii), (iii).

Conversely, if (ii) holds, then $\mc R \mc B_q^+ u(t)$ solves HNLS$_q^-$ with initial data  $\mc R \mc B_q^+ u_0$.
As we showed above, $u(t) = \mc R \mc B_q^- (\mc R \mc B_q^+ u(t))$ solves HNLS$_q^+$ with initial data $\mc R \mc B_q^- (\mc R \mc B_q^+ u_0)=u_0$, which proves (i).

If $q<0$, we consider the solution $g=g(x,z)$ of $(\partial_x - L)g=0$, $x\geq0$, $z\in\ovl{\cb^-}$,
with $g(x,z)e^{ixz/2}\to e_2$ as $x\to+\infty$.
Then, mutatis mutandis, the proof goes through as above.
\end{proof}

We will use the following definition.

\begin{definition}
Let $ u \in L^1_{\text{loc}}(\rb^+)$. Let $q \in \rb$ and define
$$ u^e(x) \equiv
\left\{
\begin{aligned}
& u(x), \ x \geq 0, \\
& \mc{R}\mc{B}_q^+ u(x), \ x < 0. \\
\end{aligned}
\right.
$$
Then, $u^e$ is called the \underline{B\"acklund extension} of $u\in L^1_{\text{loc}}(\rb^+)$ to $\rb$ with respect to $q$.
\end{definition}

\begin{remark}
We can also define the B\"acklund extension of $u\in L^1_{\text{loc}}(\rb^-)$:
$$ u^e(x) \equiv
\left\{
\begin{aligned}
& \mc{R}\mc{B}_q^- u(x), \ x \geq 0, \\
& u(x), \ x < 0. \\
\end{aligned}
\right.
$$
\end{remark}

The following procedure, which is motivated by Proposition \ref{prop:IBV_rel_backlund},
shows how to express the solution of the IVP \eqref{eq:nls_delta}
with even initial data $u_0\in H^{1,1}(\rb)$ in terms of a solution of NLS \eqref{eq:focusing_nls}.
In the generic case, the latter problem can be evaluated asymptotically as $t \to \infty$ using RHP/steepest descent methods.
In this way we are able to infer the long-time behavior of solutions of \eqref{eq:nls_delta}.

\begin{theorem}[Solution procedure for the IVP \eqref{eq:nls_delta} with even initial data; \cite{BT}]
\label{thm:Sol_proc}
Let $q \in \rb\setminus\{0\}$ be given and let $u_0$ be an even function in $H^{1,1}(\rb)$.
\vspace{0.1in}
\begin{enumerate}
\item[\emph{Step1.}] Set $u_0^+(x) = u_0(x)$, $x\geq0$.
Let $u_0^e$ be the B\"acklund extension of $u_0^+$.
\vspace{0.1in}

\item[\emph{Step2.}] Let $u^e(x,t)$ be the (unique, weak global) solution of NLS \eqref{eq:focusing_nls} in $H^{1,1}(\rb)$
with initial data $u^e(x,t=0)=u_0^e(x)$, in the sense of \eqref{eq:weak_sol_H_0}.
\vspace{0.1in}

\item[\emph{Step3.}] Set $u(x,t) \equiv  u^e(|x|,t)$.
\end{enumerate}
\vspace{0.1in}
Then $u^e|_{\rb^+}(x,t)$, $u(x,t)$ are the (unique, weak global) solutions of HNLS$_q^+$, \eqref{eq:nls_delta}
in the sense of \eqref{eq:weak_sol_H_q_plus}, \eqref{eq:weak_sol_H_q}, respectively.
\end{theorem}

Up till this point, what we have shown is that
if we have a solution $u^+(x,t)$ of HNLS$_q^+$ with certain technical properties,
then $u^+(x,t)$ should be constructible via the solution procedure in Theorem \ref{thm:Sol_proc}.
What we have to show, is that the solution procedure  in Theorem \ref{thm:Sol_proc}
indeed produce the solution $u^+(x,t)$ to HNLS$_q^+$ with $u^+(x,t=0)=u_0^+(x)\in H^{1,1}(\rb^+)$.
The proof of this fact involves various technicalities, and will be given
in the Proof of Theorem \ref{thm:Sol_proc} below at the end of this section.

By Theorem \ref{thm:Sol_proc}, if $u^+(x,t)$ solves HNLS$_q^+$,
then $u^e(x,t)$ is a solution of NLS on $\rb$ with the property that
$\mc B_q u^e(t)$ is also a solution of NLS.
The following proposition provides a converse to this statement.

\begin{proposition}
\label{prop:IBV_rel_backlund_conv}
Let $u(t)$ be a classical solution of NLS \eqref{eq:focusing_nls} on $\rb$ with $u(t=0)= u_0\in H^{1,1}(\rb)$.
Suppose that $a(i|q|)\neq 0$ where $a(z)$ is the scattering function for $u_0$.
If $\mc B_q u(t)$ solves NLS on $\rb$, then $u(t)|_{\rb^+}$ solves HNLS$_q^+$.
\end{proposition}

\begin{proof}
We only consider $q>0$. The other case $q<0$ is similar.
If $\beta_+(t=0) \neq \beta_-(t=0)$, then necessarily $\beta_+(t=0) = -q$ and $\beta_-(t=0) = q$.
As both $u(t)$ and $\mc B_q u(t)$ solve NLS,
from \eqref{eq:tilde_a_b}, $\frac{z-i\beta_-(t)}{z-i\beta_+(t)} a(z;u(t)) = a(z;\mc B_q u(t)) = a(z;\mc B_q u(t=0)) = \frac{z-i\beta_-(t=0)}{z-i\beta_+(t=0)} a(z;t=0) = \frac{z-iq}{z+iq} a(z;u(t=0)) = \frac{z-iq}{z+iq} a(z;u(t))$.
Thus, $\beta(t) = -q$ and $\beta_-(t) = q$ for all $t\geq0$ and hence $\tilde a(z;u(t))$ has a simple zero at $z=iq$.
As  $\gamma(iq;\mc B_q u(t)) = \gamma(iq;\mc B_q u(t=0))e^{\frac{i(iq)^2 t}2}$,
it follows from \eqref{eq:evolution_psi_1}\eqref{eq:evolution_psi_2} that
$$ \bal  -\frac{iq^2}2
& = \frac {\partial}{\partial t} \log (\tilde \gamma(iq,t))
= \frac {\partial}{\partial t} \log \Big(\frac{(\psi_1^+)_2(0,iq)}{(\psi_2^-)_2(0,iq)} \Big)\\
& = -\frac{iq^2}2 + \frac i2 (\ovl{u}_x + q \ovl{u}) \big|_{x=0} \Big( \frac{(\psi_1^+)_1}{(\psi_1^+)_2}-\frac{(\psi_2^-)_1}{(\psi_2^-)_2}\Big)(0,iq) \\
& = -\frac{iq^2}2 + \frac i2 (\ovl{u}_x + q \ovl{u}) \big|_{x=0} \frac{a(iq)}{(\psi_1^+)_2 (\psi_2^-)_2 (0,iq)} \\
\eal $$
and hence $u_x(0,t) + q u(0,t) = 0$. Thus $u(t)|_{\rb^+}$ solves HNLS$_q^+$.

If $\beta_+(t)=\beta_-(t)$ for all $t\geq0$, then from \eqref{eq:equiv_ZS_component_beta},
we see that either $(\psi_1^+)_2(0,iq,t)$ or $(\psi_2^-)_2(0,iq,t)$ is zero for all $t\geq0$.
As $0\neq a(iq)=\det(\psi_1^+, \psi_2^-)(0,iq,t)$,
$(\psi_1^+)_2(0,iq,t)$ and $(\psi_2^-)_2(0,iq,t)$ cannot vanish simultaneously.
But, these functions are continuous in $t$.
Thus, if $(\psi_1^+)_2(0,iq,t=0)$(respectively $(\psi_2^-)_2(0,iq,t=0)$) is zero,
then $(\psi_1^+)_2(0,iq,t)$(respectively $(\psi_2^-)_2(0,iq,t)$) is zero for all $t\geq0$.
If $(\psi_1^+)_2(0,iq,t) = 0$ for all $t\geq0$, then $(\psi_1^+)_1(0,iq,t) \neq 0$ for all $t\geq0$
and hence it follows from \eqref{eq:evolution_psi_1} that $u_x(0,t)+q u(0,t)=0$ for all $t\geq0$.
Similarly, if $(\psi_2^-)_2(0,iq,t)=0$ for all $t\geq0$,
then from \eqref{eq:evolution_psi_2}, $u_x(0,t)+q u(0,t)=0$ for all $t\geq0$.
Thus, we conclude that $u(t)|_{\rb^+}$ solves HNLS$_q^+$.
\end{proof}

The question arises whether the B\"acklund extension  method described above
provides the only way to extend the solution $u^+(t)$ of HNLS$_q^+$ to a solution $u^e(t)$ of NLS on the whole line
such that $u^e(t)|_{\rb^+} = u^+(t)$.
If $v(t)$ is any such extension, we must have
\beq
\label{eq:cond_extension}
\bal
&v(0-,t)=v(0+,t)=u^+(0+,t), \\
&v_x(0-,t)=v_x(0+,t)=u^+_x(0+,t). \\
\eal
\eeq
Thus the question reduces to showing that the solution of NLS on $\rb^-$ is uniquely specified by \eqref{eq:cond_extension}.
But this is true, as can be seen, for example from the work of Isakov \cite{Is}.
Thus the B\"acklund extension method is the \emph{only} way
in which one can solve HNLS$_q^+$ by extension to a solution of NLS on $\rb$.

We now develop further properties of the B\"acklund extensions.
\begin{definition}
Let $q\in \rb$ and let $u \in L^1_{\text{loc}}(\rb)$.
We say that $u$ is \underline{$q$-symmetric} if
\begin{equation}
\label{def:symm}
u = \mc{R}\mc{B}_q u.
\end{equation}
\end{definition}

\begin{remark}
\label{rmk:symmetric_extension}
 Note that if $u(x)$ is $q$-symmetric, then $u$ is the B\"acklund extension of $u|_{\rb^+}$.
\end{remark}

\begin{proposition}
\label{prop:backlund_ext}
The B\"acklund extension of a function $u\in L^1_{\text{loc}}(\rb)$ is $q$-symmetric.
\end{proposition}
\begin{proof}
By \eqref{eq:B_halfline} and Lemma \ref{lem:reciprocal_backlund}(ii),
$$
(\mc R \mc B_q u^e)|_{\rb^+} = \mc R ((\mc B_q u^e)|_{\rb^-}) = \mc R (\mc B^-_q (u^e|_{\rb^-}))
= \mc R \mc B_q^- \mc R \mc B_q^+ u = u,
$$
and
$$
(\mc R \mc B_q u^e)|_{\rb^-} = \mc R ((\mc B_q u^e)|_{\rb^+}) = \mc R (\mc B^+_q (u^e|_{\rb^+})) = \mc R \mc B_q^+ u. $$
\end{proof}

\begin{lemma}
\label{cor:limit_P}
Let $ u(x) \in H^{1,1}(\rb)$.
If $u$ is $q$-symmetric, then $\beta \equiv \beta_+ = \beta_-$.
In other words, $P(x) \to -i\beta \sigma_3$ as $|x| \to \infty$ where $\beta^2 = q^2$.
\end{lemma}
\begin{proof}
From the proof of Lemma \ref{lem:reciprocal_backlund}, $P(x)=\sigma_3 P(-x) \sigma_3$ and the result follows from Lemma \ref{lem:P_tail}.
\end{proof}

The above calculations together with Proposition \ref{prop:scattering_symmetry} imply the following.
\begin{proposition}
\label{prop:auto_backlund}
Suppose that $u(x)$ is $q$-symmetric. Let $P(x)$ solve \eqref{eq:backlund} with $Q(x) = \bsm 0 & u(x)\\-\ovl{u(x)}&0 \esm$
and let $\psi(x,z)$ be an eigensolution of \eqref{eq:scattering_problem}. Set $\tilde\psi(x,z) = (z+P(x))\psi(x,z)$. Then,
$$
\hat\psi(x,z) \equiv  \sigma_3\tilde\psi(-x,-z)= \sigma_3(-z+P(-x))\psi(-x,-z)
$$
and
$$
\psi^\#(x,z) \equiv  (-i)\ovl{\sigma_2 \hat\psi(x,\ovl{z})} = \sigma_1 (z-\ovl{P(-x)})\ovl{\psi(-x,-\ovl{z})}
$$
also solve \eqref{eq:scattering_problem} where $\sigma_1 = \bsm 0&1\\1&0\esm$ is the first Pauli matrix.
Thus the maps $\psi \mapsto \hat\psi$ and $\psi \mapsto \psi^\#$ are auto-B\"acklund transformations for the equation \eqref{eq:scattering_problem}.
In particular,
\beq
\label{eq:ZS_symmetry_0}
\psi^\pm (x,z) = \sigma_3 (z-P(-x))\psi^\mp(-x,-z) (z+i\beta\sigma_3)^{-1} \sigma_3, \ \ z\in\rb
\eeq
and
\beq
\label{eq:ZS_symmetry}
\psi^\pm (x,z) = \sigma_1 (z-\ovl{P(-x)})\ovl{\psi^\mp(-x,-z)} (z-i\beta\sigma_3)^{-1} \sigma_1, \ \ z\in\rb
\eeq
where $\psi^\pm(x,z)$ are the associated ZS-AKNS solutions and $P(x)\to-i\beta\sigma_3$ as $x\to \pm\infty$.
\end{proposition}

\begin{remark}
\label{rmk:symmetry_ZS_AKNS22}
As $\psi_1^+(x,z)$, $\psi_2^-(x,z)$ continue analytically into $\cb^+$,
for $z\in\cb^+\setminus\{- i\beta\}$,
\beq
\label{eq:symmetry_ZS_AKNS22_1}
\psi_1^+(x,z) = \frac{1}{z+i\beta} \sigma_1 (z-\ovl{P(-x)}) \ovl{\psi_2^-(-x,-\ovl{z})},
\eeq
and for $z\in\cb^+\setminus\{i\beta\}$,
\beq
\label{eq:symmetry_ZS_AKNS22_2}
\psi_2^-(x,z) = \frac{1}{z-i\beta} \sigma_1 (z-\ovl{P(-x)}) \ovl{\psi_1^+(-x,-\ovl{z})}.
\eeq
If $\beta>0$, say, it is easy to check that the apparent singularity in \eqref{eq:symmetry_ZS_AKNS22_2} as $z\to i\beta$,
is in fact removable, as it should be, etc.
\end{remark}

\begin{lemma}
\label{lem:symm_prop}
Suppose that $ u(x) \in H^{1,1}(\rb)$ is $q$-symmetric.
Then, the scattering function $a(z)$ of $u(x)$ does not vanish at $z=i|q|$.
\end{lemma}
\begin{proof}
Let $\psi_1^+$,  $\psi_2^-$ be the associated ZS-AKNS solutions and let $P(x)$ solve \eqref{eq:backlund}.
By Lemma \ref{cor:limit_P}, $P(x) \to -i\beta \sigma_3$ as $x \to \pm\infty$ where $\beta^2 = q^2$.
Suppose first that $\beta = -|q|<0$.
If $a(i|q|)=0$, necessarily $\psi_1^+(x,i|q|) = \gamma \psi_2^-(x,i|q|)$ for some constant $\gamma \neq 0$.
But then, it follows from \eqref{eq:symmetry_ZS_AKNS22_2} that
$$
\psi_1^+(0,i|q|) = \gamma \psi_2^-(x,i|q|)
= \frac{\gamma}{2i|q|} \sigma_1(i|q|-iq\sigma_3)\ovl{\psi_1^+(0,i|q|)}.
$$
As $i|q|-iq\sigma_3$ is either $\bsm 0&0\\0&2iq \esm$ or $\bsm -2iq&0\\0&0 \esm$,
we conclude that $\psi_1^+(0,i|q|)$ $= 0$, which is a contradiction.
If $\beta_+ = |q|>0$, we utilize \eqref{eq:symmetry_ZS_AKNS22_1} and the proof is similar.
\end{proof}

\begin{remark}
\label{rmk:back_ext_not_vanish}
It follows by Proposition \ref{prop:backlund_ext} and Lemma \ref{lem:symm_prop} that
the scattering function $a(z)$ of any B\"acklund extension of $u\in H^{1,1}(\rb^+)$ does not vanish at $z=i|q|$.
\end{remark}

$\mc B_q$, $\mc B_q^\pm$ are not continuous in $H^{1,1}(\rb)$, $H^{1,1}(\rb^\pm)$, respectively.
In other words, a small perturbation of $u(x)$ may not result in a small perturbation of the B\"acklund transformation $ \tilde{u}(x)$.
For example, consider the case when $q>0$ and $u(x) = 0$ on $\rb^+$.
Let $f_\lambda(x) = v_{\lambda}(x)$, $x\geq0$, where $v_{\lambda}$ is defined by \eqref{E:gs}.
Then, $\|f_\lambda\|_{H^{1,1}(\rb^+)} \rightarrow 0$ as $\lambda \downarrow q$.
From Remark \ref{rmk:1_soliton}, $\|\widetilde{f}_{\lambda}\|_{H^{1,1}(\rb^+)} = \|\eta_{\lambda}\|_{H^{1,1}(\rb^-)}$,
but $\|\eta_{\lambda}\|_{H^{1,1}(\rb^-)}$ is not small for any $\lambda > q > 0$.
In the case when $q<0$ and  $u(x)=v_\mu |\mathbb{R}^+$, there is a sequence of 2-solitons with symmetry conditions which converges to $v_\mu$ on $\rb^+$.

The following lemma, however, shows that $\mc B_q$, $\mc B_q^\pm$ are continuous at $u\in H^{1,1}(\rb^+)$
for which $g_2(0,iq;u) \neq 0$,
where $g(x,z)=(g_1, g_2)^T$, $x\geq0$, $z\in\ovl{\cb^+}$ be as in the proof of Proposition \ref{prop:bijectivity_backlund}.
\begin{lemma}
\label{lem:conv_backlund_ext}
Let $q>0$ be given.
Let $\{u_n\}$ be any sequence in $H^{1,1}(\rb^+)$ such that $u_n \to u$ as $n\to\infty$.
If $g_2(0,iq) \neq 0$, then $\mc R \mc B_q^+ u_n \to \mc R \mc B_q^+ u$.
In particular, $u_n^e \to u^e$ where  $u_n^e$, $u^e$ are the B\"acklund extensions of $u_n$, $u$, respectively.
\end{lemma}

\begin{proof}
Note that $g(x,z;u) = \psi_1^+(x,z;u^e)$, $x\geq0$.
From Remark \ref{rmk:back_ext_not_vanish}, we have $a(iq;u^e) \neq 0$,
 and hence $\Psi_1^0(x,iq;u^e)=c_1(u^e) \psi_1^+(x,iq;u^e)+ c_2(u^e) \psi_2^-(x,iq;u^e)$,
where $c_1(u^e) = (\psi_2^-)_2(0,iq;u^e)/a(iq;u^e)$, $c_2(u^e) = -(\psi_1^+)_2(0,iq;u^e)/a(iq;u^e)$.
Note that $c_2(u^e)$ is non-zero.
Let $\xi_n(x) = ((\xi_n)_1, (\xi_n)_2)^T = \Psi_1^0(x,iq;u_n^e) e^{-q x/2}$ and
$\xi(x) = (\xi_1, \xi_2)^T = \Psi_1^0(x,iq;u^e) e^{-q x/2}$.
From \eqref{eq:u_tilde}, we see that
$$
\mc B_q^+ u_n = u_n - 2q \frac {(\xi_n)_1\overline{(\xi_n)_2}}{|(\xi_n)_1|^2 + |(\xi_n)_2|^2}
$$
and
$$
\mc B_q^+ u = u - 2q \frac {\xi_1\overline{\xi_2}}{|\xi_1|^2 + |\xi_2|^2}.
$$
By Lemma \ref{lem:perturbed_potential}, we have $a(iq;u_n^e) \to a(iq;u^e)$
and hence $c_1(u_n^e) \to c_1(u^e)$, $c_2(u_n^e) \to c_2(u^e)$.
Set $m_2^-(x,z) = ((m_2^-)_1, (m_2^-)_2)^T \equiv \psi_2^-(x,z) e^{ixz/2}$.
Again by Lemma \ref{lem:perturbed_potential}, we have $\psi_1^+(x,iq;u_n^e) \to \psi_1^+(x,iq;u^e)$ in $H^{1,1}(\rb^+)$,
$(m_2^-)_1(x,iq;u_n^e) \to (m_2^-)_1(x,iq;u^e)$ in $H^{1,1}(\rb^+)$,
$(m_2^-)_2(x,iq;u_n^e) \to (m_2^-)_2(x,iq;u^e)$ in $L^\infty(\rb^+)$,
and $\partial_x (m_2^-)_2(x,iq;u_n^e) \to \partial_x(m_2^-)_2(x,iq;u^e)$ in $L^2(\rb^+)$,
As $|\xi_1(x)|^2 + |\xi_2(x)|^2 \geq c >0$ for all $x\geq0$,
we conclude that $\mc B_q^+ u_n \to \mc B_q^+ u$ in $H^{1,1}(\rb^+)$.
\end{proof}

The following fact is basic.
\begin{proposition}
\label{prop:mixed_cond_x_0}
If $u\in C^1(\rb)$ is $q$-symmetric, then
\beq
\label{eq:mixed_cond_x_0}
u'(0) + q u(0) = 0.
\eeq
In other words, $q$-symmetry yields the boundary condition.
\end{proposition}

\begin{proof}
Observe that $\tilde{Q}(0) = Q(0) - i [\sigma,P(0)] = Q(0)$ and
$$
\begin{aligned}  \tilde{Q}_x(0) + Q_x(0)
& = 2Q_x(0) - i[\sigma,P_x(0)] \\
& = 2Q_x(0) - i[\sigma,\tilde{Q}(0)P(0) -P(0)Q(0)] \\
& = 2Q_x(0) - i[\sigma,iq[\sigma_3, Q(0)]] \\
& = 2(Q_x(0)+q Q(0)). \\
\end{aligned}
$$
As $\tilde{Q}_x(0) = -Q_x(0)$ by the $q$-symmetry condition, the equation \eqref{eq:mixed_cond_x_0} follows.
\end{proof}


\begin{proposition}
\label{lem:backlund_ext_c2}
 If $u\in C^2(\rb^+)$ and $u'(0+) + q u(0) = 0$,
 then the B\"acklund extension of $u$ belongs to $C^2(\rb)$.
\end{proposition}

\begin{proof}
From the proof of Proposition \ref{prop:mixed_cond_x_0},
we have $u^e(0-)=u^e(0+)$ and $(u^e)'(0-)=(u^e)'(0+)$. For $u''(0)$,
 $$ \begin{aligned} \partial^2_x Q^e(0-) - \partial^2_x Q^e(0+)
& = \partial^2_x \tilde{Q}(0) - \partial^2_x Q(0) \\
& = - i[\sigma,\partial^2_x P](0)\\
 & = - i[\sigma, \tilde{Q}_x P -P Q_x + \tilde{Q} P_x - P_x Q](0) \\
 & = - i[\sigma, -Q_x P -P Q_x - [P_x, Q](0) \\
 & = - i[\sigma, -iq [[\sigma_3,Q(0)],Q(0)] =0.\\
 \end{aligned} $$
\end{proof}


We need the following lemma to prove Proposition \ref{p:symm_scattering} below.
\begin{lemma}
\label{lem:scattering_invertible}
Let $\Psi(x,z)$ be any invertible solution to \eqref{eq:scattering_problem} for $z\in\rb$. Then,
$$
S(z) = \lim_{x \to \infty}e^{ixz\sigma}\Psi(-x,z)\Psi^{-1}(x,z)e^{ixz\sigma}.
$$
\end{lemma}
\begin{proof}
Let $\psi^\pm = (\psi^\pm_1, \psi^\pm_2)$ be the ZS-AKNS solutions normalized at $\pm\infty$, respectively.
As $\psi^+$ solves \eqref{eq:scattering_problem} and is invertible, $\Psi(x,z) = \psi^+(x,z)M(z)$ for some invertible matrix $M(z)$.
$$ \begin{aligned}
& \lim_{x \to \infty}e^{ixz\sigma}\Psi(-x,z)\Psi^{-1}(x,z)e^{ixz\sigma} \\
& \quad = \lim_{x \to \infty}e^{ixz\sigma}\psi^+(-x,z)(\psi^+(x,z))^{-1}e^{ixz\sigma} \\
& \quad = \lim_{x \to \infty}e^{ixz\sigma}\psi^-(-x,z)S(z)(\psi^+(x,z))^{-1}e^{ixz\sigma} \\
\end{aligned} $$
Since $\psi^-(-x,z) e^{ixz\sigma}, \psi^+(x,z)e^{-ixz\sigma} \to I$ as $x \to \infty$, we obtain the result.
\end{proof}

For $u\in H^{1,1}(\rb)$, the $q$-symmetry condition \eqref{def:symm} can be reformulated in terms of the scattering data of $u$ as follows.
\begin{proposition}
\label{p:symm_scattering}
Let $q \in \rb\setminus\{0\}$ be given.
Suppose that $u\in \mc G \subset H^{1,1}(\rb)$ is generic
with scattering data $a(z)$, $b(z)$, $Z_+ =\{z_1, \cdots, z_n\}$, $\Gamma_+=\{\gamma(z_1), \cdots, \gamma(z_n)\}$.
Then, $u$ is $q$-symmetric if and only if
\begin{equation}
\label{eq:symm_scattering}
\left\{
\begin{aligned}
&\overline{a(- \overline{z})} = a(z), \ \ z \in \overline{\cb^+} ,\\
&b(-z) = b(z) \frac{z+i\beta}{z-i\beta}, \ \ z \in \rb, \\
&z_k \neq \pm i\beta \text{ and } \gamma (z_k) \overline{\gamma(-\overline{z_k})} = \frac{z_k - i\beta}{z_k + i\beta} , \ \ k=1,\dots,n, \\
&\beta = (-1)^n q. \\
\end{aligned}
\right.
\end{equation}
\end{proposition}

\begin{remark}
 Note that if $u$ is $q$-symmetric, then $a(i|q|)\neq 0$ by Lemma \ref{lem:symm_prop}.
 Hence $z_k \neq \pm i\beta$. For if $z_k=\pm i\beta$ then $z_k^2=-\beta^2=-q^2$
 and so $z_k=i|q|$, which contradicts $a(i|q|)\neq 0$.
\end{remark}

\begin{remark}
\label{rmk:symmetry_r_backlund}
For $z\in\rb$, we have
$$
|r(z)| = \bigg| \frac{b(z)}{a(z)} \bigg| =  \bigg| \frac{b(-z)}{a(-z)} \bigg| = |r(-z)|.
$$
By setting $z=0$ in the symmetry condition for $b(z)$, it follows that $b(0)=0$ and hence $r(0)=0$.
\end{remark}

\begin{proof}
Suppose first that $u(x)$ is $q$-symmetric.
From \eqref{eq:ZS_symmetry}, it follows that for $z\in\rb$,
\beq \bal S(z)
& = (\psi^-)^{-1}\psi^+(x,z)\\
& = \sigma_1 (z-i\beta\sigma_3) \ovl{(\psi^+)^{-1}\psi^-(-x,-z)} (z-i\beta\sigma_3)^{-1} \sigma_1\\
& = \sigma_1 (z-i\beta\sigma_3) \ovl{S^{-1}(-z)} (z-i\beta\sigma_3)^{-1} \sigma_1.\\
\eal \eeq
As $S(z) = \bsm a(z)& -\ovl{b(z)}\\ b(z)& \ovl{a(z)} \esm$, we conclude that $a(z) = \ovl{a(-z)} \text { and } b(-z) = b(z)\frac{z+i\beta}{z-i\beta}$, $z\in\rb$
and hence by analytic continuation,
\beq
\label{eq:a_b_symmetry}
a(z) = \ovl{a(-\ovl{z})}, \ \ z\in \ovl{\cb^+} \text { and } b(-z) = b(z)\frac{z+i\beta}{z-i\beta}, \ \ z\in\rb.
\eeq
Let $\Psi^0(x,z)$ be the solution of \eqref{eq:scattering_problem} with $\Psi^0(0,z)=I$.
By Proposition \ref{prop:auto_backlund}, there exists an $M_0(z)$ such that
$\sigma_3 (-z+P(-x))\Psi^0(-x,-z) = \Psi^0(x,z)M_0(z)$.
Setting $x=0$, $z=0$, we obtain $M_0(0) = -iq I$ and hence $\sigma_3 P(-x)\Psi^0(-x,0) = -iq \Psi^0(x,0)$.
Thus, by Lemma \ref{lem:scattering_invertible},
$$ S(0) = \lim_{x \to \infty}\Psi^0(-x,0)(\Psi^0)^{-1}(x,0) = \lim_{x \to \infty} -iq (\sigma_3 P(-x))^{-1} = \frac{q}{\beta}I $$
On the other hand, from \eqref{eq:a_b_symmetry}, $a(0) = \ovl{a(0)}$, $b(0)=0$ and so $S(0) = a(0)I$.
Note that $a(z)$ has the explicit formula \eqref{eq:a_exp_formula} (recall $|r|^2 = |b|^2/|a|^2, |a|^2+|b|^2=1$),
\beq
\label{eq:formula_a}
a(z) = \prod_{k=1}^n \frac{z-z_k}{z-\overline{z_k}}\exp \Big[\frac{1}{2\pi i} \int_{\rb} \frac{\log(1-|b(s)|^2)}{s-z}\rd s\Big], \ \  z\in\cb^+.
\eeq
Using this formula together with $b(0)=0$ and $|b(-z)|=|b(z)|$ from \eqref{eq:a_b_symmetry},
we see that $a(0) = \prod_{k=1}^n z_k/\prod_{k=1}^n \ovl{z_k}$.
In computing $a(0)=\lim_{z\to 0, z\in\cb^+} a(z)$ via \eqref{eq:formula_a},
we have used the fact that $b(z)\in H^{1,1}\subset H^{1,0}$ (see \eqref{eq:b_H11})
and so $|b(z)|^2 = O(|z|)$ as $z\to 0$, $z\in\rb$.
As $a(z_k)=0$ if and only if $a(-\ovl{z_k})=0$, $a(0) = \prod_{k=1}^n (-\ovl{z_k})/\prod_{k=1}^n \ovl{z_k}=(-1)^n$.
Thus, $\beta=q/a(0) = (-1)^n q$.
Now using \eqref{eq:symmetry_ZS_AKNS22_2}, we see that
$$
\bal \psi_1^+(x,z_k)
& = \gamma(z_k) \psi_2^-(x,z_k) \\
& =  \frac{\gamma(z_k)}{z_k-i\beta} \sigma_1 (z_k-\ovl{P(-x)}) \ovl{\psi_1^+(-x,-\ovl{z_k})} \\
& =  \frac{\gamma(z_k)\ovl{\gamma(-\ovl{z_k})}}{z_k-i\beta} \sigma_1 (z_k-\ovl{P(-x)}) \ovl{\psi_2^-(-x,-\ovl{z_k})} \\
\eal
$$
Comparing with \eqref{eq:symmetry_ZS_AKNS22_1},
we obtain $\gamma(z_k)\ovl{\gamma(-\ovl{z_k})}= \frac{z_k - i\beta}{z_k+i\beta}$.

Conversely, suppose that the scattering data for $u(x)$ satisfy the symmetries \eqref{eq:symm_scattering}.
Denote by $m(x,z)=(m_1(x,z), m_2(x,z))$ the solution to the normalized RHP with the jump matrix
$$v_x(z) = \bpm 1+|r_x(z)|^2 & r_x(z) \\ \overline{r_x(z)} & 1 \epm,
\ \ r_x(z) \equiv  r(z)e^{ixz}.
$$
Define
\begin{equation}
\label{eq:def_B}
{\mathfrak{b}}(x) = (m_1(x,i\beta), m_2(x,-i\beta)).
\end{equation}
From the relation in Proposition \ref{prop:scattering_symmetry}(i), we see that $m_2(x,-i\beta)=i\sigma_2 \ovl{m_1(x,i\beta)}$.
Hence $\det b(x) = |m_1(x,i\beta)|^2 >0$ and so $b(x)$ is invertible. Set
\beq
\label{eq:def_U0_Ux}
U_0 = -i\beta\sigma_3, \ U(x) = {\mathfrak{b}}(x)U_0 {\mathfrak{b}}(x)^{-1},
\eeq
and define
\begin{equation}
\label{eq:def_symm_RHP}
{\breve m}(x,z) \equiv
\left\{
\begin{aligned}
& \sigma_1 (z-U(x))^{-1} m(x,z) (z-U_0) a(z)^{\sigma_3} \sigma_1 \ , \ \ z \in \cb^+\setminus\{i |q|\},\\
& \sigma_1 (z-U(x))^{-1} m(x,z) (z-U_0) \ovl{a(\ovl{z})}^{-\sigma_3} \sigma_1 \ , \ z \in \cb^-\setminus\{-i |q|\}.\\
\end{aligned}
\right.
\end{equation}

\vspace{0.1in}
\noindent \emph{Claim}: $m(x,z)=\ovl{\breve m(-x,-\ovl{z})}$.
\vspace{0.1in}

\noindent We prove this claim in a number of steps. First, for $z \in \rb$,
$$
 \begin{aligned} ({\breve{m}_-})^{-1} {\breve{m}_+} (x,z)
& = \sigma_1 \overline{a(z)}^{\sigma_3} (z-U_0)^{-1} v_x(z) (z-U_0) a(z)^{\sigma_3} \sigma_1 \\
& = \sigma_1 \bpm 1  & r_x(z) \frac{z-i\beta}{z+i\beta} \frac{\overline{a(z)}}{a(z)} \\ \overline{r_x(z)} \frac{z+i\beta}{z-i\beta} \frac{a(z)}{\overline{a(z)}} & 1+|r_x(z)|^2 \epm \sigma_1 \\
& = \bpm 1+|r_{-x}(-z)|^2 & \overline{r_{-x}(-z)} \\ r_{-x}(-z) & 1 \epm.\\
\end{aligned}
$$
Here, we used the symmetries \eqref{eq:symm_scattering} for the scattering data.
Set ${m^\#}(x,z) \equiv \overline{{\breve m}(-x,-\overline{z})}$. Then,
$$
({m^\#}_-)^{-1} {m^\#}_+ (x,z) = \ovl{({\breve{m}_-})^{-1} {\breve{m}_+} (-x,-z)} = v_x(z).
$$
Observe that for $z\in\cb\setminus (\rb \cup \{ \pm i|q| \}$,
$$
(z-U(x))^{-1} m(x,z) (z-U_0) = {\mathfrak{b}} \bpm ({\mathfrak{b}}^{-1}m)_{11}&({\mathfrak{b}}^{-1}m)_{12}\frac{z-i\beta}{z+i\beta}\\ ({\mathfrak{b}}^{-1}m)_{21} \frac{z+i\beta}{z-i\beta}&({\mathfrak{b}}^{-1}m)_{22} \epm.
$$
Since $({\mathfrak{b}}^{-1}m)_{21}(x,i\beta) = 0$ and $({\mathfrak{b}}^{-1}m)_{12}(x,-i\beta) = 0$ for $x \in \rb$, $m^\#(x,z)$ does not have a pole at $z=\pm i \beta$ in the $z$-plane.
As $u\in\mc G$, for any pole $z_k\in\cb^+$,
$$
\Res\displaylimits_{z = z_k} m(x,z)= \lim_{z \to z_k} (z-z_k) m(x,z) = \lim_{z \to z_k} m(x,z) \bpm 0&0\\ c_x(z_k) &0 \epm.
$$
where $c_x(z_k) = e^{-ixz_k} c(z_k)$ and $c(z_k) = \frac{ \gamma(z_k)}{a'(z_k)}$.
Equivalently,
\beq
\label{eq:residue_equiv}
\lim_{z \to z_k} a(z) m_1(x,z) = a'(z_k) m_2(x,z_k) c_x(z_k) = \gamma(z_k) m_2 (x,z_k) e^{-ixz_k}.
\eeq
For $m^\#(x,z)$, we have
\beq
\label{eq:residue_m_sharp}
 \begin{aligned}
& \ovl{\Res\displaylimits_{z = -\ovl{z_k}} m^\#(-x,z)} \\
& =\lim_{z \to -\ovl{z_k}} \ovl{(z-(-\ovl{z_k}) m^\#(-x,z)} =\lim_{z \to z_k} -(z-z_k) \ovl{m^\#(-x,-\ovl{z})}\\
& = \lim_{z \to z_k} -\sigma_1 (z-U(x))^{-1} (z-z_k) m(x,z) a(z)^{\sigma_3}(z-U_0) \sigma_1 \\
& = \lim_{z \to z_k} -\sigma_1 (z-U(x))^{-1} (z-z_k) \bpm a(z) m_1(x,z) & \frac {m_2(x,z)}{a(z)}  \epm (z-U_0) \sigma_1 \\
& = -\sigma_1 (z_k -U(x))^{-1} \bpm \frac {z_k-i\beta}{a'(z_k)} m_2(x,z_k) & 0 \epm \\
\end{aligned}
\eeq
Thus at $-\ovl{z_k}$, the second column of $m^\#(x,z)$ is analytic, and the first column has a simple pole.
A similar calculation shows that at $-z_k$, the first column of $m^\#(x,z)$ is analytic, and the second column has a simple pole.
On the other hand, if we set $c^\#_x(z_k) = e^{-ixz_k} c^\#(z_k)$, then
$$ \begin{aligned}
& \lim_{z \to -\ovl{z_k}} \ovl{m^\#(-x,z) \bpm 0&0\\ c^\#_{-x}(-\ovl{z_k})&0 \epm}
=\lim_{z \to z_k} \ovl{m^\#(-x,-\ovl{z})} \bpm 0&0\\ \ovl{c^\#_{-x}(-\ovl{z_k})}&0 \epm\\
& = \lim_{z \to z_k}\sigma_1 (z -U(x))^{-1} m(x,z) a(z)^{\sigma_3}(z-U_0) \sigma_1 \bpm 0&0\\ \ovl{c^\#_{-x}(-\ovl{z_k})}&0 \epm\\
& =  \lim_{z \to z_k} \sigma_1 (z -U(x))^{-1} \bpm \ovl{c^\#_{-x}(-\ovl{z_k})} (z + i\beta)a(z) m_1(x,z) & 0\epm \\
& = \sigma_1 (z_k -U(x))^{-1} \bpm \ovl{c^\#(-\ovl{z_k})} (z_k + i\beta) \gamma(z_k) m_2(x,z_k) & 0\epm , \\
\end{aligned} $$
Here we have used \eqref{eq:residue_equiv}.
Comparing with \eqref{eq:residue_m_sharp} and using \eqref{eq:symm_scattering}, we have
$$
c^\#(-\ovl{z_k}) =  -\ovl{\frac 1{a'(z_k) \gamma(z_k)}\frac{z_k - i\beta}{z_k + i\beta}} = -\frac {\gamma(-\ovl{z_k})}{\ovl{a'(z_k)}}.
$$
As $a(z)=\ovl{a(-\ovl{z})}$, $\frac {a(z)}{z-z_k}=-\ovl{\big(\frac{a(-\ovl{z})}{-\ovl{z}-(-\ovl{z_k})}\big)}$ and so $a'(z_k) = -\ovl{a'(-\ovl{z_k})}$. Hence it follows that $c^\#(-\ovl{z_k}) = \frac {\gamma(-\ovl{z_k})}{a'(-\ovl{z_k})} = c(-\ovl{z_k})$ and so $c^\#(z_k)=c(z_k)$ for $k=1, \cdots, n$.
Assembling the above results, we conclude by the uniqueness for the RHP that ${m^\#}(x,z) = m(x,z)$, which proves the above Claim.

Now fix $x=0$ and write ${\mathfrak{b}} = {\mathfrak{b}}(0)$ and $m(z)=m(0,z)$, etc.
If $q>0$ and $n$ is even, then $i\beta \in \cb^+$ and
$$ \begin{aligned} \overline{m_{21}(i\beta)}
& = \overline{m^\#_{21}(i\beta)} = \breve{m}_{21}(i\beta) \\
& = \lim_{z \to i\beta} \big({\mathfrak{b}}(z-U_0)^{-1} {\mathfrak{b}}^{-1}m(z)(z-U_0) a(z)^{\sigma_3}\big)_{12} \\
& = \frac {\mf b_{12}}{a(i\beta)} ({\mathfrak{b}}^{-1}m)_{22}(i\beta) \\
& = \frac {m_{12}(-i\beta)}{a(i\beta)} \frac{1}{\det {\mathfrak{b}}} (-m_{21}m_{12} + m_{11} m_{22})(i\beta) \\
& = -\frac {\overline{m_{21}(i\beta)}}{a(i\beta) \det {\mathfrak{b}}}. \\
\end{aligned}
$$
Let $Z_+^{(1)} = Z_+ \cap i\rb$ and $Z_+^{(2)} = Z_+ \cap \{z\in \cb^+ : \textnormal{Re} z >0\}$.
If $z_k=i\xi \in Z_+^{(1)}$, then by the symmetries \eqref{eq:symm_scattering},
$0<|\gamma(i\xi)|^2 = \frac{\xi -\beta}{\xi+\beta}$ and so $\frac {i\beta-i\xi}{i\beta+i\xi}<0$.
If $z_k \in Z_+^{(2)}$, then $-\ovl{z_k} \in Z_+$ and
$\frac {i\beta-z_k}{i\beta-\ovl{z_k}} \frac {i\beta-(-\ovl{z_k}) }{i\beta-(\ovl{-\ovl{z_k})}}=|\frac {i\beta-z_k}{i\beta-\ovl{z_k}}|^2>0$.
Since $n$ is even, $|Z_+^{(1)}|$ is even and hence we see that $\prod_{k=1}^n \frac{i\beta-z_k}{i\beta-\overline{z_k}}>0$.
On the other hand, as $|b(-z)|=|b(z)|$, $z\in\rb$, we see $\int_{\rb} \frac{\log(1-|b(s)|^2)}{s-i\beta}\rd s \in i\rb$,
so it follows from the formula \eqref{eq:formula_a} that $a(i\beta)>0$.
Since $\det {\mathfrak{b}} = |m_1(i\beta)|^2 > 0$, it follows that
$$
m_{21}(0,iq) = m_{21}(0,i\beta) = 0.
$$
It follows in a similar way that
$$
\begin{aligned}
& m_{21}(0,iq) = m_{21}(0,i\beta) = 0, \ \ \text{if }q<0, \ \ n\text{ is even,} \\
& m_{11}(0,-iq) = m_{11}(0,i\beta) = 0, \ \ \text{if }q>0, \ \ n\text{ is odd,} \\
& m_{11}(0,-iq) = m_{11}(0,i\beta) = 0, \ \ \text{if }q<0, \ \ n\text{ is odd.} \\
\end{aligned}
$$
Set $\psi(x,z) \equiv m(x,z)e^{ixz\sigma}$.
If n is even,
$$
 \Psi_1^0 (x,iq) = \psi(x,iq)\psi(0,iq)^{-1} e_1 = c_1 \psi_1(x,iq) = c_1 \psi_1(x,i\beta),
$$
for some nonzero constant $c_1$.
If n is odd,
$$
\Psi_2^0 (x,-iq) = \psi(x,-iq)\psi(0,-iq)^{-1} e_2 = c_2 \psi_1(x,-iq) = c_2 \psi_1(x,i\beta),
$$
for some nonzero constant $c_2$.
Thus, using the symmetries $\Psi_2^0(x,\ovl{z}) = i\sigma_2 \ovl{\Psi_1^0(x,z)}$, etc., we obtain
$$
\varphi(x) = (\Psi_1^0 (x,iq), \Psi_2^0 (x,-iq)) =
\left\{
\begin{aligned}
& e^{-\frac {\beta x}2} {\mathfrak{b}}(x)\bpm c_1 & 0 \\ 0 & \overline{c_1} \epm, \text{ if } \beta=q,\\
& e^{-\frac {\beta x}2} {\mathfrak{b}}(x)\bpm 0 & c_2 \\ -\overline{c_2} & 0 \epm, \text{ if } \beta=-q. \\
\end{aligned}
\right.
$$
In both cases we see that
$$
U(x) = {\mathfrak{b}}(x) U_0 {\mathfrak{b}}(x)^{-1} = \varphi(x) (-iq \sigma_3) \varphi(x)^{-1} = P(x).
$$
where $P(x)$ solves \eqref{eq:backlund}.
Let $m(x,z) = I + \frac {m_1(x)}z + o(\frac 1z)$ and $a(z) = 1 + \frac {a_1}z + o(\frac 1z)$ be the expansions as $z \to \infty$, $|\textnormal{Im} z|> c|\textnormal{Re} z|$, $c>0$.
As $\breve m(x,z) = \ovl{m(-x,-\ovl z)}$, it follows from \eqref{eq:def_symm_RHP} that
$$
-\overline{m_1(-x)} = \sigma_1 (P(x) + m_1(x) + a_1 \sigma_3 - U_0)\sigma_1.
$$
From the symmetries $m(x,z) = \ovl{\sigma_2 m(x,\ovl z) \sigma_2}$, $P(x) = \ovl{\sigma_2 P(x) \sigma_2}$,
we see that $(m_1)_{12}(x) = -\ovl{(m_1)_{21}(x)}$ and $P_{12}(x) = -\ovl{P_{21}(x)}$.
Hence, it follows that
$$ \begin{aligned} u(x)
& = -i (m_1)_{12}(x) = i \big(\ovl{\sigma_1 (P(-x) + m_1(-x) + a_1 \sigma_3 - U_0)\sigma_1}\big)_{12} \\
& = i\big(\ovl{P(-x) + m_1(-x)})_{21} = -i (m_1)_{12}(-x) - iP_{12}(-x)\\
& = u(-x) - i P_{12}(-x) = \tilde{u}(-x), \\
\end{aligned}
$$
which completes the proof.
\end{proof}

\begin{remark}
\label{rmk:nonlinearization}
Suppose we are in a situation when the reflection coefficient $r(z) = \frac{\ovl b}{\ovl a}$ is small,
as is the case corresponding to our data $u_0 = v_\mu + \epsilon w$
(note that for $\epsilon=0$, $u_0 = v_\mu$ gives rise to a solution with $r\equiv 0$).
The solution of the RHP $m=(m_{ij})_{1\leq i,j\leq 2}$ is given by (see \eqref{eq:eq_m_polar})
\beq
\label{eq:eq_m_polar2}
m(x,z) = I + \frac 1{2\pi i} \int_\rb \frac {\mu_x(s)(v_x(s)-I)}{s-z} \rd s + \text{polar terms}.
\eeq
where $\mu_x$ solves \eqref{eq:sing_mu}.
For $r$ small, we see that $\mu \sim I + \nu$ where $\nu$ is a sum of polar terms.
Substituting $\mu_x \sim I + \nu$ into \eqref{eq:eq_m_polar2}, we obtain
$$ \bal u(x)
& = -i \lim_{z\to\infty} z m_{12}(x,z) \\
& = i \int_\rb \mu_x(v_x - I) \frac {\rd z}{2 \pi i} + \text{polar terms}\\
& \sim \frac 1{2 \pi} \int_\rb r(z) e^{i(xz-z^2 t/2)} + \text{polar terms}\\
\eal $$
Here, as $|a|^2 = 1- |b|^2$, $r = \frac{\ovl b}{\ovl a} \sim \ovl b$.
Neglecting the polar contribution we see that
this formula is of the same form as formula \eqref{eq:fou_rep} for the solution in the linear case
$u_e(x,t) = \frac 1{2\pi} \int_\rb e^{i(xz-tz^2/2)} \hat u (z) \rd z$.
Here $(iz+q)\hat u(z)$ is anti-symmetric.
But from \eqref{eq:symm_scattering}, $\ovl{b(z)}(iz+\beta)=-\ovl{b(-z)}(-iz+\beta)$,
so that in the case $\beta=q$, we see that $\ovl b$ and $\hat u$ have the same symmetry.
Thus the B\"acklund extension is not only the nonlinear analog of the method of images;
it is the non-linearization of the method of images.

The situation is even more striking in ``u-space".
Indeed, from \eqref{eq:u_tilde}, $\tilde{u} = u - 2q \frac{\xi_1\overline{\xi_2}}{|\xi_1|^2 + |\xi_2|^2}$,
where $\xi=(\xi_1, \xi_2)^T$ solve
\beq
\label{eq:lin_spec_xi}
\partial_x \xi = (iz\sigma + Q) \xi, \ \ \xi(x=0)=e_1.
\eeq
Iterating \eqref{eq:lin_spec_xi} for $u$ small, we obtain $\xi(x) = (e^{-q x/2}, -\int_0^x e^{q(x-2y)/2} \ovl{u}(y) \rd y )^T + O(u^2)$.
Thus $\tilde u(x) = u(x) + 2q \int_0^x e^{q(x-2y)/2} \ovl{u}(y) \rd y + O(u^2)$.
As $u(x)=\tilde u(-x)$ for $x<0$, we see that $u^e(x)$ is precisely the same as $u_e$ in \eqref{eq:lin_ext}, up to terms of order $u^2$.
\end{remark}

The following result is an immediate consequence of Theorem \ref{thm:sca_inverse} and Proposition \ref{p:symm_scattering}.
Recall from \eqref{eq:b_H11} that $r\in H^{1,1}(\rb)$ if and only if $b\in H^{1,1}(\rb)$.

\begin{corollary}
Let $(r,Z_+,K_+)\in H^{1,1} \times \cb_+^n \times (\cb\setminus 0)^{n}$.
Define $a(z)$ as in \eqref{eq:a_exp_formula} and set $b(z)=a(z)\ovl{r(z)}$.
If $\{a(z), b(z), \{z_k\}, \{\gamma_k\}, \beta\}$ satisfies \eqref{eq:symm_scattering}, then
$u=\mc I (r,Z_+,K_+)\in H^{1,1}(\rb)$ and is $q$-symmetric.
\end{corollary}

\begin{remark}
\label{rmk:1_soliton}
Consider the 1-parameter family of 1-soliton solutions $e^{i\lambda^2 t/2} \eta_\lambda$ of NLS,
\begin{equation}
\label{eq:1_soliton}
\eta_\lambda (x) = \lambda \sech ( \lambda  x + \tanh^{-1}( q / \lambda ) ), \ \ |q|<\lambda.
\end{equation}
A direct calculation shows that the scattering data of $Q=\bsm 0&\eta_\lambda \\ -\ovl{\eta_\lambda}&0 \esm$ is given by
$$
n=1, \ \ z_1=i\lambda, \ \ a(z) = \frac{z-i\lambda}{z+i\lambda}, \ \ b(z) \equiv 0, \ \ \gamma(i\lambda) = \sqrt{\frac{\lambda+q}{\lambda-q}},
$$
from which we immediately check by \eqref{eq:symm_scattering} that $\eta_\lambda$ is $q$-symmetric.
Note that $\eta_\lambda$ coincides with the nonlinear ground state $v_\lambda$ for $x \geq 0$.
In view of Remark \ref{rmk:symmetric_extension}, the B\"acklund extension of $v_\lambda |_{\mathbb{R}^+}$ is $\eta_\lambda$.
\end{remark}

Theorem \ref{thm:evol_scatt_data}, Proposition \ref{prop:backlund_ext},
Proposition \ref{prop:mixed_cond_x_0} and Proposition \ref{p:symm_scattering}
imply the following important result.
\begin{proposition}
\label{prop:classical_sol_NLS}
Let $u(t)$ solve NLS on $\rb$ with $u(t=0)=u_0 \in H^{1,1}(\rb)$.
If $u_0$ is $q$-symmetric and generic,
then $u(t)$ is $q$-symmetric for all $t\geq0$.
If $u(t)$ is a classical solution, then $u(t)|_{\rb^+}$ and $\mc R \mc B_q^+ (u(t)|_{\rb^+})$ solve HNLS$_q^\pm$, respectively.
\end{proposition}

From Proposition \ref{prop:bijectivity_backlund} the B\"acklund extension $u^e$ of $u\in H^{1,1}(\rb^+)$ belongs to $H^{1,1}(\rb)$.
In order to compute the scattering data for $u^e$, we need to know $u^e \in \mc G$.
We have the following proposition.
\begin{proposition}[Scattering data for the B\"acklund extension]
\label{prop:backlund_scattering}
Let $ u (x)\in H^{1,1}(\rb^+)$ and let $g(x,z)$ be the (unique) solution of \eqref{eq:scattering_problem} for $x\geq 0$, $z\in \ovl{\cb^+}$ given in \eqref{eq:tail_behavior_g} with $g(x,z)e^{-ixz/2} \to e_1$ as $x\to\infty$.
Denote $\bigl(\begin{smallmatrix} A(z) \\B(z) \end{smallmatrix}\bigr) \equiv g(0,z)$. Define
\begin{equation}
\label{eq:beta}
\begin{aligned} \beta =
\left\{ \begin{aligned}
& q, \qquad  \text{if} \ \ q>0, B(iq)=0 \ \ \text{or} \ \ q<0, A(-iq) \neq 0, \\
& -q, \ \ \text{if} \ \ q>0, B(iq)\neq0 \ \ \text{or} \ \ q<0, A(-iq) = 0, \\
\end{aligned}\right.
\end{aligned}
\end{equation}
and for $z \in \ovl{\cb^+}$, set
\begin{equation}
\label{eq:a}
a(z) \equiv \frac 1{z-i\beta}\big((z-iq)A(z)\overline{A(-\overline z)} - (z+iq)B(z)\overline{B(-\overline z)} \big).
\end{equation}
Suppose that $a(z)$ has only a finite number of zeros $z_1,\dots,z_n$, $n\geq0$, in $\ovl{\cb^+}$,
all of which are simple and lie in $\cb^+$.
For $z \in \rb$, set
\begin{equation}
\label{eq:b}
b(z) \equiv \frac 1{z+i\beta}\big((z+iq)A(-z)B(z) + (z-iq)A(z)B(-z) \big),
\end{equation}
For $1\leq k \leq n$, set
\begin{equation}
\label{eq:gamma_z_k}
\gamma(z_k) \equiv
\left\{ \begin{aligned}
& \frac {z_k-i\beta}{z_k+iq} \frac {A(z_k)}{\overline{B(-\overline{z_k})}} \ , \ \ \text{if} \ A(z_k) \neq 0 ,\\
& \frac {z_k-i\beta}{z_k-iq} \frac {B(z_k)}{\overline{A(-\overline{z_k})}} \ , \ \ \text{if} \ A(z_k) = 0 ,\\
\end{aligned}\right.
\end{equation}
Then, the B\"acklund extension $u^e$ of $u$ is generic and
$\{a(z)$, $b(z)$, $\{z_k\}$, $\{\gamma_k\}\}$ is the scattering data $\mc S(u^e)$ for $u^e$.
Moreover, $\mc S(u^e)$ satisfies the symmetries \eqref{eq:symm_scattering}.
\end{proposition}

\begin{proof}
Let $u^e$ be the B\"acklund extension of $u$ and let $P^e(x)$ solve \eqref{eq:backlund}.
As $u^e$ is $q$-symmetric, we have $P^e(x) \to -i\beta^e \sigma_3$ as $|x|\to\infty$, $(\beta^e)^2 = q^2$.
Let $(\psi^e)^\pm$ be the associated ZS-AKNS solutions.
It is easy to check from the proof of Lemma \ref{lem:P_tail} that in fact $\beta^e = \beta$.
Clearly $(\psi^e)_1^+(x,z) = g(x,z)$, $x\geq0$, $z\in\ \ovl{\cb^+}$ and hence $(\psi^e)_1^+(0,z) = g(0,z) = (A(z),B(z))^T$.
As $u^e$ is $q$-symmetric, we see from \eqref{eq:ZS_symmetry_0} that
$$ \bal (\psi^e)_1^- (0,z)
& 
 = \frac 1{z+i\beta} \bpm (z+iq)A(-z) \\-(z-iq)B(-z) \epm, \ \ z\in\rb. \\
\eal
$$
Thus, by \eqref{b_as_det}, $b^e(z) \equiv b(z;u^e) = \det((\psi^e)_1^-, (\psi^e)_1^+)(0,z) = b(z)$.
On the other hand, it follows from \eqref{eq:symmetry_ZS_AKNS22_2} that
$$ \bal (\psi^e)_2^- (0,z)
& 
 =\frac 1{z-i\beta} \bpm (z+iq)\overline{B(-\ovl{z})} \\(z-iq)\overline{A(-\ovl{z})} \epm, \ \ z\in\ovl{\cb^+}. \\
\eal
$$
So, the scattering function $a^e(z)\equiv a(z;u^e) = \det((\psi^e)_1^+,(\psi^e)_2^-)(0,z) = a(z)$.
From this fact, we see that $a^e(z)$ has only a finite number of zeros $z_1,\dots,z_n$, $n\geq0$, in $\ovl{\cb^+}$,
all of which are simple and lie in $\cb^+$.
Moreover, if $A(z_k) \neq 0$, then $((\psi^e)_1^+)_1(0,z_k) = A(z_k) \neq 0$ and so $((\psi^e)_2^-)_1(0,z_k) \neq 0$.
Thus the norming constant $\gamma^e(z_k)\equiv \gamma(z_k;u^e) = \frac {((\psi^e)_1^+)_1(0,z_k)}{((\psi^e)_2^-)_1(0,z_k)}
= \frac {z_k-i\beta}{z_k+iq} \frac {A(z_k)}{\overline{B(-\overline{z_k})}}$.
If $A(z_k) = 0$, then $B(z_k)\neq 0$ and hence $((\psi^e)_1^+)_2(0,z_k) = B(z_k) \neq 0$. So, $((\psi^e)_2^-)_2(0,z_k) \neq 0$
and hence $\gamma^e(z_k) = \frac {((\psi^e)_1^+)_2(0,z_k)}{((\psi^e)_2^-)_2(0,z_k)} = \frac {z_k-i\beta}{z_k-iq} \frac {B(z_k)}{\overline{A(-\overline{z_k})}}$.
Hence $\gamma^e(z_k)=\gamma(z_k)$.
Thus $\{a(z)$, $b(z)$, $\{z_k\}$, $\{\gamma_k\}\}$ is the scattering data $\mc S(u^e)$ for $u^e$;
the fact that $u^e\in \mc G$, follows directly from our assumption on $a(z)$.
\end{proof}

\begin{remark}
The fact that $\{a(z)$, $b(z)$, $\{z_k\}$, $\{\gamma_k\}\}$ is the scattering data $\mc S(u^e)$ for $u^e$,
implies a variety of regularity properties for $a(z)$,$b(z)$.
In particular, it follows necessarily that if $\beta=|q|$, then the apparent singularity of $a(z)$ at $z=i\beta$ is removable:
this can be seen directly as follows:
$\beta=|q|$ if $q>0$ and $B(iq)=0$, or $q<0$ and $A(-iq)=0$.
In the first case we have
$$
a(z) = A(z)\overline{A(-\overline z)} - \frac {z+iq}{z-iq}B(z)\overline{B(-\overline z)}
$$
and in the second case we have
$$
a(z) = \frac {z-iq}{z+iq} A(z)\overline{A(-\overline z)} - B(z)\overline{B(-\overline z)}.
$$
In both cases, the analyticity of $a(z)$ is clear.
Furthermore, as $z\to i|q|$, $a(z) \to |A(i|q|)|^2$ or $-|B(i|q|)|^2$, respectively.
As $A(z)$ and $B(z)$ cannot vanish simultaneously, we see that $a(i|q|) \neq 0$ in both cases.
If $q>0$ and $B(iq) \neq0$, then $\beta=-q$ and $a(i|q|)=a(iq)=-|B(i|q|)|^2 \neq 0$ and
if $q<0$ and $A(-iq) \neq0$, then $\beta=q$ and $a(i|q|)=a(-iq)=|A(i|q|)|^2 \neq 0$.
This shows that $a(z)$ is analytic in $\cb^+$ and non-zero at $z=i|q|$.

To see that the norming constant $\gamma(z_k)$ in \eqref{eq:gamma_z_k} are well-defined and non-zero, suppose that $B(-\ovl{z_k})\neq0$:
if $A(z_k)=0$, then as $a(z_k)=0$, we must have from \eqref{eq:a} that $B(z_k)=0$,
which contradicts the fact that $A(z)$ and $B(z)$ cannot vanish simultaneously.
Thus $A(z_k) \neq 0$ and $\gamma(z_k)$ is well-defined and non-zero.
If $B(-\ovl{z_k})=0$, then $A(-\ovl{z_k})\neq0$.
If $B(z_k)=0$, then again from \eqref{eq:a} we see that $A(z_k)=0$, which is a contradiction.
Thus we see that in all cases $\gamma(z_k)$ is well-defined and non-zero.
Finally, the fact that $|a(z)|^2+|b(z)|^2=1$, $z\in\rb$, in a direct calculation
using \eqref{eq:a}\eqref{eq:b} together with $|A(z)|^2+|B(z)|^2 = 1$, $z\in\rb$, from \eqref{eq:psi_pm_det}.
The symmetry properties \eqref{eq:symm_scattering} can also be verified directly.
\end{remark}

\begin{remark}
\label{rmk:limit_images}
As noted in \cite{Ta2}, p.437, in the repulsive case $\kappa=-1$ in \eqref{eq:gpe}, corresponding to the defocusing NLS equation,
one can similarly construct the B\"acklund extension $u^e$ of $u$ as above.
In this case, however, $a(z)$ may have a zero on the imaginary axis,
so that the spectral data for $u^e$ is singular.
This is the analog in the nonlinear situation that
the method of images in the linear case breaks down if $q>0$, as described in Section 1.
\end{remark}

The set of functions $u$ on $\rb^+$ whose B\"acklund extensions $u^e$ are generic, is open and dense in $H^{1,1}(\rb^+)$.
We have the following result.

\begin{proposition}
\label{prop:backlund_generic}
  Let $\mc G^+$ be the set of functions $u\in H^{1,1}(\rb^+)$ such that the scattering function $a(z)$ of the B\"acklund extension of $u$ has only a finite number zeros in $\ovl{\cb^+}$, all of which are simple and lie in $\cb^+$.
Then, $\mc G^+$ is open and dense in $H^{1,1}(\rb^+)$.
\end{proposition}

\begin{proof}
For $u \in \mc G^+$, consider a small disk $D=D(u,\epsilon)$ in $H^{1,1}(\rb^+)$ centered at $u$ with radius $\epsilon>0$.
For $v\in D$, let $g(0,z) = g(0,z;v)$, $a(z)=a(z;v)$ and $\beta = \beta(v)$ be defined as in Proposition \ref{prop:backlund_scattering}.
By Lemma \ref{lem:perturbed_potential},
\beq
\label{eq:uniform_bound}
\sup_{v\in D} \sup_{z\in \ovl{\cb^+}} |g(0,z;v)-g(0,z;u)|\leq c\epsilon.
\eeq
We have $|a(z;u)| \geq \hat c > 0$ for $z\in\rb$.
Define $a_1(z;v) \equiv a(z;v) (z-i\beta(v))$.
As $|z-i\beta(v)|=|z-i\beta(u)|$, $z\in\rb$, it follows from \eqref{eq:uniform_bound} that
$$
|a(z;v)|=\bigg|\frac{a_1(z;v)}{z-i\beta(u)}\bigg| \geq |a(z;u)| - \bigg| \frac{a_1(z;v)-a_1(z;u)}{z-i\beta(u)} \bigg|
\geq \hat c - c\epsilon >0,
$$
for sufficiently small $\epsilon>0$.
On the other hand, since $a(z;u) \to 1$ as $z\to\infty$, by \eqref{eq:uniform_bound} again,
we have for some $R>0$, $|a(z;v)|\geq c^\# > 0$ for $|z|\geq R$, uniformly for $v\in D$.
Now $a_1(z;u)$ has only a finite number of zeros in $\cb^+$, all of which are simple.
This is clear for $z\neq i\beta$.
However, if $\beta=|q|>0$, then as $a(i|q|;u)\neq 0$, $a_1(z;u)$ has only a simple zero at $z=i|q|=i\beta$.
Again by Lemma \ref{lem:perturbed_potential} together with Rouch\'e's theorem,
we see that $a_1(z;v)$ also has only a finite number of zeros in $\cb^+$ all of which are simple.
This proves that $\mc G^+$ is open.

We will show that $\mc G^+$ is dense in $H^{1,1}(\rb^+)$.
Let $u \in H^{1,1}(\rb^+)$.
We can assume that $u$ has compact support, say $[0,L]$.
Let $p(x,s) = u(x) + s v(x)$ where $s \in \rb$, and
$v\in H^{1,1}(\rb^+)$ has compact support in $[L,L']$, $L'>L$, and will be determined later.
Denote the associated ZS-AKNS solution normalized at $+\infty$ by $\psi^+(x,z;p)$, $x\geq 0$.
Let $m^+(x,z;p) = \psi^+(x,z;p) e^{-ixz\sigma} = (m_1^+, m_2^+)$,
$m_1^+(0,z;p) = \bigl(\begin{smallmatrix} A(z;p)\\B(z;p) \end{smallmatrix}\bigr)$,
$m_1^+(0,z;u) = \bigl(\begin{smallmatrix} A(z)\\B(z) \end{smallmatrix}\bigr)$ and
$f(x,z) = \frac {\partial m_1^+(x,z;p)}{\partial s}\big|_{s=0}$.
Note that since $p(x,s)$ has compact support, $m^+(x,z;p)$ is analytic in $z$-plane.
By differentiating the integral equation for $m_1^+$ with respect to $s$ at $s=0$, we obtain
$$
\begin{aligned} f(x,z)
& = -\int_x^{\infty} \bpm 1&0\\0&e^{-iz(x-y)} \epm \bpm 0&v\\-\overline{v}&0 \epm m_1^+(y,z;u) \\
& \quad -\int_x^{\infty} \bpm 1&0\\0&e^{-iz(x-y)} \epm \bpm 0&u\\-\overline{u}&0 \epm f(y,z) \\
& = \text{I} + \text{II}.
\end{aligned}
$$
Since $m_1^+(x,z;u) = e_1$ for $x \geq L$ and $v(x)=0$ for $0\leq x<L$, we have
$$
\text{I} = e^{-ixz} \bpm 0 \\ c(z) \epm, \ \ 0\leq x\leq L,
$$
where $c(z) = \int_L^{L'} \overline{v(y)} e^{izy}$. Therefore, for $0\leq x<L$,
$$
e^{ixz} f(x,z) = \bpm 0 \\ c(z) \epm - \int_x^{\infty} \bpm e^{iz(x-y)}&0\\0&1 \epm \bpm 0&u\\-\overline{u}&0 \epm e^{iyz} f(y,z).
$$
Since $c(z) m_2^+(x,z;u)$, $z\in\rb$, solves the same integral equation, it follows by uniqueness that
\beq
\label{eq:f_m_2_plus}
 e^{ixz} f(x,z) = c(z) m_2^+(x,z;u).
\eeq
for $0\leq x<L$ and $z\in\rb$.
But then by analytic continuation, \eqref{eq:f_m_2_plus} remains true for all $z\in\cb$, $0\leq x<L$.
Thus,
\beq
\label{eq:partial_s_AB}
\frac {\partial }{\partial s}\Big|_{s=0}\bpm A(z;p) \\ B(z;p)\epm = f(0,z) = c(z) m_2^+(0,z;u)
= c(z) \bpm -\overline{B(\overline{z})} \\ \overline{A(\overline{z})}\epm .
\eeq
Let $a_1(z,s) \equiv a_1(z;p) = a(z;p)(z-i\beta(p))$. Then,
\beq
\label{eq:a_1_mu_c}
\frac {\partial a_1}{ \partial s} (z,0)
= c(z) \overline{\mu(-\overline{z})} - \overline{c(-\overline{z})} \mu(z),
\eeq
where
$$
\mu(z) = (z-iq)A(z)B(-z) + (z+iq)A(-z) B(z).
$$
Now we prove that, for suitable choice of $v$,
\beq
\label{eq:partial_s_a_1}
\frac {\partial a_1}{ \partial s} (\xi,0) \neq 0
\eeq
for any zero $\xi$ of $a(z)=a(z;u)$.
We first show that $\mu(\xi)$ and $\mu(-\overline{\xi})$ do not vanish simultaneously.
Recall that
\begin{equation}
\label{eq:det_m_plus}
\det m^+(x,z) \equiv 1 \Rightarrow A(z)\overline{A(\overline{z})} + B(z)\overline{B(\overline{z})} = 1, \ z \in \cb.
\end{equation}
Suppose that $a(\xi)=0$, $\mu(\xi)=0$ and $\mu(-\overline{\xi})=0$ for some $\xi \in \cb^+$.
Note that $\xi \neq i|q|$.
If $B(\xi)=0$, $A(\xi) \neq 0$. But then as $a(\xi)=0$, $A(-\ovl\xi)=0$ by \eqref{eq:a}.
But then $B(-\ovl\xi)\neq0$ and it follows from $\mu(-\overline{\xi})=0$ that $A(\ovl\xi)=0$.
This contradicts \eqref{eq:det_m_plus} and hence $B(\xi)\neq0$.
On the other hand, if $A(-\ovl\xi)=0$, then as $a(\xi)=0$ and $B(\xi)\neq0$,
we must have $B(-\ovl\xi)=0$, which is a contradiction. Hence $A(-\ovl\xi)\neq0$.
Now as $a(\xi)=0$, $ (\xi-iq) A (\xi) = (\xi+iq){\frac {\overline{B(-\overline\xi)}}{\overline{A(-\overline\xi)}} } B (\xi)$ and hence by \eqref{eq:det_m_plus} again,
$$
\begin{aligned} \mu(\xi)
& = (\xi+iq) \bigg( {\frac {\overline{B(-\overline\xi)}}{\overline{A(-\overline\xi)}} } B (\xi) B(-\xi)  + A(-\xi)B(\xi) \bigg) \\
& = (\xi+iq) \frac {B(\xi)}{\overline{A(-\overline\xi)}}   \neq 0,\\
\end{aligned}
$$
which is a contradiction.
Now, as $a(z;u)$ is analytic across $\rb$ and $a(z;u) \to 1$ as $z\to\infty$, $z\in\ovl{\cb^+}$, $a(z;u)$ has only a finite number of zeros in $\ovl{\cb^+}$.
As $c(z) = \int_L^{L'} \overline{v(y)} e^{izy}$ and $\overline{c(-\overline{z})} = \int_L^{L'} v(y) e^{izy}$,
using the above results, we can find a function $v$ such that \eqref{eq:partial_s_a_1} is satisfied for all the zeros of $a(z;u)$.
For example, a simple way to do this is to consider $v=e^{i\theta} \delta_L(x)$,
where $ \delta_L(x)$ is a delta function at $x=L$, and then choose $\theta\in\rb$ appropriately;
smoothing out $v$ then produces the desired $v$.
But then, by Proposition \ref{prop:bdt72} below, all the zeros of $a_1(z;p)$ are simple for sufficiently small $s\neq0$
and hence so are those of $a(z;p)$.
In other words, there exists $p(x)$ which has compact support such that
$a(z;p)$ has only a finite number of zeros in $\ovl{\cb^+}$, all of which are simple.
We now show that any real zeros of $a(z;p)$, if there are any, can be moved off the real axis
by perturbing $p(x)\to p_{\hat s}(x)=p(x) + \hat s \hat v(x)$,
where $\hat v(x)$ again has compact support, and $\hat s\neq 0$ is arbitrarily small.
By standard analysis, if $\xi$ is a simple zero of $a(z;p_{\hat s=0})$,
then for $\hat s$ small, $a(z;p_{\hat s})$ has simple zero $\xi(\hat s)$ near $\xi=\xi(0)$, $a(\xi(\hat s);p_{\hat s})=0$ and $\xi(\hat s)$ is differentiable.
As $a_1(\xi(\hat s);p_{\hat s})=0$, we have
\beq
\label{eq:diff_a_zero}
0=\xi'(0) a_1'(\xi(0);p) + \frac {\partial a_1}{ \partial \hat s} (\xi(0);p).
\eeq
Here \eqref{eq:a_1_mu_c} is given with $u$ replaced by $p$.
If $a(z;p)$ has real zeros, as $a_1'(\xi(0);p) \neq 0$, it is easy to see that
we can choose $\hat v$ such that \eqref{eq:partial_s_a_1} is satisfied
and $\xi'(0)$ is not real for all (the finite number of) real zeros $\xi(0)$ of $a(z;p)$.
Thus, for sufficiently small $\hat s\neq0$, all the zeros of $a(z;p_{\hat s})$ in $\ovl{\cb^+}$ are simple and non-real.
This completes the proof that $\mc G^+$ is open and dense in $H^{1,1}(\rb^+)$.
\end{proof}

\begin{proposition}[\cite{BDT}, p.72]
\label{prop:bdt72}
Suppose $U$ is a neighborhood of the origin in $\cb \times \rb$ and suppose $h:U \to \cb$ is smooth and holomorphic with respect to the first variable. Suppose
$$
h(0,0) =0, \ \ \frac{\partial h}{\partial s} (0,0) \neq 0, \ \ h=h(z,s).
$$
Then, for small $s\neq0$, the zeros of $h(\cdot,s)$ near $z=0$ are simple.
\end{proposition}

\begin{remark}
It is of interest to see whether the scattering function $a(z)$ for the B\"acklund extension
of a function $u\in H^{1,1}(\rb^+)$, can indeed have zeros on $\rb$.
Fix any integer $n\geq0$ and let $\xi = \sqrt 2 (n+\frac 12) \pi >0$ and $q = \xi \cot \xi$ ia finite and non-zero.
Define $u(x) = \frac \xi 2 \chi_{[1,2]}$, $x\geq 0$ where $\chi$ is a characteristic function.
Then, a direct calculation shows that for $z\in\rb$,
$$
\left\{
\begin{aligned}
& A(z) = e^{\frac {iz}2} \Big\{ \cos W - \frac {iz}{2W} \sin W \Big\} , \\
& B(z) = \frac {\xi e^{\frac {3iz}2}}{2W}   \sin W ,\\
\end{aligned}
\right.
$$
where $W = W(z) \equiv \frac 12 \sqrt {\xi^2 + z^2}$. Since $\overline{A(-z)} = A(z)$ and $\overline{B(-z)} = B(z)$, by \eqref{eq:a},
$$
a(z) = \frac {z-iq}{z-i\beta} \Big\{(A(z))^2 - \frac {z+iq}{z-iq} (B(z))^2 \Big\}, \ \ z\in\rb.
$$
and it is easy to check that $a(z)$ vanishes at $\pm \xi$.
\end{remark}

Before we present the proof of Theorem \ref{thm:Sol_proc}, we need the following lemma of approximation.
\begin{lemma}
\label{lem:limit_H_11 bdry}
Suppose $u\in H^{1,1}(\rb)$.
Then there exists a sequence of smooth functions $\{u_n\}$ with compact support in $[0,\infty)$
such that
\begin{enumerate}
\item[\textnormal{(i)}] $u_n \to u$ in $H^{1,1}(\rb^+)$,
\item[\textnormal{(ii)}] $u_n'(0) + q u_n(0)=0$, and
\item[\textnormal{(iii)}] the B\"acklund extension $u_n^e$ of $u_n$ is generic.
\end{enumerate}
\end{lemma}

\begin{proof}
As the set $\mc G^+$ in Proposition \ref{prop:backlund_generic} is open and dense,
it is enough to show that there exists a sequence of smooth functions $\{u_n\}$ with compact support in $[0,\infty)$ such that (i) and (ii) hold.
By standard estimates, there exist smooth functions $u_n^{(1)}$ with compact support in $[\frac 12, \infty)$
such that $u_n^{(1)} \to u$ in $H^{1,1}[\frac 12, \infty)$.
Now let $\{f_n\}$ be a sequence of functions in $C_0^\infty(0,1)$
such that $f_n \to f\equiv u' + q u$ in $L^2[0,1]$.
Set $u_n^{(2)} \equiv e^{-q x} u(0) + \int_0^x e^{-q(x-s)} f_n (s) \rd s$, $x\in[0,1]$.
Then $u_n^{(2)}(x)$ is smooth and $(u_n^{(2)})'(x) + q u_n^{(2)}(x) = f_n(x)$,
so that $(u_n^{(2)})'(0) + q u_n^{(2)}(0) = f_n(0) = 0$.
Also $u_n^{(2)}$ converges in $L^2[0,1]$, $u_n^{(2)} \to u^{(2)}$,
and we have $u^{(2)}(x) = e^{-q x} u(0) + \int_0^x e^{-q(x-s)} (u'(s)+q u(s)) \rd s = u(x)$.
But then $(u_n^{(2)})' = f_n - q u_n^{(2)} \to (u' + q u)-q u  = u'$, also in $L^2[0,1]$.
Let $\chi \in C_0^\infty(\rb^+)$ such that $0\leq \chi(x)\leq 1$, $\chi (x) = 0$ for $0\leq x\leq \frac 12$, and $\chi(x)=1$ for $x\geq1$.
Let $u_n = u_n^{(1)} \chi + u_n^{(2)} (1-\chi)$.
Then clearly $u_n \to u$ in $H^{1,1}(\rb^+)$ and $u_n'(0) + q u_n(0)=0$.
\end{proof}

\begin{proof}[Proof of Theorem \ref{thm:Sol_proc}]
Suppose that $q>0$. The case $q<0$ is similar.
Let $g(x,z) = (g_1, g_2)^T$ be as in the proof of Proposition \ref{prop:bijectivity_backlund}.
If $g_2(0, iq; u_0^+) \neq 0$, then by Lemma \ref{lem:limit_H_11 bdry},
we can choose a sequence $\{u_0^{(n)}\}$ of $C^\infty$-functions with compact support in $[0,\infty)$ such that
\begin{enumerate}
\item $u_0^{(n)} \to u_0$ as $n\to\infty$ in $H^{1,1}(\rb^+)$,
\item $(u_0^{(n)})'(0+) + q u_0^{(n)}(0) = 0$, and
\item the B\"acklund extension $(u_0^{(n)})^e$ of $u_0^{(n)}$ is generic.
\end{enumerate}
As $u_0^{(n)}$ has compact support, it is easy to verify from \eqref{eq:u_tilde}
that $\mc R \mc B_q^+ u_0$ is also smooth and decays at least exponentially.
It follows by Proposition \ref{lem:backlund_ext_c2} that $(u_0^{(n)})^e\in H^{3,3}(\rb)$
and hence it follows by Proposition \ref{prop:NLS_sol_H_k_k}
that NLS has a solution $u_n^e(t) = u_n^e(x,t)$ in $H^{3,3}(\rb)$ with $u_n^e(t=0) = (u_0^{(n)})^e$.
In particular, $u_n^e(t)$ is a classical solution of NLS.
By Proposition \ref{prop:classical_sol_NLS} together with the properties of $u_0^{(n)}$ above,
we see  that $u_n^e(t)|_{\rb^+}$ solves HNLS$_q^+$.
On the other hand, by Lemma \ref{lem:conv_backlund_ext} below, we see that $(u_0^{(n)})^e \to u_0^e$ in $H^{1,1}(\rb)$
and hence $u_n^e(t)\to u^e(t)$ in $H^{1,1}(\rb)$ as $n\to\infty$, uniformly for $t$ in any interval $[0,T]$, $0<T<\infty$.
In particular, $u_n^e(t)|_{\rb^+} \to u^e(t)|_{\rb^+}$ as $n\to\infty$.
Hence, it follows that $u^e(t)|_{\rb^+}$ also solves HNLS$_q^+$.

If $g_2(0, iq; u_0^+) = 0$, then define $G(x,z)=(G_1, G_2)^T$ as in the proof of Proposition \ref{prop:IBV_rel_backlund}.
We obtain $G_2(0,iq; \mc{R}\mc{B}_q^+ u_0^+) \neq 0$ and hence
we are in the similar situation to the case $g_2(0, iq; u_0^+) \neq 0$,
but now on $\rb^-$ with $v_0 \equiv \mc{R}\mc{B}_q^+ u_0^+$.
We choose a sequence $\{v_0^{(n)}\}$ which converges to $v_0$ in $H^{1,1}(\rb^-)$
and gives rise to a classical solution $v_n(t)$ with $v_n(t=0)=(v_0^{(n)})^e$,
where $(v_0^{(n)})^e$ is the B\"acklund extension of $v_0^{(n)}\in H^{1,1}(\rb^-)$.
Let $v_0^e$ be the B\"acklund extension of $v_0\in H^{1,1}(\rb^-)$.
By Proposition \ref{prop:backlund_ext} below, we see that $v_0^e = u_0^e$
and hence $(v_0^{(n)})^e \to u_0^e$.
Thus, $u^e(t)|_{\rb^+} = \lim_{n\to\infty} v_n(t)|_{\rb^+}$ solves HNLS$_q^+$.

Now it follows immediately that $u(x,t) \equiv u^e(|x|,t)$ solves the IVP \eqref{eq:nls_delta}
with $u(x,0) = u_0(x)$ in the sense of \eqref{eq:weak_sol_H_q}.
\end{proof}

\begin{remark}
\label{rmk:sol_approx}
It follows from the proof of Theorem \ref{thm:Sol_proc},
that if $u(t) \in H^{1,1}(\rb^+)$, $u(t=0)=u_0$, is a solution of HNLS$_q^+$,
then there exists a sequence of solutions $u_n(t)\in H^{3,3}(\rb^+)$ of HNLS$_q^+$,
such that for $t\geq0$, $\|u_n(t) - u(t)\|_{H^{1,1}(\rb)} \to 0$ as $n\to \infty$.
In particular, any solution to HNLS$_q^+$ in $H^{1,1}(\rb^+)$ can be approximated
by a classical solution to HNLS$_q^+$.
With a little more work, one can show that the $u_n$'s can be chosen in $H^{k,k}(\rb^+)$ for any $k\geq3$.
\end{remark}

\begin{remark}
\label{rmk:sol_alpha_symm}
It also follows from the proof of Theorem \ref{thm:Sol_proc},
that if $u_0 \in H^{1,1}(\rb)$ and $u^e(t)$ is the solution of NLS on $\rb$ with $u^e(t=0)=u_0^e$,
the B\"acklund extension of $u_0$, then $u^e(t)$ is $q$-symmetric (cf. Proposition \ref{prop:classical_sol_NLS}).
\end{remark}

\begin{proposition}
\label{prop:HNLS_H_k_k_backlund}
Let $k\geq2$ and suppose $u_0\in H^{k,k}(\rb^+)$.
In addition, suppose that the B\"acklund extension $u_0^e$ of $u_0$ is in $H^{k,k}(\rb)$.
Then HNLS$_q^+$ has a unique solution $u(t)$ in $H^{k,k}(\rb^+)$ with $u(t=0)=u_0$
and satisfies the boundary condition $u_x(0+,t) +q u(0+,t) = 0$, $t\geq0$.
\end{proposition}

\begin{proof}
As $u_0^e\in H^{k,k}(\rb)$, NLS has a solution $u^e(t)$ in $H^{k,k}(\rb)$ with $u^e(t=0)=u_0^e$.
But then, by Theorem \ref{thm:Sol_proc}, as $H^{k,k} \subset H^{1,1}$,
$u^e(t)|_{\rb^+}$ satisfies \eqref{eq:weak_sol_H_q_plus},
and hence $u^e(t)|_{\rb^+}$ is a solution of HNLS$_q^+$ in $H^{k,k}(\rb^+)$.
Furthermore, as $H^{k,k}(\rb^+) \subset C^1(\rb^+)$, for $k\geq 2$,
and $u^e(t)$ is $q$-symmetric (see Remark \ref{rmk:sol_alpha_symm}),
the boundary condition at $x=0$ follows from Proposition \ref{prop:mixed_cond_x_0}.
\end{proof}

\begin{remark}
Let $k\geq2$. As $u_0^e \in H^{k,k}(\rb)$, the boundary condition $u_0'(0+) +q u_0(0+) = 0$
is automatic for $u_0$ (see Proposition \ref{prop:mixed_cond_x_0}).
\end{remark}

The condition above on $u_0^e$ is rather implicit.
Of course, the condition can be rephrased in terms of requirement on the derivatives of $u_0$ at $x=0+$.
For $k=2$ or $3$, the conditions are particularly simple.

\begin{proposition}
\label{prop:HNLS_H_k_k_23}
Suppose $u_0\in H^{k,k}(\rb^+)$, $k=2$ or $3$, and in addition, suppose that $u_0'(0+) +q u_0(0+) = 0$ .
Then HNLS$_q^+$ has a unique solution $u(t)$ in $H^{k,k}(\rb^+)$, $k=2$ or $3$, with $u(t=0)=u_0$
and satisfying the boundary condition $u_x(0+,t) +q u(0+,t) = 0$, $t\geq0$.
\end{proposition}

\begin{proof}
It is enough to show that $u_0^e$ is in $H^{k,k}$, $k=2,3$.
Let $w = \mc R \mc B_q^+ u_0$.
We must show that for $k=2$,

(i)$w(0-) = u_0(0+)$,

(ii)$w'(0-) = u_0'(0+)$,

\noindent and in addition, if $k=3$,

(iii) $w''(0-) = u_0''(0+)$.

\noindent As $H^{2,2} \subset C^1$, conditions (i)(ii) follow from (the proof of) Proposition \ref{prop:mixed_cond_x_0}.
On the other hand, as $H^{3,3} \subset C^2$, condition (iii) follows from Proposition \ref{lem:backlund_ext_c2}.
The proof that the B\"acklund extension $u_0^e$ of $u_0 \in H^{k,k}(\rb^+)$ belongs to $H^{k,k}(\rb)$, $k=2$ or $3$,
follows in a similar way to the proof that $u_0^e$ of $u_0 \in H^{1,1}(\rb^+)$ lies in $H^{1,1}(\rb)$
(cf. Proposition \ref{prop:bijectivity_backlund}).
\end{proof}

\begin{remark}
For $k=3$, the solution $u(t)$ in $H^{3,3}$ is a classical solution.
\end{remark}

\begin{remark}
By Proposition \ref{prop:HNLS_H_k_k_23}, solutions $u(t)$ of HNLS$_q^+$ in $H^{2,2}$
satisfy the boundary condition $u_x(0+,t) +q u(0+,t) = 0$, $t\geq0$.
As $H^{2,2} \subset \{u\in H^{1,1}(\rb^+): u' \in L^1(\rb^+)\}$,
we see that this result is consistent with Step 2 above for $t>0$.
\end{remark}

\section{Long-time Asymptotics}
\label{sec:longtime}
By the results of the previous section, in order to determine the long-time behavior of the solution $u(x,t)$ of the IVP \eqref{eq:nls_delta},
we need to analyze the solution $u^e(x,t)$ to the equation \eqref{eq:focusing_nls} with initial data given by the B\"acklund extension $u_0^e$ of $u_0|_{\rb^+}$.
We consider only the case when $u_0^e$ is generic
so that the associated RHP \eqref{eq:jump_inv_scat} has a finite number of simple poles in $\cb \setminus \rb$ (cf. Remark \ref{rmk:not_generic}).
Rather than analyzing the asymptotic behavior of the RHP directly,
it is convenient  to remove the poles using Darboux transformations (see the Appendix for details).
Once the long time behavior of the RHP without poles is determined,
we then add the poles back in, again using Darboux transformations.
This procedure of removing the poles is not necessary:
we can certainly apply the Deift-Zhou steepest-descent method directly to RHP's with poles.
However, in situations where there are only one or two pairs of poles,
which is of primary interest here, it is more convenient to proceed by first stripping out the poles.

In order to explain the above procedure in more detail,
suppose that the scattering data of $u^e (x, t=0) = u_0^e (x)$ is given by
the reflection coefficient function $r(z)=\ovl{b(z)}/\ovl{a(z)}$ with $\|r\|_{H^{1,1}(\rb)} < \infty$,
and one pair of simple eigenvalues at $z=z_1\equiv i\mu_1 \in i\rb^+$ with the norming constant $\gamma_1 = \gamma(z_1) \neq 0$,
and at $\ovl{z_1}$ with corresponding norming constant $-\ovl{\gamma_1}$.
Here, $a(z)$, $b(z)$ and $\gamma_1$ are constructed from $u_0 \in H^{1,1}(\rb^+)$ as in Proposition \ref{prop:backlund_scattering}.
Note that the norming constant $\gamma_1$ is of the form $\gamma_1= e^{-i\rho}\sqrt{\frac{\mu_1+q}{\mu_1-q}}$ for some $\rho \in \rb$.
Indeed, by Proposition \ref{p:symm_scattering},
\beq
\label{eq:norming_const_real}
\gamma (z_1) \overline{\gamma(-\overline{z_1})} = \frac{z_1 - i\beta}{z_1 + i\beta}, \ \ \beta=-q \ \
\Rightarrow \ \ |\gamma (z_1)| = \sqrt{\frac{\mu_1 +q}{\mu_1-q}}.
\eeq
The evolution of the scattering data is given as follows by Theorem \ref{thm:evol_scatt_data}:
\begin{equation}
\label{eq:scattering_data}
z_1 (t)=i\mu_1, \ \ \gamma_1 (t) = \gamma_1 e^{iz_1 ^2 t/2}, \ \ r(z,t) = r(z)e^{-iz^2 t/2} \text{ for } z\in\rb.
\end{equation}
Define
\begin{equation}
\label{eq:scattering_data2}
r_f(z) \equiv r(z) \frac{z-\overline{z_1}}{z-z_1}.
\end{equation}
By the Appendix, $r_f$ is the pole-free reflection coefficient that one obtains
after removing the poles at $z=z_1, \ovl{z_1}$ by using a Darboux transformation.
For each fixed $x\in\rb$ and $t>0$, let the $2 \times 2$ matrix function $m_f(x,t,z)$ solve the normalized RHP $(\mathbb{R}, v_f) $:
$m_f(x,t,z)$ is analytic in $\cb\setminus\rb$ and the jump matrix is given by
\begin{equation}
\label{eq:jump_mat}
v_f(z) = v_f(x,t,z) \equiv \bpm 1+|r_f(z)|^2&r_f(z)e^{i\theta} \\ \overline{r_f(z)}e^{-i\theta} &1 \epm,
\  \theta=\theta(x,t,z) \equiv xz - \frac{tz^2}{2}.
\end{equation}
Let $u_f(x,t)$ be the associated potential function with $Q_f = \bsm 0 & u_f \\ -\ovl{u_f} & 0 \esm$.
If we apply the Darboux transformation to add in poles at $z_1 =i\mu_1$, $\ovl{z_1}=-i\mu_1$
with the norming constant $\gamma_1 (t) = \gamma_1 e^{iz_1 ^2 t/2}$,
we obtain the solution $m(x,t,z)$ to the RHP corresponding to \eqref{eq:scattering_data}.
In particular, we have from \eqref{eq:b_darboux} and \eqref{eq:darboux_u},
\begin{equation}
\label{eq:remove_one_u}
u^e(x,t) = u_f(x,t) + i(z_1-\overline{z_1}) \mathcal{F}({\mathfrak{b}})
\end{equation}
where
\beq
\label{eq:remove_one_b}
 {\mathfrak{b}} = m_f(z_1)e^{ixz_1 \sigma} \bpm 1 \\ \frac{-c_1 (t)}{z_1-\overline{z_1}} \epm,
\ \ c_1 (t) = \frac{\gamma_1(t)}{a'(z_1)}.
\eeq
The pole-free solution $u_f(x,t)$ can be read off from the residue of $m_f(x,t,z)$, i.e.,
\beq
\label{eq:u_f_residue}
u_f(x,t) = -i((m_f)_1(x,t))_{12}, \ \  (m_f)_1(x,t) \equiv \lim_{z\to\infty} z(m_f(x,t,z)-I).
\eeq
The goal is to obtain an explicit formula for the asymptotic of $m_f(x,t,z)$ as $t \rightarrow \infty$;
the behavior of $u^e(x,t)$ as $t\to\infty$ can then be read off from \eqref{eq:remove_one_u} and \eqref{eq:remove_one_b}.
The case where there are $n\geq 2$ simple eigenvalues, is treated similarly. Here we use
$$
r_f(z) = r(z) \prod_{k=1}^n \frac{z-\overline{z_k}}{z-z_k}.
$$

We now apply the steepest-descent method introduced by Deift and Zhou in \cite{DZ3} to the focusing NLS at hand (see also \cite{DIZ}).
We will follow \cite{DZ} in which the authors analyze the long-time behavior of the defocusing NLS equation,
making suitable modification for the focusing case.
Along the way we will state many technical results (e.g. Proposition \ref{prop:prop_delta}, etc.) without proof:
the proofs are very similar to analogous results in \cite{DZ} and the details are left to the reader.

In view of the preceding discussion we must consider normalized RHP's $(\Sigma=\rb,v)$ on $\rb$
with  jump matrices of the form,
\beq
\label{eq:jump_v}
v(z) = v(x,t,z) = \bpm 1+|r(z)|^2&r(z)e^{i\theta} \\ \overline{r(z)}e^{-i\theta} &1 \epm, \ \theta= xz - \frac{tz^2}{2},
\eeq
where $r\in H^{1,1}(\rb)$.
We are interested in the behavior of the solution $m=m(x,t,z)$ of $(\rb,v)$ as $t\to \infty$.

Solutions $m=m(x,t,z)$ of the RHP as given in \eqref{eq:jump_inv_scat} are classical
in the sense that $m$ is continuous up to the boundary in $\cb^+$ and $\cb^-$.
Such solutions $m$ exist because the reflection coefficient $r\in H^1(\rb)$.
For a general RHP on an oriented contour $\Sigma \subset \cb$ with jump matrix $v\in L^\infty(\Sigma)$,
we must define the RHP in a weaker sense (see \cite{LS}, \cite{CG}, and also \cite{DZ}), as follows.

Let $C_\Sigma$ denote the Cauchy operator
\beq
\label{eq:cauchy_operator}
(C_\Sigma h)(z) = \int_\Sigma \frac {h(s)}{s-z} \frac{\rd s}{2\pi i} , \ \ z\in \cb \setminus \Sigma,
\eeq
with associated boundary operators
\beq
\label{eq:cauchy_bdry_op}
(C^\pm_\Sigma h)(z) = \lim_{z' \to z^\pm} (C_\Sigma h)(z').
\eeq
Here $z^\pm$ refer to the orientation of $\Sigma$ at the point $z\in\Sigma$:
by convention if one traverses the contour in the direction of the orientation,
the ($\pm$)-sides lie to the left (right) respectively.
For example, if $\Sigma=\rb$ with orientation from $-\infty$ to $\infty$,
then $z^\pm$ refer to limits as $z' \to z\in\rb$ from $\cb^+/\cb^-$ respectively.
If $\Sigma$ is sufficiently regular - Lipschitz would do -
the operators $C^\pm_\Sigma$ are bounded in $L^p(\Sigma)$, $1<p<\infty$ (see e.g. \cite{BK}),
and the limits in \eqref{eq:cauchy_bdry_op} exist pointwise almost everywhere in $\Sigma$.
We always have
\beq
\label{eq:cauchy_bdry_id}
C^+_\Sigma - C^-_\Sigma = \text{id}.
\eeq
Given $\Sigma$ and a matrix $v$,
we say that a pair $h_\pm$ of $L^p(\Sigma)$ functions lies in $\partial C(L^p)$
if $h_\pm = C^\pm_\Sigma g$ for some (unique) $g\in L^p(\Sigma)$.
We say $m_\pm$ solves the normalized RHP $(\Sigma, v)_{L^2}$, if
\begin{itemize}
\item $m_\pm \in I+\partial   C (L^2)$,
\item $m_+(z) = m_-(z) v(z)$, $z\in\Sigma$.
\end{itemize}
Note that if $m_\pm$ solves the normalized RHP $(\Sigma, v)$,
then $m_\pm = I+C^\pm_\Sigma g$ for some $g\in L^2(\Sigma)$.
Hence $m(z) = I+(C_\Sigma g)(z) = I+ \int_\Sigma \frac{g(s)}{s-z} \frac{\rd s}{2\pi i}$, $z\in \cb \setminus \Sigma$,
is analytic in $\cb\setminus\Sigma$, and clearly solves the RHP $(\Sigma, v)$ a.e. on $\Sigma$.
Although the original RHP $(\rb, v)$ has a classical solution,
in the process of solving for $m(x,t,z)$ as $t\to\infty$,
we will encounter RHP's which only have solutions in the generalized sense (see e.g. \eqref{eq:delta_def} below, etc.).
It is easy to see that the classical solution $m(x,t,z)$ of $(\Sigma, v)$ is a also a solution in the generalized sense.

The first step in the steepest descent method is to separate $e^{i\theta}$ and $e^{-i\theta}$ in the jump matrix $v(x,t,z)$
in such a way as to respect the signature table for $\mathrm{Re}(i\theta)$.
We have $\mathrm{Re}(i\theta)=\mathrm{Re} [ it(z_0^2-(z-z_0)^2)/2 ]=t\mathrm{Im}(z-z_0)^2/2$
where $z_0 = x/t$ is the stationary phase point for $\theta$.
Let $\delta_\pm$ be the solution of the scalar, normalized RHP
$({\mathbb{R}}_- + z_0, 1+ |r|^2)$,
\beq
\label{eq:delta_def}
\left\{\aligned
&\delta_+ = \delta_- (1+|r|^2),\qquad z\in {\mathbb{R}}_-
+z_0,\\
&\delta_\pm - 1 \in \partial   C (L^2),
\endaligned\right.
\eeq
where the contour ${\mathbb{R}}_- + z_0$ is oriented from $-\infty$ to
$z_0$.

\begin{proposition}[Properties of $\delta$]
\label{prop:prop_delta}
Suppose $r\in L^\infty({\rb}) \cap
L^2({\rb})$ and $\|r\|_{L^\infty} \le \rho < \infty$.
The extension $\delta$ of $\delta_\pm$ off ${\rb}_- +z_0$ is given by the formula
\begin{equation}
\label{eq:delta}
\delta(z) = \exp \Big[ \frac1{2\pi i} \int^{z_0}_{-\infty} \frac{\log
(1+|r(s)|^2)}{s-z} ds \Big],\qquad z\in {\cb} \setminus ({\rb}_-
+ z_0).
\end{equation}
and satisfies for $z\in {\cb}\setminus ({\rb}_- +z_0)$,
$$\delta(z) \overline{\delta(\bar z)} = 1,$$
$$ (1+\rho^2)^{-\frac12} \le
|\delta(z)|, |\delta^{-1}(z)| \le (1+\rho^2)^{\frac12}$$ and
$$|\delta^{\pm 1}(z)|\le 1 \quad \text{for}\quad
\pm \text{Im } z>0.$$
For real $z$,
$$|\delta_+(z) \delta_-(z)| = 1 \quad \text{and, in
particular,}\quad |\delta(z)|  = 1\quad \text{for} \quad
z>z_0,$$
$$|\delta_+(z)| = |\delta^{-1}_-(z)| = (1 + |r(z)|^2
)^{\frac12},\qquad z<z_0,$$ and
$$\Delta\equiv \delta_+\delta_- = \exp \Big[\frac1{i\pi}
\text{ P.V. } \int^{z_0}_{-\infty}
\frac{\log(1+|r(s)|^2)}{s-z}ds \Big], $$
where $P.V.$ denotes
the principal value and $|\Delta| =
|\delta_+\delta_-| = 1$.
$$\|\delta_\pm - 1\|_{L^2(dz)} \le c\rho (1+\rho^2)^{\frac12}\|r\|_{L^2} .$$
\end{proposition}
We write $w=(w^-, w^+)$ for a matrix $v$ and any of its factorization $v = (v^-)^{-1} v^+ = (I-w^-)^{-1}(I+w^+)$,
$v^\pm$, $(v^\pm)^{-1} \in L^\infty(\Sigma)$.
Denote the singular integral operator $C_w(h) = C^+(hw^-) + C^-(hw^+)$ acting on $L^p$-function valued matrices
where $C^\pm$ is a limit of Cauchy operator $C$ from the (+)-side(respectively, (-)-side) of the oriented contour $\Sigma$.
Note that if $w=(0,v-I)$ then $C_w = C_v$ as in \eqref{eq:cauchy_v_op}.
As shown in \cite{DZ}, if $(1-C_w)^{-1}$ is bounded in $L^p$ for one factorization $v=(1-w^{-1})^{-1}(I+w^+)$,
then it is bounded for all such factorizations.

\begin{lemma}
\label{lem:exist_resol_Cw}
Let $v$ be as in \eqref{eq:jump_v}, with $r\in H^{1,1}(\rb)$.
Then for any factorization $v = (v^-)^{-1} v^+$, $v^\pm, (v^\pm)^{-1} \in L^\infty(\rb)$,
the operator $1-C_w$ with $w = (w^-, w^+) = (I-v^-,v^+ - I)$ has a bounded inverse in $L^2(\rb)$ and
$$
\|(1-C_{w})^{-1}\|_{L^2(\mathbb{R})} \leq c(1+\rho^2),
$$
where $\|r\|_\infty \leq \rho < \infty$.
\end{lemma}

\begin{proof}
By Corollary 2.7 in \cite{DZ}, it is enough to show that $1-C_v$ is invertible in $L^2(\rb)$.
We will first show that $1-C_v$ is a Fredholm operator of index 0.
Define $w=(0,v-I)$ and $\hat{w} = (0, \inv v - I)$. Then for $f \in L^2 (\rb)$, using \eqref{eq:cauchy_bd_id} and \eqref{eq:cauchy_v_op},
$$
\begin{aligned} C_w (C_{\hat{w}}f)
& = C^-[(C_{\hat{w}}f)(v-I)] = C^-\big[C^-(f(\inv v - I))(v-I)\big] \\
& = C^-\big[\big(C^+(f(\inv v - I)) - f(\inv v - I)\big)(v-I)\big] \\
& = C^-\big[C^+(f(\inv v - I))(v-I) + f( v + \inv v -2I)\big] \\
& = C^-\big[C^+(f(\inv v - I))(v-I)\big] + C_w f + C_{\hat{w}}f. \\
\end{aligned}
$$
As $v$ is continuous and $v-I\to 0$ as $|z|\to \infty$,
it follows that we can approximate $v-I$ by a sequence of rational functions of the form,
$$ \sum_{i=1}^{n} \frac{B_i}{z-a_i}, \ \ a_i \in \cb \setminus \rb.$$
Let $h = f(\inv v - I)$. Then,
$$
\begin{aligned}
C^-\Big((C^+h)(\diamondsuit) \frac{1}{\diamondsuit - a_i}\Big)(z)
& = \lim_{\epsilon \downarrow 0} \int_{\rb} \frac{(C^+h)(s)}{s - a_i}\frac 1{s-(z-i\epsilon )} \frac{\rd s}{2 \pi i} \\
& =\left\{ \begin{aligned}
& 0 \qquad \quad\quad \text{if } a_i \in \cb^-, \\
& \frac{Ch(a_i)}{a_i-z} \quad \ \text{if } a_i \in \cb^+. \\
\end{aligned} \right. \\
\end{aligned}
$$
which implies that $f \mapsto C^-\Big((C^+(f(\inv v - I))) \frac{1}{\diamondsuit - a_i}\Big)$
has either rank 0 or 1 and hence is compact in $L^2(\rb)$.
It follows that $C^-[C^+(f(\inv v - I))(v-I)]$ is operator limit of compact operators and hence is compact.
But $(1-C_w)(1-C_{\hat{w}}) = 1 - C_w - C_{\hat{w}} + C_w C_{\hat{w}}= 1+ C^-[C^+(\diamondsuit(\inv v - I))(v-I)]$.
Similarly, we see that $(1-C_{\hat{w}})(1-C_w)-1$ is compact, and hence $1-C_w$ is a Fredholm operator.
For $0\leq\gamma\leq1$, set $v(\gamma) \equiv \bsm 1+\gamma^2|r|^2&\gamma r e^{i \theta}\\ \gamma\overline{r}e^{-i \theta}&1   \esm $.
Clearly $v(1) = v$. The same argument shows that $v(\gamma)$ is Fredholm for all $\gamma\in[0,1]$.
But for $\gamma = 0$, $v(0) = I$ and so $1-C_{v(0)} = 1$;
we conclude that ind$(1-C_{v(\gamma=1)}) = $ind$(1-C_{v(\gamma=0)}) = 0$.
Thus $1-C_v$ is Fredholm of index zero.
To complete the proof, we must show that $(1-C_v)f=0$, $f\in L^2(\rb)$, implies that $f=0$.
Set $m_\pm = C^\pm(f(v-I)) \in \partial C(L^2)$.
Then as $f = C_v f = C^- (f(v-I)) = m_-$, we have
$$\bal
m_+
& = C^+ (f(v-I)) = C^-(f(v-I)) + f(v-I) \\
& = m_- + m_- (v-I) = m_- v.\\
\eal
$$
On the other hand, by Cauchy's theorem,
$$
0 = \int_{\rb} m_+ m_-^* \rd z = \int_{\rb} m_- v m_-^* \rd z.
$$
But $v$ is clearly strictly positive definite and hence $f=m_-\equiv 0$.

Suppose that $m_\pm \in \partial C(L^2)$ and $m_+ = m_- v + f$ for $f\in L^2(\rb)$.
As above, we have
$$
0 = \int_{\rb} m_+ m_-^* \rd z = \int_{\rb} m_- v m_-^* + \int_{\rb} f m_-^*.
$$
and hence
$$
\bigg|\int_{\rb} m_- v m_-^* \bigg| \leq \|f\|_{L^2} \|m_-\|_{L^2}.
$$
But $v\geq \lambda_0$, where $\lambda_0 \geq \frac{2+\rho^2 - \sqrt{(2+\rho^2)^2 - 4}}2 >0$ is the smallest eigenvalue of $v$.
Hence $\lambda_0 \|m_-\|_{L^2}^2 \leq \|f\|_{L^2} \|m_-\|_{L^2}$ or $\|m_-\|_{L^2} \leq \frac 1{\lambda_0} \|f\|_{L^2} $.
But this bound then implies $\| (1-C_w)^{-1}\| \leq \frac c{\lambda_0}$, for any factorization $v=(I-w_-)^{-1} (I+w_+)$ of $v$,
by Proposition 2.6 and Corollary 2.7 in \cite{DZ}.
\end{proof}

Let $m_\delta \equiv m \delta^{-\sigma_3}$ solve the normalized RHP $(\rb, v_{\delta})$ where $v_{\delta} = \delta_-^{\sigma_3} v \delta_+^{-\sigma_3}$. In other words,
$$ \begin{aligned}
w_\delta &= (w_\delta^- , w_\delta^+ ) \\
&= \left\{
\begin{aligned}
& \bigg( \bpm 0&r \delta^2 e^{i\theta} \\ 0&0 \epm , \bpm 0&0 \\ \overline{r}  \delta^{-2}e^{-i\theta}&0 \epm \bigg) , \qquad \qquad   z > z_0, \\
& \Bigg( \bpm 0&0 \\ \frac{\overline{r}}{1+|r|^2} \delta_-^{-2} e^{-i\theta} & 0 \epm ,
\bpm 0 & \frac{r}{1+|r|^2} \delta_+^2 e^{i\theta} \\ 0 & 0 \epm  \Bigg) , \ \ z < z_0.
\end{aligned}
\right.
\end{aligned} $$

\begin{notation}
We say that the bounds on two operators $A$ and $B$ in a Banach space are \emph{equivalent}
if and only if $c^{-1}\|B\| \leq \|A\| \leq c\|B\|$ for some $c>1$.
\end{notation}

\begin{lemma}[cf. \cite{DZ}, p. 1042]
\label{lem:conj_equiv}
The operator $(1-C_{w_\delta})^{-1}$ exists in $L^2(\mathbb{R})$ and the bound on $(1-C_{w_\delta})^{-1}$ is equivalent to the bound on $(1-C_{w})^{-1}$.
\end{lemma}

For a function $f$ on $\mathbb{R}$ we introduce a rational function which coincides with $f$ at $z_0$,
\begin{equation}
\label{eq:bracket}
[f](z)=\frac {f(z_0)}{(1+i(z-z_0))^2},\ \ \ z\in\mathbb{R} \ .
\end{equation}
Replacing $r$ by the rational function $[r]$ etc., we obtain
$$ \begin{aligned}
w_\delta^{[\cdot]} &= ((w_\delta^{[\cdot]})^- , (w_\delta^{[\cdot]})^+ ) \\
&\equiv \left\{
\begin{aligned}
& \bigg( \bpm 0&[r] \delta^2 e^{i\theta} \\ 0&0 \epm , \bpm 0&0 \\ [ \overline{r} ] \delta^{-2}e^{-i\theta} &0 \epm \bigg) , \ \ \qquad \qquad  \ \ z > z_0, \\
& \Bigg( \bpm 0&0 \\ \Big[ \frac{\overline{r}}{1+|r|^2} \Big] \delta_-^{-2} e^{-i\theta} & 0 \epm ,
\bpm 0 & \Big[ \frac{r}{1+|r|^2} \Big] \delta_+^2 e^{i\theta}  \\ 0 & 0 \epm  \Bigg) , \ \ z < z_0.
\end{aligned}
\right.
\end{aligned} $$
Let $\Sigma^e\equiv\rb\cup(z_0+e^{i\pi/4}\rb)\cup(z_0+e^{-i\pi/4}\rb)$
be oriented from left to right (see Figure 1).
\begin{figure}
\begin{center}
\includegraphics{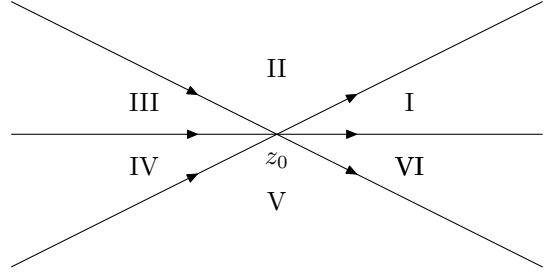}
\caption{Contour $\Sigma^e$}
\end{center}
\end{figure}

Introduce the trivial extension $v_e^{[\cdot]}$ of $v_\delta^{[\cdot]}$ onto $\Sigma^e$
 by setting $v_e^{[\cdot]}=v_\delta^{[\cdot]}$ on $\rb$ and $v_e^{[\cdot]}=I$ on $\Sigma^e \setminus \rb$.

\begin{lemma}[cf. \cite{DZ}, pp. 1043--1044]
\label{lem:triv_ext_equiv}
Suppose that $v_e^{[\cdot]} = (1-(w_e^{[\cdot]})^-)^{-1}(1+(w_e^{[\cdot]})^+)$ is a factorization as above.
Then, $(1-C_{ w_\delta^{[\cdot]}})^{-1}$ is bounded in $L^2(\rb)$ if and only if $(1-C_{ w_e^{[\cdot]}})^{-1}$ is bounded in $L^2(\Sigma^e)$
and the bounds are equivalent.
\end{lemma}

Now, set
\begin{equation}
\label{eq:Phi_deform}
 \Phi =
\left\{
\begin{aligned}
& I \quad \  \qquad\qquad\qquad\qquad\text{on II or V}\\
& \bpm 1&0\\  -[\bar{r}]\delta^{-2} e^{-i\theta} &1\epm\qquad\quad \ \text{on I}\qquad\qquad\qquad\hfill\\
& \bpm 1&-\Big[\frac{r}{1+|r|^2} \Big]\delta^2 e^{i\theta} \\ 0&1\epm  \ \quad \text{on III}\hfill\\
& \bpm 1&0\\ \Big[\frac{\ovl{r}}{1+ |r|^2}\Big]
\delta^{-2}e^{-i\theta}  &1\epm \ \quad
\text{on IV}
\hfill\\
& \bpm 1&[r] \delta^2 e^{i\theta} \\
0&1\epm\  \ \quad \qquad\quad\text{on VI}\hfill
\end{aligned}
\right.
\end{equation}
and set $v_d^{[\cdot]} \equiv \Phi_-^{-1} v_e^{[\cdot]} \Phi_+$ where $\Phi\pm$ denote the $(\pm)$-limits on the contour $\Sigma^e$.
Direct calculation shows that

\begin{equation}
\label{eq:v_d_deform}
 v_d^{[\cdot]} =
\left\{
\begin{aligned}
& I \ \ \ \qquad\qquad\qquad\qquad\text{on}\quad \rb\hfill\\
& \bpm 1&0\\  [\bar{r}]\delta^{-2} e^{-i\theta}&1\epm\qquad\quad \ \ \text{on}\quad z_0 + e^{i\pi/4} \rb^+\hfill\\
& \bpm 1&\Big[\frac{r}{1+|r|^2}\Big]\delta^2 e^{i\theta}\\ 0&1\epm \ \ \quad \text{on}\quad z_0 + e^{-i\pi/4} \rb^-\hfill\\
& \bpm 1&0\\ \Big[\frac{\ovl{r}}{1+ |r|^2}\Big]
\delta^{-2}e^{-i\theta}&1\epm  \quad
\text{on}\quad z_0 + e^{i\pi/4} \rb^-
\hfill\\
& \bpm 1&[r] \delta^2 e^{i\theta}\\
0&1\epm\  \quad \qquad\quad\text{on}\quad z_0 + e^{-i\pi/4}
\rb^+.\hfill
\end{aligned}
\right.
\end{equation}

Because $\Phi$ is analytic and uniformly bounded in $\cb \setminus \Sigma^e$ for all $x \in \rb$, $t \geq 0$, we obtain the following.
\begin{lemma}[cf. \cite{DZ}, pp. 1044--1045]
\label{lem:deform_equiv}
The operator $(1-C_{ v_d^{[\cdot]}})^{-1}$ is bounded if and only if $(1-C_{ v_e^{[\cdot]}})^{-1}$ is bounded in $L^2(\Sigma^e)$
and the bounds are equivalent.

The bound on , if it exists, is equivalent to the bound on .
\end{lemma}
As noted above, bounds on $(1-C_{ v_d^{[\cdot]}})^{-1}$ and $(1-C_{ v_e^{[\cdot]}})^{-1}$
imply bounds on $(1-C_{ w_d^{[\cdot]}})^{-1}$ and $(1-C_{ w_e^{[\cdot]}})^{-1}$
for any factorizations $v_d^{[\cdot]} = \big(1-(w_d^{[\cdot]})^-\big)^{-1}\big(1+(w_d^{[\cdot]})^+\big)$
and $v_e^{[\cdot]} = \big(1-(w_e^{[\cdot]})^-\big)^{-1}\big(1+(w_e^{[\cdot]})^+\big)$.
For $z<z_0$, define
$$
\bal & \beta(z,z_0) \equiv \\
& \ \int^{z_0}_{-\infty} \{\log(1+|r(s)|^2) -\log(1+|r(z_0)|^2) \chi^0(s) (s-z_0+1)\} \frac{ds}{2\pi i(s-z)}, \\
\eal
$$
where $\nu(z_0) = \frac1{2\pi} \log(1+|r(z_0)|^2)$ and
$\chi^0(s)$ denotes the characteristics function of the interval $(z_0-1,z_0)$.
Define
\beq
\label{eq:delta_0_def}
\delta_0(z) \equiv e^{\beta(z_0,z_0) - i\nu(z_0)} (z-z_0)^{-i\nu(z_0)}.
\eeq
Set
\begin{equation}
\label{def:v_D_L}
 v_d^L =
\left\{
\begin{aligned}
& I \qquad\qquad\qquad\qquad\qquad\qquad\text{on}\quad \rb\hfill\\
& \bpm 1&0\\ \bar r(z_0)\delta^{-2}_0 e^{-i\theta}&1\epm\qquad\quad \quad \text{on}\quad z_0 + e^{i\pi/4} \rb^+\hfill\\
& \bpm 1& \frac{r(z_0)}{1+|r(z_0)|^2}\delta^2_0 e^{i\theta}\\ 0&1\epm \qquad \quad \ \ \text{on}\quad z_0 + e^{-i\pi/4} \rb^-\hfill\\
& \bpm 1&0\\ \frac{\bar r(z_0)}{1+ |r(z_0)|^2}
\delta^{-2}_0 e^{-i\theta}&1\epm  \qquad \ \
\text{on}\quad z_0 + e^{i\pi/4} \rb^-
\hfill\\
& \bpm 1&r(z_0) \delta^2_0 e^{i\theta}\\
0&1\epm \ \qquad\qquad\quad\text{on}\quad z_0 + e^{-i\pi/4}
\rb^+.\hfill
\end{aligned}
\right.
\end{equation}
\begin{lemma}[cf. \cite{DZ}, Lemma 2.14]
\label{lem:close_to_local}
Let $r \in H^{1,0}(\rb)$ with $\|r\|_{H^{1,0}(\rb)} \leq \lambda < \infty$. Then, for $1\le p \le \infty$,
$$
\| v_d^{[\cdot]} -  v_d^L\|_{L^p(\Sigma^e)} \le
\frac{c}{t^{\frac14+\frac1{2p}}}$$
uniformly for $t\ge 1$
and all $x\in\rb$ where $c$ depends on $\lambda$.
Hence, $(1-C_{ v_d^{[\cdot]}})^{-1}$ is bounded if and only if $(1-C_{ v_d^L})^{-1}$ is bounded in $L^2(\Sigma^e)$
and the bounds are equivalent uniformly for $t\ge 1$ and all $x\in\rb$.
\end{lemma}

\begin{remark}
If $r$ had more decay, say $r \in H^{1,1}(\rb)$, one sees by integration by part that
\begin{equation}
\beta(z_0,z_0) - i\nu(z_0) = -\frac1{2\pi i} \int^{z_0}_{-\infty} \log(z_0-s) d \log (1+|r(s)|^2) \ .
\end{equation}
\end{remark}

\begin{lemma}[cf. \cite{DZ}, p. 1049]
\label{lem:deform_local}
The operator $(1-C_{ v_d^L})^{-1}$ exists in $L^2(\Sigma^e)$ and its bound is equivalent to the bound on $(1-C_{ v^L})^{-1}$ in $L^2(\rb)$
where $ v^L \equiv \bpm 1+|r(z_0)|^2&r(z_0)e^{i\theta}\\ \overline{r(z_0)}e^{-i\theta}&1 \epm$.
\end{lemma}

\begin{proposition}[cf. \cite{DZ}, Proposition 2.13]
For all sufficiently large $t$, say $t \geq t_0$, $(1-C_{ w_\delta^{[\cdot]}})^{-1}$ exists in $L^2(\rb)$, and for some uniform constant c
$$ \| (1-C_{ w_\delta^{[\cdot]}})^{-1} \|_{L^2(\rb)} \leq c \text{ for }  x \in \rb \text{ and } t \geq t_0 . $$
\end{proposition}


From the previous results, we see that for $t\geq t_0$,
$(1-C_{w_\delta^{[\cdot]}})^{-1}$, $(1-C_{v_e^{[\cdot]}})^{-1}$, $(1-C_{v_d^{[\cdot]}})^{-1}$ and $(1-C_{v_d^L})^{-1}$ exist and bounded in $L^2$
and the bounds are equivalent.
It follows that the normalized RHPs $(\rb, v_\delta^{[\cdot]})$, $(\Sigma^e, v_e^{[\cdot]})$, $(\Sigma^e, v_d^{[\cdot]})$ and  $(\Sigma^e, v_d^L)$
have unique solutions $m_\delta^{[\cdot]}$, $m_e^{[\cdot]}$, $m_d^{[\cdot]}$ and $m_d^L \in I+\partial C(L^2)$, respectively.
Note that
\beq
\label{eq:phi_delta_d}
m_\delta^{[\cdot]} = m_d^{[\cdot]}\Phi^{-1}.
\eeq

For reasons that will become clear (see \eqref{comp_contour} below),
we now reverse the orientation on $\rb_-+z_0$ to obtain a contour $\widetilde{\rb}_{z_0} = e^{i\pi}(\rb^++z_0) \cup (\rb^++z_0)$
with associated jump matrix $\widetilde{v}_\delta = v_\delta $
for $z>z_0$ and $\widetilde{v}_\delta = v_\delta^{-1}$ for $z<z_0$.
From the form of $v_\delta$, we see that $\widetilde{v}_\delta = (I- \widetilde{w}_\delta^-)^{-1}(I + \widetilde{w}_\delta^+)$ where
$$ \begin{aligned}
\widetilde w_\delta &= (\widetilde w_\delta^- , \widetilde w_\delta^+ ) \\
&= \left\{
\begin{aligned}
& \bigg( \bpm 0&r \delta^2 e^{i\theta} \\ 0&0 \epm , \bpm 0&0 \\  \overline{r}  \delta^{-2} e^{-i\theta}&0 \epm \bigg) , \ \ \qquad  \quad \ \ z > z_0, \\
& \Bigg( \bpm 0&-\frac{r}{1+|r|^2} \tilde\delta_-^2  e^{i\theta}\\ 0& 0 \epm ,
\bpm 0 & 0  \\ -\frac{\overline{r}}{1+|r|^2}  \tilde\delta_+^{-2} e^{-i\theta} & 0 \epm  \Bigg) , \ \ z < z_0.
\end{aligned}
\right.
\end{aligned} $$
where $\tilde\delta_\pm(z)$ denotes the boundary
values of $\delta(z)$ on $\widetilde{\rb}_{z_0}$. Thus
$\tilde\delta_\pm(z) = \delta_\pm(z)$ for $z>z_0$ and
$\tilde\delta_\pm(z) = \delta_\mp(z)$ for $z<z_0$.

We use $\rho,\lambda$ and $\eta$ to denote $L^\infty$, $H^{1,0}$ and $H^{1,1}$ bounds for $r$ respectively. Thus
$$
\|r\|_{L^\infty} \le \rho,\quad \|r\|_{H^{1,0}} \le \lambda,\quad \|r\|_{H^{1,1}} \le \eta .
$$
Extend $\widetilde{\rb}_{z_0}$ to a contour $\Gamma_{z_0} =
\widetilde{\rb}_{z_0} \cup (z_0 + e^{i\pi/2} \rb_-) \cup (z_0 +
e^{-i\frac\pi2}\rb_-)$.
As a complete contour, $\Gamma_{z_0}$ has the important property
\beq
\label{comp_contour}
C^+_{\Gamma_{z_0}} C^-_{\Gamma_{z_0}} = C^-_{\Gamma_{z_0}}
C^+_{\Gamma_{z_0}} = 0.
\eeq
In the following lemmas we will
assume for convenience that $x=0$. Thus $z_0=0$, $\delta =
\delta_{z_0=0}$, $\Delta = \Delta_{z_0=0}$, $\widetilde{\rb}
\equiv \widetilde{\rb}_{z_0=0}$, $\Gamma \equiv
\widetilde\Gamma_{z_0=0}$, and $\theta = -tz^2/2$.

\begin{lemma}[cf. \cite{DZ}, Lemma 4.1]
For $z\in \rb \backslash 0$,
$$
|\Delta'(z)|\le \text{\rm I} + \text{\rm II},$$ where
$$\begin{aligned}
&\|\text{\rm I}\|_{L^2} \leq c\rho\lambda\\
&|\text{\rm II}| \leq c\rho^2 \frac1{|z|}.
\end{aligned}$$
\end{lemma}


\begin{lemma}[\cite{DZ}, p.1064]
Suppose $f\in H^{1,0}$. Then for all $t>1$,
\begin{equation}
\label{eq:osci_line}
\left|\int_{\rb} f\Delta^{\pm1} e^{\mp itz^2}\ dz\right| \le
\frac{c(1+\rho\lambda)}{t^{{\frac 12}}}
\|f\|_{H^{1,0}}.
\end{equation}
In addition, if $f(0) = 0$, then for
all $t>1$,
\begin{equation}
\label{eq:osci_line_0}
\left|\int_{\rb} f\Delta^{\pm 1}e^{\mp itz^2} \ dz\right| \le
\frac{c(1+\rho\lambda)}{t^{3/4}}
\|f\|_{H^{1,0}}.
\end{equation}
\end{lemma}

\begin{figure}
\label{fig:gamma}
\begin{center}
$$\vbox{\offinterlineskip \halign{\strut \quad \hfil $#$\hfil\quad &
\vrule\quad \hfil $#$\hfil\quad\cr D_2&D_1\cr \noalign{\hrule}
D_3&D_4\cr}}$$
\caption{Contour $\Gamma$}
\end{center}
\end{figure}

Let $D_j, j=1,\ldots, 4$, be the $j^{\text{\rm th}}$ quadrant in
$\cb\backslash\Gamma$
(see Figure 2).
In the Lemma below $H^q$ denotes Hardy space.

\begin{lemma}[\cite{DZ}, pp.1065--1069]
Suppose $f\in H^{1,0}$, then for $2\le p < \infty$ and for all $t\ge 0$,
\begin{equation}
\label{eq:osci_gamma}
\left\{
\begin{aligned}
	&\begin{aligned}
	\|C^-_{\rb^+\to\Gamma} \delta^{-2} fe^{it\diamondsuit^2}\|_{L^p}
		& \le {\dsize\frac{c}{(1+t)^{{\frac 12}p}}} \|\delta^{-2}\|_{L^\infty(D_1)} \|f\|_{H^{1,0}} \\
		& \le {\dsize\frac{c (1+\rho^2)}{(1+t)^{{\frac 12}p}}} \|f\|_{H^{1,0}},\hfill
	\end{aligned} \\
	&\begin{aligned}
	\|C^-_{e^{i\pi} \rb^+\to \Gamma} \tilde \delta^{-2}_+ fe^{it\diamondsuit^2} \|_{L^p}
		& \le {\dsize\frac{c}{(1+t)^{{\frac 12}p}}} \|\delta^{-2}\|_{L^\infty(D_3)} \|f\|_{H^{1,0}} \\	
		& \le {\dsize\frac{c (1+\rho^2)}{(1+t)^{{\frac 12}p}}} \|f\|_{H^{1,0}},\hfill
	\end{aligned} \\
	&\begin{aligned}
	\|C^+_{\rb^+\to\Gamma} \delta^2 fe^{-it\diamondsuit^2}\|_{L^p}
		& \le {\dsize\frac{c}{(1+t)^{{\frac 12}p}}} \|\delta^2\|_{L^\infty(D_4)} \|f\|_{H^{1,0}} \\	
		& \le {\dsize\frac{c (1+\rho^2)}{(1+t)^{{\frac 12}p}}} \|f\|_{H^{1,0}},\hfill
	\end{aligned} \\
	&\begin{aligned}
	\|C^+_{e^{i\pi}\rb^+\to \Gamma} \tilde\delta^2_-fe^{-it\diamondsuit^2}\|_{L^p}
		& \le {\dsize\frac{c}{(1+t)^{{\frac 12}p}}} \|\delta^2\|_{L^\infty(D_2)} \|f\|_{H^{1,0}} \\
		& \le {\dsize\frac{c (1+\rho^2)}{(1+t)^{{\frac 12}p}}} \|f\|_{H^{1,0}},\hfill
	\end{aligned} \\
\end{aligned}
\right.
\end{equation}
Suppose in addition that $f(0) = 0$ and that $g$ is a function in
the Hardy space $H^q(\cb\backslash \rb)$ for some $2\le q\le
\infty$. Then for all $t\ge 0$,
\begin{equation}
\label{eq:osci_gamma_0}
\left\{
\begin{aligned}
& \|C^-_{\rb^+\to\Gamma} g_+ fe^{it\diamondsuit^2}\|_{L^2} \le {\dsize\frac{c}{(1+t)^{\frac12-\frac1q}}} \|g\|_{H^q(\cb \backslash \rb)} \|f\|_{H^{1,0}},\hfill\\
& \|C^-_{e^{i\pi}\rb^+\to\Gamma} \tilde g_+ fe^{it\diamondsuit^2}\|_{L^2} \le {\dsize\frac{c}{(1+t)^{\frac12-\frac1q}}} \|g\|_{H^q(\cb\backslash\rb)} \|f\|_{H^{1,0}},\hfill\\
& \|C^+_{\rb^+\to\Gamma} g_-fe^{-it\diamondsuit^2}\|_{L^2} \le
{\dsize\frac{c}{(1+t)^{\frac12-\frac1q}}}\|g\|_{H^q(\cb
\backslash\rb)} \|f\|_{H^{1,0}},\hfill\\
& \|C^+_{e^{i\pi}\rb^+\to\Gamma} \tilde g_- fe^{-it\diamondsuit^2}
\|_{L^2} \le {\dsize\frac{c}{(1+t)^{\frac12-\frac1q}}}
\|g\|_{H^q(\cb\backslash\rb)}
\|f\|_{H^{1,0}},\hfill
\end{aligned}
\right.
\end{equation}
where $g_\pm$
are the boundary values of $g$ on $\rb$ and $\tilde g_\pm = g_\mp$
on $e^{i\pi}\rb^+$.
\end{lemma}

Since the $H^{1,0}$- and $L^\infty$-norms of $r(z)$ is invariant under translation,
appropriate choices of $f$ in \eqref{eq:osci_gamma} give rise to the following.
\begin{corollary}[cf. \cite{DZ}, Corollary 4.4]
\label{cor:cw}
For any $2\le p < \infty$, and for all $t\ge 0$
$$
\|C^\pm_{\widetilde{\rb}_{z_0}\to \Gamma_{z_0}} \widetilde{
{w}}_\delta^\pm\|_{L^p(\Gamma_{z_0})} \le \frac{c \lambda (1+\rho^2)}{(1+t)^{{\frac 12}p}}. $$
\end{corollary}

We need the following $L^p$ bound on solutions of RHPs.
\begin{proposition}[cf. \cite{DZ}, Proposition 4.5]
\label{prop:L_p_bound_op}
Suppose that r is a continuous function on $\rb$ such that
$$\lim_{z\to\infty}r(z)=0, \ \text{ and } \ \|r\|_{L^\infty(\rb)}\leq \rho < \infty. $$
Let $v$ be given as in \eqref{eq:jump_v} and $v=(1-w^-)^{-1}(1+w^+)$ is any factorization as above.
Then for any $p\ge2$, there exists $t_0=t_0(r,p)$ such
that for $t\ge t_0$ and all $x\in\rb$, $(1-C_{ w})^{-1}$
exists in $L^p(\rb)$ and
$$\|(1-C_{ w})^{-1}\|_{L^p(\rb)\rightarrow L^p(\rb)}\le
c $$ for $t\ge t_0$ and all $x\in\rb$.
\end{proposition}

By Lemma \ref{lem:conj_equiv} and Proposition \ref{prop:L_p_bound_op}, we obtain for $r\in H^{1,0}$ and $2\le p<\infty$,
$$
\|(1-C_{\widetilde{
 w}_\delta})^{-1}\|_{L^p(\widetilde{\rb}_{z_0})\to
L^p(\widetilde{\rb}_{z_0})} \le c_p,
$$
where $c_p$ is
uniform for all $x\in\rb$ and all $t\ge t_0$.
Reversing the orientation on $\rb$ as above, we see that (cf. Proposition 2.8 in \cite{DZ})
\begin{equation}
\label{eq:Lp_bound}
\|(1-C_{\widetilde{
 w}_\delta})^{-1}\|_{L^p(\widetilde{\rb}_{z_0})\to
L^p(\widetilde{\rb}_{z_0})} \le c'_p,
\end{equation}
where $c'_p$ is uniform for all $x\in\rb$ and all $t\ge t_0$.
Let $\widetilde \mu_\delta \in I + L^2(\widetilde \rb)$ be the solution of $(1-C_{\widetilde w_\delta}) \widetilde\mu_\delta = I$.
Using Corollary \ref{cor:cw} and \eqref{eq:Lp_bound}, we obtain the following result.

\begin{corollary}[cf. \cite{DZ}, Corollary 4.6]
\label{cor:mu_minus_I}
For any $2\le p < \infty$ and for all $t\ge 0$
$$
\|\first \mu - I\|_{L^p} \le \frac{c_p \lambda (1+\rho^2)}{(1+t)^{{\frac 12}p}}. $$
\end{corollary}

\begin{lemma}[cf. \cite{DZ}, Lemma 4.7]
\label{lem:estimate}
For any $2\le p < \infty$ and for $t$ sufficiently large,
$$\|C^\pm \first\mu (\widetilde{ w}_\delta^{[\cdot]\pm} - \widetilde{ w}_\delta^\mp) \|_{L^2} \le ct^{-\frac12+\frac1{2p}} .$$
\end{lemma}


\begin{lemma}[cf. \cite{DZ}, Lemma 4.8]
\label{lem:bracket_close}
For any $2\le p<\infty$ and for $t$ sufficiently large,
\beq
\label{eq:lem_bracket_close}
\Bigg|\int_{\widetilde\rb} \first{\mu}\firste{w} - \int_{\widetilde\rb} \second{\mu}\seconde{w} \Bigg| \le ct^{-\frac34+\frac1{2p}}
\eeq
and for $z\in\cb\setminus\rb$,
\beq
\label{eq:lem_bracket_close2}
 \big|\first{m}(z) - \second{m}(z)\big| \le c(z) t^{-\frac34+\frac1{2p}}
 \eeq
where $c(z)$ is a constant depending on $z$.
\end{lemma}
\begin{proof}
The proof of \eqref{eq:lem_bracket_close} follows the proof of Lemma 4.8 in \cite{DZ}.
The proof of \eqref{eq:lem_bracket_close2} is similar using the relation
$\first m(z) = I + \int \frac{\first \mu(s)\first w(s)}{s-z} \frac{ ds}{2\pi i}$, and
$\second m(z) = I + \int \frac{\second \mu(s)\second w(s)}{s-z} \frac{ ds}{2\pi i}$.
\end{proof}

Since $\Phi=I$ in the region II, $\widetilde{m}_\delta^{[\cdot]} = m_\delta^{[\cdot]} = m_d^{[\cdot]}\Phi^{-1} = m_d^{[\cdot]}$ on II
and hence the residues at $z=\infty$ are the same,
i.e. $\lim_{z\to\infty,z\in \text{II}} z \widetilde m_\delta^{[\cdot]} = \lim_{z\to\infty,z\in \text{II}} z m_d^{[\cdot]}$,
which implies $\int_{\widetilde\rb} \widetilde{{\mu}}_\delta^{[\cdot]} \widetilde{ w}_\delta^{[\cdot]} = \int_{\Sigma^e}{{\mu}}_d^{[\cdot]} {{w}}_d^{[\cdot]} $.

\begin{lemma}[cf. \cite{DZ}, pp. 1073--1074]
For $t$ sufficiently large,
\beq
\label{eq:lem_bracket_close_local}
\Bigg|\int_{\Sigma^e} {{\mu}}_d^L{ w}_d^L - \int_{\Sigma^e}{{\mu}}_d^{[\cdot]} {{w}}_d^{[\cdot]} \Bigg| \le ct^{-\frac34}.
\eeq
and for $z\in\cb\setminus\rb$,
\beq
\label{eq:lem_bracket_close_local2}
\big|m_d^L(z) - m_d^{[\cdot]}(z) \big| \le c(z)t^{-\frac34}
\eeq
where $c(z)$ is a constant depending on $z$.
\end{lemma}

\begin{proof}
As in Lemma \ref{lem:bracket_close}, the proof of \eqref{eq:lem_bracket_close_local2} is similar to the proof of \eqref{eq:lem_bracket_close_local}.
\end{proof}


If $m_1$ and $\delta_1$ are the residues at $z=\infty$ of $m(z)
= m(x,t,z)$ and $\delta(z)^{\sigma_3}$ respectively, we see that
$$ m_1(x,t) = -\frac{1}{2 \pi i} \int_{\rb}{\mu}{w} \quad \text{and} \quad \delta_1(x,t) = -\frac1{2\pi i} \int_{-\infty}^{z_0} \log (1 + |r|^2) \sigma_3. $$
Since $m_\delta(z) = \widetilde{m}_\delta(z) = m(z)\delta(z)^{-\sigma_3}$,
the residue $(\widetilde{m}_\delta)_1(x,t)$ at $z=\infty$ of $\widetilde{m}_\delta(x,t,z)$ is given by $(\widetilde{m}_\delta)_1 = m_1  - \delta_1$.
Assembling the above results, we conclude that for any $p \geq 2$ and all sufficiently large $t$,
\begin{equation}
\label{eq:u_c_localized}
\begin{aligned}
u_f(x,t) &= -i (m_1(x,t))_{12} = -i((\widetilde m_\delta)_1(x,t))_{12} \\
&= \frac{1}{2 \pi} \int_{\widetilde\rb} (\widetilde{{\mu}}_\delta\widetilde{ w}_\delta )_{12}
=  \frac{1}{2 \pi} \int_{\Sigma^e} ({{\mu}}_d^L{ w}_d^L )_{12} + O(t^{-\frac34+\frac1{2p}}) \\
&= -i ((m_d^L)_1(x,t))_{12} + O(t^{-\frac34+\frac1{2p}}), \\
\end{aligned}
\end{equation}
where $(m_d^L)_1$ is the residue of $m_d^L$ at $z=\infty$.
Similarly, for any fixed $z \in \cb\setminus\rb$,
\begin{equation}
\label{eq:m_z1_localized}
\begin{aligned} m(z)
&= \widetilde{m}_\delta(z)\delta^{\sigma_3}(z) \\
&= \widetilde m_\delta^{[\cdot]}(z) \delta^{\sigma_3}(z) + O(t^{-\frac34+\frac1{2p}}) \\
&= m_d^{[\cdot]}(z)\Phi^{-1}(z)\delta^{\sigma_3}(z)
+ O(t^{-\frac34+\frac1{2p}}) \\
&= m_d^L(z)\Phi^{-1}(z)\delta^{\sigma_3}(z)
+ O(t^{-\frac34+\frac1{2p}}). \\
\end{aligned}
\end{equation}

\begin{remark}
\label{rmk:error_prop_norm}
Clearly by the above estimates, the error terms in \eqref{eq:u_c_localized} and \eqref{eq:m_z1_localized}
are controlled by  $H^{1,0}$ norm of the reflection coefficient function $r(z)$.
In particular, if $\|r\|_{L^\infty} \leq c \epsilon$, $\epsilon \ll 1$, then it is easy to check that
the error terms in \eqref{eq:u_c_localized} and \eqref{eq:m_z1_localized}
are $O(\|r\|_{H^{1,0}} t^{-\frac34+\frac1{2p}})$, $O(\big\|\frac {r(\cdot)}{\cdot - z}\big\|_{H^{1,0}} t^{-\frac34+\frac1{2p}})$, respectively.
\end{remark}

If $\|r\|_{H^{1,0}}$ is sufficiently small, then the singular operators $(1-C_{w_{\delta}^{[\cdot]}})^{-1}$, etc. are bounded uniformly on $t\geq0$,
which leads the following result.
\begin{lemma}
\label{lem:small_r_estimate}
If $r(z)$ has sufficiently small norm in $H^{1,0}$, we have for all $t\geq 1$,
$$
u_f(x,t)= -i ((m_d^L)_1(x,t))_{12} + O(\|r\|_{H^{1,0}} t^{-\frac34+\frac1{2p}})
$$
and for any fixed $z\in \cb\setminus\rb$,
$$
m(z)= m_d^L(z)\Phi^{-1}(z)\delta^{\sigma_3}(z) + O\bigg(\Big\|\frac {r(\cdot)}{\cdot-z} \Big\|_{H^{1,0}} t^{-\frac34+\frac1{2p}}\bigg).
$$
\end{lemma}

\section{Localized RHP}
\label{sec:Localized RHP}
We compute the solution of the localized RHP $(\Sigma^e , {v}_d^L)$ explicitly in terms of parabolic cylinder functions (see \cite{AS}).
Introduce the scaling $\zeta  = \sqrt{t}(z-z_0)$ and set (see \eqref{eq:delta_0_def})
\beq
\label{eq:phi_of_zeta}
\phi (\zeta) = \delta_0^2(z)e^{i \theta(z)} = \alpha_0^2 \zeta^{-2 i \nu(z_0)} e^{-i\zeta^2/2}
\eeq
where
$$\alpha_0 = \exp \Big[\frac{i x^2}{4t} + \frac {i}2 \nu(z_0)\log t + \beta(z_0,z_0)-i \nu(z_0)\Big].$$
Define
\beq
\label{eq:def_M_infty}
M_d^L(x,t,\zeta) \equiv m_d^L(x,t,z), \ \ M^\infty(\zeta) \equiv M_d^L(\zeta)\phi^{\sigma}(\zeta), \ \ z \in \cb\setminus\Sigma^e,
\eeq
and
\begin{equation}
\label{def:v_D_L}
\bal V^\infty
& \equiv \phi (\zeta)^{-\text{ad}\sigma} v_d^L(z) \\
&
=\left\{
\begin{aligned}
& I \qquad\qquad \ \qquad\qquad\qquad\qquad\text{on } z_0 + \rb^+\hfill\\
& \bpm 1&0\\ \bar r(z_0)&1\epm \qquad\qquad \ \quad\quad \quad \ \text{on } z_0 + e^{i\pi/4} \rb^+\hfill\\
& \bpm 1+ |r(z_0)|^2&0\\ 0&\frac 1{1+ |r(z_0)|^2}\epm  \ \quad \ \
\text{on }  z_0 + \rb^-
\hfill\\
& \bpm 1& \frac{r(z_0)}{1+|r(z_0)|^2}\\ 0&1\epm \qquad\qquad \ \qquad   \text{on } z_0 + e^{-i\pi/4} \rb^-\hfill\\
& \bpm 1&0\\ \frac{\bar r(z_0)}{1+ |r(z_0)|^2}
&1\epm  \ \qquad\qquad \ \quad \ \
\text{on }  z_0 + e^{i\pi/4} \rb^-
\hfill\\
& \bpm 1&r(z_0) \\
0&1\epm \ \qquad\qquad \ \qquad\quad\text{on } z_0 + e^{-i\pi/4}
\rb^+.\hfill
\end{aligned}
\right. \\
\eal
\end{equation}
Then, it follows that $M^\infty_+(\zeta) = M^\infty_-(\zeta) V^\infty$ on $\zeta \in \Sigma^e-z_0$.
Since $\det V_d^L \equiv 1$ on $\Sigma^e$, $\det M_d^L(z) \equiv 1$ (see Remark \ref{rmk:det_1}), $\det M^\infty = \det(M_d^L \phi^{-\sigma}) = 1$,
and hence $(M_d^L)^{-1}, (M_\infty)^{-1}$ exist.
\begin{figure}
\begin{center}
\includegraphics{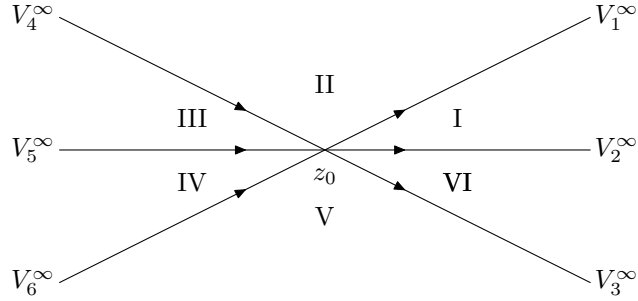}
\caption{Numbering of $V^\infty$ on $\Sigma^e$}
\end{center}
\end{figure}
Numbering the restrictions of piecewise constant jump matrix $V^\infty$ on $\Sigma^e$ as in Figure 3,
we have the cyclic condition $V^\infty_3 V^\infty_2 V^\infty_1 = V^\infty_6 V^\infty_5 V^\infty_4$,
which implies, in particular (see e.g. \cite{BDT}), that in each sector I, II, $\cdots$, VI,
$M^\infty(\zeta)$ is the restriction of an entire function.
In particular this implies that $M^\infty$ is differentiable with respect to $\zeta$
and $\partial_\zeta M^\infty_+(\zeta) = \partial_\zeta M^\infty_-(\zeta) V^\infty$ on $\zeta \in \Sigma^e-z_0$.
Thus, $(\partial_{\zeta}M^\infty) (M^\infty)^{-1}$ has a trivial jump matrix and hence is entire.
As $\zeta \rightarrow \infty$, $\zeta \in \cb\setminus(\Sigma^e-z_0)$,
$$
\begin{aligned}  \partial_{\zeta}M^\infty (M^\infty)^{-1}
& = \partial_{\zeta}(M_d^L \phi^\sigma)(\phi^{-\sigma}(M_d^L)^{-1}) \\
& = \Big[ \partial_{\zeta} M_d^L \phi^\sigma + M_d^L \Big(\frac{-2i \nu(z_0)\sigma}{\zeta} - i\zeta\sigma \Big)\phi^\sigma \Big](\phi^{-\sigma}(M_d^L)^{-1}) \\
& = M_d^L (- i\zeta\sigma ) (M_d^L)^{-1} + O\Big(\frac 1\zeta \Big).
\end{aligned}
$$
Here, we have used the fact that $\partial_\zeta M_d^L = O(\zeta^{-1})$.
As $M_d^L = I + \frac{(M_d^L)_1}\zeta + o(\zeta^{-1})$ as $\zeta \to \infty$, we have
$$
\begin{aligned} \hat{M}
& \equiv (\partial_{\zeta}M^\infty + i \zeta \sigma M^\infty) (M^\infty)^{-1} = \partial_{\zeta}M^\infty (M^\infty)^{-1} + i \zeta \sigma \\
& = i \zeta [\sigma, M_d^L] (M_d^L)^{-1} + O\Big(\frac 1{\zeta} \Big) \\
& = i \zeta \Big[\sigma, I+\frac{(M_d^L)_1}{\zeta}+o\Big(\frac 1{\zeta} \Big)\Big] \Big(I - \frac{(M_d^L)_1}{\zeta}+o\Big(\frac 1{\zeta} \Big) \Big) + O\Big(\frac 1{\zeta} \Big) \\
& = i [\sigma, (M_d^L)_1] + o(1) .\\
\end{aligned}
$$
But $\hat{M}$ has no jumps across $\Sigma^e$ and hence by Liouville's theorem, $\hat{M} \equiv i [\sigma, (M_d^L)_1]$.
Then, from the definition of $\hat{M}$, we obtain the following system of ODEs
$$(\partial_{\zeta}M^\infty + i \zeta \sigma M^\infty) = i [\sigma, (M_d^L)_1] M^\infty. $$
Set
\beq
\label{eq:k_j_def}
 \bpm 0&k_1 \\ k_2&0 \epm = i [\sigma, (M_d^L)_1].
 \eeq
The differential equations for $M^\infty_{11}$ and $M^\infty_{21}$ are given by
\beq
\label{eq:ode_parabolic}
 \left\{
\begin{aligned}
& \partial_{\zeta}M^\infty_{11} + \frac {i \zeta}2  M^\infty_{11} = k_1 M^\infty_{21}, \\
& \partial_{\zeta}M^\infty_{21} - \frac {i \zeta}2 M^\infty_{21} = k_2 M^\infty_{11}. \\
\end{aligned}
\right.
\eeq
Then,
$$ \begin{aligned} k_1 k_2 M^\infty_{11}
& =\partial_{\zeta} \Big(\partial_\zeta M^\infty_{11} + \frac i2 \zeta M^\infty_{11} \Big) - \frac{i \zeta}2 \Big(\partial_{\zeta}M^\infty_{11} + \frac i2 \zeta M^\infty_{11} \Big) \\
& = \partial_\zeta^2 M^\infty_{11} + \frac {\zeta^2}4 M^\infty_{11} + \frac i2 M^\infty_{11}
\end{aligned}
$$
Set $\eta \equiv e^{-\frac{3i \pi}4} \zeta$ and $g(\eta) \equiv M^\infty_{11}(\zeta)$.
Then, we see that $g$ satisfies the parabolic cylinder equation
$$
\partial_\eta^2 g + \Big( \frac 12 - \frac{\eta^2}4 + a \Big) g = 0, \ \text{ where } \ a=ik_1 k_2.
$$
General solutions of this equation can be written in terms of standard entire parabolic cylinder functions $D_a(\cdot)$ as follows.
\beq
\label{eq:gen_sol_parabolic}
 M^\infty_{11}(\zeta) = c_1 D_a \Big( e^{-\frac{3i \pi}4} \zeta \Big) + c_2 D_a \Big( - e^{-\frac{3i \pi}4} \zeta \Big) .
\eeq
Since $M_d^L(\zeta) \to I$ as $\zeta \to \infty, z\in i\rb$, we have
\beq
\label{eq:asym_parabolic}
\begin{aligned} (M^\infty_+)_{11}(\zeta)
 & = (1+o(1))\phi^\sigma_{11} \\
& = (1+o(1))\alpha_0 e^{\frac{3 \pi}4 \nu(z_0)}\eta^{-i \nu(z_0)}e^{-\frac14 \eta^2} .
\end{aligned}
\eeq
But from \cite{AS}, as $\eta \rightarrow \infty$,
\begin{equation}
\label{eq:asymp_para_cylin}
D_a(\eta) =
\left\{
\begin{aligned}
& \eta^a e^{-\frac14 \eta^2} (1+O(a\eta^{-2})), \\
& \qquad\qquad \text{for } |\arg\eta| < \frac{3 \pi}4 , \\
& \eta^a e^{-\frac14 \eta^2} (1+O(a\eta^{-2})) \\
& - (2\pi)^{\frac 12}(\Gamma(-a))^{-1}e^{a\pi i}\eta^{-a-1}e^{\frac{\eta^2}4}(1+O(a\eta^{-2})), \\
& \qquad\qquad \text{for } \frac{\pi}4 < \arg\eta < \frac{5 \pi}4 , \\
& \eta^a e^{-\frac14 \eta^2} (1+O(a\eta^{-2})) \\
& - (2\pi)^{\frac 12}(\Gamma(-a))^{-1}e^{-a\pi i}\eta^{-a-1}e^{\frac{\eta^2}4}(1+O(a\eta^{-2})), \\
& \qquad\qquad \text{for } -\frac{5\pi}4 < \arg\eta < -\frac{\pi}4 , \\
\end{aligned}
\right.
\end{equation}
Inserting \eqref{eq:asymp_para_cylin} into \eqref{eq:gen_sol_parabolic} and comparing with \eqref{eq:asym_parabolic},
we see that $a=-i\nu(z_0)$, and $c_1 = \alpha_0 e^{\frac{3 \pi}4  \nu(z_0)}$.
Utilizing \eqref{eq:ode_parabolic}, we thus have for $\zeta \in \text{II} - z_0 \equiv \text{II}_{z_0}$,
\beq
\label{eq:M_for_II_1}
\begin{aligned}
& M^\infty_{11}(\zeta) = \alpha_0 e^{\frac{3 \pi}4  \nu(z_0)} D_a \Big( e^{-\frac{3i \pi}4} \zeta \Big), \\
& M^\infty_{21}(\zeta) = \frac{\alpha_0}{k_1} e^{\frac{3 \pi}4  \nu(z_0)} \Big[ \partial_\zeta \Big(D_a \Big( e^{-\frac{3i \pi}4} \zeta \Big)\Big) + \frac{i \zeta}2 D_a \Big( e^{-\frac{3i \pi}4} \zeta \Big) \Big], \\
\end{aligned}
\eeq
Similarly, we have for $\zeta \in \text{II}_{z_0}$,
\beq
\label{eq:M_for_II_2}
\begin{aligned}
& M^\infty_{12}(\zeta) = \frac{\alpha_0^{-1}}{k_2} e^{-\frac{ \pi}4  \nu(z_0)} \Big[ \partial_\zeta \Big(D_{-a} \Big( e^{-\frac{i \pi}4} \zeta \Big)\Big) - \frac{i \zeta}2 D_{-a} \Big( e^{-\frac{i \pi}4} \zeta \Big) \Big], \\
& M^\infty_{22}(\zeta) = \alpha_0^{-1} e^{-\frac{ \pi}4  \nu(z_0)} D_{-a} \Big( e^{-\frac{i \pi}4} \zeta \Big). \\
\end{aligned}
\eeq
Also, for $\zeta \in \text{V}_{z_0}$,
\beq
\label{eq:M_for_V}
\begin{aligned}
& M^\infty_{11}(\zeta) = \alpha_0 e^{-\frac{ \pi}4  \nu(z_0)} D_a \Big( e^{\frac{i \pi}4} \zeta \Big), \\
& M^\infty_{21}(\zeta) = \frac{\alpha_0}{k_1} e^{-\frac{ \pi}4  \nu(z_0)} \Big[ \partial_\zeta \Big(D_a \Big( e^{\frac{i \pi}4} \zeta \Big)\Big) + \frac{i \zeta}2 D_a \Big( e^{\frac{i \pi}4} \zeta \Big) \Big], \\
& M^\infty_{12}(\zeta) = \frac{\alpha_0^{-1}}{k_2} e^{\frac{3 \pi}4  \nu(z_0)} \Big[ \partial_\zeta \Big(D_{-a} \Big( e^{\frac{3i \pi}4} \zeta \Big)\Big) - \frac{i \zeta}2 D_{-a} \Big( e^{\frac{3i \pi}4} \zeta \Big) \Big], \\
& M^\infty_{22}(\zeta) = \alpha_0^{-1} e^{\frac{3 \pi}4  \nu(z_0)} D_{-a} \Big( e^{\frac{3i \pi}4} \zeta \Big). \\
\end{aligned}
\eeq
Let $M^\infty_{\text{II}}, M^\infty_{\text{V}}$ be the analytic extensions to $\cb$ of $M^\infty|_{\text{II}_{z_0}}$, $M^\infty|_{\text{V}_{z_0}}$, respectively. Then, from the jump matrix $V^\infty$ we observe that for all $\zeta \in \cb$,
$$
M^\infty_{\text{II}}(\zeta)  = M^\infty_{\text{V}}(\zeta) \bpm1&r(z_0)\\ 0&1 \epm \bpm1&0\\ \overline{r(z_0)}&1 \epm
= M^\infty_{\text{V}}(\zeta) \bpm1+|r(z_0)|^2&r(z_0)\\  \overline{r(z_0)}&1  \epm.
$$
Therefore, as $\det M_V^\infty = 1$,
$$
\begin{aligned} \overline{r(z_0)}
& =  (M^\infty_{\text{II}})_{21} (M^\infty_{\text{V}})_{11} - (M^\infty_{\text{II}})_{11}(M^\infty_{\text{V}})_{21} \\
& = \frac{\alpha_0^2}{k_1} e^{\frac{ \pi}2 \nu(z_0)} W\Big[D_a \Big( e^{\frac{i \pi}4} \zeta \Big), D_a \Big( e^{\frac{-3i \pi}4} \zeta \Big)\Big] \\
& = \frac{\alpha_0^2}{k_1} e^{\frac{ \pi}2 \nu(z_0)} \Big[\frac{\sqrt{2 \pi}}{\Gamma(-a)} e^{\frac{i \pi}4}\Big] \\
\end{aligned}
$$
where $W(f,g) = fg'-f'g$ is the Wronskian of $f$ and $g$.
Hence,
\begin{equation}
\label{eq:k_1}
\begin{aligned}
& k_1 = \frac{\sqrt{2 \pi} e^{\frac{i \pi}4} \alpha_0^2 e^{\frac{ \pi}2 \nu(z_0)}} {\overline{r(z_0)} \Gamma(-a)}
	= \beta(z_0) e^{i[x^2/(2t) + \nu(z_0)\log t]}, \\
\end{aligned}
\end{equation}
where
\begin{equation}
\label{eq:beta_z_zero}
\begin{aligned}
& \begin{aligned} |\beta(z_0)|^2
& = \frac{2\pi e^{\pi \nu(z_0)}}{|r(z_0)|^2|\Gamma(i \nu(z_0))|^2} = \frac{2\pi e^{\pi \nu(z_0)}}{|r(z_0)|^2} \Big(\frac{\nu(z_0) \sinh(\pi \nu(z_0))}{\pi}\Big) \\
& = \nu(z_0) ,\\
\end{aligned} \\
& \begin{aligned} \arg\beta(z_0)
 & = \frac{\pi}4 - 2i (\beta(z_0,z_0)-i \nu(z_0)) + \arg(r(z_0)) - \arg(\Gamma(i \nu(z_0))) \\
 & = \frac{\pi}4 + \frac 1{\pi} \int_{-\infty}^{z_0} \log(z_0 -s) \rd  \log(1+|r(s)|^2)\\
 & \qquad  + \arg(r(z_0)) - \arg(\Gamma(i \nu(z_0)) . \\
\end{aligned} \\
\end{aligned}
\end{equation}
and
\beq
\label{eq:k_2}
	k_2 = \frac a{i k_1} = -\frac{\nu(z_0)}{k_1} =  -\ovl{\beta(z_0)} e^{-i[x^2/(2t) + \nu(z_0)\log t]}.
\eeq
Inserting \eqref{eq:k_1}, \eqref{eq:k_2} into \eqref{eq:M_for_II_1}, \eqref{eq:M_for_II_2} respectively,
we obtain the explicit formulae for $M^\infty$ in region II.
On the other regions, $M^\infty$ is obtained by simply using the constant jump matrix $V^\infty$.

\section{Proof of Theorem \ref{T:main}}
For functions $u$, $\check{u}\in H^{1,1}(\rb)$,
let $\psi^\pm(x,z)$, $\check \psi^\pm(x,z)$ be the associated ZS-AKNS solutions, respectively (see Section 3).
Set $m^{\pm}=(m_1^\pm, m_2^\pm) \equiv \psi^\pm e^{-ixz\sigma}$,
$\check{m}^{\pm}=(\check{m}_1^\pm, \check{m}_2^\pm) \equiv \check \psi^\pm e^{-ixz\sigma}$.
Denote $\Delta u = \check{u} - u$, $\Delta m_1^+ = \check{m}_1^+ - m_1^+$, etc.
We write $m_1^+ = ((m_1^+)_1, (m_1^+)_2)^T$, etc.
\begin{lemma}
\label{lem:perturbed_potential}
Let $u\in H^{1,1}(\rb)$. For any $\check{u}\in H^{1,1}(\rb)$ such that $\|\check{u} - u \|_{H^{1,1}(\rb)} \leq \epsilon$,
$$
\|\Delta m_1^+(\cdot,z)\|_{L^\infty(\rb)} \leq c_1 \|\Delta u\|_{H^{1,1}} \leq c_1 \epsilon,
$$
where $c_1$ depends on $u$ and is uniform on $z\in \ovl{\cb^+}$.
Moreover, for any fixed $z\in \cb^+$,
$$
\|\Delta (m_1^+)_2(\cdot,z)\|_{H^{0,1}(\rb)} \leq c_2 \epsilon, \ \  \|\Delta \partial_x m_1^+(\cdot,z)\|_{L^2(\rb)} \leq c_2 \epsilon
$$
where $c_2$ depends on $u$ and $z$. A similar result holds for $m_2^-(x,z)$.
\end{lemma}
\begin{proof}
The difference $\Delta m_1^+$ solves the integral equation (cf. \eqref{eq:m_1_BC})
\beq
\label{eq:diff_T_+}
 \Delta m_1^+ = \mf b_+ + T_+(\Delta m_1^+)
 \eeq
where
$$
\begin{aligned}  \mf b_+
&= -\int_x^\infty \bpm1&0\\0& e^{-iz (x-y)}\epm \bpm0&\Delta u\\-\Delta\overline{u}&0\epm m_1^+ \rd y \\
&= \int_x^\infty \bpm -\Delta u (m_1^+)_{2} \\ \Delta \overline{u} e^{-iz (x-y)}  (m_1^+)_{1}\epm \rd y .\\
\end{aligned}
$$
and $T_+$ is the operator acting on column vectors $\mf f$ defined by
$$
T_+(\mf f)(x) = -\int_x^\infty \bpm1&0\\0& e^{-iz (x-y)}\epm \bpm0&\check{u}\\-\overline{\check{u}}&0\epm \mf f(y)  \rd y
$$
Since $(m_1^+)_{2} \in H^{1,1}(\rb) \subset L^1(\rb)$ and $(m_1^+)_{1} \in L^\infty(\rb)$ (see Proposition \ref{prop:BC_tail}),
$$ \|\mf b_+\|_{L^\infty(\rb)}  \leq c \|\Delta u \|_{H^{1,1}} $$
uniformly on $z\in \ovl{\cb^+}$.
For $n \geq 1$,
$$
\begin{aligned} |T_+^n(\mf b_+)|
& \leq \|\mf b_+\|_\infty \int_x^\infty|\check{u}(y_1)| \int_{y_1}^\infty|\check{u}(y_2)|\cdots\int_{y_{n-1}}^\infty|\check{u}(y_{n})| \\
& \leq \|\mf b_+\|_\infty \frac{1}{n!} \Big[ \int_x^\infty|\check{u}(y)| \rd y \Big]^{n+1} .\\
\end{aligned}
$$
and hence
$$ \|\Delta m_1^+ \|_{L^\infty(\rb)} \leq e^{\int_{-\infty}^\infty|\check{u}(y)| \rd y} \|\mf b_+\|_\infty \leq c \|\Delta u \|_{H^{1,1}} .$$
Now we fix $z\in\cb^+$. By Lemma \ref{lem:H11_integral},
$$ \|(\mf b_+)_2\|_{H^{1,1}(\rb)} \leq c \|\Delta u (m_1^+)_1 \|_{H^{0,1}} \leq c \|\Delta u \|_{H^{0,1}}. $$
Assembling these results, we obtain
$$ \begin{aligned} \|\Delta (m_1^+)_2 \|_{H^{0,1}(\rb)}
 & = \|(\mf b_+)_2 + (T_+(\Delta m_1^+))_2 \|_{H^{0,1}}  \\
 & \leq c \|\Delta u \|_{H^{0,1}} + c \|\check{u} \|_{H^{0,1}} \|\Delta (m_1^+)_1 \|_{L^\infty} \leq c \|\Delta u \|_{H^{1,1}}. \\
 \end{aligned} $$
For $\Delta \partial_x m_1^+$, since $\Delta \partial_x m_1^+ = \bigl(\begin{smallmatrix} 0&0\\0&-iz\end{smallmatrix}\bigr) \Delta m_1^+ + \Delta Q m_1^+ + \check Q \Delta m_1^+$, we see that
$$ \begin{aligned} \|\Delta \partial_x m_1^+ \|_{L^2(\rb)}
& \leq c\|\Delta (m_1^+)_2\|_{L^2} + \|\Delta Q\|_{L^2}\|m_1^+\|_{L^\infty} + \|\check Q\|_{L^2}\|\Delta m_1^+\|_{L^\infty} \\
& \leq c \|\Delta u \|_{H^{1,1}}. \\
 \end{aligned} $$
and hence the proof is done.

\end{proof}

\begin{lemma}[cf. \cite{DZ}, Theorem 3.2]
\label{lem:ZS_H_1_0}
Let $u(x)\in H^{0,1}(\rb)$. Fix $x \in\rb$. Then, the associated ZS-AKNS solutions $m^\pm(x,\cdot) \in I+H^{1,0}(dz)$.
Moreover, if $\|\check{u}-u\|_{H^{0,1}} < \epsilon$, then $\|\check{m}^\pm(x,\cdot) - m^\pm(x,\cdot) \|_{H^{1,0}} \leq c \epsilon $.
\end{lemma}
\begin{proof}
We will provide the proof only for $m^+$.
Define operators $K_u$ acting on $2 \times 2$ matrix functions $B$ as follows,
\beq
\label{eq:operator_K}
 (K_uB)(x,z) = \int^x_{+\infty} e^{i(x-y)z\text{ ad } \sigma} Q(y) B (y,z)dy, \ x,z\in \rb
\eeq
Then, ZS-AKNS solutions for real $z$ satisfy $m^+ = I + K_um^+$.
We use the following notation.
If $\mf M$ is a measure space and $\mf B$ is a Banach space,
then $\mf B \otimes L^p(\mf M) \equiv L^p(\mf M \to \mf B)$ denotes the space of $\mf B$-valued $L^p$ functions
with norm $\|f\|_{\mf B \otimes L^p(\mf M)} = \|\|f\|_{\mf B} \|_{L^p(\mf M)}$.
Direct calculations show that
\beq
\label{L_2_cross_L_inf}
\|K_uB\|_{L^2(dz)\otimes L^\infty(dx)} \leq c \|u\|_{L^1} \|B\|_{L^2(dz)\otimes L^\infty(dx)},
\eeq
\beq
\label{L_2_cross_L_2}
\|K_uB\|_{L^2(dz)\otimes L_{\rb^+}^2(dx)} \leq c \|u\|_{H^{0,1}} \|B\|_{L^2(dz)\otimes L^\infty(dx)},
\eeq
and
\beq
\label{L_2_cross_L_inf_L_2}
\|K_uB\|_{L^2(dz)\otimes L^\infty(dx)} \leq c \|u\|_{L^2} \|B\|_{L^2(dz)\otimes L_{\rb^+}^2(dx)},
\eeq
Standard iterations for the Volterra integral equation give
$$
\|(1-K_u)^{-1}\|_{L^2(dz)\otimes L^\infty(dx)\to L^2(dz) \otimes L^\infty(dx)}\le c e^{\|u\|_{L^1}}.
$$
By Fourier theory and Hardy's inequality,
$$\begin{aligned}
& \|K_uI\|_{L^2(dz)\otimes L^\infty(dx)} = c \left\|\left(\int^{\langle x\rangle}_{+\infty} |u|^2dy\right)^{\frac12} \right\|_{L^\infty(dx)} \le c \|u\|_{L^2}\\
& \begin{aligned} \|K_uI\|_{L^2(dz) \otimes L^2_{\rb^+}(dx)}
	&= c \left\|\left(\int^{\langle x\rangle}_{+\infty} |u|^2dy\right)^{\frac12} \right\|_{L^2_{\rb^+}(dx)}  \le c\|u\|_{H^{0,1}}. \\
	\end{aligned}
\\
\end{aligned}
$$
As $m^+ = I + (1-K_u)^{-1}K_uI$, by using \eqref{L_2_cross_L_inf} we see that
$$
\|m^+ - I\|_{L^2(dz)\otimes L^\infty(dx)} \le c\|K_uI\|_{L^2(dz)\otimes L^\infty(dx)} \leq c,
$$
and hence by \eqref{L_2_cross_L_2},
$$
\bal \|m^+ - I\|_{L^2(dz)\otimes L_{\rb^+}^2(dx)}
&=  \|K_u (m^+-I) + K_u I\|_{L^2(dz)\otimes L_{\rb^+}^2(dx)} \leq c.\\
\eal
$$
Define $M \equiv (\partial_z -i x \text{ ad }\sigma)m^+$. Then $M$ satisfies the equation
$$\begin{aligned} M(x,z)
&= -i \int^x_{+\infty} e^{i(x-y)z \text{ ad } \sigma} \text{ad } \sigma (yQ(y))m^{+} (y,z)dy\\
&\quad + \int^x_{+\infty} e^{i(x-y)z \text{ ad } \sigma} Q(y) M(y,z)dy\\
&= -i K'_{\cdot u(\cdot)} m^+ + K_uM
\end{aligned}$$
where $K'$ is the operator defined in \eqref{eq:operator_K} with $\text{ad } \sigma (Q)$ in place of $Q$.
As $M = -i(1-K_u)^{-1}K'_{\cdot \Delta u(\cdot)} m^+$, it follows by \eqref{L_2_cross_L_inf_L_2} that
$$
\begin{aligned} \|M\|_{L^2(dz)\otimes L^\infty(dx)}
& \leq c\|K'_{\cdot u(\cdot)} (m^+-I) + K'_{\cdot u(\cdot)}I \|_{L^2(dz)\otimes L^\infty(dx)}\\
& \leq c\|\cdot u(\cdot)\|_{L^2} \big(\|m^+ - I\|_{L^2(dz)\otimes L_{\rb^+}^2(dx)} + 1\big) \leq c.\\
\end{aligned} $$
The equations for $\Delta m^+=\check{m}^+ - m^+$, $\Delta M=\check{M} - M$ are given by (cf. \eqref{eq:diff_T_+})
$$
\begin{aligned}
& \Delta m^+ = K_{\Delta u} \check{m}^+ + K_u \Delta m^+ ,\\
& \Delta M = -i (K'_{\cdot \Delta u(\cdot)} \check{m}^+ + K'_{\cdot u(\cdot)} \Delta m^+) + K_{\Delta u} \check M + K_u \Delta M .\\
\end{aligned}
$$
Similarly, we have
$$\bal \|\Delta m^+\|_{L^2(dz)\otimes L^\infty(dx)}
& \leq c \|K_{\Delta u}(\check{m}^+ -I) + K_{\Delta u}I\|_{L^2(dz)\otimes L^\infty(dx)}\\
&  \leq c\|\Delta u\|_{H^{0,1}} \\
\eal$$
and hence
$$\bal \|\Delta m^+\|_{L^2(dz)\otimes L^2_{\rb^+}(dx)}
& = \|K_{\Delta u}(\check{m}^+ -I) + K_{\Delta u}I + K_{u}\Delta m^+\|_{L^2(dz)\otimes L^2_{\rb^+}(dx)}\\
&  \leq c\|\Delta u\|_{H^{0,1}}\big(\|\check{m}^+ -I\|_{L^2(dz)\otimes L^\infty(dx)}+1) \\
& \qquad + c\|u\|_{H^{0,1}} \|\Delta m^+\|_{L^2(dz)\otimes L^\infty(dx)}\\
&  \leq c\|\Delta u\|_{H^{0,1}} \\
\eal$$
Therefore for $\Delta M$ we see that
$$
\begin{aligned}
\|\Delta M\|_{L^2(dz)\otimes L^\infty(dx)}
& \leq c \|K'_{\cdot \Delta u(\cdot)} (\check{m}^+ - I) + K'_{\cdot \Delta u(\cdot)}I\|_{L^2(dz)\otimes L^\infty(dx)} \\
& \qquad+ c \|-iK'_{\cdot u(\cdot)} \Delta m^+ + K_{\Delta u} \check M \|_{L^2(dz)\otimes L^\infty(dx)} \\
& \leq c\|\Delta u\|_{H^{0,1}} \\
\end{aligned}
$$
In particular, $\partial_z m^+(0,z) = M(0,z) \in L^2(\rb)$. Also, $\partial_z \Delta m^+(0,z) = \Delta M(0,z)\in L^2(\rb)$ and hence $\|\Delta m^+ \|_{H^{1,0}} \leq c \|\Delta u\|_{H^{0,1}}$.
\end{proof}

Let $\psi_1^+$, $(\psi_0)_1^+$ be the first columns of the ZS-AKNS solutions for $u_0^e$, $\eta_{\mu_0}$ in Theorem \ref{T:main} and \eqref{eq:1_soliton}, respectively.
Set $m_1^+ \equiv \psi_1^+ e^{-ixz/2}$, $(m_0)_1^+ \equiv (\psi_0)_1^+ e^{-ixz/2}$ ($x\geq0$), respectively.
By Remark \ref{rmk:1_soliton}, $\eta_{\mu_0}(x)$ in \eqref{eq:1_soliton} is the B\"acklund extension of $v_{\mu_0}(x)$
and it is straightforward to see that $(m_0)_1^+$ and $(m_0)_2^-$ have the explicit formulae (see \cite{RS})
\beq
\label{eq:m_0_1_plus}
\bal
& (m_0)_1^+(x,z) = \frac 1{z + i\mu_0}\bpm z+i \mu_0 \tanh (\mu_0  x + \tanh^{-1}( q / \mu_0 ))\\ i \mu_0 \sech (\mu_0  x + \tanh^{-1}( q / \mu_0 )) \epm, \\
& (m_0)_2^-(x,z) = \frac 1{z + i\mu_0}\bpm i \mu_0 \sech (\mu_0  x + \tanh^{-1}( q / \mu_0 )) \\
z-i \mu_0 \tanh (\mu_0  x + \tanh^{-1}( q / \mu_0 )) \epm. \\
\eal
\eeq
Define (cf. Proposition \ref{prop:backlund_scattering})
\beq
\label{eq:AB_0_AB_1}
\bal
& \bpm A(z) \\B(z) \epm \equiv m_1^+(0,z), \ \
\bpm A_0(z) \\B_0(z) \epm \equiv (m_0)_1^+(0,z), \\
& A_1(z) \equiv A(z) - A_0(z), \ \ B_1(z) \equiv B(z) - B_0(z). \\
\eal
\eeq
Then
\beq
\label{A_0_B_0}
\bpm A_0(z) \\B_0(z) \epm = \frac{1}{z+i\mu_0} \bpm z+iq \\ i\sqrt{\mu_0^2-q^2} \epm.
\eeq
By Proposition \ref{prop:backlund_scattering}, the scattering data for the B\"acklund extension $u_0^e$ of $u_0(x)|_{\rb^+}=u(x,0)|_{\rb^+}$ is given by
\begin{equation}
\label{eq:scattering_expansion}
\bal
& a(z)= \frac 1{z-i\beta}\Big[(z+iq)\frac {z-i\mu_0}{z+i\mu_0} + (z-iq)g_1(z) - (z+iq)g_2(z) \Big], \\
& b(z) = \frac 1{z+i\beta}\big[(z+iq)g_3(z) + (z-iq)g_3(-z)\big]. \\
\eal
\end{equation}
where
\beq
\label{eq:scattering_expansion2}
\bal
& g_1(z) = A_0(z)\overline{A_1(-\ovl{z})}+A_1(z)\overline{A_0(-\ovl{z})}+A_1(z)\overline{A_1(-\ovl{z})} \\
& g_2(z) = B_0(z)\overline{B_1(-\ovl{z})}+B_1(z)\overline{B_0(-\ovl{z})}+B_1(z)\overline{B_1(-\ovl{z})}\\
& g_3(z) = A_0(-z) B_1(z)+A_1(-z) B_0(z)+A_1(-z) B_1(z) \\
\eal
\eeq
and $\beta$ is determined as in \eqref{eq:beta}.
By Lemma \ref{lem:ZS_H_1_0}, we have
\beq
\label{eq:A_1_B_1_small}
\|A_1\|_{H^{1,0}(\rb)} \leq c \epsilon, \ \ \|B_1\|_{H^{1,0}(\rb)} \leq c \epsilon,
\eeq
and hence
\beq
\label{g_j_small}
\|g_j\|_{H^{1,0}(\rb)}  \leq c \epsilon, \ \ 1\leq j \leq 3.
\eeq
Thus,
\beq
\label{eq:b_norm}
\|b\|_{L^2(\rb)}  \leq c \epsilon, \ \|b\|_{L^\infty(\rb)}  \leq c \epsilon \text{ and } \|b\|_{H^{1,0}(\rb)}  \leq c \epsilon |q|^{-\frac 12}.
\eeq
Define $a_0(z) \equiv \frac{z+iq}{z-i\beta}\frac{z-i\mu_0}{z+i\mu_0}$. Then,
\beq
\label{eq:a_norm}
\|a-a_0\|_{L^\infty(\rb)}  \leq c \epsilon , \ \ \|a -a_0\|_{H^{1,0}(\rb)}  \leq c \epsilon |q|^{-\frac 12}.
\eeq
As $a(z;\eta_{\mu_0})=\frac{z-i\mu_0}{z+i\mu_0}$, $\eta_{\mu_0}$ is generic
and it follows from the proof of Proposition \ref{prop:backlund_generic} that
for all sufficiently small $\epsilon>0$ $a(z) = a(z;u_0^e)$ is also generic.
As $a(z;\eta_{\mu_0})$ has only one zero in $\cb^+$, it follows that $\beta(\eta_{\mu_0}) = (-1)^1 q = -q$.
Suppose $q>0$.
Then $a_1(z;\eta_{\mu_0}) \equiv a(z;\eta_{\mu_0})(z-i\beta(\eta_{\mu_0})) = \frac{z-i\mu_0}{z+i\mu_0} (z+iq)$ has only one zero in $\cb^+$,
and hence, $a_1(z;u_0^e)$ also has only one zero $z_1 \sim i\mu_0$ in $\cb^+$, by Rouch\'e's theorem.
But
\beq
\label{eq:a_1_a}
a_1(z;u_0^e) = a(z;u_0^e)(z-i\beta(u_0^e)),
\eeq
and as $|\beta(u_0^e)| = |q| < \mu_0$, we see that $a(z_1;u_0^e) = 0$.
Hence $a(z;u_0^e)$ also has one zero $z_1 \sim i \mu_0$ in $\cb^+$ and necessarily $\beta(u_0^e) = (-1)^1 q = -q$,
which is consistent with the fact that $a_1(iq;u_0^e)\neq 0$ (see \eqref{eq:a_1_a}).
Now suppose $q<0$.
Then $a_1(z;\eta_{\mu_0})$ has two zeros $i\mu_0, -iq$.
Again by Rouch\'e, $a_1(z;u_0^e)$ has two zeros $z_1 \sim i \mu_0, z_2 \sim -iq$ in $\cb^+$.
If $z_2 = -iq$, then as $a(-iq;u_0^e)\neq 0$ by Lemma \ref{lem:symm_prop},
it follows from \eqref{eq:a_1_a} that $\beta(u_0^e)=-q$,
and that $a(z;u_0^e)=\frac{a_1(z;u_0^e)}{z+iq}$ has only one zero $z_1 \sim i\mu_0$ in $\cb^+$
(note that this is consistent with $\beta(u_0^e)=-q$).
On the other hand if $z_2 \neq -iq$, then it follows from \eqref{eq:a_1_a} that $\beta(u_0^e) = q$
(otherwise $a_1(z;u_0^e)$ would have three zeros in $\cb^+$).
But then, again from  \eqref{eq:a_1_a}, we see that $a(z;u_0^e)$ has two zeros $z_1 \sim i \mu_0, z_2 \sim -iq$ in $\cb^+$
(note again that this is consistent with $\beta(u_0^e)=q$).
By the symmetry condition $a(z;u_0^e) = \ovl{a(-\ovl z;u_0^e)}$ from \eqref{eq:symm_scattering},
we see that the zeros $z_1, z_2$ must lie on $i\rb^+$.
We have proved the following result.

\begin{proposition}[Zeros of $a(z;u_0^e)$ in $\cb^+$]
\label{prop:zeros_a} $\ $

\begin{enumerate}
\item[\textnormal{(i)}] If $q>0$, $a(z;u_0^e)$ has one simple zero $z_1 = i\mu_1 \in i\rb^+$, $\mu_1 \sim \mu_0$.
\item[\textnormal{(ii)}] If $q<0$, and $a_1(-iq;u_0^e)=0$, then $a(z;u_0^e)$ has one simple zero $z_1 = i\mu_1 \in i\rb^+$, $\mu_1 \sim \mu_0$.
\item[\textnormal{(iii)}] If $q<0$, and $a_1(-iq;u_0^e)\neq0$, then $a(z;u_0^e)$ has two simple zeros $z_1 = i\mu_1 \in i\rb^+$, $\mu_1 \sim \mu_0$
and $z_2 = i\mu_2 \in i\rb^+$, $\mu_2 \sim -q$, $\mu_2 \neq -q$.
\end{enumerate}
\end{proposition}

\begin{lemma}
\label{lem:zeros_a}
In the notation of Proposition \ref{prop:zeros_a},
$$
\bal
& \mu_1 = \mu_0 + \epsilon w_1 +O(\epsilon q +\epsilon^2 ), \\
& \mu_2 = -q + O(\epsilon^2 q).\\
\eal
$$
where $w_1 = \int_\rb \mathrm{Re} w(y)v_{\mu_0}(y) \rd y$.
Moreover, $\sqrt{\frac {\mu_2+q}{\mu_2-q}} = O(\epsilon)$.
\end{lemma}

\begin{proof}
For $z_2 = i\mu_2$, as $\beta = +q$ and $\mu_2 \neq -q$ by Proposition \ref{prop:zeros_a} (iii), we have from \eqref{eq:scattering_expansion},
\beq
\label{eq:a_zero_z_2}
0=a(z_2) = \frac{\mu_2+q}{\mu_2-q} \frac{\mu_2-\mu_0}{\mu_2+\mu_0} + g_1(z_2) - \frac{\mu_2+q}{\mu_1-q} g_2(z_2).
\eeq
Set $K=\frac {\mu_2+q}{\mu_2-q}$.
From \eqref{g_j_small} and \eqref{eq:a_zero_z_2} we see that
$K =  -g_1(z_2)\big(\frac{\mu_2-\mu_0}{\mu_2+\mu_0}- g_2(z_2)\big)^{-1} = O(\epsilon)$.
Hence $\mu_2 = -q\frac{1+K}{1-K} = -q + O(\epsilon q)$,
which implies that $A_0(z_2)= \frac{\mu_2+q}{\mu_2+\mu_0}=O(\epsilon q)$.
But then from \eqref{eq:scattering_expansion2} $g_1(z_2) = A_0(z_2)(A_1(z_2)+\ovl{A_1(z_2)}) +O(\epsilon^2)=O(\epsilon^2)$,
which now implies that $K = O(\epsilon^2)$ and hence $\mu_2 = -q + O(\epsilon^2 q)$.

For $z_1 = i\mu_1$, let $(\psi_1^+, \psi_2^-)$, $((\psi_0)_1^+, (\psi_0)_2^-)$ be the ZS-AKNS solutions
for $u_0^e$, $\eta_{\mu_0}$ in $\ovl{\cb^+}$, respectively.
Define $\varphi (x,z) \equiv \dot\psi_1^+ - \gamma_1 \dot\psi_2^-$, $z\in\cb^+$
where $\dot\psi_1^+ = \partial_z \psi_1^+$ and $\dot\psi_2^- = \partial_z \psi_2^-$
and $\gamma_1$ is the norming constant for $z_1$.
Fix $z=z_1=i\mu_1$ and set $\phi(x) = (\psi_1^+, \varphi)(x,z_1)$.
As
\beq
\label{eq:det_phi}
 \bal 0 \neq a'(z_1)
& = \det (\dot\psi_1^+ , \psi_2^-) + \det (\psi_1^+ , \dot\psi_2^-) \\
& = \det (\dot\psi_1^+ , \gamma_1^{-1} \psi_1^+) + \det (\psi_1^+ , \dot\psi_2^-) \\
& = -\gamma_1^{-1} \det \phi, \\
\eal
\eeq
$\phi^{-1}$ exists for all $x\in\rb$.
A simple computation shows that $\phi_x = (iz_1\sigma + Q) \phi$ where $Q = \bsm 0 & u_0^e\\ -\ovl{u_0^e} & 0 \esm$.
As $((\psi_0)_1^+)_x = (iz_1\sigma + Q_0) (\psi_0)_1^+$ where $Q_0 = \bsm 0 & \eta_{\mu_0}\\ -\ovl{\eta}_{\mu_0} & 0 \esm$,
we see that
\beq
\label{eq:diff_phi_psi}
(\phi^{-1} (\psi_0)_1^+)_x = -(\phi^{-1} \Delta Q \phi)(\phi^{-1} (\psi_0)_1^+), \ \ \Delta Q = Q - Q_0.
\eeq
Let $M = \phi e^{-ixz_1\sigma}$ and $(m_0)_1^+ = (\psi_0)_1^+ e^{-ixz_1/2}$. Then,
$$
\phi^{-1} (\psi_0)_1^+ = \bpm 1 & 0\\0&e^{ixz_1} \epm M^{-1} (m_0)_1^+ .
$$
Since $(m_0)_1^+ \to e_1$ and $M\to \bsm 1 & 0\\0& -\gamma_1 a'(z_1) \esm$ as $x\to +\infty$,
it follows from \eqref{eq:diff_phi_psi} that
$$
\phi^{-1} (\psi_0)_1^+ = e_1 + \int_x^\infty  \phi^{-1} \Delta Q \phi (\phi^{-1} (\psi_0)_1^+),
$$
and hence
$$
M^{-1} (m_0)_1^+ = e_1 + \int_x^\infty  \bpm 1 & 0\\0&e^{-i(x-y)z_1} \epm G M^{-1} (m_0)_1^+ \rd y,
$$
where $G = M^{-1} \Delta Q M$.
As $M(x)$ is uniformly bounded on $\rb$ and $\|\Delta Q\|_{H^{0,1}} \leq c \epsilon$,
it follows by a standard iteration that
$$
M^{-1} (m_0)_1^+ = e_1 + \int_x^\infty \bpm 1 & 0\\0&e^{-i(x-y)z_1} \epm G e_1 \rd y+ O(\epsilon^2),
$$
uniformly on $x\geq0$. In particular, multiplying by $M$ and setting $x=0$, we obtain
\beq
\label{eq:diff_expansion}
\bpm A_1(z_1) \\ B_1(z_1) \epm = m_1^+ (0,z_1) - (m_0)_1^+ (0,z_1) = -M(0) \bpm f_1 \\ f_2 \epm + O(\epsilon^2),
\eeq
where $f_1 = \int_0^\infty G_{11}$ and $f_2 = \int_0^\infty  e^{-\mu_1 y} G_{21}$.
Set
$$
M_0 \equiv ((\psi_0)_1^+, (\dot\psi_0)_1^+ - \gamma_1 (\dot\psi_0)_2^-) e^{-ixz_1\sigma}.
$$
As $A_0(z_1) \neq 0$ by \eqref{A_0_B_0}, it follows from \eqref{eq:gamma_z_k} and \eqref{eq:diff_expansion} that
the norming constant $\gamma_1 = \gamma(z_1)$ for $u_0^e$ is given by
\beq
\label{eq:gamma_1_1_small}
\gamma_1 = \frac{z_1 - i\beta}{z_1 +iq} \frac {A(z_1)}{B(z_1)} = (1+O(\epsilon+q)) \frac {A_0(z_1)}{B_0(z_1)} = 1+O(\epsilon+q).
\eeq
From \eqref{eq:scattering_expansion} and \eqref{g_j_small}, we have
\beq
\label{eq:a_z_1}
0=a(z_1) = \frac{\mu_1+q}{\mu_1-\beta} \frac{\mu_1-\mu_0}{\mu_1+\mu_0} + g_1(z_1) - g_2(z_1) + O(\epsilon q),
\eeq
and hence, again by \eqref{g_j_small},
\beq
\label{eq:mu_mu_0}
\mu_1 = \mu_0 + O(q+\epsilon).
\eeq
As $\bsm A_0(z_1) \\ B_0(z_1) \esm = \frac 12 \bsm 1 \\ 1 \esm + O(\epsilon)$ by \eqref{A_0_B_0},
it follows from \eqref{eq:scattering_expansion2} and \eqref{eq:A_1_B_1_small} that
\beq
\label{eq:g_1_minus_g_2}
\bal g_1(z_1) - g_2(z_1)
& = \mathrm{Re}(A_1(z_1)-B_1(z_1)) + O(\epsilon q + \epsilon^2) \\
\eal
\eeq
Direct calculation using \eqref{eq:m_0_1_plus} and \eqref{eq:mu_mu_0} shows that
$M_0(0,z_1) = \frac 12 \bsm 1& -i \mu_0^{-1} \\1& i \mu_0^{-1} \esm + O(q+\epsilon)$.
By Lemma \ref{lem:perturbed_potential}, $|m_1^+(0, z) - (m_0)_1^+(0, z)| \leq c \epsilon$
and $|m_2^-(0, z) - (m_0)_2^-(0, z)| \leq c \epsilon$ for $z\in \ovl{\cb^+}$.
Then it follows by analyticity of $m_1^+(0, z)$, etc. that
\beq
\label{eq:M_M_0_close}
|M(0, z_1) - (M_0)(0, z_1)| \leq c \epsilon
\eeq
Hence from \eqref{eq:diff_expansion} and \eqref{eq:gamma_1_1_small},
\beq
\label{eq:diff_expansion2}
\bal
A_1(z_1) - B_1(z_1)
& = - (1, -1) M(0) \bpm f_1 \\ f_2 \epm + O(\epsilon^2)\\
& = - (1, -1) M_0(0) \bpm f_1 \\ f_2 \epm + O(\epsilon^2)\\
& = \frac{if_2}{\mu_0} + O(\epsilon q + \epsilon^2). \\
\eal
\eeq
Now, from \eqref{eq:a_exp_formula}, $a'(z_1) = \frac 1{z_1 - \ovl{z_1}} e^{-l(z_1)}$ if $z_1$ is the only (simple) zero of $a(z)$ in $\cb^+$
and $a'(z_1) = \frac 1{z_1 - \ovl{z_1}}\frac {z_1 - z_2}{z_1 - \ovl{z_2}} e^{-l(z_1)} = \frac 1{z_1 - \ovl{z_1}} e^{-l(z_1)} + O(q)$
if $a(z)$ has a second (simple) zero at $i\mu_2 \sim O(q)$.
As $l(z_1) = O(\|r\|_{L^2}^2) = O(\epsilon^2)$,
it follows from \eqref{eq:det_phi} that $\det M = \det \phi = -\gamma_1 a'(z_1) = -\frac 1{2i\mu_1} + O(q+ \epsilon)$.
Assembling the above results,
$$ \bal \mathrm{Re}(if_2)
& = \mathrm{Re} \int_0^\infty i e^{-\mu_1 y} (M^{-1} \Delta Q M)_{21} \rd y \\
& = \mathrm{Re} \int_0^\infty  \frac {ie^{-\mu_1 y}}{\det M}  (- \epsilon w M_{21}^2 -  \epsilon \ovl{w} M_{11}^2) \rd y  \\
& = \mathrm{Re} \int_0^\infty  \frac {ie^{-\mu_1 y}}{\det M}  (- \epsilon w (M_0)_{21}^2 -  \epsilon \ovl{w} (M_0)_{11}^2) \rd y  + O(\epsilon^2) \\
& = -\epsilon \int_0^\infty  \mathrm{Re} w(y) v_{\mu_0}(y) \rd y  + O(\epsilon q + \epsilon^2) \\
& = -\frac{\epsilon w_1}2 + O(\epsilon q + \epsilon^2) \\
\eal$$
Thus it follows from \eqref{eq:a_z_1}, \eqref{eq:g_1_minus_g_2} and \eqref{eq:diff_expansion2}
that $\mu_1 = \mu_0 + \epsilon w_1 +O(\epsilon q +\epsilon^2 )$.
\end{proof}

\begin{remark}
The zeros of $a(z)$ can be written in terms of $r(z)$ and $w^e \equiv u_0^e - \eta_{\mu_0}$ using the conserved integrals for NLS. Recall that the conserved integrals can be read off from the expansion of $\log a(z) = \sum_{n=1}^{\infty} I_n z^{-n}$ as $z\to\infty$ (see \cite{ZS}). The first 3 coefficients in the expansion are given by
$$ \bal
& I_1 = -i\int_\rb |u(x,t)|^2 \rd x, \ \ I_2 = \frac 12 \int_\rb (u\ovl{u}_x-u_x \ovl{u})(x,t) \rd x, \\
& I_3 = i\int_\rb (|u|^4 - |u_x|^2)(x,t) \rd x. \\
\eal $$
Also, if $a(z)$ has one zero, it follows from \eqref{eq:a_exp_formula} that as $z\to\infty$, $\textnormal{Im} z> c|\textnormal{Re} z|$, $c>0$,
$$ \bal \log a(z)
& = \log \Big(\frac{z-i\mu_1}{z+i\mu_1} \Big) - l(z) \\
& =\frac{1}{z} \Big[ -2i\mu_1 + \frac{1}{2 \pi i} \int_{\mathbb{R}} \log(1+|r(s)|^2) \rd s  \Big] + O\Big(\frac{1}{z^2}\Big), \\
\eal $$
where $l(z)$ is given in \eqref{eq:l_z} and so we have
\beq
\label{eq:perturbed_mu}
\mu_1 = \mu_0 + \frac{1}{2}\int_{\mathbb{R}}  \eta_{\mu_0}( w^e+\overline{w^e} ) + |w^e|^2 - \frac{1}{4\pi}\int_{\rb} \log(1+|r|^2).
\eeq
This formula should be compared with Lemma \ref{lem:zeros_a}.
We see, in particular, that the contribution of $w^e \equiv u_0^e - \eta_{\mu_0}$ is approximated by $w_1$.
Similarly, if $a(z)$ has two zeros,
\beq
\label{eq:two_zeros}
\bal
& \mu_1 + \mu_2 = \frac 12 \int_\rb |u_0^e|^2 - \frac{1}{4\pi}\int_\rb \log(1+|r(s)|^2)\rd s,\\
& \mu_1^3 + \mu_2^3 = \frac 32 \int_\rb (|u_0^e|^4 - |(u_0^e)_x|^2) - \frac{3}{4\pi} \int_\rb s^2\log(1+|r(s)|^2)\rd s. \\
\eal
\eeq
\end{remark}

\begin{proof}[Proof of Theorem \ref{T:main}]
Fix $0< \kappa < \frac{1}{4}$. It is enough to consider $x\geq0$. First, we assume that $a(z)$ has one simple zero $z_1 = i\mu_1$.
From \eqref{eq:symm_scattering}, we have $\beta=-q$.
By \eqref{eq:norming_const_real} and Theorem \ref{thm:evol_scatt_data}, the norming constant $\gamma_1(t)=\gamma(t;z_1)$ is given by
\beq
\label{eq:gamma_one_zero}
\gamma_1(t)= \gamma_1 e^{iz_1^2 t/2} \text{ where } \gamma_1 = e^{-i\rho_1} \sqrt{\frac{\mu_1+q}{\mu_1-q}}, \text{ for some } \rho_1 \in \rb.
\eeq
Recalling the construction in the Appendix of the solution of a RHP with one pair of poles
in terms of the solution of a RHP without poles using a Darboux transformation, we set
$$
r_f(z) \equiv r(z) \frac{z-\overline{z_1}}{z-z_1}, \ \ z\in\rb, \text{ and }
c_1(t) = \frac {\gamma_1(t)}{a'(z_1)} .
$$
Note from \eqref{eq:a_exp_formula} that $a(z) = \frac{z-i\mu_1}{z+i\mu_1} e^{-l(z)}$
where $l(z)$ is defined in \eqref{eq:l_z}.
As $r(z) = \frac{\ovl{b(z)}}{\ovl{a(z)}}$, by \eqref{eq:b_norm} and \eqref{eq:a_norm}, we see that
$$
\|r\|_{L^\infty} \leq c\|b\|_{L^\infty} \leq c \epsilon, \ \ \|r\|_{L^2} \leq c\|b\|_{L^2} \leq c \epsilon,
$$
and
$$
\bal \|r'\|_{L^2}
&  = \bigg\|\frac {b'} a - \frac{ba'}{a^2} \bigg\|_{L^2}  \leq c \big(\|b'\|_{L^2} + \|b\|_{L^\infty} \|a'\|_{L^2} \big) \leq c \epsilon|q|^{-\frac 12}.\\
\eal
$$
Therefore,
$$
\|r_f\|_{L^\infty} = \|r\|_{L^\infty} \leq c \epsilon, \ \ \|r_f\|_{H^{1,0}} \leq c \|r\|_{H^{1,0}} \leq c \epsilon|q|^{-\frac 12}.
$$
Assume that $\epsilon\leq \chi_0 |q|^{\frac 12}$ where $\chi_0 >0$ will be determined further on.
Denote by $(m_f)_\pm \in I+\partial C(L^2)$ the solution of the normalized RHP $(\rb, v_f(z))$ without poles
where $v_f(z) = \bsm 1+|r_f(z)|^2&r_f(z)e^{i\theta} \\ \overline{r_f(z)}e^{-i\theta} &1 \esm$.
Such a solution exists by the general theory of Section \ref{sec:longtime}.
As before, $m_f(x,t,z)$, $z\in \cb\setminus\rb$, denote the extension of $(m_f)_\pm$ off the axis.
We now apply the steepest-descent analysis in Section \ref{sec:longtime} with $r(z)$ replaced by $r_f(z)$.
In the analysis in Section \ref{sec:longtime}, various auxiliary functions such as $m_d^L$ etc. are introduced:
in the present context we should properly use the notation $(m_f)_d^L$ etc., but we simply write $m_d^L$.

Let $\zeta_1 = \sqrt t (z_1 - z_0) = \sqrt t (i\mu_1 - z_0)$.
By the estimate \eqref{eq:m_z1_localized}, Lemma \ref{lem:small_r_estimate} and \eqref{eq:def_M_infty} for $t\geq 1$,
$$
\bal m_f(z_1)
& =  m_d^L(z_1)\Phi^{-1}(z_1)\delta(z_1)^{\sigma_3}+ O(\epsilon |q|^{-\frac 12} t^{-(\frac 12 + \kappa)}) \\
& = M^\infty(\zeta_1)\phi^{-\sigma}(\zeta_1) \Phi^{-1}(z_1)\delta(z_1)^{\sigma_3} + O(\epsilon |q|^{-\frac 12} t^{-(\frac 12 + \kappa)}) \\
& = M_{\text{II}}^\infty(\zeta_1)\phi^{-\sigma}(\zeta_1) \phi^{\text{ad} \sigma}(K^\infty)(\zeta_1) \Phi^{-1}(z_1)\delta(z_1)^{\sigma_3} \\
& \qquad\qquad + O(\epsilon |q|^{-\frac 12} t^{-(\frac 12 + \kappa)}) .\\
\eal
$$
where $\Phi$ is defined in \eqref{eq:Phi_deform} with $r_f(z)$ and
\beq
\label{eq:K_infty}
\bal K^\infty(\zeta) =
\left\{\bal
& I, \qquad\qquad\qquad\quad \  \zeta \in \text{II}_{z_0}, \\
& \bpm 1 & \frac {r_f(z_0)}{1+|r_f(z_0)|^2} \\ 0 & 1 \epm ^{-1}, \zeta \in \text{III}_{z_0}. \\
\eal \right.
\eal
\eeq
Note that if $\zeta_1 \in \text{II}_{z_0}$,
$$
\phi^{\text{ad} \sigma}(K^\infty)(\zeta_1)=I, \ \ \Phi^{-1}(z_1) = I,
$$
If $\zeta_1 \in \text{III}_{z_0} $, then $x\geq \mu_1 t$ and hence
$|e^{i\theta(\zeta_1)}|=|e^{-i\zeta_1^2/2}|=e^{-\mu_1 x}$ decay exponentially as $t\to \infty$. As $r_f(z_0) = O(\epsilon)$,
$$
\phi^{\text{ad} \sigma}(K^\infty)(\zeta_1)=I+O(\epsilon e^{-c t}), \ \ \Phi^{-1}(z_1) = I+O(\epsilon  e^{-c t}).
$$
Let $k_1$, $k_2$ be given in \eqref{eq:k_1} and \eqref{eq:k_2} with $r_f(z)$.
We use from \cite{AS} the identity $D_a'(\eta) = - \frac 12 \eta D_a(\eta) + a D_{a-1}(\eta)$ for $\eta \in \cb$ to obtain
$$
\begin{aligned} \partial_\zeta \Big(D_a \Big( e^{-\frac{3i \pi}4} \zeta \Big)\Big)
& = e^{-\frac{3i \pi}4} D_a'\Big( e^{-\frac{3i \pi}4} \zeta \Big)\\
& = e^{-\frac{3i \pi}4} \Big[ -\frac 12 e^{-\frac{3i \pi}4} \zeta D_a\Big( e^{-\frac{3i \pi}4} \zeta \Big) + a D_{a-1}\Big( e^{-\frac{3i \pi}4} \zeta \Big)\Big]\\
& = -\frac{i}2 \zeta D_a\Big( e^{-\frac{3i \pi}4} \zeta \Big) + a e^{-\frac{3i \pi}4} D_{a-1}\Big( e^{-\frac{3i \pi}4} \zeta \Big)\\
\end{aligned}
$$
and hence by \eqref{eq:M_for_II_1},
$$
(M_{\text{II}}^\infty)_{21}(\zeta) = \frac{\alpha_0}{k_1} e^{\frac{3 \pi}4  \nu(z_0)} a e^{-\frac{3i \pi}4} D_{a-1} \Big( e^{-\frac{3i \pi}4} \zeta \Big).
$$
Similarly, it follows from \eqref{eq:M_for_II_2} that
$$
(M_{\text{II}}^\infty)_{12}(\zeta) = - \frac{\alpha_0^{-1}}{k_2} e^{-\frac{\pi}4  \nu(z_0)} a e^{-\frac{i \pi}4} D_{-a-1} \Big( e^{-\frac{i \pi}4} \zeta \Big).
$$
Thus as $a = -i\nu(z_0) = O(\epsilon^2)$, we see from asymptotic expansion \eqref{eq:asymp_para_cylin} of $D_a(\cdot)$ that
$$
M_{\text{II}}^\infty(\zeta_1)\phi^{-\sigma}(\zeta_1) = \bpm 1&p_1 \\ p_2&1 \epm + O(\epsilon^2 t^{-1}),
$$
where
\beq
\label{eq:p_1_p_2}
p_1 = -\frac{ik_1}{\zeta_1} = \frac{ik_1}{\sqrt t (z_0 - i\mu_1)} \text{ and }
p_2 = \frac{ik_2}{\zeta_1} = -\frac{ik_2}{\sqrt t (z_0 - i\mu_1)}.
\eeq
Let $\delta(z)$ be given in \eqref{eq:delta} with $r_f(z)$.
Since $|r_f(z)| = |r(z)| = |r(-z)| = |r_f(-z)|$ for $z\in\rb$ by Remark \ref{rmk:symmetry_r_backlund},
we have $e^{l(z_1)} \delta(z_1)^{-2} = e^{-il_1} \hat{\delta}_1^{-2}$
where $l_1$, $\delta_1$ are given in \eqref{eq:l_j} and \eqref{eq:delta_j}, respectively.
Assembling the above results, we obtain
$$
\begin{aligned}
{\mathfrak{b}}(x,t) &\equiv m_f(z_1)e^{ixz_1 \sigma} \bpm 1 \\ \frac{-c_1 (t)}{z_1-\overline{z_1}} \epm \\
& =  \big( M_{\text{II}}^\infty(\zeta_1)\phi^{-\sigma}(\zeta_1) + O(\epsilon |q|^{-\frac 12} t^{-(\frac 12 + \kappa)}) \big)
\delta(z_1)^{\sigma_3} e^{ixz_1 \sigma}  \bpm 1 \\ \frac{-c_1 (t)}{z_1-\overline{z_1}} \epm\\
& = \upsilon_1^{-1} \delta(z_1) e^{\frac{\mu_1 x}2} \bigg[\bpm \upsilon_1^x + p_1  \\ p_2 \upsilon_1^x + 1 \epm + O(\epsilon |q|^{-\frac 12} t^{-({\frac 12}+\kappa)})\bigg].\\
\end{aligned}
$$
where $\upsilon_1$, $\upsilon_1^x$ are defined in \eqref{eq:upsilon_j}.
For $u_f(x,t)$, we have the expansion $ m_d^L(z) = M_d^L(\zeta) = I + (M_d^L)_1 \zeta^{-1} + o(\zeta^{-1}) = I + (M_d^L)_1 (\sqrt{t}z)^{-1} + o(z^{-1})$ as $z \to \infty$, which implies that $(m_d^L)_1 = \frac 1{\sqrt t}(M_d^L)_1$.
Hence, by \eqref{eq:u_c_localized} and \eqref{eq:k_j_def},
$$
\begin{aligned} u_f(x,t)
& = -i((m_d^L)_1)_{12}  + O(\epsilon|q|^{-\frac 12} t^{-({\frac 12}+\kappa)}) \\
& = -\frac{i}{\sqrt t}((M_d^L)_1)_{12}  + O(\epsilon|q|^{-\frac 12} t^{-({\frac 12}+\kappa)}) \\
& = -\frac {k_1}{\sqrt t} + O(\epsilon|q|^{-\frac 12} t^{-({\frac 12}+\kappa)}). \\
\end{aligned}
$$
It follows by the Darboux transformation that
\beq
\label{eq:u_one_zero_eq}
\bal  u(x,t)
& = u_f(x,t) + i(z_1 - \ovl{z_1})\mathcal{F}({\mathfrak{b}}(x,t)) \\
& = -\frac{k_1}{\sqrt t} - \frac {2\mu_1 (\upsilon_1^x + p_1 )(\ovl{p_2} \ovl{\upsilon_1^x} +1)}
{|\upsilon_1^x + p_1|^2 + |p_2\upsilon_1^x +1|^2}+O( \epsilon |q|^{-\frac 12} t^{-({\frac 12}+\kappa)}), \\
\eal
\eeq
where $\mathcal{F}$ is defined in \eqref{eq:backlund_ratio_ftn}.
This equation holds, in particular, for $x\geq 1/M$.

For $0\leq x \leq M$, $z_0 = O(t^{-1})$ and hence $\hat{\delta}_1 = 1+O(\epsilon^2 t^{-1})$.
By Remark \ref{rmk:symmetry_r_backlund}, $r_f(0) = r(0)=0$.
Since $r_f(z)= \int_0^z r_f'(s) \rd s$, we have $r_f(z_0) = O(\epsilon|q|^{-\frac 12} |z_0|^{\frac 12}) = O(\epsilon|q|^{-\frac 12} t^{-\frac12})$
and hence $k_1, k_2 = O(\epsilon|q|^{-\frac 12} t^{-\frac12})$.
Thus,
$$
u_f(x,t) = O(\epsilon|q|^{-\frac 12} t^{-({\frac 12}+\kappa)}), \ \ p_1, p_2 = O(\epsilon|q|^{-\frac 12} t^{-1}).
$$
and hence it follows by \eqref{eq:u_one_zero_eq} that
$$
\bal u(x,t)
& =  \frac {- 2\mu_1 \upsilon_1^x}{|\upsilon_1^x|^2+1}  + O(\epsilon |q|^{-\frac 12} t^{-({\frac 12}+\kappa)})\\
& = e^{i( \mu_1^2 t/2 + \rho_1 + l_1)} v_{\mu_1}(x) + O(\epsilon |q|^{-\frac 12} t^{-({\frac 12}+\kappa)}), \\
\eal
$$

Now, suppose that $a(z)$ has two zeros $z_1 = i\mu_1$ and $z_2=i\mu_2$.
We have $\beta=q$ from \eqref{eq:symm_scattering}.
Again by \eqref{eq:norming_const_real} and Theorem \ref{thm:evol_scatt_data}, the norming constants $\gamma_j(t)=\gamma(t;z_j)$, $j=1,2$ are given by
\beq
\label{eq:gamma_two_zero}
\gamma_j(t) = \gamma_j e^{iz_j^2 t/2}, \text{ where } \gamma_j = e^{-i\rho_j} \sqrt{\frac{\mu_j-q}{\mu_j+q}}, \text{ for some } \rho_j \in \rb.
\eeq
Set $\zeta_2 = \sqrt t (z_2 - z_0) = \sqrt t (i\mu_2 - z_0)$.
Recalling from the Appendix the action of repeated Darboux transformation, we set
$$
r_f(z) \equiv r(z) \frac{z-\overline{z_1}}{z-z_1}\frac{z-\overline{z_2}}{z-z_2}, \  \text{ and }
c_1(t) = \frac {\gamma_1(t)}{a_1'(z_1)}, \ c_2(t) = \frac {\gamma_2(t)}{a'(z_2)},
$$
where $a_1(z) = \frac{z-z_1}{z-\overline{z_1}} e^{-l(z)}$ and $a(z) = \frac{z-z_2}{z-\overline{z_2}}a_1(z)$.
Denote by $m_1(x,t,z)$ the solution of the normalized RHP with the reflection coefficient $r_1(z) \equiv r(z) \frac{z-\overline{z_2}}{z-z_2}$ and one pair of simple poles at $z=\pm i\mu_1$ whose norming constants are $\gamma_1(t), -\ovl{\gamma_1(t)}$, respectively.
We have
$$
\bal {\mathfrak{b}_1}(x,t)
& \equiv m_f(z_1) e^{iz_1 x \sigma} \bpm 1 \\ \frac{-c_1(t)}{z_1 - \ovl{z_1}} \epm \\
& = \hat\upsilon_1^{-1} \delta(z_1) e^{\frac{\mu_1 x}2} \bigg[\bpm \hat \upsilon_1^x+p_1  \\ p_2 \hat \upsilon_1^x + 1 \epm
+ O( \epsilon|q|^{-\frac 12} t^{-(\frac 12+\kappa)})\bigg]. \\
\eal
$$
where $\hat \upsilon_1$, $\hat \upsilon_1^x$ are given in \eqref{eq:upsilon_j}.
Write $\mf b_1 = ((\mf b_1)_1, (\mf b_1)_2)^T$ and define
$$\bal
& \lambda_1 \equiv \frac{z_2 - \ovl{z_1}}{z_2 - z_1}, \ \
\hat {{\mathfrak{b}}}_1 \equiv \bpm ({\mathfrak{b}}_1)_1 & -(\ovl{{\mathfrak{b}}_1})_2 \\ ({\mathfrak{b}}_1)_2 & (\ovl{{\mathfrak{b}}_1})_1 \epm, \\
& \hat \mu \equiv \bpm z_2-z_1 & 0 \\ 0 & z_2-\ovl{z_1}\epm = (z_2-z_1) \bpm 1 & 0 \\ 0 & \lambda_1 \epm, \\
\eal $$
Let $\psi_1(x,t,z) = m_1(x,t,z) e^{ixz\sigma}$.
It follows from the Darboux transformation (see \eqref{eq:darboux}) that
\beq
\label{eq:mf_b_2}
\bal {\mathfrak{b}_2}
& \equiv \psi_1(z_2) \bpm 1 \\ \frac{-c_2(t)}{z_2 - \ovl{z_2}} \epm \\
& = \hat {{\mathfrak{b}}}_1 \hat\mu \hat {{\mathfrak{b}}}_1^{-1} \big(m_f(z_2)e^{ixz_2\sigma}\big) \hat\mu^{-1}
\bpm 1 \\ \frac{-c_2(t)}{z_2 - \ovl{z_2}} \epm \\
&=  \hat {{\mathfrak{b}}}_1 \hat\mu \hat {{\mathfrak{b}}}_1^{-1} \big(m_f(z_2) \delta(z_2)^{-\sigma_3} \big)
\bpm \hat\upsilon_2^x \\ 1 \epm \frac{e^{\frac{\mu_2 x}2} \delta(z_2) }{\hat\upsilon_2(z_2 - z_1)}. \\
\eal
\eeq
where $\hat \upsilon_2$, $\hat \upsilon_2^x$ are given in \eqref{eq:upsilon_j}.
As $z_2=-iq+O(\epsilon^2 q)$ by Lemma \ref{lem:zeros_a}, we have
$$
\bigg\|\frac {r_f(\cdot)}{\cdot-z_2}\bigg\|_{H^{1,0}} \leq c \epsilon |q|^{-\frac 32}.
$$
Now set $\tau = |q|\sqrt t$ and assume $t \gg q^{-2}$.
Then, $\tau \gg 1$. It follows by \eqref{eq:K_infty} that
$$\phi^{\text{ad} \sigma}(K^\infty)(\zeta_2) \Phi^{-1}(z_2) =
\left\{\bal
& I+O(\epsilon e^{-\mu_2 x}) = I+O(\epsilon e^{-\tau^2}), \text{ if } x\geq \mu_2 t, \\
& I \qquad\qquad\qquad\qquad\qquad\qquad, \text{ otherwise}.\\
\eal\right.
$$
As $|\zeta_2|^{-2} = (t (z_0^2 + \mu_2^2))^{-1} = O(\tau^{-2})$, it follows by \eqref{eq:asymp_para_cylin} again that
\beq
\label{localRHP_approx}
M_{\text{II}}^\infty(\zeta_2)\phi^{-\sigma}(\zeta_2) = \bpm 1&p_3 \\ p_4&1 \epm + O(\epsilon^2 (z_0^2 + q^2)^{-1} t^{-1}),
\eeq
where
\beq
\label{eq:p_3_p_4}
p_3 = -\frac{ik_1}{\zeta_2} = \frac{ik_1}{\sqrt t (z_0 - i\mu_2)}, \ \
p_4 = \frac{ik_2}{\zeta_2} = -\frac{ik_2}{\sqrt t (z_0 - i\mu_2)},
\eeq
Thus, using Lemma \ref{lem:small_r_estimate} again, we see that
$$\bal
& m_f(z_2) \delta(z_2)^{-\sigma_3} \\
& \qquad = M_{\text{II}}^\infty(\zeta_2)\phi^{-\sigma}(\zeta_2) \phi^{\text{ad} \sigma}(K^\infty)(\zeta_2) \Phi^{-1}(z_2)
+ O(\epsilon |q|^{-\frac 32} t^{-(\frac 12 + \kappa)}) \\
& \qquad= \bpm 1&p_3 \\ p_4&1 \epm + O(\epsilon e^{-\tau^2} + \epsilon^2 (z_0^2 + q^2)^{-1} t^{-1} + \epsilon |q|^{-\frac 32} t^{-(\frac 12 + \kappa)})\\
\eal
$$
where $\epsilon e^{-\tau^2}$ is dropped from the error term if $x< \mu_2 t$.
Inserting
$\hat {{\mathfrak{b}}}_1 \hat\mu \hat {{\mathfrak{b}}}_1^{-1} =  \frac {z_2-z_1}{|(\mf b_1)_1|^2 +|(\mf b_1)_2|^2}
\bsm |(\mf b_1)_1|^2 + \lambda_1 |(\mf b_1)_2|^2 & (\mf b_1)_1 \ovl{(\mf b_1)_2}(1-\lambda_1)
\\  \ovl{(\mf b_1)_1} (\mf b_1)_2(1-\lambda_1) & \lambda_1 |(\mf b_1)_1|^2 + |(\mf b_1)_2|^2 \esm$
into \eqref{eq:mf_b_2} and assembling the above results, we obtain
\beq
\label{eq:mf_b_2_2}
\bal
 {\mathfrak{b}_2}
&  = e^{\frac{\mu_2 x}2} \delta(z_2) \hat\upsilon_2  \bigg[\bpm s_1\\s_2 \epm \\
& \qquad + O(\epsilon e^{-\tau^2} + \epsilon^2 (z_0^2 + q^2)^{-1} t^{-1} + \epsilon |q|^{-\frac 32} t^{-(\frac 12 + \kappa)}) \bigg], \\
\eal
\eeq
where $s_1$ and $s_2$ are given, after some calculations, by \eqref{eq:s_j}.
Again by the Darboux transformation, it follows that
$$
\bal
& u(x,t) \\
& = u_f(x,t) + i(z_1 - \ovl{z_1})\mathcal{F}({\mathfrak{b}_1}(x,t))  + i(z_2 - \ovl{z_2})\mathcal{F}({\mathfrak{b}_2}(x,t)),\\
& = -\frac{k_1}{\sqrt t} - \frac {2\mu_1 (\hat\upsilon_1^x + p_1 )(\ovl{p_2} \ovl{\hat\upsilon_1^x} +1)}
{|\hat\upsilon_1^x + p_1|^2 + |p_2\hat\upsilon_1^x +1|^2} -  \frac { 2\mu_2 s_1 \ovl{s_2}} {|s_1|^2 + |s_2|^2}\\
& \qquad\qquad  + O(\epsilon qe^{-\tau^2} + \epsilon^2 q(z_0^2 + q^2)^{-1} t^{-1} + \epsilon |q|^{-\frac 12} t^{-(\frac 12 + \kappa)}). \\
\eal
$$
This equation holds, in particular, for $x\geq 1/M$.

If $0\leq x \leq M$, we have $z_0 = O(t^{-1})$ and $k_1, k_2 = O(\epsilon|q|^{-\frac 12} t^{-\frac 12})$ as above.
Hence $p_1, p_2 = O(\epsilon|q|^{-\frac 12} t^{-1})$ and $p_3, p_4 =O(\epsilon|q|^{-\frac 32} t^{-1})$.
Thus,
$$
\bal  u(x,t)
& = - \frac {2\mu_1 \hat\upsilon_1^x}{|\hat\upsilon_1^x|^2+1} -  \frac { 2\mu_2 s_1 \ovl{s_2}} {|s_1|^2 + |s_2|^2}
 + O(\epsilon |q|^{-\frac 12} t^{-(\frac 12 + \kappa)}), \\
& = e^{i(\mu_1^2 t/2 + \rho_1 + l_1)} \mu_1 \sech (\mu_1 x - \tanh^{-1}(q/\mu_1)) \\
& \quad - \frac {2\mu_2 (\hat\upsilon_2^x - s_0)(1+ \hat \upsilon_1^x \ovl{s_0})}
    {|\hat\upsilon_2^x - s_0|^2  + |1+ \hat \upsilon_1^x \ovl{s_0}|^2}
    + O(\epsilon |q|^{-\frac 12} t^{-(\frac 12 + \kappa)}), \\
\eal
$$
where $s_0$ is given in \eqref{eq:s_0}.
This completes the proof of Theorem \ref{T:main}.
\end{proof}

\section{Proof of Theorem \ref{T:main2}}
\begin{lemma}
\label{lem:real_norming_const}
Let $u(x)$ be a real-valued function on the half line $x\geq0$ and
let $u^e(x)$ be the B\"acklund extension of $u(x)$ with respect to $q$.
Let $\psi^+(x,z)$ be the ZS-AKNS solution associated with $u^e$.
Then,
\beq
\label{eq:psi_+_real}
\psi^+(x,z) = \ovl{\psi^+(x,-z)}, \ \ x\geq0, \ \ z\in \rb.
\eeq

Suppose that the scattering function $a(z)$ of $u^e(x)$ has n simple zeros in $\cb^+$.
If $z=i\mu\in i\rb^+$ is a zero of $a(z)$, the corresponding norming constant $\gamma(i\mu)$ of $z=i\mu$ is real.
For $u(x)=v_{\mu_0} (x) + qw(x)$, $w$ real, $q\ll 1$, in particular, we have
$$
	\gamma(i\mu) = \sqrt{\frac{\mu - \beta}{\mu+\beta}}, \ \ \beta=(-1)^n q.
$$
\end{lemma}

\begin{proof}
Let $Q(x)=\bsm0&u(x) \\-\ovl{u(x)}&0\esm$.
As $u(x), x\geq 0$, is real, we see that
$$
\ovl{\psi^+_x(x,-z)} = (iz \sigma + Q) \ovl{\psi^+(x,-z)}, \ \ x\geq 0, \ \ z\in \rb.
$$
and $\ovl{\psi^+(x,-z)}e^{-ixz\sigma} = \ovl{\psi^+(x,-z)e^{-ix(-z)\sigma}} \to I$ as $x\to\infty$,
which implies \eqref{eq:psi_+_real}.

For the norming constant $\gamma(i\mu)$, as $\mu$ and $u(x)$, $x\geq0$, are real, we have
$$
\ovl{\psi^+_x(x,i\mu)} = (i(i\mu) \sigma + Q) \ovl{\psi^+(x,i\mu)}, \ \ x\geq 0,
$$
and $\ovl{\psi^+(x,i\mu)} e^{-ix(i\mu)\sigma} \to \text{I}$ as $x \to +\infty$,
which implies that $\psi^+(x,i\mu) = \ovl{\psi^+(x,i\mu)}$, $x\geq0$.
Therefore, $\psi^+(0,i\mu)$ is real
and hence the result follows from \eqref{eq:symm_scattering} and \eqref{eq:gamma_z_k}.
\end{proof}

Define
$$
\bal
& E_+^0(z) \equiv \int_z^\infty e^{-is^2/2} \rd s,
\ \ E_+(z) \equiv \int_z^\infty e^{-its^2/2} \rd s = \frac 1{\sqrt t} E_+^0(\sqrt tz),\\
\eal
$$
and
$$
\bal
& E_-^0(z) \equiv \int_{-^\infty}^z e^{-is^2/2} \rd s,
\ \ E_-(z) \equiv \int_{-^\infty}^z e^{-its^2/2} \rd s = \frac 1{\sqrt t} E_-^0(\sqrt tz),\\
\eal
$$
Note that $E_0 \equiv \int_{-^\infty}^\infty e^{-is^2/2} \rd s = \sqrt{2 \pi} e^{-\frac {i\pi}4}$.
\begin{lemma}
\label{lem:g_E_+}
For any $g \in L^2(\rb^+)$, $t\geq1$,
$$
	\int_0^\infty g(z)E_+(z) \rd z = O\Big( \frac {\|g\|_{L^2}}{t^{3/4}} \Big).
$$
\end{lemma}

\begin{proof}
As $g \in L^2(\rb^+)$, it follows that
$$
\bal
& \Bigg| \int_0^\infty g(z)E_+(z) \rd z \Bigg|
=\Bigg| \int_0^{1/\sqrt{t}} + \int_{1/\sqrt{t}}^\infty \  g(z) \frac {E_+^0(\sqrt{t}z) }{\sqrt{t}} \rd z \Bigg| \\
&\qquad \qquad \leq \frac c{\sqrt{t}} \Bigg[ \|g\|_{L^2} \cdot \frac 1{t^{1/4}} +  \|g\|_{L^2} \cdot \Big(\int_{1/\sqrt{t}}^\infty |E_+^0(\sqrt{t}z)|^2 \rd z \Big)^{{\frac 12}} \Bigg]. \\
\eal
$$
Since $E_+^0(z) = \frac 1{iz} e^{-iz^2/2} + O\Big(\frac 1{z^2}\Big)$ as $z \to \infty$, we have
$$
\int_{1/\sqrt{t}}^\infty |E_+^0(\sqrt{t}z)|^2 \rd z = \frac 1{\sqrt{t}} \int_1^\infty |E_+^0(u)|^2 \rd u = O\Big(\frac 1{\sqrt{t}}\Big).
$$
\end{proof}

\begin{lemma}
\label{lem:stationary_phase_estimate}
Let $|q| \ll 1 $, $\tau = |q| \sqrt t$ and $z_0 \in \rb$. Suppose that
$$
\bal
& g(z)=g_1(z) + h(z)g_2(z), \\
& h(z)=\prod_{k=1}^n \frac{i q_k}{z+ iq_k}, \ \ q_k = \pm q, \ \ 1\leq k \leq n, \text{ and }\\
& \|g_i\|_{H^{1,0}(\rb)} \leq c |q|, \ \ i=1, 2. \\
\eal
$$
Then, for $C>1$,
$$\bal
& \int_{-\infty}^\infty g(z) e^{-it(z-z_0)^2/2} \rd z\\
& \qquad = \left\{
\bal
& \frac {g_1(z_0)}{\sqrt t}  E_0 + O\Big( \frac {q}{\sqrt{t}} \big(|q|^{\frac 12} + t^{-\frac 14} + \tau |\log \tau| \big)\Big),
\ \ \tau \leq \frac 12,\\
& O(q^2), \ \ \tau\geq 1/C>0. \\
\eal
\right. \\
\eal
$$
\end{lemma}

\begin{proof}
We first consider $z>z_0$. By integrating by parts and using Lemma \ref{lem:g_E_+}, we see that
$$
\bal
& \int_{z_0}^\infty g_1(z) e^{-it(z-z_0)^2/2} \rd z \\
& \qquad = -g_1(z)E_+(z-z_0) \Big|_{z_0}^\infty + \int_{z_0}^\infty g'_1(z)E_+(z-{z_0}) \rd z \\
& \qquad = \frac {g_1(z_0)}{\sqrt t} E_+^0(0) + O\Big( \frac {q}{t^{3/4}} \Big).\\
\eal
$$
Similarly, we have
$$
\bal
& \int_{z_0}^\infty h(z) g_2(z) e^{-it(z-z_0)^2/2} \rd z \\
& \qquad = -h(z)g_2(z)E_+(z-{z_0}) \Big|_{z_0}^\infty + \int_{z_0}^\infty h(z) g'_2(z) E_+(z-{z_0}) \rd z  \\
& \qquad\qquad + \int_{z_0}^\infty h'(z) g_2(z) E_+(z-{z_0}) \rd z \\
& \qquad= \frac {g_2(z_0)}{\sqrt t} h(z_0) E_+^0(0) + \text{I} + \text{II}. \\
\eal
$$
Since $\|h \cdot g'_2\|_{L^2} \leq c |q|$, $\text{I} = O\big( \frac {q}{t^{3/4}} \big)$ by Lemma \ref{lem:g_E_+}.
Let $\hat{z_0}=z_0/|q|$ and $h_1(u) = \prod_{k=1}^n \frac{i q_k/|q|}{u+ iq_k/|q|}$.
By change of variable $z=|q|u$,
$$
\bal \text{II}
& = \frac 1{\sqrt t} \int_{\hat{z_0}}^\infty h'_1(u) \big(g_2(|q|u)-g_2(0)\big) E_+^0(\tau u - \sqrt t z_0) \rd u \\
& \quad + \frac {g_2(0)}{\sqrt t} \int_{\hat{z_0}}^\infty h'_1(u) \big(E_+^0(\tau u - \sqrt t z_0)  - E_+^0(- \sqrt t z_0)\big)\rd u\\
& \quad + \frac {g_2(0)}{\sqrt t} \big(-h_1(\hat{z_0}) \big) E_+^0(- \sqrt t z_0) \\
& = \text{II}_1 + \text{II}_2 - \frac {g_2(0)}{\sqrt t} h({z_0}) E_+^0(- \sqrt t z_0). \\
\eal
$$
As $\big|g_2(|q|u)-g_2(0)\big| = \big|\int_{0}^{|q|u} g'_2\big| \leq |qu|^{{\frac 12}} \|g'_2\|_{L^2}$,
$$
|\text{II}_1| \leq \frac {|q|^{{\frac 12}} \|g_2'\|_{L^2}} {\sqrt t}  \int_{\hat{z_0}}^\infty |h'_1(u)| |u|^{{\frac 12}} \rd u
= O\Big( \frac {|q|^{3/2}}{\sqrt t}  \Big).
$$

If $\tau \leq \frac 12$, $|E_+^0(\tau u - \sqrt t z_0) - E_+^0(- \sqrt t z_0)|\leq \min\{c,\tau |u|\}$, and hence
$$\bal |\text{II}_2|
& \leq \frac {c|g_2(0)|}{\sqrt t} \Bigg[ \int_{-\infty}^{-1/\tau} |h'_1(u)| +\int_{-1/\tau}^{1/\tau} |h'_1(u)| \tau |u| + \int_{1/\tau}^{\infty} |h'_1(u)|  \Bigg] \\
& \leq \frac {c |g_2(0)|}{\sqrt t} \Bigg[ \tau \int_0^{1/\tau} \frac{u}{u^2+1} +   \int_{1/\tau}^{\infty} \frac 1{u^2}\Bigg] = O\Big(\frac q{\sqrt t} \tau |\log \tau| \Big).\\
\eal
$$
Assembling the above inequalities, we obtain
$$
\bal
& \int_{z_0}^\infty g(z)  e^{-it(z-z_0)^2/2} \rd z \\
& \quad = \frac {g_1(z_0)}{\sqrt t} E_+^0(0) + \frac {h(z_0)}{\sqrt t} \big(g_2(z_0) E_+^0(0) - g_2(0) E_+^0(-\sqrt t z_0)\big) \\
& \quad \qquad + O\Big( \frac {q}{\sqrt{t}} \big(|q|^{\frac 12} + t^{-\frac 14} + \tau |\log \tau| \big)\Big) \\
\eal
$$
Similarly using $E_-^0(z), E_-(z)$, we obtain for $z<z_0$,
$$
\bal
& \int_{-\infty}^{z_0} g(z)  e^{-it(z-z_0)^2/2} \rd z \\
& \quad = \frac {g_1(z_0)}{\sqrt t} E_-^0(0) + \frac {h(z_0)}{\sqrt t} \big(g_2(z_0) E_-^0(0) - g_2(0) E_-^0(-\sqrt t z_0)\big) \\
& \quad \qquad + O\Big( \frac {q}{\sqrt{t}} \big(|q|^{\frac 12} + t^{-\frac 14} + \tau |\log \tau| \big)\Big) \\
\eal
$$
and hence
$$
\bal
& \int_{-\infty}^\infty g(z)  e^{-it(z-z_0)^2/2} \rd z \\
& \ \ = \frac {E_0}{\sqrt t} \big(g_1(z_0) + h(z_0)(g_2(z_0) - g_2(0))\big)
 + O\Big( \frac {q}{\sqrt{t}} \big(|q|^{\frac 12} + t^{-\frac 14} + \tau |\log \tau| \big)\Big) \\
\eal
$$
As
$$
\big|h(z_0) (g_2(z_0)-g_2(0))\big| \leq \frac {c|q|}{|z_0| + |q|} \|g_2'\|_{L^2} |z_0|^{\frac 12} \leq c |q|^{3/2},
$$
the result follows for $\tau \leq \frac 12$.

For $\tau\geq 1/C>0$, note that $\frac 1{\sqrt t} \leq C|q|$. The result follows from the fact that
$$
|\text{II}_2| \leq \frac {c|g_2(0)|}{\sqrt t} \int_{\hat{z_0}}^\infty |h'_1(u)| \rd u = O(q^2),
$$
with the similar estimate for $z<z_0$.
\end{proof}

\begin{lemma}
\label{lem:expansion_rhp}
Let $m(x,t,z)$ solve the normalized RHP $(\mathbb{R}, v )$ without poles where the jump matrix $v$ is given in \eqref{eq:jump_v}.
Let $u(x,t)$ be the associated potential function.
Suppose that $\|r\|_{L^2 \cap L^\infty} \leq c|q| \ll 1$.
Then, for each fixed $z \in \cb^+$,
$$
	m(x,t,z) = I + \frac 1{2\pi i} \int_\rb \bpm 0& \frac{r(s)}{s-z} e^{i\theta(s)}\\ \frac{\overline{r(s)}}{s-z} e^{-i\theta(s)} & 0 \epm \rd s + O\Big(\frac {q^2}{|\textnormal{Im} z|^{{\frac 12}}} \Big),
$$
and
$$
	u(x,t) = \frac 1{2\pi} \int_\rb r(s) e^{i\theta(s)} \rd s +O(q^2).
$$
where $\theta(z) = xz - \frac 12 tz^2$.
\end{lemma}

\begin{proof}
We refer to Section \ref{sec:scat_inv} for the solution procedure for the RHP $(\mathbb{R}, v )$.
Let $ {w_+} = v-I$. Since $\|r\|_{L^2 \cap L^\infty(\rb)} \leq c|q| \ll 1$,
$$
	\|(1-C_{ v})^{-1}\|_{L^2(\rb)} \leq \frac 1{1-\|r\|_\infty} \leq c,
$$
where $C_v$ is defined in \eqref{eq:cauchy_v_op}.
Hence
$$
	 \mu \equiv I + (1-C_{ v})^{-1} C_{ v} I = I + (1-C_{ v})^{-1} C^-  {w_+} \in I+L^2(\rb).
$$
Here
$$
	 {w_+}(z) = \bpm |r(z)|^2& r(z) e^{i\theta(z)}\\ \overline{r(z)} e^{-i\theta(z)} & 0 \epm.
$$
As
$$
	\Bigg| \int_\rb \frac {g_1(s)g_2(s)}{s-z} \rd s \Bigg| \leq \frac {\|g_1\|_{L^2} \|g_2\|_{L^\infty}}{|\textnormal{Im} z|^{{\frac 12}}}
$$
for any $g_1 \in L^2(\rb), g_2 \in L^\infty(\rb)$, it immediately follows that (see \eqref{eq:sing_mu})
$$
\bal m(x,t,z)
& = I + C( \mu {w_+})(z) \\
& = I + C{w_+}(z) + C\big(((1-C_{ v})^{-1} C^- {w_+}){w_+}\big)(z) \\
& = I + \frac 1{2\pi i} \int_\rb \bpm 0& \frac{r(s)}{s-z} e^{i\theta(s)}\\ \frac{\overline{r(s)}}{s-z} e^{-i\theta(s)} & 0 \epm \rd s + O\Big(\frac {q^2}{|\textnormal{Im} z|^{{\frac 12}}} \Big),\\
\eal
$$
and
$$
u(x,t) = \frac 1{2\pi} \int_\rb ( \mu {w_+})_{12} = \frac 1{2\pi} \int_\rb r(s) e^{i\theta(s)} \rd s +O(q^2).
$$
\end{proof}

Let $\psi^+=(\psi_1^+,\psi_2^+)$, $\psi_0^+=((\psi_0)_1^+,(\psi_0)_2^+)$ be the ZS-AKNS solutions for $u_0^e$, $\eta_{\mu_0}$, respectively.
The following expansion is standard in the perturbation theory.
\begin{lemma}
\label{lem:psi_expansion}
For $z\in \rb$, $x\geq 0$,
$$
	\psi^+(x, z) = \psi_0^+(x, z) - q \psi_0^+(x, z) \int_x^\infty L(y, z) \rd y + O(q^2)
$$
uniformly where $L(x,z) = (\psi_0^+)^{-1}W\psi_0^+$ and $W(x) = \bsm 0&w(x)\\-\ovl{w(x)}&0 \esm$.
\end{lemma}

\begin{proof}
Let $\psi^+(x,z) = \psi_0^+ (x,z)\phi(x,z)$. Then,
$$
	\phi_x = (\psi_0^+)^{-1}(\psi^+_x- (\psi_0^+)_x \phi) = q L\phi. 
$$
As $\phi = (\psi_0^+)^{-1}\psi^+ \to I$ as $x\to +\infty$, $z\in\rb$, it follows by standard iterations that
$$
	\phi(x,z) = I-q\int_x^\infty L(y,z) \rd y + O(q^2).
$$
\end{proof}

\begin{proof}[Proof of Theorem \ref{T:main2}]
It is enough to consider $x\geq 0$.
We use $\approx$ to denote equality up to order $\frac {q}{\sqrt{t}} (t^{-\frac 14} + |q|^{\frac 12} + \tau |\log \tau|)$
if $t\leq \frac 12 |q|^{-2}$, and order $q^2$ if $t\geq \frac 1C |q|^{-2}$.
Note that $q^2 \leq  \frac {cq}{\sqrt t} \tau |\log \tau|$ if $t\leq \frac 12 |q|^{-2}$.
By the previous lemma, setting $x=0$ and using \eqref{eq:m_0_1_plus}, it follows that
\beq
\label{eq:expansion_psi_0}
\bpm A_1(z) \\B_1(z) \epm = q \bpm A_0(z)&-\ovl{B_0(z)} \\B_0(z)& \ovl{A_0(z)}  \epm \bpm f_1(z)\\ f_2(z) \epm   + O(q^2), \ \ z\in\rb,
\eeq
where
$$
\bpm f_1(z)\\ f_2(z) \epm
=  \int_0^\infty w(s) \bpm -2iz \mu_0 \sech (\mu_0  s)/(z^2 + \mu_0^2)  \\
 e^{isz} \big(z^2 + 2iz\mu_0 \tanh (\mu_0  s ) - \mu_0^2 \big)/ (z+i\mu_0)^2 \epm  \rd s,
$$
and $A_0, B_0, A_1$ and $B_1$ are defined in \eqref{eq:AB_0_AB_1}.
Let $g_3(z)$ be given in \eqref{eq:scattering_expansion2}.
By \eqref{eq:psi_+_real}, we have $g_3(-z_0) = \ovl{g_3(z_0)}$.
Hence using \eqref{A_0_B_0} and \eqref{eq:expansion_psi_0}, we obtain
\beq
\label{eq:g_3_eval_z_0}
\bal g_3(z_0) + g_3(-z_0)
& = 2\mathrm{Re} \big(A_0(-z_0) B_1 (z_0) + A_1(-z_0) B_0(z_0)\big) + O(q^2) \\
& = -2qK(z_0)+O(q^2) \\
\eal
\eeq
where $K(z)$ is given in \eqref{eq:K_z}.

If $a(z)$ has one zero $z_1 = i\mu_1$, then $\beta=-q$ and hence it follows from \eqref{eq:b} that
\beq
\label{eq:b_z_g_3}
b(z) = g_3(z)+g_3(-z)+\frac {2iq}{z-iq} g_3(z),
\eeq
By Lemma \ref{lem:real_norming_const} and Theorem \ref{thm:evol_scatt_data},
the norming constant $\gamma_1(t) = \gamma(t;z_1)$ for $z=z_1$ is given by
$$
\gamma_1(t) = \gamma_1 e^{-i\mu_1^2 t/2} = \sqrt{\frac {\mu_1+q}{\mu_1 -q}} e^{-i\mu_1^2 t/2}.
$$
Set
$$
r_f(z) \equiv r(z) \frac{z+i\mu_1}{z-i\mu_1}, \text{ and } c_1(t) = \frac {\gamma_1(t)}{a'(z_1)}.
$$
where $a(z) = \frac{z-z_1}{z-\overline{z_1}} e^{-l(z)}$.
By \eqref{eq:a_exp_formula},
$$
r_f(z) = \frac{\ovl{b(z)}}{\ovl{a(z)}}\frac{z+i\mu_1}{z-i\mu_1} = \ovl{b(z)}e^{\ovl{l(z)}}.
$$
where $l(z)$ is defined in \eqref{eq:l_z}.
Let $m_f(x,t,z)$, $u_f(x,t)$ be the solution of the normalized RHP and the associated potential function
in Lemma \ref{lem:expansion_rhp} with $r$ replaced by $r_f$.
Set $h(z) \equiv e^{\ovl{l(z)}} - 1$.
As $\|l\|_{H^{1,0}(\rb)} \leq c \|\log(1+|r|^2)\|_{H^{1,0}(\rb)}$,
it follows that $\|h(z)\|_{H^{1,0}(\rb)} = O(|q|^{3/2})$ and hence
we see by Lemma \ref{lem:stationary_phase_estimate}, Lemma \ref{lem:expansion_rhp}, \eqref{eq:g_3_eval_z_0} and \eqref{eq:b_z_g_3} that
$$
m_f(x,t,z_1) \approx \bpm 1 & p_1 \\ p_2 & 1 \epm, \ \
u_f(x,t) \approx i(z_0 - i\mu_1)p_1,
$$
where
\beq
\label{eq:p_12_small_q}
p_1 = -\frac q{\sqrt t} \frac {K(z_0) E_0 e^{\frac {itz_0^2}2}}{\pi i(z_0 - i\mu_1)}, \ \
p_2 = -\frac q{\sqrt t} \frac {K(z_0) \ovl{E_0} e^{-\frac {itz_0^2}2}}{\pi i(z_0 - i\mu_1)} .
\eeq
From Remark \ref{rmk:symmetry_r_backlund},
\beq
\label{eq:approx_l}
	l(iu) = \frac 1{2\pi i} \int_\rb \frac {iu}{s^2+u^2}{\log \big(1+\big|r(s)\big|^2\big)} \rd s = O(q^2),
\eeq
uniformly on $u >0$.
Let $\upsilon_1$, $\upsilon_1^x$ be given in \eqref{eq:nu_j_sm_t}.
It follows by \eqref{eq:approx_l} that $\frac{-c_1(t)}{z_1 - \ovl{z_1}} = \upsilon_1^{-1} + O(q^2)$.
Set
$${\mathfrak{b}} \equiv m_f(z_1) e^{ixz_1 \sigma} \bpm 1 \\ \frac{-c_1(t)}{z_1 - \ovl{z_1}} \epm.$$
Assembling the above results, we see that
\beq
\label{eq:b_one_zero_}
 \upsilon_1 e^{-\frac{\mu_1 x}2} \mf b
\approx \bpm \upsilon_1^x +p_1 \\ p_2 \upsilon_1^x + 1 \epm.
\eeq
Therefore, by the Darboux transformation,
\beq
\label{eq:one_zero_darboux}
\bal u(x,t)
& = u_f(x,t) + i(z_1 - \ovl{z_1}) \mathcal{F}({\mathfrak{b}}) \\
& \approx i(z_0 - i\mu_1)p_1 - \frac {2\mu_1 (\upsilon_1^x + p_1 )(\ovl{p_2} \ovl{\upsilon_1^x} +1)}
{|\upsilon_1^x + p_1|^2 + |p_2\upsilon_1^x +1|^2}. \\
\eal
\eeq
This relation holds, in particular, in the region $x\geq 1/M$.
If $t\geq \frac 1C |q|^{-2}$,  then $\frac q{\sqrt t} = O(q^2)$ and it follows that $p_1, p_2 = O(q^2)$,
which in turn implies that
$$
u(x,t) \approx  \frac {- 2\mu_1\upsilon_1^x}{|\upsilon_1^x|^2+1} \approx e^{i \mu_1^2 t/2} v_{\mu_1}(x).
$$
This proves \eqref{eq:sol_large_x_one_zero} and the second part of \eqref{eq:sol_small_x_one_zero}.

Now assume $t\leq \frac 12 |q|^{-2}$ and $0\leq x\leq M$.
Since $z_0 = O(t^{-1})$, $K(z_0) = \mathrm{Re} \int_0^\infty e^{-isz_0} w(s) + O(t^{-1})$.
Set $K_1(z) \equiv \int_{-\infty}^\infty e^{-isz} w(s) ds$ and $w_0 \equiv \int_0^\infty w(s) \rd s = \frac 12 K_1(0)$.
As $w\in H^{0,1}(\rb)$, $K_1 \in H^{1,0}(\rb)$ and so $|K_1(z_0) - K_1(0)| \leq |z_0|^{\frac 12} \|K_1'\|_{L^2(\rb)} = O(t^{-\frac 12})$.
As $w$ is real and even, $K(z_0) = \frac 12 \mathrm{Re} K_1(z_0) + O(t^{-1}) = w_0 + O(t^{-\frac 12})$.
Thus,
$$
 p_1 = -\frac{qw_0}{\mu_1} \sqrt{\frac 2{\pi t}} e^{i[x^2/(2t) - \pi/4]}+O(qt^{-1}), \ \ p_2 = \ovl{p_1} + O(q t^{-\frac 32}).
 $$
Combining the results, we obtain from \eqref{eq:one_zero_darboux} for $0\leq x\leq M$,
$$
u(x,t) \approx e^{i \mu_1^2 t/2} \Bigg[ v_{\mu_1}(x) - q w_0  \sqrt{\frac{2}{\pi t}} \Big(e^{i\Omega}\sech^2\mu_1 x - e^{-i\Omega} \tanh^2 \mu_1 x \Big) \Bigg],
$$
where $\Omega(x,t) = -x^2/(2t) + \mu_1^2 t/2 + \pi/4$.
This completes the proof of \eqref{eq:sol_small_x_one_zero}.

If $a(z)$ has two zeros in $\cb^+$, we have $\beta=q$ and hence it follows again from \eqref{eq:b} that
$$ b(z) = g_3(z)+g_3(-z)-\frac {2iq}{z+iq} g_3(-z).$$
Again by Lemma \ref{lem:real_norming_const} and Theorem \ref{thm:evol_scatt_data},
the norming constants $\gamma_j(t) = \gamma(t;z_j)$ for $z=z_j$, $j=1,2$ are given by
$$
\gamma_j(t) = \gamma_j e^{-i\mu_j^2 t/2} = \sqrt{\frac {\mu_j-q}{\mu_j +q}} e^{-i\mu_j^2 t/2}, \ \ j=1,2.
$$
Set
$$
r_f(z) = \frac{z+i\mu_1}{z-i\mu_1}\frac{z+i\mu_2}{z-i\mu_2} \frac{\ovl{b(z)}}{\ovl{a(z)}}, \  \text{ and }
c_1(t) = \frac {\gamma_1(t)}{a_1'(z_1)}, \ c_2(t) = \frac {\gamma_2(t)}{a'(z_2)},
$$
where $a_1(z) = \frac{z-z_1}{z-\overline{z_1}} e^{-l(z)}$ and $a(z) = \frac{z-z_2}{z-\overline{z_2}} a_1(z)$.
Denote by $m_f(x,t,z)$ the solution of the normalized RHP $(\rb, v_f)$ without poles associated with $r_f$,
and by $m_1(x,t,z)$ the solution of the normalized RHP $(\rb, v_1 = \bsm 1+|r_1|^2 & r_1 e^{i\theta} \\ r_1 e^{-i\theta} & 1 \esm )$
with $r_1(z) = \frac{z+i\mu_2}{z-i\mu_2} \frac{\ovl{b(z)}}{\ovl{a(z)}}$
and one pair of poles at $z_1=i\mu_1$ and $\ovl{z_1} = -i\mu_1$ with the norming constants $\gamma_1(t), -\ovl{\gamma_1(t)}$
(cf. \eqref{eq:norming_const_fixed}).
Set
$$
\hat \upsilon_1 \equiv -\sqrt {\frac {\mu_1 +q}{\mu_1 - q}} e^{i\mu_1^2t/2}, \ \
\hat \upsilon_1^x \equiv \hat \upsilon_1 e^{-\mu_1 x},
$$
and
$$
{\mathfrak{b}_1} \equiv m_f(z_1) e^{iz_1 x \sigma} \bpm 1 \\ \frac{-c_1(t)}{z_1 - \ovl{z_1}} \epm.
$$
Then, again by \eqref{eq:approx_l}, we have
$$
 \hat\upsilon_1 e^{-\frac{\mu_1 x}2}{\mathfrak{b}_1} \approx \bpm \hat\upsilon_1^x + p_1 \\ p_2 \hat\upsilon_1^x + 1 \epm . \\
$$
as above.
Write $\mf b_1 = ((\mf b_1)_1, (\mf b_1)_2)^T$.
Define
$$\bal
& \hat {{\mathfrak{b}}}_1 \equiv
\bpm ({\mathfrak{b}}_1)_1 & -(\ovl{{\mathfrak{b}}_1})_2 \\ ({\mathfrak{b}}_1)_2 & (\ovl{{\mathfrak{b}}_1})_1 \epm, \ \
 \lambda_1 \equiv \frac{z_2 - \ovl{z_1}}{z_2 - z_1},\\
& \hat \mu \equiv \bpm z_2-z_1 & 0 \\ 0 & z_2-\ovl{z_1}\epm = (z_2-z_1) \bpm 1 & 0 \\ 0 & \lambda_1 \epm. \\
\eal $$
From Lemma \ref{lem:expansion_rhp} and the fact that $\mu_2 = O(|q|)$ (Lemma \ref{lem:zeros_a}), we obtain
$$
m_f(x,t,z_2) = \bpm 1 & p_3 \\ p_4 & 1 \epm+ O\big(|q|^{\frac 32}\big),
$$
where
\beq
\label{eq:p_34_small_q}
 p_3 \equiv \frac 1{2\pi i} \int_\rb \frac{r_f(s)}{s-i\mu_2} e^{i\theta(s)}, \ \
p_4 \equiv \frac 1{2\pi i} \int_\rb \frac{\ovl{r_f(s)}}{s-i\mu_2} e^{-i\theta(s)}.
\eeq
Define $\hat \upsilon_2 \equiv -\sqrt{\frac{\mu_2+q}{\mu_2-q}} e^{i\mu_2^2 t/2}$.
Then, $-\frac{c_2(t)}{z_2 - \ovl{z_2}} = \lambda_1\hat \upsilon_2^{-1}(1+O(q^2))$ by \eqref{eq:approx_l}.
Let $\psi_1(x,t,z) = m_1(x,t,z) e^{ixz\sigma}$.
Note that $\hat \upsilon_2 = O(q)$ by Lemma \ref{lem:zeros_a},
$\lambda_1 = -1+O(q)$ and $\hat\upsilon_1 = - e^{i\mu_2^2 t/2} + O(q)$.
Assembling the above results, we see that
\beq
\label{eq:mf_b_2_small_q}
\bal {\mathfrak{b}_2}
& \equiv m_1(z_2) e^{ixz_2\sigma} \bpm 1 \\ \frac{-c_2(t)}{z_2 - \ovl{z_2}} \epm \\
&=  \hat {{\mathfrak{b}}}_1 \hat\mu \hat {{\mathfrak{b}}}_1^{-1} m_f(z_2) e^{-x\mu_2\sigma} \hat\mu^{-1}
 \bpm 1 \\ \frac{-c_2(t)}{z_2 - \ovl{z_2}} \epm. \\
& = \hat \upsilon_2^{-1} e^{\frac{\mu_2 x}2} \bigg[ \bpm p_3 - s \\ 1+  \ovl{\hat \upsilon_2^x} s \epm  + O(q) \bigg] \\
\eal
\eeq
where $s= \frac{2(p_3-\hat \upsilon_1^x)}{|\hat \upsilon_1^x|^2 + 1}$.
By Lemma \ref{lem:stationary_phase_estimate}, we see that
\beq
\label{eq:eq_3}
	qp_3 = -\frac 1{2\pi} \frac q{\mu_2} \int_\rb \frac{i\mu_2}{s-i\mu_2} \ovl{b(s)}(1+h(s)) e^{i\theta(s)} \rd s \approx 0.
\eeq
Set $s_0 \equiv \frac{-2\hat \upsilon_1^x}{|\hat \upsilon_1^x|^2 + 1}$.
Using \eqref{eq:mf_b_2_small_q} and \eqref{eq:eq_3}, we have
$$ \bal  i(z_2-\ovl{z_2}) \mathcal{F}({\mathfrak{b}_2})
& \approx \frac{ 2q s_0 (1+\hat \upsilon_1^x \ovl{s_0})} {|s_0|^2 + |1+ \hat \upsilon_1^x \ovl{s_0}|^2}
=  \frac{ 4q \hat\upsilon_1^x(|\hat\upsilon_1^x|^2-1)}{(|\hat\upsilon_1^x|^2 + 1)^2}. \\
\eal $$
But,
$$
\bal
& i(z_1 - \ovl{z_1}) \mathcal{F}({\mathfrak{b}}) - i(z_1 - \ovl{z_1}) \mathcal{F}({\mathfrak{b}_1}) \\
& \quad =  - \frac {2\mu_1 (\upsilon_1^x + p_1 )(\ovl{p_2} \ovl{\upsilon_1^x} +1)}
{|\upsilon_1^x + p_1|^2 + |p_2\upsilon_1^x +1|^2}
 + \frac {2\mu_1 (\hat\upsilon_1^x + p_1 )(\ovl{p_2} \ovl{\hat\upsilon_1^x} +1)}
{|\hat\upsilon_1^x + p_1|^2 + |p_2\hat\upsilon_1^x +1|^2} \\
& \quad \approx  \frac{ 4q \hat\upsilon_1^x(|\hat\upsilon_1^x|^2-1)}{(|\hat\upsilon_1^x|^2 + 1)^2}
\approx i(z_2-\ovl{z_2}) \mathcal{F}({\mathfrak{b}_2}). \\
\eal
$$
where $\mf b$ is given in \eqref{eq:b_one_zero_}.
Thus
$$
\bal u(x,t)
& = u_f + i(z_1 - \ovl{z_1}) \mathcal{F}({\mathfrak{b}_1}) + i(z_2-\ovl{z_2}) \mathcal{F}({\mathfrak{b}_2}) \\
& \approx u_f + i(z_1 - \ovl{z_1}) \mathcal{F}({\mathfrak{b}}). \\
\eal
$$
This completes the proof of Theorem \ref{T:main2}.
\end{proof}

\section{Appendix - Adding in or removing poles via a Darboux transformation}
The following calculations are standard in scattering/inverse scattering theory (see e.g. \cite{RS}).
Let $u(x)\in H^{1,1}(\rb)$ be given and consider the associated ZS-AKNS operator
$\partial_x - \bigl(iz\sigma + \bigl(\begin{smallmatrix} 0&u(x) \\ -\overline{u(x)} & 0 \end{smallmatrix}\bigr)\bigr)$
and its reflection coefficient function $r(z)\in H^{1,1}(\rb)$.
Suppose that for each $x\in \rb$, $2 \times 2$ matrix $\psi(x,z) = m(x,z)e^{ixz\sigma}$ solves the corresponding RHP
with a finite number of simple bound states at $z=z_1,  \dots, z_n\in\cb^+$, and at $z=\ovl{z_1},  \dots, \ovl{z_n}\in\cb^-$, $n\geq 0$,
\beq
\left\{
\bal
& \psi(x,z) \text{ is analytic in } z\in\cb\setminus (\rb\cup\{z_1, \ovl{z_1}, \dots, z_n, \ovl{z_n} \}),  \\
& \psi_+(x,z) = \psi_-(x,z) v(z), \ \ v(z) = \bsm 1+|r(z)|^2 & r(z) \\ \ovl{r(z)}& 1 \esm , \ \ z\in \rb,\\
& \psi(x,z)e^{-ixz\sigma} \to I \text{ as } z \to \infty, \\
& \Res\displaylimits_{z = z_k} \psi(x,z)= \lim_{z \to z_k} \psi(x,z) \bpm 0&0\\c(z_k)&0 \epm, 1\leq k \leq n,\\
& \Res\displaylimits_{z = \ovl{z_k}} \psi(x,z)= \lim_{z \to \ovl{z_k}} \psi(x,z) \bpm 0&-\ovl{c(z_k)}\\0&0 \epm. \\
\eal
\right.
\eeq
The goal is to add in another simple bound state at $z=\xi\in \cb^+\setminus \{z_1, \cdots , z_n \}$
and simultaneously at $z=\ovl\xi\in \cb^-\setminus \{\ovl {z_1}, \cdots , \ovl{z_n} \}$.
We use a Darboux transformation $(z+P)(\partial_x - L) = (\partial_x - \tilde L) (z+P)$ as in \eqref{eq:backlund_commute_x}.
By \eqref{eq:sol_P}, $P$ can be chosen in the form $P = {\mathfrak{b}}(x)P_0{\mathfrak{b}}^{-1}(x)$
where $P_0$ is a constant matrix and $\mf b =\mf b(x)$ solves the equation $b' = Qb - i \sigma b P_0$.
In contrast to the choice $P_0 = P(0) \equiv -iq \sigma_3$ in \eqref{eq:backlund},
the appropriate choice here is $P_0 = - \bsm \xi & 0 \\ 0& \ovl{\xi} \esm$; $\mf b$ is determined below.
Set
\beq
\label{eq:darboux}
\tilde\psi(x,z) \equiv {\mathfrak{b}}(x)\mu(z){\mathfrak{b}}^{-1}(x) \psi(x,z) \mu^{-1}(z),
\eeq
where $\mu(z) = z+P_0 = \bsm z-\xi & 0 \\ 0& z-\ovl{\xi} \esm$.
Note that $\tilde \psi(x,z) e^{-ixz\sigma} \to I$ as $z\to\infty$.
Let $\tilde c(\xi)$ be any nonzero constant.
We want to choose ${\mathfrak{b}}(x)$ so that $\tilde\psi$ has a simple pole in the first column at $z=\xi$
and a simple pole in the second column at $z=\ovl{\xi}$ such that for $x\in\rb$,
$$
\bal
& \Res\displaylimits_{z = \xi} \tilde\psi(x,z)= \lim_{z \to \xi} \tilde\psi(x,z) \bpm 0&0\\\tilde c(\xi)&0 \epm,\\
& \Res\displaylimits_{z = \ovl{\xi}} \tilde\psi(x,z)= \lim_{z \to \ovl{\xi}} \tilde\psi(x,z) \bpm 0&-\ovl{\tilde c(\xi)}\\0&0 \epm. \\
\eal
$$
Since
$$
{\mathfrak{b}}^{-1} \tilde\psi = \bpm ({\mathfrak{b}}^{-1}\psi)_{11} & ({\mathfrak{b}}^{-1}\psi)_{12}\frac{z-\xi}{z-\ovl{\xi}} \\ ({\mathfrak{b}}^{-1}\psi)_{21}\frac{z-\ovl{\xi}}{z-\xi} & ({\mathfrak{b}}^{-1}\psi)_{22}  \epm,
$$
we have
$$
\Res\displaylimits_{z = \xi} {\mathfrak{b}}^{-1}(x) \tilde\psi(x,z) =  \bpm 0&0\\({\mathfrak{b}}^{-1}\psi)_{21}(x,\xi)(\xi-\ovl{\xi}) & 0 \epm.
$$
But,
$$
\lim_{z \to \xi} {\mathfrak{b}}^{-1}(x)\tilde\psi(x,z) \bpm 0&0\\\tilde c(\xi)&0 \epm = \bpm 0&0\\\tilde c(\xi)({\mathfrak{b}}^{-1}\psi)_{22}(x,\xi) & 0 \epm.
$$
and hence we must have
$$
(\xi - \ovl{\xi})(e_2, {\mathfrak{b}}^{-1}\psi(x,\xi) e_1) = \tilde c(\xi)(e_2, {\mathfrak{b}}^{-1}\psi(x,\xi) e_2).
$$
Therefore, it follows necessarily that
$$
{\mathfrak{b}}(x)e_1 = c_1(x) \Big( \psi(x,\xi)e_1 - \frac{\tilde c(\xi)}{\xi-\ovl{\xi}} \psi(x,\xi)e_2 \Big)
$$
for some nonzero function $c_1(x)$.
Similarly for $z=\ovl{\xi}$, we see that
$$
{\mathfrak{b}}(x)e_2 = c_2(x) \Big( -\frac{\ovl{\tilde c(\xi)}}{\xi-\ovl{\xi}} \psi(x,\ovl{\xi})e_1+\psi(x,\ovl{\xi})e_2 \Big)
$$
for some nonzero function $c_2(x)$.
Observe that $c_1(x), c_2(x)$ factor out in the formula \eqref{eq:darboux} for $\tilde\psi(x,z)$.
Set
\beq
\label{eq:b_darboux}
{\mathfrak{b}}(x)= \bpm \psi(x,\xi) \bpm 1 \\ \frac{-\tilde c(\xi)}{\xi-\ovl{\xi}}\epm & \psi(x,\ovl{\xi}) \bpm \frac{-\ovl{\tilde c(\xi)}}{\xi-\ovl{\xi}} \\ 1\epm \epm.
\eeq
From the symmetry in Proposition \ref{prop:scattering_symmetry}(i) we see that
$\mf b_2 = \bsm 0&-1\\1&0 \esm \ovl{\mf b_1}$ where $\mf b = (\mf b_1, \mf b_2)$.
Thus, $\det {\mathfrak{b}}(x) =|(\mf b_1)_1(x)|^2 + |(\mf b_1)_2(x)|^2 > 0$ and hence $\mf b(x)$ is invertible for all $x\in \rb$.
The jump matrix $\tilde v$ for $\tilde\psi(x,z)$ is given by
$$
\bal	\tilde v(z)
& = \tilde \psi_-^{-1}(x,z)\tilde \psi_+(x,z) = \mu(z) v(z) \mu^{-1}(z) \\
& = \bpm 1+|\tilde r(z)|^2 & \tilde r(z) \\ \ovl{\tilde r(z)} & 1 \epm, \ \ z\in\rb, \\
\eal
$$
where
$$
\tilde r(z) = r(z) \frac{z-\xi}{z - \ovl{\xi}}.
$$
A straightforward calculation shows that for $1\leq k\leq n$,
$$
\bal
& \Res\displaylimits_{z = z_k} \tilde\psi(x,z)= \lim_{z \to z_k} \tilde\psi(x,z) \bpm 0&0\\\tilde c(z_k)&0 \epm,\\
& \Res\displaylimits_{z = \ovl{z_k}} \tilde\psi(x,z)= \lim_{z \to \ovl{z_k}} \tilde\psi(x,z) \bpm 0&-\ovl{\tilde c(z_k)}\\0&0 \epm. \\
\eal
$$
where
$$
\tilde c(z_k) = c(z_k) \frac{z_k-\ovl\xi}{z_k - \xi}.
$$
The above calculations show that $\tilde m(x,z) = \tilde\psi(x,z)e^{-ixz\sigma}$ solves the RHP of type \eqref{eq:jump_inv_scat}
with $r(z)$, $Z_+$, $K_+$ replaced by $\tilde r(z)$, $\tilde Z_+ =\{z_1, \cdots, z_n, \xi \}$,
$\tilde K_+ = \{\tilde c(z_1), \cdots \tilde c(z_n), \tilde c(\xi) \}$, respectively.
Note that (see \eqref{eq:a_exp_formula})
$$
\tilde a(z) = \frac{z-\xi}{z - \ovl\xi} a(z),
$$
where $a(z)$, $\tilde a(z)$ are the scattering functions for $\psi(x,z)$, $\tilde \psi(x,z)$, respectively.
Hence we see from \eqref{eq:gamma_c_a} that
\beq
\label{eq:norming_const_fixed}
\tilde \gamma (z_k) = \tilde c(z_k) \tilde a'(z_k) = c(z_k) a'(z_k) = \gamma(z_k), \ \ k=1, \cdots, n
\eeq
where $\gamma(z_k)$, $\tilde \gamma (z_k)$ are the corresponding norming constants.
By Remark \ref{rmk:RHP_unique}, $\tilde m$ is unique.
Finally, we compute the corresponding potential $\tilde u(x)$.
From the fact that $m_x = iz[\sigma, m] + Qm$, $Q = \bsm 0&u \\ -\ovl{u}&0 \esm$, we have
$$
Q = -i[\sigma, m_1], \ \ m = I + \frac {m_1}z + o(z^{-1}),
$$
as $z\to\infty$ in any cone $|\textnormal{Im} z|> c|\textnormal{Re} z|$, $c>0$.
Let $\mu_1 = \bsm \xi & 0 \\ 0 & \ovl{\xi} \esm$. For $\tilde m = \tilde\psi e^{-ixz\sigma}$
$$
\bal \tilde m
& = {\mathfrak{b}}\Big(I-\frac {\mu_1}z\Big){\mathfrak{b}}^{-1}\Big( I + \frac {m_1}z + o(z^{-1}))\Big)\Big(I-\frac {\mu_1}z\Big)^{-1} \\
& = I + \frac {m_1-{\mathfrak{b}}\mu_1{\mathfrak{b}}^{-1} +\mu_1}z + o(z^{-1}),\\
\eal
$$
and hence
\beq
\label{eq:darboux_u}
\bal \tilde u(x)
&= -i[\sigma, m_1-{\mathfrak{b}}\mu_1{\mathfrak{b}}^{-1} +\mu_1]_{12} \\
&= u(x) + i(\xi - \ovl{\xi}) \frac{(\mf b_1)_1 \ovl{(\mf b_1)_2}} {|(\mf b_1)_1|^2 +|(\mf b_1)_2|^2}.\\
\eal
\eeq

One can also use Darboux transformations similar to \eqref{eq:darboux} to remove eigenvalues.
We do not provide any further details, except to note that at each step,
if the poles at $z=z_k, \ovl{z_k}$ are removed, then $r(z)\to \tilde r(z) = r(z) \frac{z-\ovl{z_k}}{z-z_k}$, etc.

\section*{Acknowledgements}
The work of the first author was supported in part by NSF Grant DMS-0500923.


\bibliographystyle{alpha}
\bibliography{NLS_delta}

\begin{thebibliography}{AKNS74}

\bibitem[AKNS74]{AKNS}
M.~J. Ablowitz, D.~J. Kaup, A.~C. Newell, and H.~Segur.
\newblock The inverse scattering transform-fourier analysis for nonlinear
  problems.
\newblock {\em Studies in Appl. Math.}, 53(4):249--315, 1974.

\bibitem[AS65]{AS}
M.~Abramowitz and I.~A. Stegun.
\newblock {\em Handbook of Mathematical Functions}.
\newblock Dover Publications, New York, 1965.

\bibitem[BC84]{BC}
R.~Beals and R.~R. Coifman.
\newblock Scattering and inverse scattering for first order systems.
\newblock {\em Comm. Pure Appl. Math.}, 37(1):39--90, 1984.

\bibitem[BDT88]{BDT}
R.~Beals, P.~A. Deift, and C.~Tomei.
\newblock {\em Direct and Inverse Scattering on the Line}.
\newblock Number~28 in Mathematical Surveys and Monographs. American
  Mathematical Society, Providence, RI, 1988.

\bibitem[BK97]{BK}
A.~B\"ottcher and Y.~I. Karlovich.
\newblock {\em Carleson curves, Muckenhoupt weights, and Toeplitz operators},
  volume 154 of {\em Progress in mathematics}.
\newblock Birkh\"auser, Boston, 1997.

\bibitem[BT91]{BT}
R.~F. Bikbaev and V.~O. Tarasov.
\newblock Initial-boundary value problem for the nonlinear {S}chr\"odinger
  equation.
\newblock {\em J. Phys. A: Math. Gen.}, 24(11):2507--2516, 1991.

\bibitem[CG81]{CG}
K.~Clancey and I.~Gohberg.
\newblock {\em Factorization of matrix functions and singular integral
  operators}.
\newblock Birkh\"auser Verlag, Boston, 1981.

\bibitem[CL55]{CL}
E.~A. Coddington and N.~Levinson.
\newblock {\em Theory of Ordinary Differential Equations}.
\newblock McGraw-Hill, New York, 1955.

\bibitem[Dei78]{De}
P.~A. Deift.
\newblock Applications of a commutation formula.
\newblock {\em DukeMath. J.}, 45(2):267--310, 1978.

\bibitem[DIZ93]{DIZ}
P.~A. Deift, A.~Its, and X.~Zhou.
\newblock Long-time asymptotics for integrable nonlinear wave equations.
\newblock {\em Important Developments in Soliton Theory 1980-1990}, pages
  181--204, 1993.
\newblock A.S. Fokas and V.E. Zakharov, eds.

\bibitem[DZ91]{DZ4}
P.~A. Deift and X.~Zhou.
\newblock Direct and inverse scattering on the line with arbitrary
  singularities.
\newblock {\em Comm. Pure Appl. Math.}, 44(5):485--533, 1991.

\bibitem[DZ93]{DZ3}
P.~A. Deift and X.~Zhou.
\newblock A steepest descent method for oscillatory {R}iemann-{H}ilbert
  problems. {A}symptotics for the {MKdV} equation.
\newblock {\em Ann. of Math.}, 137(2):295--368, 1993.

\bibitem[DZ03]{DZ}
P.~A. Deift and X.~Zhou.
\newblock Long-time asymptotics for solutions of the {NLS} equation with
  initial data in weighted {S}obolev spaces.
\newblock {\em Comm. Pure Appl. Math.}, 56:1029--1077, 2003.

\bibitem[FIS05]{FIS}
A.~S. Fokas, A.~R. Its, and L.~Y. Sung.
\newblock The nonlinear {S}chr\"odinger equation on the half-line.
\newblock {\em Nonlinearity}, 18(4):1771--1822, 2005.

\bibitem[Fok89]{Fo}
A.~S. Fokas.
\newblock An initial-boundary value problem for the nonlinear {S}chr\"odinger
  equation.
\newblock {\em Phys. D}, 35:167--185, 1989.

\bibitem[Fok02]{Fo2}
A.~S. Fokas.
\newblock Integrable nonlinear evolution equations on the half-line.
\newblock {\em Comm. Math. Phys.}, 230:1--39, 2002.

\bibitem[FT87]{FT}
L.D. Faddeev and L.A. Takhtajan.
\newblock {\em Hamiltonian methods in the theory of solitons}.
\newblock Springer-Verlag, Berlin ; New York, 1987.

\bibitem[HZ09]{HZ}
J.~Holmer and M.~Zworski.
\newblock Breathing patterns in nonlinear relaxation.
\newblock {\em Nonlinearity}, 22(6):1259--1301, 2009.

\bibitem[Isa93]{Is}
V.~Isakov.
\newblock Carleman type estimates in an anisotropic case and applications.
\newblock {\em J. Diff. Eq.}, 105:217--238, 1993.

\bibitem[Kha91]{Kh}
I.~T. Khabibullin.
\newblock B\"acklund transformation and integrable initial-boundary value
  problems.
\newblock {\em Mat. Zametki}, 49(4):130--137, 1991.

\bibitem[LS87]{LS}
G.~S. Litvinchuk and I.~M. Spitkovsky.
\newblock {\em Factorization of measurable matrix functions}.
\newblock Number~5. Akademie-Verlag, Berlin, 1987.

\bibitem[PS03]{PS}
L.~P. Pitaevskii and S.~Stringari.
\newblock {\em Bose–Einstein Condensation}.
\newblock Clarendon Press, Oxford, 2003.

\bibitem[RS02]{RS}
C.~Rogers and W.~K. Schief.
\newblock {\em B\"acklund and Darboux transformations. Geometry and modern
  applications in soliton theory}.
\newblock Cambridge Texts in Applied Mathematics. Cambridge University Press,
  Cambridge, 2002.

\bibitem[Sha75]{Sh}
A.~B. Shabat.
\newblock The inverse scattering problem for a system of differential equations
  ({R}ussian).
\newblock {\em Funkcional. Anal. i Pril\u{o}zen.}, 9(3):75--78, 1975.

\bibitem[Skl87]{Sk}
E.~K. Sklyanin.
\newblock Boundary conditions for integrable equations.
\newblock {\em Functional Anal. Appl.}, 21(2):164--166, 1987.

\bibitem[Tar88]{Ta}
V.~O. Tarasov.
\newblock A boundary value problem for the nonlinear {S}chr\"odinger equation.
\newblock {\em Zap. Nauchn. Sem.(LOMI)}, 169(3):151--165, 1988.
\newblock translation in J. Soviet Math. 54 (1991), no. 3, 958--967.

\bibitem[Tar91]{Ta2}
V.~O. Tarasov.
\newblock The integrable initial-boundary value problem on a semiline:
  nonlinear schr\"odinger and sine-{G}ordon equations.
\newblock {\em Inverse Problem}, 7(3):435--449, 1991.

\bibitem[Zho89]{Zh1}
X.~Zhou.
\newblock The {R}iemann-{H}ilbert problem and inverse scattering.
\newblock {\em SIAM J. Math. Anal.}, 20(4):966--986, 1989.

\bibitem[Zho98]{Zh2}
X.~Zhou.
\newblock {$L\sp 2$}-{S}obolev space bijectivity of the scattering and inverse
  scattering transforms.
\newblock {\em Comm. Pure Appl. Math.}, 51(7):697--731, 1998.

\bibitem[ZS72]{ZS}
V.~E. Zakharov and A.~B. Shabat.
\newblock Exact theory of two-dimensional self-focusing and one-dimensional
  self-modulation of waves in nonlinear media.
\newblock {\em Soviet Physics JETP}, 34(1):62--69, 1972.
\newblock translated from \u{Z}. \'Eksper. Teoret. Fiz. 61 (1971), no. 1,
  118--134.

\end{thebibliography}

\end{document}